\documentclass[12pt, leqno]{article}
\usepackage{indentfirst}
\usepackage{amstext}
\usepackage{amsthm}
\usepackage{amsopn}
\usepackage{amsfonts}
\usepackage{amsmath}
\usepackage{latexsym}
\usepackage{amscd}
\usepackage{amssymb}
\usepackage{amsmath}

=\oddsidemargin \font\teneufm=eufm10 \font\seveneufm=eufm7
\font\fiveeufm=eufm5
\newfam\eufmfam
\textfont\eufmfam=\teneufm \scriptfont\eufmfam=\seveneufm
\scriptscriptfont\eufmfam=\fiveeufm
\def\frak#1{{\fam\eufmfam\relax#1}}
\let\goth\mathfrak

\def\zz{\frak z}
\def\gg{\goth g}

\def\ee{\frak e}

\def\GG{\frak G}
\def\HH{\frak H}
\def\gg{\goth g}
\def\gh{\goth h}
\def\gt{\goth t}
\def\gG{\goth G}
\def\gP{\goth P}
\def\gL{\goth L}
\def\gH{\goth H}
\def\gY{\goth Y}
\def\gF{\goth F}
\def\gX{\goth X}
\def\gT{\goth T}
\def\gZ{\goth Z}

\def\gg{\goth g}

\def\gz{\goth z}

\def\sl{\goth{s}\goth{l}}

\def\1{\mbox{\bf 1}}

\def\IP{\Bbb P}
\def\cH{\hbox{\v H}}

\def\rad{\mathop{\rm rad}\nolimits}

\def\corad{\mathop{\rm corad}\nolimits}

 \DeclareMathOperator{\Hom}{Hom}
\DeclareMathOperator{\Aut}{Aut} \DeclareMathOperator{\Int}{Int}
\DeclareMathOperator{\Out}{Out}

\DeclareMathOperator{\Ad}{Ad} \DeclareMathOperator{\Spin}{\bf Spin}
\DeclareMathOperator{\SO}{\bf SO} 

\DeclareMathOperator{\PGL}{\bf PGL}
 \DeclareMathOperator{\GL}{\bf GL} 
\DeclareMathOperator{\SL}{\bf SL}

\newtheorem{theorem}{Theorem}[section]

\newtheorem{claim}[theorem]{Claim}

\newtheorem{corollary}[theorem]{Corollary}

\newtheorem{lemma}[theorem]{Lemma}

\newtheorem{proposition}[theorem]{Proposition}

\theoremstyle{definition}
\newtheorem{remark}[theorem]{Remark}
\newtheorem{example}[theorem]{Example}
\newtheorem{examples}[theorem]{Examples}

\newtheorem{definition}[theorem]{Definition}

\numberwithin{equation}{section}

\def\Z{\Bbb Z}
\def\C{\Bbb C}
\def\Q{\Bbb Q}
\def\G{\mathbb G}

\def\C{\mathbb C}
\def\R{\mathbb R}
\def\bA{\text{\rm \bf A}}
\def\bB{\text{\rm \bf B}}
\def\bC{\text{\rm \bf C}}

\def\bD{\text{\rm \bf D}}
\def\bE{\goth E}
\def\bbE{\text{\rm \bf E}}
\def\bF{\text{\rm \bf F}}
\def\bH{\text{\rm \bf H}}
\def\bM{\text{\rm \bf M}}
\def\bG{\text{\rm \bf G}}
\def\bN{\text{\rm \bf N}}
\def\bU{\text{\rm \bf U}}
\def\bW{\text{\rm \bf W}}
\def\Mor{\text{\rm Mor}}

\def\bZ{\text{\rm \bf Z}}
\def\bV{\text{\rm \bf V}}

\def\bO{\text{\rm \bf O}}

\def\pr{\prime}
\def\lra{\longrightarrow}

\def\bX{\text{\rm \bf X}}
\def\CL{{\cal L}}
\def\cL{{\cal L}}

\def\rAut{\text{\rm Aut}}

\def\fet{\text{\it f\'et}}

\def\bGL{\text{\rm \bf GL}}
\def\bPGL{\text{\rm \bf PGL}}
\def\PSO{\text{\rm \bf PSO}}

\def\bT{\text{\rm \bf T}}
\def\bS{\text{\rm \bf S}}
\def\bL{\text{\rm \bf L}}

\def\bQ{\text{\rm \bf Q}}
\def\bP{\text{\rm \bf P}}
\def\bY{\text{\rm \bf Y}}

\def\P{\mathbb P}

\def\be{\mathbf e}
\def\bg{{\bf g}}

\def\bs{{\pmb\sigma}}
\def\bx{{\pmb x}}
\def\by{{\pmb y}}

\def\bmu{{\pmb\mu }}
\def\bnu{{\pmb\nu }}
\def\wh{\widehat}
\def\wt{\tilde}
\def\us{\underset}
\def\os{\overset}
\def\ol{\overline}
\def\id{\text{\rm id}}

\def\et{\text{\rm \'et}}

\def\q{\quad}

\def\2int{\mathop{2\int}\nolimits}

\def\rank{\mathop{\rm rank}\nolimits}
\def\dim{\mathop{\rm dim}\nolimits}

\def\Spec{\mathop{\rm Spec}\nolimits}

\def\Hom{\mathop{\rm Hom}\nolimits}

\def\Stab{\mathop{\rm Stab}\nolimits}

\def\Gal{\mathop{\rm Gal}\nolimits}

\def\Int{\mathop{\rm Int}\nolimits}

\def\Pic{\mathop{\rm Pic}\nolimits}
\def\Br{\mathop{\rm Br}\nolimits}
\def\Aut{\text{\rm{Aut}}}
\def\Out{\text{\rm{Out}}}
\def\bAut{\text{\bf{Aut}}}

\def\bOut{\text{\bf{Out}}}
\def\Int{\mathop{\rm Int}\nolimits}

\def\resp.{\mathop{\rm resp.}\nolimits}
\def\limproj{\mathop{\oalign{lim\cr
\hidewidth$\longleftarrow$\hidewidth\cr}}}
\def\limind{\mathop{\oalign{lim\cr
\hidewidth$\longrightarrow$\hidewidth\cr}}}

\def\Im{\mathop{\rm Im}\nolimits}

\def\lgr{\longrightarrow}

\font\math=cmmi10
\def\varpi{\hbox{\math\char'44}}
\def\simlgr{\buildrel\sim\over\lgr}

\def\pa{\S\kern.15em }

\def\un{\uppercase\expandafter{\romannumeral 1}}
\def\deux{\uppercase\expandafter{\romannumeral 2}}
\def\trois{\uppercase\expandafter{\romannumeral 3}}
\def\quatre{\uppercase\expandafter{\romannumeral 4}}
\def\cinq{\uppercase\expandafter{\romannumeral 5}}
\def\six{\uppercase\expandafter{\romannumeral 6}}

\def\gg{\goth g}
\def\et{\acute et}

\def\hrf{\hbox to .2in{\hrulefill}}
\def\hfl#1#2#3{\smash{\mathop{\hbox to#3{\rightarrowfill}}\limits
^{\scriptstyle#1}_{\scriptstyle#2}}}
\def\gfl#1#2#3{\smash{\mathop{\hbox to#3{\leftarrowfill}}\limits
^{\scriptstyle#1}_{\scriptstyle#2}}}

\def\kalg{k\text{--}alg}

\begin{document}

\title{Torsors, Reductive Group Schemes and Extended Affine Lie Algebras}

\author{Philippe Gille$^{\rm 1}$ and Arturo Pianzola$^{\rm 2,3}$}
\date{}
 \maketitle
$^{\rm 1}${\it UMR 8553 du CNRS, Ecole Normale Sup\'erieure, 45 rue
d'Ulm, 75005 Paris, France. } 

$^{\rm 2}${\it Department of Mathematical Sciences, University of
Alberta, Edmonton, Alberta T6G 2G1, Canada.}

$^{\rm 3}${\it Centro de Altos Estudios en Ciencia Exactas, Avenida de Mayo 866, (1084) Buenos Aires, Argentina.}
 \maketitle 
 \begin{abstract}
 \noindent We give a detailed description of the torsors that correspond to multiloop algebras. These algebras are twisted forms of simple Lie algebras extended over Laurent polynomial rings. They play a crucial role in the construction of Extended Affine Lie Algebras (which are higher nullity analogues of the affine  Kac-Moody Lie algebras). The torsor approach that we take draws heavily for the theory of reductive group schemes developed by M. Demazure and A. Grothendieck. It also allows us to find a bridge between multiloop algebras and the  work of F. Bruhat and J. Tits on reductive groups over complete local fields.\\
 

\end{abstract}
{\small
\tableofcontents }

\vskip14mm

\section{Introduction}

$\,\,\,\,\,\,\,\,\, \,\,\,\,\,\,\,\,\,\,\,\,\,\,\,\,\,\,\,\,\, \,\,\,\,\,\,\,\,\, \,\,\,\,\,\,\,\,\,\,\,\,\,\,\,\,\,\,\,\,\,\,\,\,\,\,\,\,\,\,\,\,\, \,\,\,\,\,\,\,\,\,\,\,\,\,\,\,\,\,\,\,\,\,\,\,\,\,\,\,\,\,\,\,\, \, \text{\it To our good friend Benedictus Margaux}$

\bigskip

Many interesting  infinite
dimensional Lie algebras can be thought as being ``finite
dimensional'' when viewed, not as algebras over the given base
field, but rather as algebras over their centroids. From this point
of view, the algebras in question look like ``twisted forms" of simpler
objects with which one is familiar. The quintessential example of
this type of behaviour is given by the affine Kac-Moody Lie
algebras. Indeed the algebras that we are most interested in, Extended Affine Lie Algebras (or EALAs for short), can roughly  be thought of as higher nullity analogues of the affine Kac-Moody Lie algebras. Once  the twisted form point of view is taken the theory of reductive group schemes developed by Demazure and Grothendieck [SGA3] arises naturally.

Two key concepts which are common to \cite{GP2} and the present work are those  of a {\it twisted form of an algebra,} and of a {\it multiloop algebra}. At this point we briefly recall what these objects are, not only for future reference, but also to help us redact a more comprehensive Introduction. 
\smallskip

\centerline{***}

\smallskip

Unless specific mention to the contrary throughout this paper $k$ will denote a field of characteristic $0, $ and $\overline k$ a fixed algebraic closure of $k.$ We denote $\kalg$ the category of associative unital commutative $k$--algebras, and $R$  object of $\kalg.$
Let $n \geq 0$ and $m >0$ be integers that we assume are fixed in our discussion. Consider the Laurent polynomial rings $R = R_n= k[t^{\pm 1}_1,\dots,t^{\pm
1}_n]$  and  $R'  = R_{n, m} =  k[t^{\pm
\frac{1}{m}}_1,\dots,t^{\pm \frac{1}{m}}_n].$ For convenience we also consider the direct limit $R'_{\infty}  ={\limind} R_{n, m}$ taken over $m$ which in practice will allow us to ``see" all the $R_{n,m}$ at the same time. The natural map $R\to R^\pr$ is not only faithfully flat but also
\'etale. If $k$ is algebraically closed this extension is Galois and plays a crucial role in the study of multiloop algebras. The explicit description of $\Gal(R'/R)$  is given below.

Let $A$ be a $k$--algebra. We are in general interested in understanding forms (for the $fppf$-topology) of the algebra $A \otimes_k R,$ namely algebras $\cL$ over $R$ such that 
\begin{equation}\label{trivialize}
\cL \otimes _R S\simeq A\otimes _k S \simeq (A\otimes _k
R)\otimes _R S
\end{equation}
for some  faithfully flat and finitely presented  extension $S/R.$ The case which is of most interest to us is when $S$ can be taken to be a Galois extension $R'$ of $R$ of Laurent polynomial algebras described above.\footnote{The Isotriviality Theorem of \cite{GP1} and \cite{GP3} shows that this assumption is superflous if $\bAut(A)$ is a an algebraic $k$--group whose connected component is reductive,  for example if $A$ is a finite dimensional simple Lie algebra.} 

Given a form $\cL$
as above for which (\ref{trivialize}) holds, we say that $\cL$ is
{\it trivialized} by $S.$  The $R$--isomorphism classes of such
algebras can be computed by means of cocycles, just as one does in
Galois cohomology:
\begin{equation}\label{correspondence}
 \text{\it Isomorphism classes of}\, S/R\text{\it --forms of}\; A \otimes_k R
\longleftrightarrow H_{fppf}^1\big(S/R,\bAut(A)\big).
\end{equation}
The right hand side is the part ``trivialized by $S$" of the pointed set of non-abelian cohomology  on the flat site of $\Spec (R)$ with coefficients in the sheaf of groups $\bAut (A).$ In the case when $S$ is Galois over $R$ we can indeed identify  $H_{fppf}^1\big(S/R,\bAut(A)\big)$ with the ``usual" Galois cohomology set  $H^1\big(\Gal(S/R),\bAut(A)(S)\big)$ as in [Se].

Assume now that $k$ is algebraically closed and fix a compatible set of primitive 
$m$--th roots of unity $\xi_m ,$ namely such that  $\xi _{me} ^e = \xi_m
$ for all $e > 0.$ We can then identify $\Gal(R'/R)$ with ${(\Z/m\Z )}^n$ where for each
 ${\bf e} = (e_1,\dots ,e_n)\in \Z^n$ the corresponding element $\ol{\bf e} = (\ol e_1,\cdots,\ol e_n) \in \Gal(R'/R)$  acts on $R'$
via $
^{\ol {\bf e}} t^{\frac{1}{m}}_i = \xi  ^{e_i}_{m}
t^{\frac{1}{m}}_i.$
\smallskip

 The primary example of  forms  $\cL$ of $A \otimes_k R$ which are trivialized by a Galois extension $R'/R$ as above are the multiloop algebras based on $A.$ These are defined as follows. Consider an
$n$--tuple $\bs = (\sigma  _1,\dots,\sigma  _n)$ of commuting
elements of ${\rm Aut}_k(A)$ satisfying $\sigma  ^{m}_i = 1.$  For
each $n$--tuple $(i_1,\dots ,i_n)\in \Z^n$ we consider the
simultaneous eigenspace
$
A_{i_1 \dots i_n} =\{x\in A:\sigma  _j(x) = \xi  ^{i_j}_{m} x \,
\, \text{\rm for all} \,\, 1\le j\le n\}.
$
Then $A = \sum A_{i_1 \dots i_n}, $ and $A = \bigoplus A_{i_1 \dots
i_n}$ if we restrict the sum to those $n$--tuples $(i_1,\dots ,i_n)$
for which  $0 \leq i_j < m_j.$

The  multiloop algebra corresponding to $\bs$, commonly denoted by
$L(A,\bs),$ is defined by
$$
L(A,\bs) = \us{(i_1,\dots ,i_n)\in \Z^n}\oplus\,
A_{i_1\dots i_n} \otimes t^{\frac{i_1}{m}} \dots
t^{\frac{i_n}{m}}_n \subset A\otimes _k R' \subset A \otimes_k R'_{\infty}
$$
Note that $L(A,\bs),$ which does not depend on the choice of common period $m,$ is not only a $k$--algebra (in general infinite-dimensional), but also naturally an $R$--algebra. It is when $L(A,\bs)$ is viewed as an $R$--algebra that Galois cohomology and the theory of torsors enter into the picture. Indeed a rather simple calculation shows that 
$$ L(A,\bs) \otimes _R R'\simeq A \otimes _k R' \simeq (A\otimes _k
R)\otimes _R R'.$$
 Thus $L(A,\bs)$ corresponds to a torsor over $\Spec(R)$ under $\bAut(A).$  

When $n =1$ multiloop algebras are called simply loop algebras.  To illustrate our methods, let us look at the case of (twisted) loop
algebras  as
 they appear in the theory of affine Kac-Moody Lie algebras. Here $n =1,$ $k = \C$ and $A = \gg$ is a  finite-dimensional simple Lie algebra. Any such  $\cL$  is
 naturally a Lie algebra over $R:=\C [t^{\pm 1}]$
and
$ \cL \otimes _R S\simeq \gg \otimes _\C S \simeq (\gg\otimes _\C
R)\otimes _R S $
for some (unique)  $\gg,$ and
some finite \'etale extension $S/R.$  In particular,
$\cL$ is an $S/R$--form of the $R$--algebra $\gg\otimes_\C R$, with
respect to the \'etale topology of $\Spec(R)$.  Thus $\cL$
corresponds to a torsor over $\Spec(R)$ under $\bAut(\gg)$ whose
isomorphism class is an element of the pointed set $H^1_{\text{
\'et}}\big(R,\bAut(\gg)\big).$ We may in fact take $S$ to be $R' = \C[t^{\pm
\frac{1}{m}}].$

Assume that $A$ is a finite-dimensional. The crucial point in the
classification of forms of $A\otimes_k R$ by cohomological methods
is  the exact sequence of pointed sets
\begin{equation}\label{exact}
H_{\et}^1 \big(R,\bAut^0(A)\big) \to H_{\et}^1\big(R,\bAut
(A)\big) \os\psi  \longrightarrow H_{\et}^1\big(R,\bOut(A)\big),
\end{equation}
 where $\bOut(A)$ is
the (finite constant) group of connected components of the algebraic $k$--group $\bAut(A).$\footnote{Strictly speaking we should be using the affine $R$--group scheme $\bAut(A \otimes_k R)$ instead of the algebraic $k$--group $\bAut(A).$ This harmless and useful abuse of notation will be used throughout the paper.}

 Grothendieck's theory of the algebraic fundamental group allows us
to identify  $H_ {\et}^1\big(R,\bOut(A)\big)$ with the set of
conjugacy classes of $n$--tuples of commuting elements of the corresponding finite (abstract) group
$\Out(A)$ (again under the assumption that $k$ is algebraically closed). This is an important cohomological invariant attached to any twisted form of $A \otimes_k R.$   
We point out that the cohomological information is always about the twisted forms viewed as algebras over $R$ (and {\it not} $k$). In practice, as the affine Kac-Moody case illustrates, one is interested in understanding  these algebras as objects over $k$  (and {\it not} $R$). A technical tool (the centroid trick) developed and used in \cite{ABP2} and \cite{GP2}  allows us to compare $k$ vs $R$ information.

We begin by looking at the nullity $n = 1 $ case. The map
$\psi$ of (\ref {exact}) is injective [P1]. This fundamental fact follows from a general result about the vanishing of $H^1$ for reductive group schemes over certain Dedekind rings which includes $k[t^{\pm 1}].$ This result can be thought of as an analogue of ``Serre Conjecture I" for some very special  rings of dimension $1$. 
It follows from what has been said that we can attach a
conjugacy class of the finite group $\Out(A)$ that characterizes
$\cL$ up to $R$--isomorphism. In particular, if $\bAut(A)$ is
connected, then all forms (and consequently, all twisted loop
algebras) of $A$ are {\it trivial}, i.e. isomorphic to
$A\otimes_k R$ as $R$--algebras. This yields  the classification of the affine Kac-Moody Lie algebras by purely cohomological methods. One can in fact {\it define} the affine algebras by such methods (which is a completely different approach than the classical definition by generators and relations).

Surprisingly enough the analogue of ``Serre Conjecture II" for $k[t^{\pm 1}_1, t^{\pm1}_2]$ fails, as explained in \cite{GP2}.  The single family of counterexamples known are the the so-called Margaux algebras.  The classification of forms in nullity 2 case is in fact quite interesting and challenging. Unlike the nullity one case there are forms which are not multiloop algebras (the Margaux algebra is one such example). The classification in nullity $2$  by cohomological methods, both over $R$ and over $k,$  will be given in $\S9$ as an application of one of our main results (the Acyclicity Theorem). This classification (over $k$ but not over $R$) can also be attained entirely by EALA methods \cite{ABP3}. The two approaches complement each other and are  the culmination of a project started a decade ago. We also provide  classification results for  loop Azumaya algebras in \S13.

Questions related to the classification and characterization of EALAs in arbitrary nullity are at the heart of our work. In this situation $A = \gg$ is a finite dimensional simple Lie algebra over $k.$ The twisted forms relevant to EALA theory are always multiloop algebras based on $\gg$ \cite {ABFP}.  It is therefore desirable to try to characterize and understand the part of $H_ {\et}^1 \big(R,\bAut(\gg)\big)$ corresponding to multiloop algebras.
We address this problem by introducing the concept of {\it loop} and {\it toral} torsors (with $k$ not necessarily algebraically closed). These concepts are key ideas within our work. It is easy to show using a theorem of Borel and Mostow that  a multiloop algebra based on $\gg,$ viewed as a Lie algebra over $R_n$, always admits a Cartan subalgebra (in the sense of \cite{SGA3}). We establish that the converse also holds.

Central to our work is the study of the canonical map 
\begin{equation} \label{RtoF}
H_ {\et} ^1 \big(R_n,\bAut(\gg)\big) \to H_ {\et} ^1 \big(F_n,\bAut(\gg)\big)
\end{equation}
where $F_n$ stands for the iterated Laurent series field $k((t_1)) \dots ((t_n)).$ The Acyclicity Theorem proved in \S8  shows that the restriction of the canonical map (\ref{RtoF}) to the subset  $H_{loop}^1 \big(R_n,\bAut(\gg)\big) \subset H_ {\text{\rm
\'et}} ^1 \big(F_n,\bAut(\gg)\big)$ of classes of loop torsors is bijective. This has strong  applications to the classification of EALAs. Indeed $H_{\et}^1 \big(F_n,\bAut(\gg)\big)$ can be studied using Tits' methods for algebraic groups over complete local fields. In  particular EALAs can be naturally attached Tits indices and diagrams, combinatorial root data and relative and absolute types. These are important invariants which are extremely useful for classification purposes. Setting any applications aside, and perhaps more importantly,  we believe that  the theory and methods that we are putting forward display an intrinsic beauty, and show just how powerful the methods developed in \cite{SGA3} really are.
\medskip

\noindent{\bf Acknowledgement} The authors would like to thank M. Brion, V. Chernousov and the referee for their valuable comments.

\medskip

\section{Generalities on the algebraic fundamental group,
torsors, and reductive group schemes}

Throughout this section $\gX$ will denote a scheme, and $\gG$ a group scheme over $\gX.$  
\subsection{The fundamental group}\label{subfundamentalgroup}

Assume that $\gX$ is connected and locally noetherian.  Fix a
geometric point $a$ of $\gX$ i.e. a morphism $a:\;\Spec(\Omega) \to \gX$
where $\Omega   $ is an algebraically closed field.

Let $\gX_{\fet}$ be the category of finite \'etale covers of $\gX,$ and
$F$ the covariant functor from $\gX_{\fet}$ to the category of finite
sets given by
$$
F(\gX^\pr) =\{\text{\rm geometric points of}\; \gX^\pr\;\text{\rm
above}\; a\}.
$$
That is, $F(\gX')$ consists of all morphisms $a^\pr : \;\Spec\,(\Omega
) \to \gX^\pr$ for which the diagram
$$
\begin{matrix}
{} &{} &\gX^\pr\\
&\os {a^\pr}\nearrow &\downarrow\\
\Spec\, (\Omega  ) &\us a\rightarrow &\gX
\end{matrix}
$$
commutes. The group of automorphism  of the functor $F$ is called
the {\it algebraic fundamental group of} $\gX$ {\it at} $a,$ and  is
denoted by $\pi_1(\gX,a).$ If $\gX = \Spec(R)$ is affine, then $a$ corresponds to a ring homomorphism $R \to \Omega$  and we will denote the fundamental group by $\pi_1(R, a).$

The functor $F$ is pro-representable: There exists a directed set
$I,$ objects $(\gX_i)_{i\in I} $ of $\gX_{\fet},$ surjective morphisms
$\varphi  _{ij} \in \,\Hom_\gX(\gX_j,\gX_i)$ for $i\le j$ and geometric
points $a_i\in F(\gX_i)$ such that
\begin{equation}\label{FG1}
  a_i =\varphi  _{ij} \circ
a_j
\end{equation}
\begin{equation} \label{FG2}
 \text{\it The canonical map}\;  f:\limind\; \Hom_\gX(\gX_i,\gX^\pr)\to F(\gX^\pr) \,\, \text{\it is bijective,}
\end{equation}
where the map $f$ of (\ref{FG2}) is as follows:  Given
$\varphi :\gX_i\to \gX^\pr$ then $f(\varphi  ) = F(\varphi  )(a_i).$
 The elements of
$\limind\; \Hom_\gX(\gX_i,\gX^\pr)$ appearing in (\ref{FG2}) are by
definition the morphisms  in the category of pro-objects over $\gX$
(see \cite{EGAIV} \S8.13 for details). It is in this sense that
$\limind\;\Hom(\gX_i,-)$ pro-represents $F.$

Since the $\gX_i$ are finite and \'etale over $\gX$ the morphisms
$\varphi _{ij}$ are affine. Thus the inverse limit
$$
\gX^{sc} = \limproj\;\gX_i
$$
exist in the category of schemes over $\gX$   \cite{EGAIV} \S8.2.  For any
scheme $\gX^\pr$ over $\gX$ we thus have a canonical map
\begin{equation}\label{natural}
\Hom_{Pro-\gX}(\gX^{sc},\gX^\pr)  \buildrel {\rm def} \over = \limind\;\Hom_\gX(\gX_i,\gX^\pr) \simeq F(\gX') \to \;\Hom_\gX(\gX^{sc},\gX^\pr)
\end{equation}
obtained by considering the canonical morphisms $\varphi  _i: \gX^{\rm
sc}\to \gX_i.$

\begin{proposition}\label{representability} Assume $\gX$ is noetherian. Then  $F$ is represented
by $\gX^{sc};$ that is, there exists a
 bijection
$$
F(\gX^\pr) \simeq \Hom_\gX(\gX^{sc},\gX^\pr)
$$
which is functorial on the objects $\gX^\pr$ of $\gX_{\fet}.$
\end{proposition}

\begin{proof} Because the $\gX_i$ are affine over $\gX$ and $\gX$ is noetherian,
each $\gX_i$ is noetherian; in particular, quasicompact and
quasiseparated.
 Thus, for $\gX^\pr/\gX$ locally of finite presentation, in particular for $\gX'$ in $\gX_{\fet}$,  the map
(\ref{natural}) is bijective \cite[prop 8.13.1]{EGAIV}. The
Proposition now follows from (\ref{FG2}).
\end{proof}

\begin{remark} The bijection of Proposition \ref{representability}  could
be thought along the same lines as those of (\ref{FG2})  by
considering the ``geometric point'' $a^{ sc}\in \limproj F(\gX_i)$
satisfying $a^{sc}\mapsto a_i$ for all $i\in I.$
\end{remark}

In computing $\gX^{ sc} = \limproj \gX_i$ we may replace $(\gX_i)_{i\in
I}$ by any cofinal family.  This allows us to assume that the $\gX_i$
are (connected) Galois, i.e. the $\gX_i$ are connected and the (left)
action of $\text{\rm Aut}_\gX(\gX_i)$ on $F(\gX_i)$ is transitive. We then
have
$$
F(\gX_i) \simeq \;\Hom_{Pro-\gX}(\gX^{sc},\gX_i) \simeq
\;\Hom_\gX(\gX_i,\gX_i) = \text{\rm Aut}_\gX(\gX_i).
$$
Thus $\pi_1(\gX,a)$ can be identified with the group
$ \limproj \,\text{\rm Aut} _\gX(\gX_i)^{opp}.$  Each $\text{\rm
Aut}_\gX(\gX_i)$ is finite, and this endows $\pi_1(\gX,a)$ with the
structure of a profinite topological group.

The group $\pi_1(\gX,a)$ acts on the right on $\gX^{sc}$ as the
inverse limit of the finite groups $\text{\rm Aut}_\gX(\gX_i).$
Thus, the group $\pi_1(\gX,a)$ acts on the left on each set $F(\gX^\pr)
=\;\Hom_{Pro-\gX}(\gX^{sc},\gX^\pr)$ for all $\gX^\pr\in \gX_{\fet}.$ This
action is continuous since the structure morphism $\gX^\pr\to \gX$
``factors  at the finite level", i.e  there exists a morphism $\gX_i\to
\gX^\pr$ of $\gX$--schemes for some $i \in I.$  If $u:\gX^\pr \to \gX^{\pr\pr}$ is a morphism of
$\gX_{\fet},$ then the map $F(u) :F(\gX^\pr) \to F(\gX^{\pr\pr})$ clearly
commutes with the action of $\pi_1(\gX,a).$ This construction provides
an equivalence between $\gX_{\fet}$ and the category of finite sets equipped with a continuous
$\pi_1(\gX,a)$--action.

The right action of $
\pi_1(\gX,a)$ on $\gX^{sc}$ induces an action of $\pi_1(\gX,a)$ on
$\gG(\gX^{sc}) = \;\Mor_\gX(\gX^{sc},\gG),$ namely
$$
^\gamma  f(z) = f(z^\gamma  )  \q \forall \gamma \in \pi_1(\gX,a), \q f\in
\gG(\gX^{sc}), \q z\in \gX^{sc}.
$$

\begin{proposition}\label{discrete} Assume $\gX$ is noetherian and that 
 $\gG$ is locally of finite
presentation over $\gX.$ Then $\gG(\gX^{sc})$ is a discrete
$\pi_1(\gX,a)$--module and the canonical map
$$
\us\longrightarrow{\text{\rm lim}}\,H^1\big(\text{\rm
Aut}_\gX(\gX_i),\gG(\gX_i)\big) \to H^1\big(\pi_1(\gX,a),\gG(\gX^{sc})\big)
$$
is bijective.
\end{proposition}

\begin{remark}\label{continuous} Here and elsewhere when a profinite group $A$ acts discretely on a module $M$ the corresponding cohomology $ H^1(A,M) $ is the {\it continuous} cohomology as defined in \cite{Se}. Similarly, if a group $H$ acts in both $A$ and $M,$ then $\Hom_H(A,M)$ stands for the continuous group homomorphism of $A$ into $M$ that commute with the action of $H.$
\end{remark}

\begin{proof}
To show that $\gG(\gX^{sc})$ is discrete amounts to showing that
the stabilizer in $\pi_1(\gX,a)$ of a point of $f\in \gG(\gX^{sc})$
is open. But if $\gG$ is locally of finite presentation then
$\gG(\gX^{sc}) = \gG(\limproj\,\gX_i) = \limind\, \gG(\gX_i)$
(\cite{EGAIV} prop. 8.13.1), so we may assume that $f\in \gG(\gX_i)$
for some $i.$ The result is then clear, for the stabilizer we are
after is the inverse image
 under the continuous map $\pi_1(\gX,a) \to \text{\rm
Aut}_\gX(\gX_i)$  of the stabilizer of $f$ in $\text{\rm Aut}_\gX(\gX_i)$
(which is then open  since
  $\text{\rm Aut}_\gX(\gX_i)$ is given the
discrete topology).

By definition
$$
H^1\big(\pi_1(\gX,a),\gG(\gX^{sc})\big) =\limind
\big(\pi_1(\gX,a)/U, \gG(\gX^{sc})^U\big)
$$
where the limit is taken over all open normal subgroups $U$ of
$\pi_1(\gX,a).$ But for each such $U$ we can find $U_i\subset U$ so
that $ U_i = \;\text{\rm ker}\big(\pi_1(\gX,a) \to \,\text{\rm
Aut}_\gX(\gX_i)\big). $ Thus
$$
H^1\big(\pi_1(\gX,a),\gG(\gX^{sc}\big) = \limind\;H^1\big(\text{\rm Aut}_\gX(\gX_i),\gG(\gX_i)\big)
$$
as desired.
\end{proof}

Suppose now that our $\gX$ is a geometrically connected $k$--scheme, where $k$ is of arbitrary characteristic.
We will denote
$\gX \times_k {\ol k}$ by $\ol{\gX}.$  Fix a  geometric point $\ol a :\Spec(\ol k) \to \ol \gX.$ Let $a$ (resp. $b$) be the geometric points
of $\gX$ [resp. $\Spec(k)$]  given by the composite maps
 $a : \Spec(\ol k) \buildrel \ol a \over \to  \ol{\gX} \to \gX$ [resp. $b : \Spec(\ol k)
\buildrel \ol a \over \to   \gX \to \Spec(k)].$ Then by
\cite[th\'eor\`eme IX.6.1]{SGA1} $\pi_1\big(\Spec(k), b\big) \simeq \Gal(k) = \Gal(k_s/k)$
where $k_s$ is the separable closure of
$k$ in $\ol k$, and the sequence
\begin{equation}\label{fundamentalexact}
1 \to \pi_1(\ol{\gX}, \ol a) \to  \pi_1(\gX,  a) \buildrel p \over \lgr
\Gal(k) \to 1
\end{equation}
is exact.

\subsection{Torsors}
 
 Recall that a (right)  {\it torsor over
$\gX$ under $\gG$}  (or simply a $\gG$--torsor if $\gX$ is understood) is a scheme $\bE$ over $\gX$ equipped with a right action
of $\gG$ for which there exists a faithfully flat morphism $\gY \to \gX$,
locally of finite presentation, such that $\bE \times_\gX \gY \simeq \gG
\times_\gX \gY = \gG_{\gY},$ where $\gG_{\gY}$ acts on itself by right translation.

A $\gG$--torsor $\bE$ is {\it locally trivial} (resp. {\it \'etale locally
trivial}) if it admits a trivialization by an open Zariski (resp.
\'etale) covering of $\gX.$  If $\gG$ is affine, flat and locally of finite
presentation over $\gX$, then $\gG$--torsors over $\gX$ are classified by the pointed
set of cohomology $H^1_{fppf}(\gX,\gG)$ defined by means of cocycles
\`a la \v Cech. If $\gG$ is smooth, any $\gG$--torsor is \'etale
locally trivial (cf. \cite{SGA3}, Exp.
{\uppercase\expandafter{\romannumeral 24}}), and their classes are
then measured by $H^1_{\et}(\gX,\gG)$. In what follows the
$fppf$-topology will be our default choice, and we will for
convenience denote $H^1_{fppf}$ simply by $H^1.$
Given a base change $\gY\to \gX$, we denote by  $H^1(\gY/\gX,
\gG)$ the kernel of the base change map
$H^1(\gX,\gG) \to H^1(\gY,\gG_{\gY}).$ As it is customary, and when no confusion is possible,  we will  denote in what follows $H^1(\gY,\gG_{\gY})$ simply by $H^1(\gY,\gG)$

Recall that a torsor $\bE$ over $\gX$ under $\gG$ is called {\it
isotrivial} if it is trivialized by some {\it finite} \'etale
extension of $\gX,$ that is,
$$
[\bE] \in H^1(\gX^\pr/\gX,\gG) \subset H^1(\gX,\gG)
$$
for some $\gX^\pr$ in $ \gX_{\fet}.$ We denote by $H^1_{iso}(\gX,\gG)$ the 
subset of $H^1(\gX,\gG)$ consisting of classes of isotrivial torsors.

\begin{proposition}  
Assume that $\gX$  is  noetherian and that $\gG$ is locally of finite
presentation over $\gX.$  Then
$$
H^1_{iso}(\gX,\gG) = \;\text{\rm ker}\, \big(H^1(\gX,\gG) \to H^1(\gX^{sc},\gG)\big).
$$
\end{proposition}

\begin{proof}
Assume $\bE$ is trivialized by $\gX^\pr\in \gX_{\fet}.$  Since the
connected components of $\gX^\pr$ are also in $\gX_{\fet}$\footnote{There exists a finite Galois connected covering $\gF \to \gX$ such that $\gF \times_\gX \gX^\pr \cong \gF \sqcup \dots \sqcup \gF$ ($r$ times).
If we decompose  $\gX^\pr = \gY_1  \sqcup \dots \sqcup \gY_m$ into its connected components we have
$$
\gX^\pr \times_\gX \gF = \gY_1 \times_\gX \gF  \sqcup \dots \sqcup \gY_m \times_\gX \gF =  \gF \sqcup \dots \sqcup \gF.
$$
It follows that each $\gY_i \times_\gX \gF$ is a disjoint union of copies of $\gF$, 
hence $\gY_i \times_\gX \gF$ is finite \'etale over $\gF$ for $i=1,..,m$. 
Then each  $\gY_i$ is  \'etale over $\gX$ by proposition 17.7.4.vi of \cite{EGAIV}.
 By  descent, each $\gY_i$ is  finite  over $\gX$ \cite[prop. 2.7.1 xv]{EGAIV}, so
 each  $\gY_i/\gX$ is finite \'etale.
} there exists
a morphism $\gX_i\to \gX^\pr$ for some $i\in I.$  But then $\bE\times_\gX
\gX_i = \bE\times_\gX \gX^\pr \times_{\gX^\pr}\gX_i = \gG_{\gX^\pr}\times_{\gX^\pr}
\gX_i = \gG_{\gX_i}$ so that $\bE$ is trivialized by $\gX_i.$  The image of
$[\bE]$ on $H^1(\gX^{sc},\gG)$ is thus trivial.
 
Conversely assume $[\bE]\in H^1(\gX,\gG)$ vanishes under the base change
$\gX^{sc}\to \gX.$  Since the $\gX_i$ are quasicompact and
quasiseparated and $\gG$ is locally of finite presentation, a
theorem of Grothendieck-Margaux \cite{Mg} shows that the canonical
map
$$
\limind \;H^1(\gX_i,\gG) \to H^1(\gX^{sc},\gG)
$$
is bijective. Thus $\bE\times_{\gX} \gX_i\simeq \gG_{\gX_i}$ for some $i\in
I.$
\end{proof}

\subsection{An example: Laurent polynomials in characteristic $0$}\label{laurent} We look
in detail at an example that is of central
importance to this work, namely the case when $\gX=\;\Spec\, (R_n)$
where $ R_n = k[t^{\pm 1}_1, \dots,  t^{\pm 1}_n]$ is the Laurent
polynomial ring in $n$--variables with coefficients on a field $k$ of
characteristic $0.$ 

Fix once and for all a compatible set
$(\xi _m)_{m\ge 0}$ of primitive $m$--roots of unity in $\ol k$
(i.e. $\xi ^\ell_{m\ell} = \xi _m).$ Let $\{k_\lambda \}_{\lambda
\in \Lambda }$ be the set of finite Galois extensions of $k$ which
are included in $\ol k.$  Let $\Gamma _\lambda =\;\Gal\,(k_\lambda
/k)$ and $\Gamma  = \limproj\;\Gamma _\lambda .$ Then $\Gamma  $
coincides with the algebraic fundamental group of $\Spec\,(k)$ at
the geometric point $\Spec\,(\ol k).$

Let $\varepsilon  :R_n \to k$ be the evaluation map at $t_i 
=1.$ The composite map $R_n\os\varepsilon  \to k\hookrightarrow
\ol k$ defines a geometric point $a$ of $\gX$ and a geometric point $\ol a$ of  $\ol{\gX}=\;\Spec\, (\ol{R}_n)$ where $\ol{R}_n = \ol{k}[t^{\pm 1}_1, \dots,  t^{\pm
1}_n].$

Let $I$ be the subset of $\Lambda  \times \Z_{>0}$ consisting of all
pairs $(\lambda  ,m)$ for which $k_\lambda  $ contains $\xi_m.$ Make
$I$ into a directed set by declaring that $(\lambda  ,\ell)\le (\mu
,n) \Longleftrightarrow k_\lambda  \subset k_\mu  $ and $\ell\vert
n.$

Each
$$
R^\lambda  _{n,m} = k_\lambda  [t^{\pm \frac{1}{m}}_1,\dots, t^{\pm \frac{1}{m}}_n]
$$
is a Galois extension of $R_n$ with Galois group $\Gamma _{m,\lambda
} = (\Z/m\Z)^n\rtimes \Gamma  _\lambda  $ as follows:  For ${\bf e} = (e_1,\dots,e_n)\in \Z^n$ we have $^{\ol{\bf e}} t_j^\frac{1}{m} = \xi  ^{e_j}_j
t_j^\frac{1}{m}$ where $^- :\Z^n \to (\Z/mZ)^n$ is the canonical map, and the
group $\Gamma  _\lambda  $ acts naturally on $R^\lambda  _{n,m}$
through its action on $k_\lambda.$ It is immediate from the
definition that for $\gamma \in \Gamma _\lambda $ we have $\gamma
{\bf e}  \gamma ^{-1}:t_j^\frac{1}{m}\mapsto (^\gamma  \xi _j)^{e_j}t_j^\frac{1}{m}.$  Thus
if $^\gamma  \xi _j = \xi ^{f_j}_j$ then $\gamma  {\bf e}\gamma
^{-1} = {\bf e}^\pr$ where ${\bf e}^\pr = (f_1e_1,\dots f_n e_n).$

If $(\lambda  ,\ell)\le (\mu  ,n)$ we have a canonical inclusion
$R^\lambda  _{n,\ell}\subset R^\mu  _{n,m}.$  For $i = (\lambda
,\ell)\in I$ we let $\gX_i =\;\Spec\,(R^\lambda  _{n,\ell}).$  The
above gives morphisms $\varphi_{ij} : \gX_j\to \gX_i$ of $\gX$--schemes
whenever $i \le j.$

We have \cite{GP3}
$$
\begin{aligned}
\gX^{sc} = \limproj\;\gX_i &=\;\Spec(\limind\, \gX_i) \\
&=\;\Spec\,(\ol R_{n,\infty  })\\
\end{aligned}
$$
where $ \ol R_{n,\infty  } = \us m{\limind} \; \ol R_{n,m }$ with $\ol R_{n,m  } = \ol k[t^{\pm
\frac{1}{m}}_1,\dots,t^{\pm \frac{1}{m}}_n].$ Thus
\begin{equation}\label{FGLaurent}
\pi_1(\gX,a) = \limproj\Gamma  _{m,\lambda  } = \wh{\Z}(1)^n \rtimes
\;\Gal \,(k).
\end{equation}

\noindent where $\wh{\Z}(1)$ denotes the abstract group $\wh{\Z} = \limproj_m \,  \bmu_m(\ol k)$ equipped with the natural action of the absolute Galois group $\Gal(k) = \Gal(\ol k/ k).$

\begin{remark}\label{salade} Consider the affine $k$--group scheme
${_\infty \bmu}= \limproj_m \bmu_m$. It corresponds to the Hopf algebra
$$
k[{_\infty \bmu}] = \limind_m \, k[\bmu_m] = \limind_m \, \frac{ k[t]}{t^m - 1}.
$$
Then we have ${_\infty\bmu}(\overline k) \simeq \wh{\Z}$ and
 ${_\infty\bmu}(\overline k)$ is equipped with a (canonical) structure
of profinite $\Gal( k)$--module.
\end{remark}

\begin{remark}\label{preGLn} Let the notation be that of Example \ref{laurent}. Since  $\Z^n$ is the character group of the 
the torus $\GG_{m,k}^{n}$, we have 
an automorphism $\GL_n(\Z) \simeq \bAut_{gr}( \GG_m^{n,k})^{\rm op}$. This defines a  
left action  $\GL_n(\Z)$ on $R_n$ and 
 a right  action of    $\GL_n(\Z)$ on the torus $\GG_{m,k}^{n}$.
Furthermore, by universal nonsense,  this action extends uniquely to the simply connected covering $\ol R_{n,\infty}$ at the geometric point $\ol a$. The extended action on the torus $\Spec(\ol R_{n,m})$ with character group
$(\frac{1}{m}\Z)^n$ is given by the extension of the action of $\GL_n(\Z)$ from $\Z^n$ to $(\frac{1}{m}\Z)^n$ inside $\Q^n$. The group  $\GL_n(\Z)$ acts (on the right) 
on $\pi_1(R_n)$, so we can  consider the semidirect product of groups 
$\GL_n(\Z) \ltimes \pi_1(R_n)$ which  acts then on $\ol R_{n,\infty}$ (see   \S8.4 for details).
\end{remark}

By taking the action of  $\GL_n(\Z)$  on $R_n$ described in Remark \ref{preGLn}  we can twist
the  $R_n$--module $R_n$ by an element $g \in \GL_n(\Z)$. We denote the resulting  
twisted $R_n$--algebra by $R_n^g.$\footnote{The multiplication on the $R_n$-algebras  $R_n^g$ and $R_n$ coincide. 
It is the action of $R_n$  that is different. See \S\ref{subsecsemi} for details.}

\begin{lemma}\label{galois}
Let $S$ be a connected finite  \'etale cover of $R_n.$
Let $L \subset S$ be the integral closure
of $k$ in $S$. Then there exists $g \in \GL_n(\Z)$,
 $a_1,...,a_n \in L^\times$ and positive integers $d_1,...,d_n$ such that
 $d_1 \mid d_2 \cdots \mid d_n$ and  
$$
S \otimes_{R_n} {R^g_n} \simeq_{R_n\text{-alg}} (R_n \otimes_k L)\bigl[ \sqrt[d_1]{a_1t_1},  \cdots , \sqrt[d_n]{a_nt_n} \bigr].
$$
In particular, $S$ is $k$--isomorphic to $R_n \otimes_k L$ and $\Pic(S)=0$.
\end{lemma}

\begin{proof} 
Note that  $R_n^g$ and  $R_n$ have the same units. For convenience 
in what follows we will for simplicity denote   $(R_n)^g$ by $R_n^g,$  $(R_{n,m})^g$ by $R^g_{n,m}$ and $S \otimes_{R_n} {R_n^g}$ 
by $S^g.$

 By Galois theory there exists a
finite Galois extension
$k'/k$ and a positive integer $m$ such that $\mu_m(\ol k) \subset k'$ and $S \simeq_{R_n} ( R_{n,m} \otimes_k k')^H$
where $H$ is a subgroup of ${\Gal( R_{n,m} \otimes_k k'/ R_n)}
=\bmu_m(k')^n \rtimes \Gal(k'/k)$. 
Hence $S$ is geometrically connected and $S$ is a finite \'etale cover of $R_n \otimes_k L$.  
We can assume without loss of generality  that $k=L$.
To say that $L=k$ is to say that 
the map $H \to  \Gal(k'/k)$ is onto. We consider the following commutative
diagram
$$
\begin{CD}
1 @>>> \bmu_m(k')^n \cap H  @>>> H  @>>>  \Gal(k'/k) @>>>  1 \\
&&@VVV @VVV \mid \mid \\
1 @>>> \bmu_m(k')^n @>>> \Gal( R_{n,m} \otimes_k k'/ R_n) @>>>  \Gal(k'/k)
@>>>  1. \\
\end{CD}
$$
Note that the action  $ \Gal(k'/k)$ on  $\bmu_m(k')^n$ normalizes
$\bmu_m(k')^n \cap H$.
Hence $\bmu_m^n(k') \cap H$ is the group of $k'$--points of  a split $k$--group $\bnu$ of multiplicative type.
By considering the corresponding characters groups,
we get a   surjective homomorphism
$(\Z/m\Z)^n = \widehat{  \bmu_m^n }  \to
\widehat {\bnu}$ of finite abelian groups.
An element  $g \in \GL_n(\Z) \subset \Aut_k(R_n)$ transforms the diagram above
to yield
$$
\begin{CD}
1 @>>> \bnu(k')  @>>> H  @>>>  \Gal(k'/k) @>>>  1 \\
&&@VVV @VVV \mid \mid \\
1 @>>> \bmu_m(k')^n @>>> \Gal( R_{n,m} \otimes_k k'/ R_n) @>>>  \Gal(k'/k)
@>>>  1. \\
&& @V{g^*}VV @V{g^*}VV   \mid \mid  \\
1 @>>> \bmu_m(k')^n @>>> \Gal( R_{n,m}^g \otimes_k k'/ R_n^g) @>>>  \Gal(k'/k)
@>>>  1. \\
&& @AAA @AAA  \mid \mid   \\
1 @>>> ^g\bnu(k')  @>>> H^g  @>>>   \Gal(k'/k)  @>>>  1 \\
\end{CD}
$$
The action of $\GL_n(\Z)$
on  $(\Z/m\Z)^n = \widehat{  \bmu_m^n }$ is the 
left action provided by the homomorphism  $\GL_n(\Z) \to \GL_n(\Z/ m\Z)$. 
By elementary facts about generators of finite abelian groups
 there exists $g \in \GL_n(\Z)$  and
positive integers $d_1,...,d_n$ such that $d_1 \mid d_2 \cdots \mid d_n \mid m$
 for which  the following
holds
$$
\begin{CD}
(\Z/m\Z)^n =  \widehat{  \bmu_m^n }  \enskip
@>>> \enskip  \widehat { \bnu}  \\
@V{g^*}V{\simeq}V @V{\simeq}VV \\
(\Z/m\Z)^n =  \widehat{  \bmu_m^n }  \enskip
@>>>  \enskip \Z/ (m/d_1)\Z \oplus \cdots
  \Z / (m/d_n)\Z. \\ 
\end{CD}
$$
This base change leads to the following commutative diagram
$$
\begin{CD}
1 @>>> \bmu_{m/d_1}(k') \times \cdots \times \bmu_{m/d_n}(k') @>>> H^g  @>>>
\Gal(k'/k) @>>>  1 \\
&&@VVV @VVV \mid \mid \\
1 @>>> \bmu_m(k')^n @>>> \Gal( R_{n,m}^g \otimes_k k'/ R_n^g) @>>>  \Gal(k'/k)
@>>>  1. \\
\end{CD}
$$
We claim that  $S^g$ is equipped with an $R_{n}^g$--torsor structure under
$\bmu:=\bmu_{d_1} \times \cdots \bmu_{d_n}$.
The diagram above provides a bijection
$$
\Gal( R_{n,m}^g \otimes_k k'/ R_n^g)/ H^g\simlgr \bmu(k'),
$$
hence a map $\psi :  \Gal( R_{n,m}^g \otimes_k k'/ R_n^g)  \lgr \bmu( k')$
which is a cocycle for the standard action
of  $\Gal( R_{n,m}^g \otimes_k k'/ R_n^g)$ on  $\bmu(k')$
as we now check. We shall use the following two facts 

\smallskip

(I) $\psi$ is right $H^g$--invariant;

\smallskip

(II) the restriction of $\psi$ to  $\bmu_m(k')^n$ is a morphism
of $\Gal(k'/k)$--modules.

\medskip

\noindent We are given $\gamma_1, \gamma_2  \in
 \Gal( R_{n,m}^g \otimes_k k'/ R_n^g)$. Since $^gH$ surjects onto
$\Gal(k'/k)$,  we can write $\gamma_i=   \alpha_i h_i $
with $ \alpha_i \in  \bmu_m(k')^n$ and  $ h_i \in {H^g}$ for
$i=1,2$. We have 

\begin{eqnarray} \nonumber
\psi(\gamma_1 \gamma_2) &= &
\psi( \alpha_1 \,  h_1 \,  \alpha_2 \, h_2) \\ \nonumber
&= &
\psi( \alpha_1 \,  h_1 \,  \alpha_2 \,  h_1^{-1} h_1 h_2 ) \\ \nonumber
&= &
\psi( \alpha_1 \,  h_1 \,  \alpha_2 \,  h_1^{-1}  ) \qquad [by \,\, \text{\rm (I)}\,]  \\ \nonumber
&=&
\psi( \alpha_1 )\,  \psi(  h_1 \, \alpha_2 \, h_1^{-1}) \qquad [by \,\, \text{\rm (II)}\,] \\ \nonumber
&= & \psi(\alpha_1)  \,  h_1  \, \psi(\alpha_2)\,  h_1^{-1} \qquad [by\,\, \text{\rm (I)}\,]   \\ \nonumber
&= & \psi(\gamma_1) \, \gamma_1 . \psi( \gamma_2).
\end{eqnarray}

\noindent Denote by $\tilde S$ the $\bmu$--torsor over $R_n^g$ defined by $\psi$, 
that is
$$
\tilde S := \Bigl\{ \,  x \in {R}_{n,m}^g \otimes_k k' \, \mid \, 
\psi(\gamma) \, . \gamma(x)= x \quad \forall \gamma \in\Gal( R_{n,m}^g \otimes_k k'/ R_n^g)\, 
\Bigr\}.
$$
Since $\psi$ is trivial over $H^g$, we have $\tilde S \subset S^g$.
But $\tilde S$ and $S^g $ are both finite \'etale coverings of
$R_n^g$ of  degree $\mid \bmu(k') \!\mid$, hence $\tilde S = S^g$.
Since $\Pic(R_n^g)=0$, we can use  Kummer theory (see \cite{M} III 4.10), 
namely the isomorphism 
$$
H^1(R_n^g, \bmu) = \prod\limits_{j=1,...,n} H^1(R_n^g, \bmu_{d_j}) \simeq 
 \prod\limits_{j=1,...,n} {R_n^g}^\times/ ({R_n^g}^\times)^{ d_j}. 
$$ 
for determining the structure of $\tilde S.$  Since ${R_n^g}^\times=k^\times \times \Z^n$, there exist 
scalars $a_1,...,a_n \in k^\times$ and monomials $x_1,...,x_n$ in the $t_i$ 
such that the class of $\tilde S/R_n^g$ in 
$H^1(R_n, \bmu)$ is given by $(a_1 \, x_1, \cdots, a_n x_n)$.
In terms of covering, this means that 
$\tilde S= k\bigl[ \sqrt[d_1]{a_1x_1},  \cdots , \sqrt[d_n]{a_nx_n} \bigr]$.
Extending  scalars to  $k'$, we have
$$
\tilde S\otimes_k k' = ( R_{n,m}^g \otimes_k k')^{\bmu(k')}
= k'\bigl[ \sqrt[d_1]{t_1},  \cdots , \sqrt[d_n]{t_n} \bigr].
$$
From this it follows that  $x_i=t_i$ mod $\big((R_n^g \otimes_k k')^\times\big)^{^{ d_i}}$
and  $x_i=t_i$ mod $({R_n^g}^\times)^{^{d_i}}$. We conclude that 
$\tilde S= k\bigl[ \sqrt[d_1]{a_1t_1},  \cdots , \sqrt[d_n]{a_nt_n} \bigr]$.
\end{proof}

\subsection{Reductive group schemes: Irreducibility and isotropy}\label{sec-irr}

The notation that we are going to use throughout the paper deserves some comments. We will tend to use boldface characters, such as $\bG,$ for algebraic groups over $k,$  as also for  group schemes over $\gX$ that are obtained from groups over $k.$ A quintessential example is $\bG_\gX = \bG \times_k \gX.$ For arbitrary group schemes, or more generally group functors,  over $\gX$ we shall tend to use german characters, such as $\gG.$ This duality of notation will be particularly useful when dealing with twisted forms over $\gX$ of groups that come from $k.$ 
 
 The concept of reductive group scheme over  $\gX$ and all related terminology is that of \cite{SGA3}.\footnote{The references to \cite{SGA3} are so prevalent that they will henceforth be given by simply listing the Expos\'e number.   Thus   XII. 1.3, for example, refers to section 1.3 of  Expos\'e XII of \cite{SGA3}. } 
 
 For convenience we now recall and introduce some concepts and notation attached to a reductive group scheme $\gH$  over $\gX.$ We denote by $\rad(\gH)$ (resp. $\corad(\gH)$)  its radical (resp. coradical) torus, 
that is its  maximal central subtorus (resp. its maximal toral quotient) of $\gH$ [XII.1.3].
 
 We say that a  $\gH$ is {\it reducible} if it admits 
 a proper  parabolic subgroup  $\gP$  such that $\gP$ contains a Levi
subgroup $\gL$  (see XXVI).\footnote{The concept of proper parabolic subgroup is not defined in \cite{SGA3}. By  proper  we mean that $\gP_{\overline s}$ is a proper 
 subgroup of $\gG_{\overline x}$ for all geometric points 
$\overline x$ of $\gX$.} The opposite notion is  {\it irreducible.}
If $\gX$ is affine, the notion of reducibility for $\gH$ is equivalent to the existence of a proper parabolic subgroup  $\gP$ (XXVI.2.3), so there is no ambiguity with the terminology of \cite{CGP} and \cite{GP2}.

By extension, if an affine group $\gG$ over $\gX$ acts on $\gH$,
we say that the action is {\it reducible} if 
it normalizes  a couple $(\gP,\gL)$ where $\gP$ is a  proper  parabolic subgroup 
of $\gH$ and $\gL$ a Levi subgroup of $\gP$. The action is otherwise called {\it irreducible.}

We say that $\gH$ over is {\it  isotropic}
if   $\gH$ admits a subgroup isomorphic to $\G_{m,\gX}$. The opposite notion is {\it anisotropic}.

\medskip

If the base scheme $\gX$ is semi-local and connected (resp. normal),
one can show that $\gH$ is anisotropic  if and only if   $\gH$ is irreducible 
and  the torus $\rad(\gH)$  (or equivalently $\corad(\gH)$) is anisotropic (XXVI.2.3, resp. 
\cite{Gi4}).

\medskip

Similarly  we say that the  action of $\gG$ on $\gH$ is  {\it isotropic} if 
it centralizes  a   split subtorus $\gT$ of $\gH$ with the property that all geometric fibers of $\gT$ are non-trivial.
Otherwise the action  is {\it anisotropic.}  One checks that this is the case if and only if the action of $\gG$ on $\gH$ is irreducible 
and the action of $\gG$ on the torus $\rad(\gH)$  (or equivalently to  $\corad(\gH)$)
 is anisotropic.

\section{Loop, finite and toral torsors}

 Throughout this section $k$ will denote a
field of arbitrary characteristic, $\gX$ a geometrically connected noetherian scheme over $k$, and
$\gX^{sc} = \limproj\;\gX_i$ its simply connected cover as
described in \S\ref{subfundamentalgroup}. Let $\bG$ a group scheme over $k$ which is locally of finite presentation. We will maintain the notation of
the previous section, and  assume that $\Omega =\ol k.$ Consider the
fundamental exact sequence (\ref{fundamentalexact}). The geometric
point $a$ corresponds to a point of $\gX(\ol k).$

\subsection{Loop torsors}

Because of (\ref{FG1}), the geometric points $a_i :\;\Spec\,(\ol k)
\to \gX_i$ induce
 a geometric point $a^{sc}:\,\Spec\,(\ol k)
\to \limproj \,\gX_i = \gX^{sc}.$ We thus have a group homomorphism
\begin{equation}
\label{Eq3} \bG( k_s) \to \bG(\ol k) \os{\bG(a^{sc})}\longrightarrow \bG(\gX^{sc}).
\end{equation}

The group $\pi_1(\gX,a)$ acts on $k_s,$ hence on $\bG(k_s),$ via the
group homomorphism $\pi_1(\gX,a)\to\,\Gal \,(k)$ of
(\ref{fundamentalexact}). This action is continuous, and together
with (\ref{Eq3}) yields a map
$$
H^1\big(\pi_1(\gX,a),\bG(k_s)\big) \to
H^1\big(\pi_1(\gX,a),\bG(\gX^{sc})\big),
$$
where we remind the reader that these $H^1$ are defined in the ``continuous" sense (see Remark \ref{continuous}). On the other hand, by Proposition \ref{discrete} and basic
properties of torsors trivialized by Galois extensions, we have
$$
\begin{aligned}
H^1\big(\pi_1(\gX,a),\bG(\gX^{sc})\big)
&= \limind \;H^1\big(\text{\rm Aut}_\gX(\gX_i),\bG(\gX_i)\big)\\
&= \limind \;H^1(\gX_i/\gX,\bG)\subset H^1(\gX,\bG).
\end{aligned}
$$

By means of the foregoing observations we make the following.
\begin{definition} \label{looptorsor} A torsor $\bE$
over $\gX$ under $\bG$ is called a {\it loop torsor} if its
isomorphism class $[\bE]$ in $H^1(\gX,\bG)$ belongs to the image of the
composite map
\begin{equation}\label{looptorsor}
H^1\big(\pi_1(\gX,a),\bG(k_s)\big) \to
H^1\big(\pi_1(\gX,a),\bG(\gX^{sc})\big)\subset H^1(\gX,\bG).
\end{equation} 
\end{definition}
\medskip

We will denote by $H^1_{loop}(\gX, \bG) $ the subset of
$H^1(\gX, \bG)$ consisting of classes of loop torsors. They are given by (continuous) cocycles in the image of the natural map $Z^1\big(\pi_1(\gX,a),\bG(k_s)\big)  \to Z^1(\gX,\bG),$ which  we call {\it loop cocycles.}

\begin{examples}\label{kloopexamples} (a) If $\gX =\,\Spec\,(k)$ then  $H^1_{loop}(\gX, \bG)$ is nothing but the usual Galois
cohomology of $k$
with coefficients in $\bG.$

\vskip.2cm
(b) Assume that  $k$ is separably closed. Then the action of $\pi_1(\gX,a)$ on
$\bG(k_s)$ is trivial, so that
$$
H^1\big(\pi_1(\gX,a),\bG(k_s)\big)
=\;\Hom\big(\pi_1(\gX,a),\bG(k_s)\big)/\text{\rm Int}\, \bG(k_s)
$$
where the group $\text{\rm Int}\;\bG(k_s)$ of inner automorphisms of
$\bG(k_s)$ acts naturally on
\newline $\Hom\big(\pi_1(\gX,a),\bG(k_s)\big).$ To be precise,  $\text{\rm Int}(g)(\phi) : x \to g^{-1}\phi(x)g$ for all 
$g \in \bG(k_s),$ $\phi \in \Hom\big(\pi_1(\gX,a),\bG(k_s)\big)$ and $x \in \pi_1(\gX,a).$
Two particular cases are important:

\vskip.2cm (b1) $\bG$ abelian: In this case
$H^1\big(\pi_1(\gX,a),\bG(k_s)\big)$ is just the group of
continuous homomorphisms from $\pi_1(\gX,a)$ to $\bG(k_s).$

\vskip.2cm (b2) $\pi_1(\gX,a) =\widehat{\Z}^n:$ In this case
$H^1\big(\pi_1(\gX,a),\bG(k_s)\big)$ is the set of conjugacy
classes of $n$--tuples ${\bf \sigma} = (\sigma_1,\dots,\sigma_n)$ of commuting
elements of finite order of $\bG(k_s).$ That the elements are of finite order follows from the continuity assumption. 

\vskip.2cm (c) Let $\gX=\;\Spec\,(k[t^{\pm 1}])$ with $k$
algebraically closed of characteristic $0.$  If $\bG$ is a {\it
connected } linear algebraic group over $k$ then $H^1(\gX,\bG) = 1$
(\cite[prop. 5]{P1}).  We see from  (b2) above that the canonical map 
$$
H^1\big(\pi_1(\gX,a),\bG(k_s)\big) \to H^1(\gX,\bG)
$$
of (\ref{looptorsor}) need not be injective. It need not be surjective either (take $\gX = \IP^1_k$ and $\bG = \bG_{m,k}).$

\vskip.2cm (d) If the canonical map $\bG(k_s) \to \bG(\gX^{sc})$
is bijective, e.g. if $\gX$ is  a geometrically integral
projective variety\ over $k$ (i.e. a geometrically integral
closed subscheme of $\mathbb{P}_k^n$ for some $n,$)  then
$H^1_{loop}(\gX,\bG) = H^1\big(\pi_1(\gX,a),\bG(k_s)\big).$
\end{examples}

\begin{remark}\label{klooptorsion} The notion of loop torsor behaves well  under twisting by a Galois cocycle  
$z \in Z^1\big(\Gal(k), \bG(k_s)\big).$  Indeed   the torsion map $
\tau_z^{-1}: H^1(\gX, \bG) \to
H^1(\gX, {_z\bG})
$
maps loop classes to loop classes. 
\end{remark}

\subsection{Loop reductive groups} \label{seckloopgroup}
Let $\gH$ be a reductive group scheme over $\gX.$  Since $\gX$ is connected, for all $x \in \gX$ the geometric fibers $\gH_{\ol x}$ are reductive group schemes of the same ``type" (see \cite[XXII.2.3]{SGA3}.  By Demazure's  theorem there exists  a unique split reductive group $\bH_0$ over $k$ such that $\gH$ is a twisted form (in the \'etale topology of $\gX$) of $\gH_0 = \bH_0 \times_k \gX.$ We will call $\bH_0$ the Chevalley $k$--form of $\gH.$ The $\gX$--group $\gH$ corresponds to a torsor $\bE$ over $\gX$ under the group scheme $\bAut(\gH_0),$ namely $\bE = {\rm \bf Isom}_{\rm group}(\gH_0, \gH).$  We recall that $\bAut(\gH_0)$ is representable by a smooth and separated group scheme over $\gX$ by XXII 2.3. By definition $\gH$ is then the  contracted product  $\bE \wedge^{\bAut(\gH_0)} \gH_0$ (see \cite{DG} III \S 4 n$^{\rm o}$3 for details.)

We now define one of the central concepts of our work.

\begin{definition} \label{defkloopgroup}
We say that a group scheme $\gH$ over $\gX$ is  {\it loop reductive}
if it is reductive  and if $\bE$  is
a loop torsor.
\end{definition}

We look more closely to the affine case  $\gX=\Spec(R)$.
Concretely, let $\bH_0 = \Spec(k[\bH_0])$ be a split reductive
$k$--group and consider the corresponding $R$--group 
$\gH_0  = \bH_0 \times_k R,$ whose Hopf algebra is $R[\gH_0]  = k[\bH_0] \otimes_k R.$  

 Let $\gH$ be an $R$--group which is a twisted form of $\gH_0$ 
trivialized by the universal covering $R^{sc}.$ Then to a
 trivialization $ \gH_0 \times_R R^{sc}
\cong  \gH \times_R R^{sc}$,  we can attach a cocycle $u \in
Z^1 \big(\pi_1(R,a), \bAut(\gH_0)(R^{sc})\big)$ from which
 $\gH$ can be recovered by Galois descent as we now explain in the form of a Remark for future reference.

\begin{remark}\label{cocycleconvention} 
{\rm There are two possible conventions as to the meaning of the cocycles $u$ of
 $Z^1 \big(\pi_1(R,a), \bAut(\gH_0)(R^{sc})\big).$
 On the one hand $\gH_0$ can be thought of as the affine scheme 
$\Spec(R[\gH_0]),$  $\bAut(\gH_0)(R^{sc})$ as the (abstract) 
group of automorphisms of the $R^{sc}$--group 
$\Spec(R^{sc}[\gH_0])$ where $R^{sc}[\gH_0]  =  R[\gH_0] \otimes_R R^{sc},$ and $\pi_1(X,a)$ as the  opposite 
 group of automorphisms of 
  $\Spec(R^{sc}) /\Spec(R)$ acting naturally 
on $\bAut(\gH_0)(R^{sc})$.

 We will adopt an (anti) equivalent second point of view that
 is much more convenient for our calculations. We view 
$\bAut(\gH_0)(R^{sc})$ as the group of automorphisms of 
the $R^{sc}$--Hopf algebra $R^{sc}[\bH_0] =  R[\gH_0]  \otimes_R R^{sc} \simeq 
 k[\bH_0] \otimes_k R^{sc}$ on which 
the Galois group $\pi_1(R,a)$ acts naturally. 
Then the $R$--Hopf algebra
 $R[\gH]$  corresponding to $\gH$ is given by 
$$R[\gH] = \{ x \in R^{sc}[\bH_0]  : u_{{\gamma}}{{^\gamma}}x = 
a \,\, \text{\rm for all} \, \, {\gamma} \in \pi_1(R,a) \}.$$ 
To say then that $\gH$ is $k$--loop reductive is to say that $u$ can be chosen 
so that  $u_\gamma \in \bAut(\bH_0)(\ol k) \subset \bAut(\bH_0)(R^{sc}) = \bAut(\gH_0)(R^{sc})$ for all $\gamma \in \pi_1(R,a).$}
\end{remark}

\subsection{Loop torsors at a rational base point} \label{basepoint}

  If our geometric point $a$ lies above a $k$--rational
point $x$ of $\gX,$ then $x$ corresponds to a
  section of the structure morphism $\gX \to \Spec(k)$ which maps $b$ to $a.$ This yields
  a group homomorphism $x^* : \Gal(k) \to \pi_1(\gX,a)$ that splits the sequence (\ref{fundamentalexact})
  above. This
splitting defines an  action of $\Gal(k)$ on the profinite group  $\pi_1(\ol \gX , \ol a)$, hence
a semidirect product identification
\begin{equation}\label{splitfundamental}
 \pi_1(\gX,a) \simeq  \pi_1(\ol \gX,\ol a) \rtimes \Gal(k).
\end{equation}
We have seen an example of (\ref{splitfundamental}) in Example \ref{laurent}.

\begin{remark}\label{profinite} {\rm
By the structure of extensions of profinite groups \cite[\S 6.8]{RZ},
it follows that  $\pi_1(\ol{\gX}, \ol a)$ is the projective limit
of a system $\big(H_\alpha  \rtimes \Gal(k)\big)$ where the $H_\alpha$'s are
finite groups.  The Galois action on each $H_\alpha$
defines a twisted finite constant $k$--group $\bnu_\alpha$.
We define
$$
\bnu= \limproj_\alpha \bnu_\alpha.
$$
The $\bnu_\alpha$ are affine $k$--groups such that
$$
\bnu(\ol k )= \limproj_\alpha \, \bnu_\alpha(\ol k  ) \, = \,  \pi_1( \ol{\gX},\ol a).
$$
Note that $\bnu(k_s )=\bnu(\ol k)$. 
In the case when $\gX = \Spec(R_n),$ where as before $R_n = k[t_1^{\pm{1}},\dots, t_n^{\pm{1}}]$ with $k$ of  characteristic zero and $a$ is the geometric point described in Example \ref{laurent},   the above construction yields the affine $k$--group $_\infty \bmu$
defined in Remark \ref{salade}.}
\end{remark}

By means of the decomposition (\ref{splitfundamental}) we can think
of loop torsors as being comprised of a geometric and an arithmetic
part, as we now explain.

Let $\eta \in  Z^1\big(\pi_1(\gX,a), \bG(k_s)\big).$ The restriction
$\eta_{\mid \Gal(k)}$ is called the {\it arithmetic part} of $\eta$
and its denoted by $\eta^{ar}.$
  It is easily seen that $\eta^{ar}$ is in fact a cocycle in $Z^1\big(\Gal(k), \bG(k_s)\big)$. If $\eta$ is fixed in our discussion,
   we will at times denote  the cocycle $\eta^{ar}$ by the more
  traditional notation $z.$ In particular, for $s \in \Gal(k)$
 we write $z_s$ instead   of $\eta^{ar}_s.$ 

   Next we consider the restriction of $\eta$ to
  $\pi_1( \ol{\gX} ,\ol a)$ that we denote by $\eta^{geo}$
and called the {\it geometric part} of $\eta.$

We thus have a  map 
$$
\begin{CD}
\Theta \, \, : \, \,   & Z^1 \big(\pi_1(\gX,a), \bG(k_s)\big) @>>> Z^1\big(
\Gal(k), \bG(k_s)\big) \times
\Hom\big( \pi_1(\ol{\gX} ,\ol a), \bG(k_s)\big) \\ 
&\eta & \mapsto & \bigl( \quad \eta^{ar} \quad , \quad \eta^{geo} \quad \bigr)
\end{CD}
$$ 

The group $\Gal(k)$ acts on $\pi_1(\ol{\gX} ,\ol a)$ by conjugation. On
$\bG(k_s),$ the Galois group $\Gal(k)$ acts on two different ways. 
There is the natural action arising for the action of $\Gal(k)$ on
$k_s$ that as customary we will denote by $^sg,$ and there is also
the twisted action given by the cocycle $\eta^{ar} = z$.
Following Serre we denote this last by $^{s'}g$. Thus $^{s'}g = z_s
{^sg} {z_{s}}^{-1}.$ Following  standard practice to view the abstract group
$\bG(k_s)$ as a $\Gal(k)$--module with the twisted action by $z$ we
write $_z{\bG(k_s)}.$

 For $s \in \Gal(k)$ and $h \in \pi_1( \ol{\gX}, \ol a)$, we have

\begin{eqnarray} \nonumber
\eta^{geo}_{shs^{-1}} & = & \eta_{shs^{-1}}= \eta_{s}{^s(\eta_{hs^{-1}})}
\qquad \hbox{ [$\eta$ is a cocycle]}\\ \nonumber
 & =&  z_s \,\,  ^s(\eta_{hs^{-1})}  \qquad \hbox{ [$\eta_s = \eta^{\rm ar}_s = z_s$]}\\ \nonumber
& =&  z_s \,\,  ^s( \eta^{geo}_h \,  z_{s^{-1}} )\qquad \hbox{ [$\eta$ is a cocycle and $h$ acts trivially on $\bG(k_s)$]} \\ \nonumber
& =&  z_s \, \, \,  {^s \eta^{\rm geo}_h} \, z_s^{-1} \qquad \hbox{ [$1 = z_s \, ^s z_{s^{-1}}$].}
\end{eqnarray}

\noindent This shows that  $\eta^{\rm geo}: \pi_1(\ol{\gX} ,\ol a) \to
{_z\bG(k_s)}$ commutes with the action of $\Gal(k).$ In other words,
$\eta^{geo} \in \Hom_{\Gal(k)}\big( \pi_1( \ol{\gX} ,\ol a), {_z\bG}(k_s)\big)$.

\begin{lemma}\label{Theta} The  map $\Theta$ defines a bijection between
$Z^1\big(\pi_1(\gX,a), \bG(k_s)\big)$ and couples $( z, \eta^{geo})$ with $z
\in  Z^1\big(\pi_1(\gX,a), \bG(k_s)\big)$ and $\eta^{geo} \in \Hom_{\Gal(k)}\big(
\pi_1(\overline{\gX}, \overline{a}), {_z\bG}(k_s)\big)$.
\end{lemma}

\begin{proof} Since a $1$--cocycle is determined by its image on generators,
the map $\Theta$ is injective. For the surjectivity, assume we are
given $z \in  Z^1(\pi_1(\gX,a), \bG(k_s))$ and $\eta^{geo} \in
\Hom_{\Gal(k)}\big( \pi_1({\ol \gX} ,\ol a), {_z\bG}(k_s)\big)$. We define then 
${\eta: \pi_1(\gX,a) \to \bG(k_s)}$ by $\eta_{hs}:= \eta^{geo}_h \, z_s$
This map is continuous, its restriction to $\pi_1( \ol{\gX} ,\ol a)$ 
(resp. $\Gal(k)$ ) is $\eta^{geo}$ (resp. $z$). Finally, since
$\eta$ is a section of the projection map $\bG(k_s) \rtimes
\pi_1(\gX,a) \to \pi_1(\gX,a)$, it is a cocycle.
\end{proof}

We finish this section by recalling some basic properties of the
twisting bijection. Let $\eta \in Z^1\big(\pi_1(\gX,a), \bG(k_s)\big)$ and
consider its corresponding pair $\Theta(\eta) = (z, \eta^{geo}).$ We
can apply the same construction to the twisted $k$--group ${_z\bG}.$
This leads to a map $\Theta_z$ that  attaches to a cocycle
$\eta' \in Z^1\big(\pi_1(\gX,a), {_z\bG}(k_s)\big)$ a pair $(z', {\eta'}^{\rm
geo})$ along the lines explained above. Note that by Lemma
\ref{Theta} the pair $(1, \eta^{\rm geo})$ is in the image of
$\Theta_z.$ More precisely.

\begin{lemma} \label{looptwist} Let $\eta \in Z^1\big(\pi_1(\gX,a), \bG(k_s)\big)$. With the above notation, the inverse of the twisting
map {\rm \cite{Se}}
$$
\tau_z^{-1}:   Z^1\big(\pi_1(\gX,a), \bG(k_s)\big) \simlgr Z^1\big(\pi_1(\gX,a), {
_z\bG}(k_s)\big)
$$
satisfies $\Theta_z \circ \tau_z^{-1}(\eta) = (1, \eta^{\rm geo})$.
\end{lemma}

\begin{remark} \label{direct}{\rm Consider the special case when  the semi-direct
 product  is direct, i.e.
$\pi_1(\gX,a)= \pi_1(\ol{\gX} ,\ol a) \times \Gal(k)$. In other words, the affine $k$--group $\bnu$ defined above is constant  so that  
$$ \eta^{geo}_h = z_s \,\,  {^s\eta^{geo}_h}
\, z_s^{-1}
$$
for all $h \in \pi_1( \ol{\gX} ,\ol a)$ and  $s \in \Gal(k)$. The
torsion map  
$$
\tau_z^{-1}: Z^1\big(\pi_1({\ol \gX} ,\ol a), \bG(k_s)\big) \to
Z^1\big(\pi_1(\ol{\gX},\ol a), {_z\bG}(k_s)\big) 
$$ 
maps the cocycle $\eta$  to the homomorphism  $\eta^{\rm geo}:
\pi_1(\ol{\gX} ,\ol a) \to {_z\bG}(k_s)$.}

\end{remark}

We give now one more reason to call $\eta^{geo}$ the
geometric part  of $\eta$.

\begin{lemma}\label{salade2} 
 Let $\bnu$ be the affine $k$--group scheme defined in
  Remark \ref{profinite}. Then for each linear algebraic $k$--group
  $\bH$, there  is a natural bijection
$$
\Hom_{k-gp}( \bnu, {\bH}) \simlgr
\Hom_{\Gal(k)}\bigl(  \pi_1( {\ol \gX} ,\ol a) ,  {\bH}(k_s) \bigr) .
$$
\end{lemma}

\begin{proof} First we recall that $\Hom_{\Gal(k)}\bigl(  \pi_1( {\ol \gX} ,\ol a) ,  {\bH}(k_s) \bigr)$ stands for the {\it continuous} homomorphisms from  $  \pi_1( {\ol \gX} ,\ol a)$ to $ {\bH}(k_s)$ that commute with the action of $\Gal(k)$.

Write $\bnu= \limproj \bnu_\alpha$ as an inverse limit
of twisted constant  finite $k$--groups.
Since $\bH$ and the $\bnu_\alpha$ are of finite presentation  we have by applying \cite{SGA3} ${\rm VI}_{B}$ 10.4  that
$$
\Hom_{k-gp}( \bnu, {\bH}) = \limind_\alpha
\Hom_{k-gp}( \bnu_\alpha, \bH)
$$
$$=   \limind_\alpha \Hom_{\Gal(k)}\big( \bnu_\alpha( k_s), \bH( k_s)\big)
= \Hom_{\Gal(k)}\big(  \pi_1({\ol \gX} ,\ol a) ,  {\bH}(k_s) \big). 
$$
\end{proof}

This permits to see purely geometric $k$--loop torsors
in terms of homomorphisms of affine  $k$--group schemes.

\subsection{Finite torsors}\label{finitetorsors} 

Throughout this section we assume that $\bG$ is a smooth
 affine $k$--group, and $\gX$ a scheme over $k.$  Let $\bG_\gX = \bG
 \times_k \gX$ be the $\gX$--group obtained from $\bG$ by base change.

 Following our  convention  a torsor over $\gX$ under $\bG$ means under $\bG_\gX,$ and we write $H^1(\gX, \bG)$ instead of $H^1(\gX, \bG_\gX).$

\begin{definition}\label{finitetorsor} A torsor $\bE$ over $\gX$ under $\bG$  is said to be {\it finite} if it admits
a reduction to a finite $k$--subgroup $\bS$ of $\bG$; this is to say, the class of $\bE$ belongs to the image of the natural map $H^1(\gX, \bS) \to H^1(\gX, \bG) $ for some finite subgroup $\bS$ of $\bG$. 
\end{definition}

We denote by
$H^1_{finite}(\gX, \bG)$ the subset of  $H^1(\gX, \bG)$ consisting of
classes of finite torsors, that is
$$
H^1_{finite}(\gX, \bG):= \bigcup\limits_{\bS \subset \bG} {\rm
Im}\big( H^1(\gX, \bS) \to H^1(\gX, \bG) \big) .
$$
where $\bS$ runs over all finite $k$--subgroups of $\bG$. 
\medskip

The case when $k$ is of characteristic $0$  is well known.

\begin{lemma}\label{lem1} Assume that $k$ is of characteristic $0.$
Then  $H^1_{finite}(\gX, \bG) \subset H^1_{loop}(\gX, \bG).$ If
in addition $k$ is algebraically closed, then $H^1_{finite}(\gX, \bG)
= H^1_{loop}(\gX, \bG).$
\end{lemma}

\begin{proof} Let $\bS$ be a finite subgroup of the $k$--group  $\bG.$ Since $k$ is of
characteristic $0$ the group $\bS$ is \'etale.
Thus $\bS$ corresponds to a finite abstract group $S$ together with
a continuous action of $\Gal(k)$ by group automorphisms. More
precisely (see \cite{SGA1} or \cite{K2} pg.184) $S = \bS(\overline{k})$ with the natural action of
$\Gal(k).$ Similarly the \'etale $\gX$-group $\bS_\gX$ corresponds to
$S$ with the action of $\pi_1(\gX, a)$ induced from the homomorphism
$\pi_1(\gX,a) \to \Gal(k).$

By Exp. XI of \cite{SGA1} we have
\begin{equation}
H^1(\gX, \bS) \buildrel {\rm def} \over = H^1(\gX, \bS_\gX) = H^1\big(\pi_1(\gX,a), S\big) = H^1(\pi_1(\gX,a), \bS(\ol{k}))
\end{equation}
which shows that $H^1(\gX, \bS) \subset H^1_{loop}(\gX, \bG).$

If $k$ is algebraically closed any $k$-loop torsor $\bE$ is given by a
continuous group homomorphism $f_{\bE} : \pi_1(\gX,a) \to \bG(k),$ as explained in  Example \ref{kloopexamples}(b). Then the
image of $f_{\bE}$ is a finite subgroup of $\bG(k)$ which gives rise to
a finite (constant) algebraic subgroup $\bS$ of $\bG.$ By
construction $[\bE]$ comes from $H^1(\gX, \bS).$
\end{proof}

\subsection{Toral torsors}\label{toralclasses}

Let $k$, $\bG$ and $\gX$ be as in the previous section.  Given a torsor $\bE$ over $\gX$ under $\bG_\gX$ we can
consider the twisted $\gX$--group $_{\bE}\bG_\gX = {\bE} \wedge^{\bG_\gX} \bG_\gX.$  Since no confusion will
arise we will denote $_{\bE}\bG_\gX$ simply by $_{\bE}\bG.$ We say that
our torsor $\bE$ is {\it toral} if the twisted $\gX$--group $_{\bE}\bG$
admits a  maximal torus (XII.1.3). We denote by $H^1_{
toral}(\gX, \bG) \subset H^1(\gX, \bG)$ the set of classes of toral torsors.

We recall the following useful result.

\begin{lemma}\label{toral}

\begin{enumerate}

\item  Let $\bT$ be a maximal torus of
$\bG$.\footnote{We remind the reader that we are abiding by \cite{SGA3} conventions and terminology. In the expression ``maximal torus of $\bG$ "  we view $\bG$ as a $k$--group, namely a group scheme over $\Spec(k).$ In particular $\bT$ is a $k$-group...}  Then
$$
H^1_{toral}(\gX,\bG) = {\rm Im}\Big(
H^1\big(\gX,\bN_{\bG}(\bT)\big) \to H^1(\gX, \bG)\Big).
$$

\item Let $1 \to \bS \to \bG' \buildrel p \over \to \bG \to 1$ be a central extension of
$\bG$ by a $k$--group $\bS$ of multiplicative type. Then the  diagram
$$
\begin{CD}
H^1_{toral}(\gX, \bG') & \enskip \subset \enskip  & H^1(\gX, \bG') \\
@V{p_*}VV @V{p_*}VV \\
H^1_{toral}(\gX, \bG) & \enskip  \subset \enskip  & H^1(\gX, \bG) \\
\end{CD}
$$
 is cartesian.

\end{enumerate}

\end{lemma}

\begin{proof} (1) is established in \cite[3.1]{CGR2}.

\smallskip

\noindent (2) Consider first the case when $\bS$ is the reductive center of $\bG'$,
We are given an $\gX$--torsor $\bE'$ under $\bG'$
and consider the surjective morphism of $\gX$--group schemes ${_{\bE'}\bG'} \to {_{\bE'}\bG}$
whose  kernel is $\bS \times_k \gX$.
By   XII 4.7 there is a natural one-to-one correspondence between
maximal tori of the $\gX$--groups  ${_{\bE'}\bG'}$ and  ${_{\bE'}\bG}$.
Hence $\bE'$ is a toral $\bG'$--torsor if and only if $\bE' \wedge^{\bG'_\gX} \bG_\gX$ is
a toral $\bG$--torsor.
The  general case follows form the fact that $\bG'/\bZ' \simeq \bG/\bZ$  where
$\bZ'$ (resp. $\bZ$) is the reductive center of $\bG$.
\end{proof}

In an important case the property of a torsor being toral is of
infinitesimal nature.

\begin{lemma}Assume that
$\bG$ is semisimple of adjoint type. For a $\gX$--torsor $\bE$ under $\gG$  the following conditions are equivalent:

\begin{enumerate}

\item $\bE$ is  toral.
\item   The Lie algebra ${\mathcal  Lie}(_{\bE}\bG)$  admits a Cartan subalgebra.
\end{enumerate} 
\end{lemma}

\begin{proof}
 By  XIV th\'eor\`emes 3.9 and 3.18 there exists a natural  one-to-one correspondence between
 the maximal tori of $_{\bE}\bG$ and Cartan subalgebras of ${\mathcal Lie}(_{\bE}\bG).$
\end{proof}

\noindent Recall the following result \cite{CGR2}.

\begin{theorem} \label{bogo}
Let $R$ be a commutative ring and $\bG$ a smooth affine
group scheme over $R$ whose connected component of the identity $\bG^0$
is reductive.  Assume further that one of the following holds:
\medskip

(a) $R$ is an algebraically closed field, or

\medskip

(b) $R=\Z$, $\bG^0$ is a  Chevalley group, and
the order of the Weyl group of the geometric fiber
$\bG_{\overline{s}}$ is independent of $s \in \Spec(\Z)$, or

\medskip

(c) $R$ is a semilocal ring, $\bG$ is connected, and
the radical torus ${\rm rad}(\bG)$ is isotrivial.

\medskip

\noindent
Then there exist a maximal torus $\bT $ of  $\bG,$ and
a finite $R$--subgroup $\bS \subset \bN_{\bG}(\bT)$,
such that

\smallskip
\begin{enumerate}
\item $\bS$ is an extension of a finite
 twisted constant $R$--group
 by a finite $R$--group of multiplicative type,

\smallskip 

\item the natural map $ H^1_{fppf}(\gX, \bS) \lra H^1_{fppf}\big(\gX, \bN_{\bG}(\bT)\big) $
is surjective for any  $R$--scheme $\gX$ satisfying the  condition:
\begin{equation} \label{e.C}
\text{$\Pic(\gX') = 0$ for every generalized Galois cover $\gX'/\gX$,}
\end{equation} where by a generalized
 Galois cover $\gX' \to \gX$ we understand a $\Gamma$--torsor
for some twisted finite constant $\gX$--group scheme $\Gamma.$ \qed

\end{enumerate}
\end{theorem}

\begin{corollary}\label{corbogo} Let $\bG$ be a linear algebraic  $k$--group
whose connected component of the identity $\bG^0$ is reductive.
Assume that one of the following holds:

\medskip

(i) $k$ is algebraically closed;

\medskip

(ii) $\bG$ is obtained by base change from a smooth affine $\Z$--group
satisfying the hypothesis of Theorem \ref{bogo}(b);

\medskip

(iii) $\bG$ is reductive.

\smallskip

If the $k$--scheme $\gX$ satisfies  condition (\ref{e.C}), then

\begin{enumerate}
\item $H^1_{toral}(\gX, \bG) \subset H^1_{finite}(\gX,\bG)$.

\item If furthermore ${\rm char}(k)=0$, we have
$H^1_{toral}(\gX, \bG) \subset H^1_{finite}(\gX,\bG) \subset H^1_{loop}(\gX,\bG)$.
\end{enumerate}

\end{corollary}

The first statement is immediate.
The second one follows from Lemma \ref{lem1}. \qed

\section{Semilinear considerations}\label{secsemi}

Throughout this section $\tilde{k}$ will denote  an object of $\kalg.$ We will denote by  $\Gamma$ a subgroup of the group $ \rAut_{k-alg}(\tilde{k}).$ The elements of $\Gamma$ are thus $k$--linear automorphisms of the ring $\tilde{k}.$ For convenience we will denote the action of an element  $\gamma \in \Gamma$ on an element $\lambda \in \tilde{k}$ by $^\gamma{\lambda}.$

\subsection{Semilinear morphisms}\label{subsecsemi} Given an object $R$ of $\tilde{k}$--alg (the category of associative
unital commutative $\tilde{k}$--algebras), we will denote the action of and element $\lambda \in \tilde{k}$ on an element $r \in R$ by $\lambda_{R} \cdot r,$ or simply  $\lambda_{R} \, r$ or  $\lambda r$ if no confusion is possible.

Given an element $\gamma \in \Gamma,$ we denote by $R^\gamma$ the object of
$\tilde{k}$--alg which coincides with $R$ as
a ring, but where the $\tilde{k}$--module structure is now obtained
by ``twisting " by $\gamma:$
$$ \lambda_ {{R^\gamma}} \cdot r = ({^\gamma{\lambda}}) _{R} \cdot r
$$

One verifies immediately that 
\begin{equation}\label{twistR}
({R^\gamma})^\tau = R^{\gamma \tau}
\end{equation}
for all $\gamma,
\tau \in \Gamma.$ It is important to emphasize that (\ref{twistR}) is an {\it equality} and not a canonical identification.

Given a morphism $\psi : A \to R$ of $\tilde{k}$--algebras and an element $\gamma \in \Gamma$
 we can view $\psi$ as a map $\psi_\gamma : A^\gamma \to R^\gamma$ (recall that $A= A^\gamma$ and $R = R^\gamma$ as rings, hence also as sets). It is immediate to verify that $\psi_\gamma$ is also a morphism of $\tilde{k}$--algebras. By (\ref{twistR}) we have $(\psi_\gamma)_\tau = \psi_{\gamma \tau}$ for all $\gamma,
\tau \in \Gamma.$

The map $\psi \to \psi_\gamma$ gives a natural correspondence
\begin{equation}\label{twistRmor}
\Hom_{\tilde{k}-alg}(A, R) \to \Hom_{\tilde{k}-alg}({A^\gamma}, {R^\gamma}).
\end{equation}
In view of (\ref{twistR}) we have also a natural (and equivalent) correspondence
\begin{equation}\label{twistRmorbis}
\Hom_{\tilde{k}-alg}(A, R^\gamma) \to \Hom_{\tilde{k}-alg}({A^\gamma}^{^{-1}}, R).
\end{equation}
that we record for future use.

\begin{remark}\label{equalities} {\rm 
 (i)  Let $\gamma, \sigma \in \Gamma.$ It is clear from the definitions that the $k$-algebra isomorphism $\gamma : \tilde{k} \to \tilde{k}$ induces a $\tilde{k}$-algebra isomorphism  $\gamma_{\sigma} : \tilde{k}^\sigma \to \tilde{k}^{\gamma \sigma}.$ If no confusion is possible we will denote $\gamma_{\sigma}$ simply by $\gamma.$

One checks that the $\tilde{k}$--algebras $(R \otimes_k \tilde{k})^\gamma$ and  $ R \otimes_k \tilde{k}^\gamma$ are {\it equal} (recall that both algebras have $R \otimes_k \tilde{k}$ as underlying sets). We thus have a
$\tilde{k}$--algebra isomorphism $$1 \otimes \gamma : R \otimes_k \tilde{k} \to  R \otimes_k \tilde{k}^\gamma = (R\otimes \tilde{k})^\gamma,$$ 
or more generally
$$1 \otimes \gamma_{\sigma} : R \otimes_k \tilde{k}^{\sigma} \to  R \otimes_k \tilde{k}^{\gamma \sigma} = (R\otimes \tilde{k})^{\gamma \sigma}.$$ 
\smallskip
(ii) If $A$ is an object $\tilde{k}$--alg, and $\gamma \in \Gamma$, then the $\tilde{k}$--algebras $A$ and $A^{\gamma}$ have the same ideals.}
\end{remark}

 Given a $\tilde{k}$--functor $\gX,$  that is a functor from the category $\tilde{k}$--alg to the category of sets (see \cite{DG} for details), and an element $\gamma \in \Gamma$
we can  define a new the $\tilde{k}$--functor $^\gamma\gX$ by setting
 \begin{equation}\label{twistXfunctor}
 ^\gamma\gX(R) = \gX(R^\gamma)
 \end{equation}
 and  $^\gamma\gX(\psi) = \gX(\psi_\gamma)$ where $\psi : R \to S$ is as above. The diagram
 
 $$
\begin{CD}\label{twistmap}
 \gX'(R^{\gamma})@>{=}>>  \gX(R^{\gamma}) \\
@V{^\gamma\gX(\psi)}VV @V{\gX(\psi_\gamma)}VV\\
{^\gamma\gX'}(R)@>{=}>> {^\gamma\gX} (R)\\
\end{CD}
$$ 
 then commutes by definition, and one can indeed easily verify  that $^\gamma\gX$ is a $\tilde{k}$--functor. We call $^\gamma \gX$ the {\it twist of} $\gX$ {\it by} $\gamma.$
 
 Similarly to the case of $\tilde{k}$--algebras described in (\ref{twistR}) we have the {\it equality} of functors
 \begin{equation}\label{twistX}
 ^\gamma(^\tau\gX) = {^{\gamma \tau}\gX}
 \end{equation} for all $\gamma,
\tau \in \Gamma.$

A  morphism $f  : \gX' \to \gX$ induces a morphism ${^\gamma f }  : {^\gamma\gX'} \to {^\gamma\gX}$ by setting $^{\gamma}f(R) = f(R^{\gamma}).$ We thus have the commutative diagram
 $$
\begin{CD}\label{twistmap}
 \gX'(R^{\gamma})@>{f(R^{\gamma})}>>  \gX(R^{\gamma}) \\
@V{=}VV @V{=}VV\\
{^\gamma\gX'}(R)@>{^{\gamma}f}(R)>> {^\gamma\gX} (R)\\
\end{CD}
$$
This gives  a natural bijection
\begin{equation}\label{twistXmor}
\Hom_{\tilde{k}-fun}(\gX', \gX) \to \Hom_{\tilde{k}-fun}({^\gamma\gX'}, {^\gamma\gX})
\end{equation}
given by $f \mapsto {^\gamma f }.$ This correspondence is compatible with the action of $\Gamma,$ this is $^\gamma(^\tau f) = {^{\gamma \tau}f}.$ As before we will for future use explicitly write down an equivalent version
of this last bijection, namely
\begin{equation}\label{twistXmorbis}
\Hom_{\tilde{k}-fun}({^\gamma}^{^{-1}}\gX', \gX) \to \Hom_{\tilde{k}-fun}(\gX', {^\gamma\gX})
\end{equation}

\subsection{Semilinear morphisms}
A $\tilde{k}$--functor morphism $f :  {^\gamma\gX} \to \gZ$ is
called a {\it semilinear morphism  of type}
$\gamma$ {\it from} $\gX$ {\it to}  $\gZ.$ We denoted the set of such morphisms by $\Hom_\gamma(\gX, \gZ),$ and
set $\Hom_\Gamma(\gX, \gZ) = \cup_{\gamma \in \Gamma} \Hom_\gamma(\gX, \gZ).$\footnote{The alert reader may question whether the ``type" is well defined.  Indeed it may happen that $^\gamma\gX$ and $\gX$ are the {\it same} $\tilde{k}$-functor even though $\gamma \neq 1.$ This ambiguity can be formally resolved by defining  semilinear morphism  of type
$\gamma$ as pairs ($f :  {^\gamma\gX} \to \gZ, \gamma).$ We will omit this complication of notation in what follows since no confusion will be  possible within our context. Note that the union of sets $\cup_{\gamma \in \Gamma} \Hom_\gamma(\gX, \gZ)$ is thus disjoint by definition.}
These are the $\Gamma$--{\it semilinear morphisms} from
$\gX$ to $\gZ.$

If $f  : {^\gamma\gX} \to \gY$ and $g :  {^\tau\gY} \to \gZ$ are
semilinear of type $\gamma$ and $\tau$ respectively, then the map
$gf : {^{\tau \gamma}\gX }\to \gZ$ defined by $(gf)(R) = g(R) \circ
f(R^\tau)$ according to the sequence
\begin{equation}
{^{\tau \gamma}\gX }(R) =  {^{\gamma}\gX }(R^{\tau}) \buildrel f(R^{\tau}) \over
\to \gY(R^{\tau}) = {^{\tau}\gY }(R) \buildrel g(R) \over
\to \gZ(R)
\end{equation}
is semilinear  of type $\tau \gamma.$

The above considerations give the set $\Aut_\Gamma(\gX)$ of invertible elements of $\Hom_\Gamma(\gX, \gX)$ a group structure whose elements are   called $\Gamma$--{\it semilinear automorphisms of} $\gX.$ There is a
group homomorphism $t : \Aut_\Gamma(\gX) \to \Gamma$ that assigns to
a semilinear automorphism of $\gX$ its type.

\begin{remark}\label{fiberproduct} Fix a $\tilde{k}$--functor $\gY$. Recall that the category of $\tilde{k}$--functors over $\gY$ consists of $\tilde{k}$--functors $\gX$ equipped with a structure morphism $\gX \to \gY$. This category admits fiber products: Given $f_1 : \gX_1 \to \gY$ and $f_2 : \gX_2 \to \gY$
then  $\gX_1 \times_{\gY} \gX_2$ is given by
$$(\gX_1 \times_{\gY} \gX_2)(R) =  \{(x_1,x_2) \in \gX_1(R) \times \gX_2(R) : f_1(R)(x_1) = f_2(R)(x_2)\}.$$

Semilinearity extends to fiber products under the right conditions. Suppose $f_1 : \gX_1 \to \gY$ and $f_2 : \gX_2 \to \gY$ are as above, and that the action of $\Gamma$ in  $\gX_i$ and $\gY$ is  compatible in the obvious  way. Then for each $\gamma
\in \Gamma$ the ``structure morphisms"  $^{\gamma}f_i : {^{\gamma}\gX}_i \to {^{\gamma}\gY}$ defined above can be seen to verify 
\begin{equation}\label{fiberproduct}
{^\gamma}(\gX_1 \times_{\gY} \gX_2) = {^\gamma \gX_1} \times_{^\gamma \gY}  {^\gamma \gX_2}
\end{equation}
for all $\gamma \in \Gamma.$

\end{remark}

\subsection{Case of affine schemes}
Assume that  $\gX$ is affine, that is $\gX = {\rm Sp}_{\tilde{k}}A = \Hom_{\tilde{k}-alg}(A, -).$ If $\gamma \in \Gamma$ then \begin{equation}\label{affinetwist}
{^\gamma \gX} = {\rm Sp}_{\tilde{k}}{A^\gamma}^{^{-1}}
\end{equation}
 as can be seen from (\ref{twistRmorbis}). In particular ${^\gamma \gX}$ is also affine. Our next step is to show that  semilinear twists of  schemes are also schemes.

Assume that $\gY $ is an open subfunctor of $\gX.$  We claim that ${^\gamma \gY}$ is an open subfunctor of ${^\gamma \gX}.$  We must show that for all affine functor $ {\rm Sp}_{\tilde{k}}A$ and all morphism $f : {\rm Sp}_{\tilde{k}}A \to {^\gamma \gX}$ there exists an
ideal $I$ of $A$ such that $f^{-1}(^\gamma \gY) = D(I)$ where
$$ D(I)(R) = \{ \alpha \in \Hom(A, R) : Rf(I) = R\}.$$

Let us for convenience denote $ {\rm Sp}_{\tilde{k}}A$ by $\gX',$ and $\gamma^{-1}$ by $\gamma'.$ Our morphism $f$ induces  a morphism ${{^\gamma}'}f  : {{^\gamma}'}\gX' \to \gX$ by the considerations described above. Because $\gY$ is open in $\gX$ and
${{^\gamma}^{'}\gX'}  = {\rm Sp}_{\tilde{k}-alg}A{^\gamma}$  is affine there exists and ideal $I$ of $A{^\gamma}$ such that $({{^\gamma}'}f)^{-1}(\gY) = D(I).$ Applying this to the $\tilde{k}$--algebra $R^{\gamma}$  we obtain
\begin{equation}\label{open}
{{^\gamma}'}f(R^{\gamma})^{-1}(\gY(R^{\gamma}) = \{ \alpha \in \Hom_{\tilde{k}-alg}(A{^\gamma}^{'}, R^{\gamma}) : R^{\gamma}\alpha(I) = R^{\gamma} \}.
\end{equation}
On the other hand ${{^\gamma}'}f(R^{\gamma})^{-1} = f(R)^{-1}$ and $\gY(R^{\gamma}) = {^{\gamma} \gY}(R).$ Finally in the right hand side of (\ref{open}) we have  $\Hom_{\tilde{k}-alg}(A{^\gamma}^{'}, R^{\gamma}) =  \Hom_{\tilde{k}-alg}(A, R) $ and $R^{\gamma}\alpha(I) = R^{\gamma}$ if and only if $R\alpha(I) = R.$ Since $I$ is also an ideal of the $\tilde{k}$--algebra $A$ this completes the proof that  ${^\gamma \gY}$ is an open subfunctor of ${^\gamma \gX}.$

If $\gX$ is local then so is ${^{\gamma} \gX}.$ Indeed, given a $\tilde{k}$-algebra $R$ and and element  $f \in R$ then $f$ can naturally be viewed as  an element of $R^{\gamma}$ (since $R$ and $R^{\gamma}$ coincide as rings), and it is immediate to verify
that $(R_f)^{\gamma} = (R^{\gamma})_f.$ Using that it is then clear that the sequence

\begin{equation}\label{local}
{{^\gamma}}\gX(R) \to {{^\gamma}}\gX(R_{f_i})  \rightrightarrows {{^\gamma}}\gX(R_{f_if_j})
\end{equation}
is exact whenever $ 1 = f_1 + \dots +f_n.$

Since $R$ is a field if and only if $R^{\gamma}$ is a field it is clear that if $\gX$ is covered by a family of open subfunctors $(\gY_i)_{i \in I}$, then ${^{\gamma} \gX}$ is covered by the open subfunctors ${^{\gamma} \gY_i}.$ From this it follows
that if $\gX$ is a scheme then so is ${^{\gamma} \gX}.$

\begin{remark}\label{twistedschemes} {\rm Let $\gX$ is a $\tilde{k}$--scheme defined along traditional lines (and not as a special type of $\tilde{k}$--functor), and let $\gX$ also denote the restriction to the category of affine $\tilde{k}$--schemes of the functor of points of $\gX.$  If we define (again along traditional lines) ${^{\gamma}\gX} =  \gX \times_{\Spec(\tilde{k})} \Spec(\tilde{k}{^{\gamma}}^{^{-1}}),$ then it can be shown that the functor of points of ${^{\gamma}\gX}$ (restricted to the category  of affine $\tilde{k}$--schemes) coincides with the twist by $\gamma$ of $\gX$ that we have defined.}
\end{remark}

\begin{remark}\label {affinecase} We look in detail at the case when our  $\tilde{k}$--scheme  is an affine group scheme $\gG.$ Thus $\gG = {\rm Sp}_{\tilde{k}}\tilde{k}[\gG]$ for some $\tilde{k}$--Hopf algebra $\tilde{k}[\gG]$ (see \cite{DG} II \S1 for details). 

Let $\epsilon_{\gG} : \tilde{k}[\gG] \to \tilde{k}$ denote the counit map. As $\tilde{k}$-modules we have $\tilde{k}[\gG] = \tilde{k} \oplus I_{\gG}$ where $I_{\gG}$ is the kernel of $\epsilon_{\gG}.$  Let $\gamma \in \Gamma.$ As explained in (\ref{affinetwist}) we have  $^\gamma{\gG} = {\rm Sp}_{\tilde{k}}\tilde{k}[\gG]{^\gamma}{^{^{-1}}}= \Hom_{\tilde{k}-alg}(\tilde{k}[\gG]{^\gamma}{^{^{-1}}}, -).$ We leave it to the reader to verify that $\epsilon_{_{^\gamma{\gG}}} : \gamma \circ \epsilon_{\gG}.$ As an abelian group $I_{\gG} = I_{_{^\gamma{\gG}}},$ but in this last the action of $\tilde{k}$ is obtained through the action of $\tilde{k}$ in $\tilde{k}[\gG]^\gamma.$ 

Next we make some relevant observations about Lie algebras from a functorial point of view (\cite{DG} II \S4). Recall that the group functor $\mathfrak{Lie}(\gG)$  attaches to an object $R$ in $\tilde{k}$-alg the kernel of the group homomorphism $\gG(R[\epsilon]) \to \gG(R)$ where $R[\epsilon]$ is the $\tilde{k}$--algebra of dual numbers of $R,$ and the group homomorphism comes from the functorial nature of $\gG$ applied to the morphism $R[\epsilon] \to R$ in $\tilde{k}$--alg that maps $\epsilon \mapsto 0.$ By definition ${\mathcal Lie}(\gG) = \mathfrak{Lie}(\gG)(\tilde{k}).$ In particular ${\mathcal  Lie}(\gG) \subset \gG(\tilde{k}[\epsilon] )= \Hom_{\tilde{k}{\text{\rm -alg}}}(\tilde{k}[\gG], \tilde{k}[\epsilon]).$ Every element $x \in {\mathcal Lie}(\gG)$ is given by 
\begin{equation}\label {Liealgebra element}
x : a \mapsto \epsilon_{\gG}(a) + \delta_x(a)\epsilon 
\end{equation}
with $\delta_x \in {\rm Der}_{\tilde{k}}(\tilde{k}[\gG], \tilde{k})$ where $\tilde{k}$ is viewed as a $\tilde{k}[\gG]$--module via the counit map of $\gG$. In what follows we write $x = \epsilon_{\gG} + \delta_x\epsilon.$ The map $x \mapsto \delta_x$  is in fact a $\tilde{k}$-module isomorphism
${\mathcal Lie}(\gG) \simeq {\rm Der}_{\tilde{k}}(\tilde{k}[\gG], \tilde{k}).$ In  particular if $\lambda \in \tilde{k}$ then $\lambda x  \in {\mathcal Lie}(\gG)$ is such that $\delta_{\lambda x} = \lambda \delta_x.$  

Similar considerations apply to the affine $\tilde{k}$-group $^\gamma{\gG}.$  We have
 ${\mathcal  Lie}(^\gamma{\gG})  = \break {\rm Der}_{\tilde{k}}(\tilde{k}[\gG]{^\gamma}{^{^{-1}}}, \tilde{k}).$  Note that if $y \in {\mathcal Lie}(^\gamma{\gG})$ corresponds to $\delta_y \in  {\rm  Der}_{\tilde{k}}(\tilde{k}[\gG]{^\gamma}{^{^{-1}} }, \tilde{k}),$ then under the action of $\tilde{k}$ on ${\mathcal Lie}(^\gamma{\gG})$ 
 the element $\lambda y$ corresponds to the derivation $\lambda \delta_y$ and {\it not} to $(^{\gamma}\lambda)\delta_y$: The ``$\gamma$ part"  is taken into consideration already by the fact that $y \in {\mathcal Lie}(^{\gamma}\gG)$ and that $\delta_y \in {\rm Der}_{\tilde{k}}(\tilde{k}[\gG]{^\gamma}{^{^{-1}}}, \tilde{k}).$ 

\end{remark}

\subsection{Group functors}
Let from now on $\gG$ denote a $\tilde{k}$-group functor. If $\gH$ is a subgroup functor of $\gG$ we let
$$
\Aut_\gamma(\gG, \gH) = \Bigl\{f \in \Aut_\gamma(\gG) \, \,  | \, \,  {^\gamma \gH} = f^{-1}(\gH) \Bigr\}.$$ It is easy to verify then
that $\Aut_\Gamma(\gG, \gH) = \cup_{\gamma \in \Gamma} \Aut_\gamma(\gG, \gH)$
is a subgroup of $\Aut_\Gamma(\gG).$

\begin{proposition}\label{semilinear} Let $\Pi_{\tilde{k}/k}\gG$ be the Weil
restriction of $\gG$ to $k$ (which we view
as a $k$--group functor). There exists a canonical group homomorphism
$$\tilde{} \,  : \Aut_\Gamma(\gG) \to \Aut(\Pi_{\tilde{k}/k}\gG).$$
\end{proposition}

\begin{proof} As observed in Remark \ref{equalities} the map $\gamma : \tilde{k} \to \tilde{k}^\gamma$ is an isomorphism
of $\tilde{k}$--alg, and for $R$  in $k$--alg $(R \otimes_k \tilde{k})^\gamma = R \otimes_k \tilde{k}^\gamma$. We thus have a
$\tilde{k}$--algebra isomorphism $1 \otimes \gamma : R \otimes_k \tilde{k} \to
(R\otimes \tilde{k})^\gamma.$  For a given $f \in \Aut_\Gamma(\gG),$ the
composite map

$$\tilde{f}(R) : (\Pi_{\tilde{k}/k}\gG)(R) = \gG(R \otimes_k \tilde{k})
\buildrel \gG(1 \otimes \gamma)
\over \lra \gG\big((R \otimes_k \tilde{k})^\gamma\big) = $$

$$ = {^\gamma \gG}(R \otimes_k \tilde{k}) \buildrel f(R \otimes_k \tilde{k}) \over
\to \gG(R \otimes_k \tilde{k}) = (\Pi_{\tilde{k}/k}\gG)(R)$$
is an  automorphism of the group $(\Pi_{\tilde{k}/k}\gG)(R).$ One readily
verifies that the
family $\tilde{f} = \tilde{f}(R)_{R \in k-{\rm alg}}$ is
functorial on $R,$ hence an
automorphism of $\Pi_{\tilde{k}/k}\gG.$

To check that \,\, $\tilde{}$ \,\,  is a group homomorphism we consider two elements $f_1, f_2  \in \Aut_\Gamma(\gG)$ of type $\gamma_1$ and $\gamma_2$ respectively. Recall that $\gamma_2$ induces a $\tilde{k}$-algebra homomorphism $1 \otimes{\gamma_2}_\sigma: R \otimes \tilde{k}^{\gamma} \to R \otimes \tilde{k}^{^{{\gamma_2} \gamma}}$  for all $\sigma \in \Gamma $ [see Remark \ref{equalities}(i)]. Since $\gamma$ will be understood from the context we will denote this homomorphism simply by $1 \otimes \gamma_2.$  By functoriality we get the following commutative diagram [see Remark \ref{equalities}(1)]

$$
\begin{CD}
\gG(R \otimes \tilde{k}^{\gamma_1})@>\gG(1 \otimes \gamma_2)>>
\gG(R \otimes \tilde{k}^{\gamma_2 \gamma_1}) \\
@VV{=}V @VV{=} V\\
{^{\gamma_1}}\gG(R \otimes \tilde{k}) @>{^\gamma}\gG(1 \otimes \gamma_2)>>{^{\gamma_1}}\gG(R \otimes \tilde{k}^{\gamma_2}) \\
@VV{f_1(R \otimes \tilde{k})}V @VV{f_1(R \otimes \tilde{k}^{\gamma_2}})V \\
\gG(R \otimes \tilde{k}) @>\gG(1 \otimes \gamma_2)>> \gG(R \otimes \tilde{k}^{\gamma_2} )\\
\end{CD}
$$

Since $f_2 \circ f_1$ is of type $\gamma_2 \gamma_1,$ by definition we have
$$\widetilde{f_2 \circ f_1}(R \otimes \tilde{k}) = (f_2 \circ f_1)(R \otimes \tilde{k})  \circ \gG(1 \otimes \gamma_2  \gamma_1).$$Thus
$$ \widetilde{f_2 \circ f_1}(R \otimes \tilde{k}) =  (f_2 \circ f_1)(R \otimes \tilde{k}) \circ \gG(1 \otimes \gamma_2 \circ 1 \otimes \gamma_1) $$
$$ =  (f_2 \circ f_1)(R \otimes \tilde{k}) \circ \gG(1 \otimes \gamma_2) \circ \gG(1 \otimes \gamma_1) $$
$$=    f_2(R \otimes \tilde{k}) \circ f_1(R \otimes \tilde{k}^{\gamma_2} ) \circ  \gG(1 \otimes \gamma_2) \circ \gG(1 \otimes \gamma_1) $$
$$=  f_2(R \otimes \tilde{k} )\circ  \gG(1 \otimes \gamma_2) \circ f_1(R \otimes \tilde{k}) \circ \gG(1 \otimes \gamma_1) $$
$$= \tilde{f_2}(R \otimes \tilde{k} )\circ \tilde{f_1}(R \otimes \tilde{k}).$$
\end{proof}

\begin{example} (a) Consider the case of the trivial $\tilde{k}$--group
$\ee_{\tilde{k}}.$ Each set $\Aut_\gamma(\ee_{\tilde{k}}) = {\rm Isom}({^\gamma\ee_{\tilde{k}}}, \ee_{\tilde{k}})$ consists of one element which we
denote by $\gamma_{*}.$  Then
$\Aut_\Gamma(\ee_{\tilde{k}}) \simeq \Gamma.$ We have $\Pi_{\tilde{k}/k}\ee_{\tilde{k}} =  \ee_k.$ In particular $\Aut(\Pi_{\tilde{k}/k}\ee_{\tilde{k}}) = 1$ and the
homomorphism \,\, $\tilde{} \,  : \Aut_\Gamma(\gG) \to
\Aut(\Pi_{\tilde{k}/k}\gG)$ is in this case necessarily trivial. In affine
terms $\ee_{\tilde{k}}$ is represented by $\tilde{k}$ and ${^\gamma\ee_{\tilde{k}}}$ by
$\tilde{k}^{\gamma^{-1}}.$ Then the $\tilde{k}$--group isomorphism $\gamma_{*} :
{^\gamma\ee_{\tilde{k}}} \to \ee_{\tilde{k}}$ corresponds to the $\tilde{k}$--Hopf algebra
isomorphism $\gamma^{-1} : \tilde{k} \to \tilde{k}^{\gamma^{-1}}.$

(b) Consider the case when $\Gamma$ is the Galois group of the
extension $\Bbb C/ \Bbb R,$ and $\gG$ is the additive $\Bbb
C$--group. Then $\Aut_\Gamma(\gG)$ can be identified with the group of automorphisms of
$(\Bbb C, +)$ which are of the form $z \mapsto \lambda z$ or $z
\mapsto \lambda \overline{z}$ for some $\lambda \in \Bbb C^\times.$
The Weil restriction of $\gG$ to $\Bbb R$ is the two-dimensional
additive $\Bbb R$--group. Thus $\Aut(\Pi_{\tilde{k}/k}\gG) = {\rm GL}_2(\Bbb
R).$

The above examples show that, even if $\tilde{k}/k$ is a finite
Galois extension of fields and $\gG$ is a connected linear algebraic
group over $\tilde{k}$, the
 homomorphism $f \mapsto \tilde{f}$
need be neither injective nor surjective
\end{example}

\begin{corollary}\label{semilinearcor1} Assume that $\gG = {\rm Sp}_{\tilde{k}} \tilde{k}[\gG]$ is an affine $\tilde{k}$-group. The group $\Aut_\Gamma(\gG)$
acts naturally on the
groups $\gG(\tilde{k})$ and $\gG(\tilde{k}[\epsilon]).$ Furthermore the action of
an element $f \in \Aut_\gamma(\gG)$ on
 $\gG(\tilde{k}[\epsilon])$ stabilizes ${\mathcal Lie}(\gG) \subset \gG(\tilde{k}[\epsilon]).$ 
The induced
 map  ${\mathcal  Lie}(f) : {\mathcal  Lie}(\gG) \to {\mathcal  Lie}(\gG)$ is an
automorphism of ${\mathcal Lie}(\gG)$ viewed as a
Lie algebra over $k.$ This automorphism is $\tilde{k}$--semilinear, i.e.,
${\mathcal Lie}(f )(\lambda x) = ({^\gamma\lambda}) {\mathcal  Lie}(f )(x)$ for
all $\lambda \in \tilde{k}$ and $x \in {\mathcal  Lie}(\gG).$ 
\end{corollary}

\begin{proof}  We maintain the notation and use the facts presented in Remark \ref{affinecase}. Let $x \in {\mathcal  Lie}(\gG)$ and write $x = \epsilon_{\gG} + \delta_x\epsilon.$   If $\lambda \in \tilde{k}$ then $\lambda x  \in {\mathcal  Lie}(\gG)$ is such that $\delta_{\lambda x} = \lambda \delta_x.$ 

By definition $(\Pi_{\tilde{k}/k}\gG)(k) = \gG(\tilde{k})$ and
$(\Pi_{\tilde{k}/k}\gG)(k[\epsilon]) = \gG(\tilde{k}[\epsilon]).$ The action of
an element $f \in \Aut_\Gamma(\gG)$ on these two groups is then given
by the automorphisms $\tilde{f}_k$ and $\tilde{f}_{k[\epsilon]}$
of the previous Proposition. Thus if  we let  $\gamma_\epsilon : \tilde{k}[\epsilon] \to \tilde{k}[\epsilon]^\gamma$ denote the isomorphism of $\tilde{k}$--alg induced by $\gamma$ the map $\tilde{f}_{k[\epsilon]}$ is then obtained by restricting to ${\mathcal  Lie}(\gG)$ the composite map
$$\gG(\tilde{k}[\epsilon]) = \Hom_{\tilde{k}-alg}(\tilde{k}[\gG], \tilde{k}[\epsilon]) \buildrel \gG(\gamma_\epsilon) \over\to \Hom_{\tilde{k}-alg}(\tilde{k}[\gG], \tilde{k}[\epsilon]^\gamma)=$$
$$ = \Hom_{\tilde{k}-alg}(\tilde{k}[\gG]{^\gamma}{^{^{-1}}}, \tilde{k}[\epsilon]) = {^\gamma{\gG}}(\tilde{k}[\epsilon])\buildrel f(\tilde{k}[\epsilon]) \over\to \Hom_{\tilde{k}-alg}(\tilde{k}[\gG], \tilde{k}[\epsilon]) = \gG(\tilde{k}[\epsilon]).
$$
Using the fact that $\gamma_{\epsilon}  \circ \epsilon_{\gG} = \gamma \circ \epsilon_{\gG} = \epsilon_{_{^\gamma{\gG}}} $ it easily follows that 
\begin{equation}\label{semiLie}
{\mathcal  Lie}(f)(x) = \tilde{f}_{k[\epsilon]}(x) = f(\tilde{k}[\epsilon]) \circ \big(\epsilon_{_{^\gamma{\gG}}} + (\gamma \circ \delta_x)\epsilon \big)
\end{equation}
Let $ y = \epsilon_{_{^\gamma{\gG}}} + (\gamma \circ \delta_x)\epsilon \in {\mathcal  Lie}({^\gamma{\gG}}).$  If $\lambda \in \tilde{k}$ then we have 
$${\mathcal  Lie}(f)(\lambda x) = f(\tilde{k}[\epsilon]) \circ \big(\epsilon_{_{^\gamma{\gG}}} + (\gamma \circ \delta_{\lambda x})\epsilon \big)$$
$$= f(\tilde{k}[\epsilon]) \circ \big(\epsilon_{_{^\gamma{\gG}}} + (^\gamma (\lambda \delta_x)\epsilon \big)$$
$$= f(\tilde{k}[\epsilon]) \circ \big(\epsilon_{_{^\gamma{\gG}}} + {^\gamma} \lambda (\gamma \circ \delta_x)\epsilon \big)$$
$$= f(\tilde{k}[\epsilon]) \big( ({^\gamma}\lambda) y )\big)$$
where $({^\gamma}\lambda) y $ is the action of the element ${^\gamma}\lambda \in \tilde{k}$ on the element $y \in {\mathcal  Lie}({^\gamma}\gG)$, as explained in the  last paragraph of Remark \ref{affinecase}. Since the restriction of $f(\tilde{k}[\epsilon]) $ to ${\mathcal  Lie}({^\gamma}\gG)$ induces an isomorphism ${\mathcal  Lie}({^\gamma}\gG) \to {\mathcal  Lie}(\gG)$ of $\tilde{k}$-Lie algebras, this restriction is in particular $\tilde{k}$-linear. It follows that
$${\mathcal  Lie}(f)(\lambda x) = f(\tilde{k}[\epsilon]) \big(({^\gamma}\lambda) y \big) = ({^\gamma}\lambda)  f(\tilde{k}[\epsilon]) (y ) = ( {^\gamma}\lambda) {\mathcal  Lie}(f)(x).$$
This shows that ${\mathcal  Lie}(f)$ is semilinear. We leave it to the reader to verify that ${\mathcal  Lie}(f)$ is an automorphism of ${\mathcal  Lie}(\gG)$ as a Lie algebra over $k.$
\end{proof}

\begin{remark} There is no natural action of $\Aut_\Gamma(\gG)$ on $\gG.$
\end{remark}

\subsection{Semilinear version of a theorem of Borel-Mostow}
Throughout this section $k$ denotes a field of characteristic $0.$

\begin{theorem}\label{semilinearBM} (Semilinear Borel-Mostow)
 Let $\tilde{k}/k$ be
a finite Galois extension of fields with Galois group $\Gamma.$ Suppose we are
given a quintuple
$\big(\gg, H, \psi, \phi, (H_i)_{0 \leq i \leq s}\big)$ where

$\gg$ is a (finite dimensional) reductive Lie algebra over $\tilde{k},$

$H$ is a group,

$\psi$ is a group homomorphism from $H$ into the Galois group $\Gamma,$

$\phi$ is a group homomorphism from $H$ into the group $\Aut_k(\gg)$
of automorphisms of $\gg$ viewed as a
Lie algebra over $k,$

 $(H_i)_{1 \leq i \leq s}$ is a finite family of subgroups of $H$ for which the following two conditions hold:

(i) If we let the group $H$ act on $\gg$ via $\phi$ and
on $\Gamma$ via $\psi$, namely
 $^h x = {^{\phi(h)}}x$ and ${^h}\lambda = {^{\psi(h)}\lambda}$
for all $h \in H,$  $x \in \gg,$ and $\lambda \in \tilde{k},$ then the
action of $H$ in $\gg$ is semilinear, i.e.,  ${^h}(\lambda x) {^h} = {^h}{\lambda} {^h} x.$

(ii) $\ker(\psi) = H_s \supset H_{s - 1} \supset ... \supset H_1
\supset H_0 = 0.$ Furthermore, each $H_i$ is normal in $H,$ the
elements of $\phi(H_i)$ are semisimple,\footnote{Because $H_i \subset
\ker(\psi)$ the action of the elements of $H_i$  on $\gg$ is
$\tilde{k}$--linear. The assumption is that $\phi(\theta)$ be semisimple
as  a $\tilde{k}$--linear endomorphisms of $\gg$ for all $\theta \in H_i.$} and the quotients $H_i /
H_{i-1}$ are cyclic.

Then there exists a Cartan subalgebra  of $\gg$ which is stable
under the action of $H$.
\end{theorem}

\begin{proof}
We will reason by  induction on  $s.$ If $s = 0$ we can identify by assumption (ii) $H$
with a subgroup $\Gamma_0$ of $\Gamma$ via $\psi.$ Let $\tilde{k}_0 = \tilde{k}^{\Gamma_0}.$ This yields a
semilinear action of $\Gamma_0$ on $\gg.$ By Galois descent the
fixed point $\gg^{\Gamma_0}$ is a Lie algebra over $\tilde{k}_0$ for which
the canonical map $\rho : \gg^{\Gamma_0} \otimes_{\tilde{k}_0} \tilde{k}\simeq \gg$
is a $\tilde{k}$--Lie algebra isomorphism. If $\gh_0$ is a Cartan subalgebra
of $\gg^{\Gamma_0}$ then $\rho(\gh_0 \otimes_{\tilde{k}_0} \tilde{k})$ is a Cartan
subalgebra of $\gg$ which is $H$--stable as one can easily verify with the aid of assumption (i).

Assume $s \geq 1$ and consider a generator $\theta$ of the cyclic group
$H_1.$ As we have already
observed the action of $\theta$ on $\gg$ is $\tilde{k}$--linear.   If $V$ is a
$\tilde{k}$--subspace of $\gg$ stable
under $\theta$ we will denote by $V^\theta$  the subspace of fixed points.
Before continuing with
we establish the following crucial fact:

\begin{claim}\label{claim}  $\gg^\theta$ is a reductive Lie algebra over $\tilde{k}.$
If $\gh$ is a Cartan
subalgebra of $\gg^{\theta},$ then $\zz_{\gg}(\gh)$ is a Cartan subalgebra
of $\gg.$
\end{claim} 
  Since $\phi(\theta)$ is an automorphism of the $\tilde{k}$--Lie algebra $\gg$ we
see that  $\gg^\theta$ is
  indeed a Lie subalgebra of $\gg.$ Let $\gg'$ and $\gz$ denote the
derived algebra and the centre of $\gg$
  respectively. Because $\gg$ is reductive $\gg'$ is semisimple and
$\gg = \gg' \times \gz.$ Clearly
  $\theta$ induces by restriction automorphisms (also denoted by $\theta$)
of $\gg'$ and of $\gz.$ By
  \cite{Bbk} Ch. 8 \S1 cor. to prop. 12. $(\gg')^\theta$ is reductive, and
  therefore $\gg^\theta = (\gg')^\theta \times \gz^\theta$ is also reductive.

 Every Cartan subalgebra $\gh$ of $\gg^\theta$ is  of the form $\gh= \gh'
\times \gz^\theta$ for
 some Cartan subalgebra $\gh'$ of $({\gg'})^{\theta}.$
Clearly $\zz_{\gg}(\gh) = \zz_{\gg'}(\gh') \times \gz.$ By
 \cite{P3} theorem 9 the centralizer  $\zz_{\gg'}(\gh')$ is a Cartan
subalgebra of $\gg',$ so the claim follows.
\medskip

We now return to the proof of the Theorem. Since $H_1$ is normal in
$H$ we have an induced action (via $\phi$) of $H' =  H/H_1$ on the reductive $\tilde{k}$--Lie algebra $\gg^{\theta}.$  We have
induced group homomorphisms $\phi' : H'  \to \Aut_k(\gg^{\theta})$
and $\psi' : H' \to \Gamma$ (this last since $H_1 \subset
\ker(\psi)$). For $0 \leq i < s$ define $H'_i = H_{i + 1}/H_1.$ We
apply the induction assumption to the quintuple $\big(\gg^\theta,
H', \psi', \phi', (H'_i)_{0 \leq i \leq s - 1}\big).$ This yields
the existence of a Cartan subalgebra   $\gh$ of $\gg^{\theta}$ which
is stable under the action of $H'$ given by $\phi'.$ This means
that, back in $\gg$, the algebra $\gh$ is stable under our original
action of $H$ given by $\phi.$ But then the centralizer of $\gh$ in
$\gg$ is also stable under this action, and we can now conclude by
(\ref{claim}) .
\end{proof}

\begin{remark} If $\psi$ is the trivial map  the Theorem reduces
to the ``Main result (B)" of Borel and
Mostow \cite{BM} for $\gg.$  The use of  (\ref{claim}) allows for a
slightly more direct proof of this result.
\end{remark}

We shall use the above semilinear version
of Borel-Mostow's theorem \ref{corBM1} to establish the following result
which will play a crucial role in the the proof of the existence 
of maximal tori
 on twisted groups corresponding to loop torsors.

\begin{corollary}\label{corBM1} Let $\tilde{k}/k$ be a finite Galois
extension with Galois group $\Gamma.$
Let $\bG$ be a reductive group over $\tilde{k}.$ Let $H$ be a group,
and assume we are given a group homomorphism $\rho : H  \to
\Aut_\Gamma(\bG)$  for which we can find a family of subgroups
$(H_i)_{0 \leq i \leq s}$ of $H$ as in the Theorem, that is  $\ker(t
\circ \rho) = H_s \supset H_{s - 1} \supset ... \supset H_1 \supset
H_0 = 0$ where $t : \Aut_\Gamma(\bG) \to \Gamma$ is the type
morphism, each $H_i$ is normal in $H,$ the elements of $\rho(H_i)$
act semisimply on the $\tilde{k}$--Lie algebra ${\mathcal  Lie}(\bG)$, and the
quotients $H_i / H_{i-1}$ are cyclic. 
\noindent Then there exists a maximal torus $\bT$ of $\bG$ such
that
 $\rho$ has values in  $\Aut_\Gamma(\bG,\bT) \subset \Aut_\Gamma(\bG)$. Namely if $h \in H$ and $(t \circ \rho)(h) = \gamma \in \Gamma,$ then $\rho(h) : {^\gamma}\bG \to \bG$ induces by restriction an isomorphism $ {^\gamma}\bT \to \bT.$
\end{corollary}
\begin{proof} Let $h \in H.$  If $(t \circ \rho)(h) = \gamma$ then according to the various definitions we have the following commutative diagram.

$$
\begin{CD}
   \bG(\tilde{k}) @>\bG(\gamma)>> \bG(\tilde{k}^\gamma)  = {^\gamma}\bG(\tilde{k})@>\rho(h)(\tilde{k})>> \bG(\tilde{k})  \\
 @VVV @VVV @VVV \\
  \bG(\tilde{k}[\epsilon]) @>\bG(\gamma_\epsilon)>> \bG(\tilde{k}[\epsilon]^\gamma)  = {^\gamma}\bG(\tilde{k}[\epsilon]) @>\rho(h)(\tilde{k}[\epsilon])>> \bG(\tilde{k}[\epsilon])\\
 @AAA @AAA @AAA \\
 {\mathcal  Lie}(\bG) @>\mathfrak{Lie}(\gG)(\gamma)>> {\mathcal  Lie}(\bG^\gamma)  @>\mathfrak{Lie}(\gG)(\rho(h))>> {\mathcal  Lie}(\bG).
\end{CD}
$$
where we have denoted by $\gamma_\epsilon : \tilde{k}[\epsilon] \to \tilde{k}[\epsilon]^\gamma$ the $\tilde{k}$-algebra isomorphism induced by $\gamma.$
For convenience in what follows we will denote ${\mathcal  Lie}(\bG)$ by $\gg.$  By Corollary
\ref{semilinearcor1} we obtain by composing $\rho$ with the map \, $\tilde{}$ \, defined in Proposition \ref{semilinear} a group homomorphism
$\phi : H \to \Aut_k(\gg),$ namely $\phi(h) = \widetilde{\rho(h)}$, which together with the group
homomorphism $\psi = t \circ \rho : H \to \Gamma$ and the $H_i$
satisfy the assumptions of Theorem \ref{semilinearBM}. It follows that there exists
a Cartan subalgebra $\gt$ of $\gg$ which is stable under the action
of $H$ defined by $\phi.$ 

Note that by definition 
\begin{equation}\label{rho}
 \widetilde{\rho(h)} = \rho(h)(\tilde{k}) \circ \bG(\gamma)
 \end{equation}
 which is nothing but the top row of our diagram above. Similarly with the notation of Corollary \ref{semilinearcor1} we have 
 \begin{equation}\label{rhoLie}
 {\mathcal  Lie}(\widetilde{\rho(h)}) = \widetilde{\rho_\epsilon(h)}|_{_{\gg}} = \mathfrak{Lie}(\bG)(\rho(h)) \circ \mathfrak{Lie}(\bG)(\gamma)
 \end{equation}
 where $\widetilde{\rho_\epsilon(h)}$ stands for the middle row of our diagram, namely $\rho(h)(\tilde{k}[\epsilon]) \circ \bG(\gamma_\epsilon).$ 

Let $\bT$ be the maximal torus of $\bG$ whose Lie algebra is
$\gt$ [XIV.6.6.c].  We have $\bT = \bZ_{\bG}(\gt)$ where the centralizer is taken respect to the adjoint action of $\bG$ on $\gg.$\footnote{We could not find a reference for this basic fact in the literature. By [XIII 5.3] we have $\bN_{\bG}(\bT) = \bN_{\bG}(\gt).$ Since the natural homomorphism $\bN_{\bG}(\bT)/ \bT \to \Aut(\gt)$ is injective we obtain
$\bT = \bZ_{\bG}(\gt).$}  

Given an element $g \in \bG(\tilde{k})$ we will denote its natural image in $\bG(\tilde{k}[\epsilon])$ by $g_\epsilon.$ Since we are working over a base field the $\tilde{k}$-points of $\bT = \bZ_{\bG}(\gt)$ can be computed in the naive way, namely
\begin{equation}\label{T(k')}
 \bT(\tilde{k}) = \{ g \in \bG(\tilde{k}) : g_\epsilon x {g_\epsilon}^{-1} = x \,\,\text{\rm for all}\,\, x \in \gt \subset \bG(\tilde{k}[\epsilon])\}
\end{equation}
Since $\tilde{\rho_\epsilon(h)}$ is an automorphism of the (abstract) group $\bG(\tilde{k}[\epsilon])$ we obtain
\begin{equation}
 \bT(\tilde{k}) = \{ g \in \bG(\tilde{k}) : \tilde{\rho_\epsilon(h)}(g_\epsilon)\tilde{\rho_\epsilon(h)}(x) \big(\tilde{\rho_\epsilon(h)}({g_\epsilon})\big)^{-1} = \tilde{\rho_\epsilon(h)}(x) \,\,\text{\rm for all}\,\, x \in \gt \}
\end{equation}
But since $\tilde{\rho_\epsilon(h)}$ stabilizes $\gt$ this last reads
\begin{equation}\label{T(k')bis}
 \bT(\tilde{k}) = \{ g \in \bG(\tilde{k}) : \tilde{\rho_\epsilon(h)}(g_\epsilon)x \big(\tilde{\rho_\epsilon(h)}({g_\epsilon})\big)^{-1} = x \,\,\text{\rm for all}\,\, x \in \gt \}
\end{equation}

Note that by the commutativity of the top square of our diagram we have $\big(\tilde{\rho(h)}(g)\big)_\epsilon = \tilde{\rho_\epsilon(h)}(g_\epsilon).$ Thus from (\ref{T(k')bis}) we obtain that $\tilde{\rho_\epsilon(h)}\big(\bT(\tilde{k})\big) = \bT(\tilde{k}).$ On the other hand by (\ref{rho}) we have  $\tilde{\rho_\epsilon(h)}\big(\bT(\tilde{k})\big) =  \rho(h)(\tilde{k}) \Big(\bG(\gamma)\big(\bT(\tilde{k})\big)\Big).$ But by definition $\Big(\bG(\gamma)\big(\bT(\tilde{k})\big)\Big) = {^\gamma}\bT(\tilde{k}).$  Thus our $\tilde{k}$-group homomorphism $\rho(h) : {^\gamma}\bG : \to \bG$ is such that the two tori $\rho(h)({^\gamma}\bT)$ and $\bT$ of $\bG$ have the same $\tilde{k}$-points. This forces 
 $\rho(h)({^\gamma}\bT) = \bT.$ 
\end{proof}

Next we give a crucial application of the semilinear considerations 
developed thus far to the existence of maximal tori for
 certain loop groups.

\subsection{Existence of maximal tori in loop groups}

We come back to the case of $R=R_n = k[t^{\pm 1}_1,\dots, t^{\pm 1}_n]$ where $k$ is a field of characteristic zero. This is the ring that plays a central role in all applications to infinite-dimensional Lie theory.
It is not true in general that a reductive $R_n$--group admits a maximal
 torus; however.

\begin{proposition}\label{existenceoftori} 
Let $\gG$ be a loop reductive group scheme over $R_n$
(see definition \ref{defkloopgroup}).
 Then  $\gG$ admits a maximal torus.
\end{proposition}

\begin{proof}  We try to recreate the situation of the
 semilinear Borel-Mostow theorem. We can assume that $\gG$ is split after 
base extension to the Galois covering 
 $ \tilde{R} = \tilde{k}[t^{\pm 1/m}_1,\dots, t^{\pm 1/m}_n]$
where  $m$ is a  positive integers and 
 $\tilde{k}/k$ is a finite Galois extension of fields 
  containing all primitive $m$-th
 roots of unity of $\overline{k}.$  
 Recall from Example \ref{laurent}  that
$\tilde{R}$
is a Galois extension of $R$ with Galois group 
$\tilde{\Gamma} = (\Z/m\Z)^n\rtimes \Gamma $ as follows:  For ${\bf e} =  (e_1,\dots,e_n)\in \Z^n$ we have $^{\ol{\bf e}}(\lambda t_j^{\frac{1}{m}}) = \lambda \xi  ^{e_j}_m
t_j^{\frac{1}{m}}$ for all $\lambda \in \tilde{k}$, where $^- :\Z^n \to (\Z/mZ)^n$ is the canonical map, while the Galois 
group $\Gamma = \Gal(\tilde{k}/k)$ acts naturally on $\tilde{R}$
through its action on $\tilde{k}.$

Let $\bG_0$ be the Chevalley $k$--form of $\gG$ (see \S \ref{seckloopgroup}). By assumption, we can assume that
$\gG$ is the twist of $\gG_0 = \bG_0 \times_k  R$ by a loop cocycle
$$
u: \tilde \Gamma \to \bAut(\bG_0)(\tilde k).
$$
 The homomorphism 
$\psi  : \tilde{\Gamma} = (\Z/m\Z)^n\rtimes \Gamma \to \Gamma$ is defined to be the natural projection. For convenience we will adopt the following notational convention. The elements of $\tilde{\Gamma}$ will be denoted by $\tilde{\gamma}$, and the image under $\phi$ of such an element (which belongs to $\Gamma$), by the corresponding greek character: that is $\psi(\tilde{\gamma}) = \gamma.$ 

Consider the reductive $\tilde{k}$--group $\bG = \Spec(\tilde{k}[\bG_0])$ where, as usual, $\tilde{k}[\bG_0]$ denotes the $\tilde{k}$--Hopf algebra $\tilde{k} \otimes_k k[\bG_0].$ Consider for each $\tilde{\gamma}$ the map $f(\tilde{\gamma}) : \tilde{k}[\bG] \to \tilde{k}[\bG]$ defined by
\begin{equation}\label{f}
f(\tilde{\gamma}) =  u_{\tilde{\gamma}} \circ \gamma.
\end{equation}
Since each $u_{\tilde{\gamma}}$ is an automorphism of the $\tilde{k}$--Hopf algebra $\tilde{k}[\bG],$ it follows that $f(\tilde{\gamma})$ is in fact a $\tilde{k}$-Hopf algebra isomorphism $\tilde{k}[\bG] \to \tilde{k}[\bG]^{\gamma}.$ As such it can be thought of, by Yoneda considerations and (\ref{affinetwist}), as an element of $\Aut_\gamma(\bG)$ of type $\gamma^{-1}$ which we will denote by $\rho(\tilde{\gamma}).$ 

Since the restriction of the action of $\tilde{\gamma}$ on $\tilde{R}[\bG]$ to $\tilde{k}[\bG]$ is given by $\gamma,$ the cocycle condition on $u$ shows that for all $\tilde{\alpha}, \tilde{\beta} \in \tilde{\Gamma}$ we have
\begin{equation}\label{fmor}
\rho(\tilde{\alpha}\tilde{\beta}) =  \rho(\tilde{\alpha}) \rho(\tilde{\beta})
\end{equation}
where this last product takes place in $\Aut_{\Gamma}(\bG).$ Thus $\rho$ is a group homomorphism and $f(\tilde{\gamma})$ can be viewed as a $\tilde{k}$-Hopf algebra morphism from $\tilde{k}[\bG]$ to $\tilde{k}[\bG]^{\gamma}$ 

From (\ref{fmor}), the various definitions and the ``anti equivalent" nature of Yoneda's correspondence it follows that the map $\tilde{\gamma} \to \rho(\tilde{\gamma})$ can be viewed as a group homomorphism $\rho : \tilde{\Gamma}^{\rm opp}  \to \Aut_\Gamma(\bG),$  where $\tilde{\Gamma}^{\rm opp} $ is the opposite group of $\tilde{\Gamma}.$ Since $\rho(\tilde{\gamma})$ is of type $\gamma^{-1}$ we can complete the necessary semilinear picture by defining $\phi : \tilde{\Gamma}^{\rm opp}  \to \Gamma$ to be the map $\tilde{\gamma} \to \gamma^{-1}.$ The kernel of the composite map $t \circ \rho$ is precisely $ (\Z/m\Z)^n$, and the elements of this kernel act trivially on $\tilde{k}[\bG]$, in particular their corresponding action on the Lie algebra of $\bG$ is trivial, hence semisimple. We can thus apply Corollary \ref{corBM1}; the role of $H$ now being played by $\tilde{\Gamma}^{\rm opp}.$

Let $\bT$ be a torus $\bG$ such that  $\rho(\tilde{\gamma})({^\gamma}^{^{-1}}\bT) = \bT$ for all $\tilde{\gamma} \in \tilde{\Gamma}.$ The torus $\bT$ corresponds to a Hopf ideal $I$ of the Hopf $\tilde{k}$--algebra $\tilde{k}[\bG]$ representing $\bG.$ Each $\rho(\tilde{\gamma}),$ which corresponds to the $\tilde{k}$-Hopf algebra isomorphism $f(\tilde{\gamma})$ described in (\ref{f}), induces a $\tilde{k}$--Hopf algebra isomorphism $\overline{f}(\tilde{\gamma})$ from $\tilde{k}[\bT]$ to $\tilde{k}[\bT]^\gamma$ where $\tilde{k}[\bG]/I = \tilde{k}[\bT]$ is the Hopf algebra representing $\bT.$  For future use we observe that the  resulting action of $\tilde{\Gamma}$ on $\tilde{k}[\bT]$ is $\Gamma$--{\it semilinear} in the sense that if $\lambda \in \tilde{k}$ and $a \in \tilde{k}[\bT]$ then
\begin{equation}\label{gammasemilinear}
 \overline{f}(\tilde{\gamma})(\lambda a) = \overline{f}(\tilde{\gamma})(\lambda_{\tilde{k}[\bT]} . a) = \lambda_{\tilde{k}[\bT]^\gamma} .\big( \overline{f}(\tilde{\gamma})(a)\big)  = ({^\gamma}\lambda)\overline{f}(\tilde{\gamma})(a)
 \end{equation}
This follows immediately from the definition of $f(\tilde{\gamma}).$

Consider the reductive $\tilde{R}$-group $\tilde{\gG} = \bG \times_{\tilde{k}} \tilde{R}$ and its maximal torus   $\tilde{\gT} = \bT \times_{\tilde{k}} \tilde{R}.$ We want to define an action of $\tilde{\Gamma}$ as automorphisms of the $R$--Hopf algebra $\tilde{R}[\gT] = \tilde{k}[\bT] \otimes_{\tilde{k}} \tilde{R}$ so that the action is $\tilde{\Gamma}$--semilinear, this is
\begin{equation}\label{tildesemilinear}
^{\tilde{\gamma}}(xs) = {^{\tilde{\gamma}}}x {^{\tilde{\gamma}}}s 
\end{equation}
for all $\tilde{\gamma} \in \tilde{\Gamma}$, $s \in \tilde{R}$ and $x \in \tilde{R}[\gT].$ By Galois descent this will show that the maximal torus $\tilde{\gT} $ of $\tilde{\gG}$ descends to a torus (necessarily maximal) $\gT$ of $\gG.$

To give the desired semilinear action consider, for a given fixed $\tilde{\gamma} \in \tilde{\Gamma},$ the map
$$
 \tilde{k}[\bT] \times \tilde{R} \to \tilde{k}[\bT] \otimes_{\tilde{k}}\tilde{R} = \tilde{R}[\gT]
$$
defined by
\begin{equation}\label{descentmap}
(a,s) \mapsto \overline{f}(\tilde{\gamma}) (a)\otimes \,{^{\tilde{\gamma}}s}
\end{equation}
for all $a \in \tilde{k}[\bT]$ and $s \in \tilde{R}.$ From (\ref{gammasemilinear}) and the fact that ${^{\tilde{\gamma}}}s = {^\gamma}s$ if $s \in \tilde{k} \subset \tilde{R}$ it follows that the above map is $\tilde{k}$-balanced, hence that induces a morphism of $\tilde{k}$-spaces
\begin{equation}\label{descenttensormap}
\widehat{f}(\tilde{\gamma}) :  \tilde{k}[\bT] \otimes_{\tilde{k}}\tilde{R}= \tilde{R}[\gT]  \to  \tilde{R}[\gT]
\end{equation}
satisfying
\begin{equation}\label{descenttensormapprop}
\widehat{f}(\tilde{\gamma}) : a \otimes s \mapsto \overline{\rho}(\tilde{\gamma}) (a)\otimes{^{\tilde{\gamma}}}s
\end{equation}
for all $a \in \tilde{k}[\bT]$ and $s \in \tilde{R}.$ From (\ref{gammasemilinear}) and (\ref{descenttensormapprop}) we then obtain  an action of the group $\tilde{\Gamma}$ on the Hopf algebra $\tilde{R}[\gT]$ as prescribed by (\ref{tildesemilinear}). 
\end{proof}

\subsection{Variations of a result of Sansuc.}

We shall need the following variantion of a well-known and useful result \cite[1.13]{Sa}.

\begin{lemma} \label{sansuc} Assume that $k$ is of characteristic zero.
 Let $\bH$ be a linear algebraic group over $k$ and let $\bU$ be a normal unipotent 
subgroup of $\bH$.

\begin{enumerate}
 \item Let $k'/k$ be a finite Galois extension of fields. Let $\Gamma$ be a finite group acting 
on $k'/k$.
 Then the map
$$
H^1\big(\Gamma, \bH(k')\big) \to H^1\big(\Gamma, (\bH/ \bU)(k')\big)  
$$
is bijective.

\item Let $R$ be an object in $\kalg.$ Then the map
$$
H^1(R, \bH) \to H^1(R, \bH/ \bU)
$$
is bijective.
\end{enumerate}

\end{lemma}

\begin{proof} The $k$--group $\bU$ 
admits a non-trivial characteristic   central split unipotent subgroup 
$\bU_0 \simeq \GG_a^n$ \cite[IV.4.3.13]{DG}. We can then form the following commutative diagram
of exact sequence of algebraic $k$--groups
$$
\begin{CD}
&&&&&&1 \\
&&&&&& @VVV \\
&& 1 &&  && \bU / \bU_0 \\
&& @VVV &&  @VVV \\ 
 1 @>>> \bU_0 @>>> \bH @>>> \bH/ \bU_0  @>>> 1 \\ 
&& @VVV @V {\cong}VV @VVV \\
 1 @>>> \bU @>>> \bH @>>> \bH/ \bU  @>>> 1 \\ 
&& @VVV && @VVV \\
&& \bU / \bU_0 &&  && 1 \\
&& @VVV &&   \\ 
&& 1 \\
\end{CD}
$$
If the Lemma holds for the morphisms $\bH \to \bH / \bU_0$ and
$\bH/ \bU_0  \to  \bH/ \bU$, it holds for  $\bH \to \bH/ \bU$.
Without loss of generality, we can therefore assume by  devissage that  
$\bU=\GG_a^n$.

\smallskip
\noindent (1) Since by Hilbert's Theorem 90 (additive form) and devissage $H^1(k',\bU)=0$,\footnote{See \cite{GMB} Lemme 7.3 for a more general result.} we have an exact sequence of $\Gamma$--groups
$$
1 \to \bU(k') \to \bH(k') \to (\bH/ \bU)(k') \to 1.
$$
For each   $c \in Z^1\big(\Gamma, (\bH/ \bU)(k') \big)$, 
the group $_c{\big(\bU(k')\big)}$ is a uniquely divisible 
abelian group, so  $H^i\big(\Gamma, {_c(\bU(k'))}\big)=0$ for all $i>0$.
By applying a basic result on non-abelian cohomology
\cite[\S I.5, corollary to prop. 41]{Se},
the vanishing of these $H^2$ implies  that the map 
$H^1\big(\Gamma, \bH(k')\big) \to H^1\big(\Gamma, (\bH/ \bU)(k')\big) $
is surjective. Similarly, 
for each $z \in Z^1\big(\Gamma, \bH(k') \big)$,
the group $0= H^1\big(\Gamma, {_z(\bU(k'))}\big) $ maps onto the subset of $H^1\big(\Gamma, \bH(k')\big) $ consisting of classes of cocycles whose image in $H^1\big(\Gamma, (\bH/ \bU)(k')\big) $ coincides with that of $[z].$ We conclude
 that the map
$H^1\big(\Gamma, \bH(k')\big) \to H^1\big(\Gamma, (\bH/ \bU)(k')\big) $
is bijective. 

\smallskip

\noindent(2) 
Let us first prove the injectivity by using the classical torsion trick.
We are given  a $\bH/ \bU$--torsor $\bE$ over $\Spec(R)$. 
We can twist the exact sequence of $R$--group schemes 
$1 \to \bU_R \to \bH_R  \to \bH_R/ \bU_R \to 1$ by $\bE$ and get the 
twisted sequence $1 \to {_\bE\bU} \to {_\bE\bH}  \to {_\bE\bH} / {_\bE\bU} \to 1$,
where as usual we write ${_\bE\bU}$ instead of ${_\bE\bU}_R  $ and ${_\bE\bH}$ instead of ${_\bE\bH}_R.$
We consider the following commutative diagram of sets \cite[III.3.3.4]{Gi}
$$
\begin{CD}
 && H^1( R, \bH) @>>>  H^1( R, \bH/ \bU) \\
&& @A{torsion}A{\simeq}A @A{torsion}A{\simeq}A \\
 H^1( R, {_\bE\bU}) @>>>  H^1( R, {_\bE\bH}) @>>>  H^1( R, {_\bE\bH} / {_\bE\bU} ) \\
\end{CD}
$$
where the bottom map is an exact sequence of pointed sets. 
 Indeed $\bGL_n$ is the group of automorphisms of the group scheme $\bG_a^n$ (see Lemma \ref{rep} below).
 It follows that 
$_\bE\bU$ corresponds to  a  locally free sheaf over $\Spec(R)$.
By \cite[pp 16-17]{Gr1} (or \cite[III.3.7]{M}), we have  
$\cH^i(R, {_\bE\bU}) = 0$ for all $i>0$.\footnote{ All the $\cH^i$ that we consider  coincide with the corresponding  $H^i$ defined in terms of derived functors.}  So the the map $H^1( R, {_\bE\bH}) \to H^1( R, {_\bE\bH} / {_\bE\bU})$ has trivial kernel
 and the fiber of $ H^1( R, \bH) \to   H^1( R, \bH/ \bU)$ is only $[\bE]$.

For   surjectivity, if we are given a 
 $\bH/ \bU$--torsor $\bE$ over $\Spec(R)$ then by \cite[IV.3.6.1]{Gi}
there is a class 
$$
\Delta([\bE]) \in {\cH}^2(R, {_\bE\bU})
$$
which is the obstruction to the existence of a lift of $[\bE]$ to $H^1(R, \bH)$.
Here $_\bE\bU$ is the $R$--group scheme obtained by twisting $\bG_a^n$ by the $R$--torsor $\bE$. 
Since $_\bE\bU$ corresponds to  a locally free sheaf, 
the same reasoning used above  
shows that the obstruction
 $\Delta([\bE])$ vanishes as desired. 
\end{proof}

\begin{lemma}\label{rep} Let $\gX$ be a scheme of characteristic $0.$ 
Let $\cal E$ be a locally free $\gX$--sheaf of finite rank and  let $\bV(\cal E)$
be the associated ``additive" $\gX$--group scheme.  Then the 
natural homomorphism of fpqc sheaves
$$
\alpha : \GL(\cal E) \to \bAut_{\gX-\text{\it gr}}(\bV(\cal E)) 
$$
is an $\gX$--group sheaf isomorphism. In particular,  $\bAut_{\gX-gr}\big(\bV(\cal E)\big)$ is an $\gX$--group scheme.\end{lemma}

Our convention is that of \cite[\S 2]{DG}, namely
$\bV(\cal E)(\gX')$ $= H^0(\gX', \cal E \otimes_{ {\mathcal O}_\gX }  {\mathcal O}_{X'} )$ for every scheme $\gX'$ over $\gX$.

\begin{proof} 
It is clear that $\alpha$ is a morphism of $\gX$--groups. For showing that $\alpha$ is an isomorphism of sheaves, we may assume that $\gX= \Spec(R)$ is affine and that ${\cal E} = R^n$.
This in turn reduces to the case of $R = \Q$ and ${\cal E}= \Q^n$.
By descent, it will suffice to establish the result for  $R = \ol \Q$ and ${\cal E} = (\ol \Q)^n$.
Now on $\ol \Q$-schemes the functor  $S \mapsto  \bAut_{gr}(\bV({\cal E}) )(S)$
 is representable by a linear algebraic  $\ol \Q$-group $\bH$  according to  Hochschild-Mostow's criterion \cite[th. 3.2]{HM}. Therefore
we can check the fact  that   $\alpha : \bGL_n \to  \bH$ is an isomorphism on $\ol \Q$--points.
But this readily follows from the equivalence of categories 
between nilpotent  Lie algebras and algebraic unipotent groups \cite[\S IV \S2 cor.4.5]{DG}. Since $\GL(\cal E)$ is an $\gX$--group scheme, $\alpha$ is an isomorphism of $\gX$--group schemes.
\end{proof}

\section{Maximal tori of group schemes over the punctured line}

Let $\bG$ be a linear algebraic $k$--group. One of the central results of \cite{CGP} is the existence of maximal tori for twisted groups of the form ${_{\bE}\bG}$ where $[\bE] \in H^1(k[t^{\pm 1}], \bG).$\footnote{If the  characteristic of $k$ is sufficiently large.} This result is used to describe
the nature of torsors over $k[t^{\pm 1}]$ under $\bG.$ In our present work we are ultimately interested in the classification of reductive groups over Laurent polynomial rings when $k$ is  of characteristic $0,$ and applications to infinite
dimensional Lie theory. In understanding twisted forms of $\bG$ the relevant objects are torsors under $\bAut(\bG)$, and not $\bG.$ It is therefore essential to have an analogue of the \cite{CGP} result mentioned above, but for arbitrary twisted groups, not just inner forms.\footnote{$\bAut(\bG)$ need not be an algebraic group. Even if it is, the fact that it need not be connected leads to considerable technical complications (stemming from the fact that, unlike the affine line, the punctured line has non-trivial geometric \'etale coverings). As already mentioned, these difficulties have to be dealt with if one is interested in the study of twisted forms of $\bG_R$ or its Lie algebra.} This is one of the crucial theorems of our paper.

\begin{theorem}\label{punctural}
Let $R=k[t^{\pm 1}]$  where $k$ is a field of characteristic $0.$
Every reductive group scheme $\gG$ over $R$ admits a maximal torus.
\end{theorem}

\begin{corollary} \label{corpunctural}
Let $k$ and $R$ be as above. Let $\gG$ be a
smooth affine group scheme over $R$ whose connected component of the identity $\gG^0$
is reductive. Then

\begin{enumerate}
 \item $H^1_{toral}(R, \gG) = H^1(R, \gG).$
\item If $\gG$ is constant, i.e. $\gG= \bG \times_k R$
for some linear algebraic $k$--group $\bG$, then
$$
H^1_{toral}(R, \gG) = H^1_{loop}(R, \gG) =  H^1(R, \gG).
$$
\end{enumerate}

\end{corollary}

The first assertion is an immediate corollary of the Theorem while
the second then follows  from Corollary \ref{corbogo}.2 and Lemma \ref{galois}. \qed
\medskip

The proof of the Theorem relies on Bruhat-Tits twin buildings and
 Galois  descent considerations. We begin by establishing the following useful reduction.

\begin{lemma}\label{lem-punctural} It suffices to establish Theorem \ref{punctural}
under the assumption that $\gG$ is a  twisted form of a  simple simply connected Chevalley $R$--group.\footnote{The usual  algebraic group literature would use the term ``almost simple" in this situation. We adhere  throughout to  the terminology of \cite{SGA3}. }
\end{lemma}

\begin{proof}  Assume  that Theorem \ref{punctural} holds in the simple simply connected case.
By [XII.4.7.c], there is a natural one-to-one correspondence
between the maximal tori of $\gG$, its adjoint group $\gG_{ad}$ and
those of the simply connected covering $\tilde{\gG}_{ad}$ of $\gG_{ad}$.
We can thus assume  without lost of generality that
$\gG$ is simply connected. By
 [XXIV.5.10] we have
$$
\gG = \prod\limits_{i=1,...,l}  \, \prod\limits_{S_i/R}\gG_i
$$
where each $S_i$ is a connected finite \'etale covering of $R$ and each
$\gG_i$ a simple  simply connected $S_i$--group scheme.
By Demazure's main  theorem, the $S_i$--groups
$\gG_i$ are twisted forms of simple simply connected Chevalley groups.
Since  by Lemma \ref{galois}  $S_i$ is a  Laurent polynomial
ring, our hypothesis implies
that each of the $S_i$--groups $\gG_i$ admits a maximal torus $\gT_i.$
Then our $R$--group
$\gG$ admits the maximal torus $\prod\limits_{i=1,...,l} \,  \prod\limits_{S_i/R} \gT_i$.
\end{proof}

\subsection{Twin buildings}\label{sect-twin}

Throughout this section $k$ denotes a field of characteristic $0.$ We set $R =k[t^{\pm 1}],$ $K = k(t)$ and $\widehat{K} = K((t)).$ For a ``survival kit" on euclidean buildings, we recommend Landvogt's paper \cite{L}.

Let $\tilde{R}$ be  a finite Galois extension of $R$ of the form $\tilde{R}=\tilde{k}[t^{\pm {1 \over n}}]$ where $\tilde{k}/k$ is
a finite Galois extension of $k$ containing all $n$-roots of unity in $\overline{k}.$
Then as we have already seen
$\tilde{\Gamma} :=\Gal(\tilde{R}/R) = \bmu_n(\tilde{k}) \rtimes \Gamma$ where $\Gamma =\Gal(\tilde{k}/k)$.

Set   $\tilde{t}= t^{1 \over n}$ We let $L = \tilde{k}(\, \tilde{t}\, ) = \tilde{k}( \tilde{t}^{^{\,-1}})$, and consider the 
two completions $\widehat{L}_+ = \tilde{k}((\, \tilde{t} \, ))$ and $\widehat{L}_- = \tilde{k}((\tilde{t}^{^{\,-1}}))\,$ of $L$ at $0$ and
$\infty$ respectively, as well as their corresponding valuation rings $\widehat{A}_+ = \tilde{k}[[\, \tilde{t} \,]]$ and  $\widehat{A}_{-} = \tilde{k}[[\tilde{t}^{^{\,-1}}]]$. 

Let $\bG$ be a split simple simply connected group over $k.$ Let  $\bT$ be a maximal split torus of $\bG, \,$ $\bB^+$  a Borel subgroup of 
$\bG$ which contains $\bT,$ and $\bB^{-}$ the corresponding opposite Borel subgroup (which also contains $\bT$).
We denote by $\bW= \bN_\bG(\bT)$ the corresponding Weyl group and
by $\Delta_\pm$ the Dynkin diagram attached to  $(\bG, \bB^{\pm},\bT)$.

Following Tits \cite{T3}, we consider the twin building
 ${\cal B}= {\cal B}_+ \times {\cal B}_{-}$ of $\bG_{0} \times_k L$ with respect to
the two completions $\widehat{L}_+$ and $\widehat{L}_-$. Recall that
 $\cal B$ comes equipped with an action of the group
$\bG(L)$, hence also of $\bG(\tilde{R}).$ The split torus $\bT_{0} \times_k L$
gives rise to a twin apartment ${\cal A}= {\cal A}_+ \times {\cal A}_{-}$ of ${\cal B}$.
The Borel subgroups $\bB^\pm$  define the fundamental chambers   ${\cal C}_\pm$ of  ${\cal A}_\pm$, each of which
is an open simplex whose vertices are given by the extended Dynkin diagram 
$\tilde \Delta_\pm$ of  $\Delta_\pm$.

Recall that the group functor $\bAut(\bG)$ is an affine group scheme. The group $\bAut(\bG)(L)$ acts on $\cal B$ by ``transport of
structure" \cite{L} 1.3.4.\footnote{The group in question acts on the set of maximal split tori, hence permutes the apartments around.} This leads to an action of $\bG(L)$ on $\cal B$ via 
 $\Int : \bG \to \bAut(\bG).$ This  action
coincides with the ``standard"  action of $\bG(L)$ mentioned
before because $\bG$ is semisimple. By taking into account the
natural action of $\tilde{\Gamma} \simeq \Gal(\widehat{L}_{+}/\widehat{K})$
on $\cal B$ we conclude
that the twin building ${\cal B}$ is equipped with an action of the
semi-direct product ${\Aut(\bG)(\tilde{R}) \rtimes \Gamma}$ which is 
compatible (via the adjoint action) with the action of
$\bG(\tilde{R}).$

The hyperspecial  group $\bG(\widehat A_\pm)$ fixes a unique point
$\phi_{\pm}$ of ${\cal A}_{\pm}$ \cite[\S 9.1.19.c]{BT1}. 
Recall that the hyperspecial points of ${\cal B}_\pm$ are $\bG(\widehat{L}_\pm)$-conjugate to 
$\phi_\pm$ of   ${\cal B}_\pm$, and can therefore be identified with the set of left cosets
$$
\bG(\widehat{L}_\pm) / \bG(\widehat{A}_\pm)  \simeq \bG(\widehat{L}_\pm) . \, \phi_\pm  \, \subset \, {\cal B}_\pm  . 
$$
More generally each facet of the building ${\cal B}_\pm$ has a type \cite[\S2.1.1 ]{BT1}
which is a  subset of $\tilde \Delta_\pm$ and the  type of a point $x \in {\cal B}_\pm$ 
is the type of its underlying facet $F_x.$\footnote{Namely the smallest facet containing $x$ in its closure.}
The type of the chamber ${\cal C}_\pm$ is $\emptyset$ and the type of 
an hyperspecial point is $\tilde \Delta_\pm \setminus \Delta_\pm,$ namely the extra vertex of the affine Dynkin diagram.

\subsection{Proof of Theorem \ref{punctural}}

By Lemma \ref{lem-punctural}, we can assume that $\gG$ is simple simply connected.
By the Isotriviality Theorem
\cite[cor. 2.16]{GP1}, we know that our  $R$--group $\gG$ is
isotrivial. This means that there exists a finite Galois covering
$S/ R$ and  a ``trivialization" $f: \bG \times_k S \simeq \gG
\times_R S$ where $\bG$ is a  split simple simply connected  $k$--group. In our terminology,  $\bG$ is the Chevalley $k$--form of $\gG.$

Because of the structure of the algebraic fundamental group of $R$
we may assume without loss of generality that $S = \tilde{R}$ is as in \S \ref{sect-twin}, and we keep all the notation therein. What is so special about this situation is that
$R$ and $\tilde{R}$ ``look the same", namely they are both Laurent polynomial rings in one variable with coefficients in a field.

We have  $\Spec(\tilde{R}) =  {\Bbb P}^1_{\tilde{k}} \setminus \{0, \infty\}$ and the action of $\tilde{\Gamma}$ on $\tilde{R}$ extends to ${\Bbb P}^1_{\tilde{k}}$ since $\tilde{R}$ is regular of dimension $1$. 

For $\tilde{\gamma} \in \tilde{\Gamma}$ consider the map
$z_{\tilde\gamma}= f^{-1} \circ \, {^{\tilde\gamma} f}: \tilde\Gamma \to \bAut(\bG)(\tilde{R})$,
where $\bAut(\bG)$ stands for the group scheme of automorphisms of
the $\Z$-group $\bG$.
  Then
$z = (z_{\tilde\gamma})_{\tilde\gamma \in \tilde\Gamma}$ is a cocycle in $Z^1\big(\tilde\Gamma,
\bAut(\bG)(\tilde{R})\big)$ where the Galois group $\tilde\Gamma$ acts
naturally on $\bAut(\bG)(\tilde{R})$ via its action on $\tilde{R}$. Descent theory tells us  that $\bG $ is isomorphic
to the twisted $R$-group ${_z\bG} $.\footnote{Recall that for
convenience ${_z\bG}$ is shorthand notation for 
${_z(\bG \times_{k} R}) = {_z(\bG_R})$.}

The action of $\tilde\Gamma$  on $\bAut(\bG)(\tilde{R})$ allows   us  to consider
the semidirect product group $\bAut(\bG)(\tilde{R}) \rtimes \tilde\Gamma.$
We then
have a group homomorphism
\begin{equation}\label{psi}
\psi_z: \tilde\Gamma \to
 \bAut(\bG)(\tilde{R}) \rtimes \tilde\Gamma
 \end{equation}
  given by 
$\psi_z(\gamma)= z_\gamma \, \gamma$ which is a section 
of the projection map $\bAut(\bG)(\tilde{R}) \rtimes \tilde\Gamma \to \tilde\Gamma$.

Let  $\bT$ be a maximal split of $\bG$.
Set $L= \tilde k( \tilde t),$  and let $A_{+}$ (resp. $A_{-}$) be the local
 ring of ${\Bbb P}^1_{\tilde{k}}$ at $0$ (resp. $\infty$).
The composite map (see \S 5.1)
$$
\tilde\Gamma \buildrel \psi_z \over \longrightarrow \bAut(\bG)(\tilde{R}) \rtimes \tilde\Gamma \to \Aut( {\cal B})
$$
is a group homomorphism. The corresponding action of $\tilde\Gamma$ on $\cal{B}$ will be referred to as the  {\it twisted action} of $\tilde\Gamma$ on the building.
 We now appeal to  the Bruhat-Tits
 fixed point theorem \cite[\S 3.2]{BT1} to obtain a point
$p=(p_+,p_{-} ) \in {\cal B}$ 
 which is fixed under the twisted action, i.e. $\psi_z(\tilde\gamma).p=z_{\tilde\gamma} \, \tilde\gamma(p)=p$ for all $\tilde\gamma \in \tilde\Gamma$.
Abramenko's result \cite[Proposition 5]{A} states that $\bG(\tilde{R}).  {\cal A} =  {\cal B}.$
  There thus exists $g \in \bG(\tilde{R})$
such that $p$ belongs to the apartment $g . {\cal A}$. Up to replacing
$z_{\tilde\gamma}$ by $\Int(g)^{-1} \, z_{\tilde\gamma} \, \Int(\tilde\gamma(g))$,
we can therefore assume that $p$ belongs to ${\cal A}$.

We shall use several times that $\tilde\Gamma$ acts trivially on ${\cal A}$ under the standard action.
To see this one reduces to the action of $\widetilde \Gamma$  on  
${\cal A}_+= \phi_+ + \widehat T \otimes_\Z \R$.
Firstly  $\widetilde \Gamma$ stabilizes the group  $\bG(\widehat A_+)$ so it  fixes $\phi_+$.
Secondly it acts trivially on
$\widehat T$ so acts trivially on ${\cal A}_+$.

Observe that since $\tilde\gamma(p)=p$, we have that $z_{\tilde\gamma} \, . \, p=p$ for all
$\tilde\gamma \in \tilde\Gamma$.

Let $F_{p_\pm}$ be the facet associated to $p_\pm$ and choose a vertex $q_\pm$ of
$\overline F_{p_\pm}$. The transforms of  $q_\pm$ by $z_{\tilde\gamma}$ and $\tilde\gamma$
are vertices of $\overline F_{p_\pm}$, so $\psi_z(\tilde\gamma).q_\pm$
belongs  to ${\cal A}$.
We define
$$
x_\pm = {\rm Barycentre}\Bigl(  \psi_z(\tilde\gamma).q_\pm, \tilde\gamma \in \tilde\Gamma \Bigr) \in
{\cal A},
$$
where the barycentre stands for the riemannian's one as defined by Pansu \cite[\S 4.2]{P}.

Let $d$ be the  integer attached to $\bG$  in \cite[\S 2]{Gi1}, and set
$m= d  \mid \! \tilde\Gamma \! \mid$.
Let $s \in \overline L$ (a fixed algebraic closure of $L$) be such that $s^m = \tilde{t}.$  We have accordingly $s^{mn} = t.$ Set
$$
R'= k'[\, s^{\pm 1} \,]
$$
where $k'$ is a Galois extension of $k$ which contains
$\tilde{k}$ and all $mn$-roots of unity in $\overline{k}.$
Then $R'$ is Galois over $R$ of Galois group
$$
\Gamma'= \bmu_{mn}(k') \rtimes \Gal(k'/k).
$$
By Galois theory the map $\upsilon : \Gamma' \to \tilde\Gamma$ given by
\begin{equation}\label{Galoismap}
\upsilon : (\xi, \theta) \mapsto (\xi^m, \theta_{| \tilde{k}})
\end{equation}
is a surjective group homomorphism.

We consider the twin building ${\cal B'}= {\cal B'}_+ \times
{\cal B'}_{-} $  of $\bG$
which is constructed in the manner described above after replacing, mutatis mutandis, the relevant objects attached to $\tilde{R}$ by those of  $R'$ .  We have
a restriction map  \cite[\S II.4]{Ro}
$\rho_\pm : {\cal B}_\pm \to  {\cal B'}_\pm$
which gives rise to
$$
\rho = (\rho_+, \rho_{-}): {\cal B} \to {\cal B'}.
$$
Furthermore, if $\gamma' \in \Gamma'$ and we set $\upsilon(\gamma') = \tilde\gamma$ then the following diagram commutes
$$
\begin{array}{lc}
   \big{(*)\ \ \ }&
\begin{CD}
\cal{B} & \enskip \buildrel \rho \over \lgr \enskip  & \cal{B'}\\
@V{\tilde\gamma}VV @V{\gamma'}VV \\
\cal{B} & \enskip \buildrel \rho \over \lgr \enskip  & \cal{B'}\ \\
\end{CD}
\end{array}
$$
where the actions of $\tilde\Gamma$ and $\Gamma'$ are the twisted actions. If we define $z' : \Gamma' \to \bAut(\bG)(R')$ by
$$z': \gamma' \mapsto z'_{\gamma'} = z_{\tilde\gamma} \in \bAut(\bG)(\tilde{R}) \subset \bAut(\bG)(R')$$
then $z'$ is a cocycle and the classes $[z]$ and $[z']$  in $H^1\big(R, \bAut(\bG)\big)$ are the same.

\begin{lemma} $\rho_\pm(x_\pm)$ is a  hyperspecial point of  ${\cal B'}_\pm$.
\end{lemma}

\begin{proof} We look at the case $\rho_{+}(x_+)$.
 We need to consider the intermediate extensions
$$
\widehat L_+= \tilde{k}(( \, {\tilde{t}}\, )) \enskip \subset  \enskip
k'(( \, \tilde{t} \,)) \enskip   \subset  \enskip  k'(( \, s^d\, ))
\enskip \subset  \enskip  k'(( \,s\, )).
$$
The map $\rho_{+}$ is the composite of the corresponding maps for the intermediate fields, namely
$$
{\cal B}_+ = {\cal B}_+ (\bG_{0,\,  \tilde{k}(( \, {\tilde{t}}\, ))} )\enskip  \buildrel \rho_{1,+} \over \lgr  \enskip  {\cal B}\bigl(\bG_{0, \, k'(( \, \tilde{t} \,))} \bigr) \enskip
\buildrel \rho_{d,+} \over \lgr  \enskip  {\cal B}\bigl(\bG_{0, \,   k'(( \, s^d\, )) } \bigr)  \enskip
 \buildrel \rho_{\frac{m}{d},+} \over \lgr  \enskip
{\cal B}\bigl(\bG_{0, \, k'(( \,s\, ))} \bigr){\cal B'}_+ .
$$
The first map does not change the type.
By \cite[lemma 2.2.a]{Gi1}, the image under $\rho_{d,+}$ of any vertex is
a  hyperspecial point. We have

\begin{eqnarray} \nonumber 
\rho_+( x_+) & = &
\rho_+\Bigl[ {\rm Barycenter}\Bigl(  \psi_z(\tilde\gamma).q_\pm, \gamma \in \tilde\Gamma\Bigr) \Bigr] \\ \nonumber
&= & \rho_{\frac{m}{d},+}\Bigl[ {\rm Barycenter}\Bigl(  \rho_{d,+} \circ \rho_{1,+}(\psi_z(\tilde\gamma).q_\pm), \gamma \in \tilde\Gamma\Bigr) \Bigr].
\end{eqnarray}

\noindent By \cite[ Lemma 2.3' in the errata]{Gi1},
we know that the image under $\rho_{\frac{m}{d}, +}$ of the barycentre
of $\frac{m}{d}$ hyperspecial points of a  common apartment (namely the one attached to the torus $\bT$) is  a hyperspecial point, so we conclude that
$\rho_+( x_+)$ is a  hyperspecial point.
\end{proof}

In view of diagram (*) above it follows that by  replacing $\tilde{R}$ by $R'$ we may  assume
without loss of generality
that the points $p_\pm \in {\cal A}_\pm$ are hyperspecial. Note that by construction, the points $\rho_\pm( x_\pm)$ of  ${\cal B'}_{\pm}$
are fixed by both actions (standard and twisted)  of $\Gamma'$, so that after our further extension of base  ring we may assume that
$\rho_\pm( x_\pm)$ of  ${\cal B}_{\pm}$ are fixed by both actions of $\tilde\Gamma.$

Since $\bT(\tilde{R})$ acts transitively on the sets
$ \phi_{\pm}  + (\widehat T)^0 \subset
 {\cal A}_{\pm}$ of hyperspecial points of  ${\cal A}_{\pm}$, there exists $g \in \bT(\tilde{R})$ and a cocharacter
$\lambda \in (\widehat {\bT})^0 $ such that
$$
g . (\psi_+, x^\lambda_{-} ) = (x_{+}, x_{-})=x.
$$
where  $ x^\lambda_{-} := \phi_{-} + \lambda$ (recall that we have a map
$(\widehat {\bT})^0  \to {\cal A}_{-}=  (\widehat {\bT})^0 \otimes_\Z \R$ defined by
$\theta \mapsto \phi_{-}  + \theta$). Up to replacing
the cocycle $z$ by $z'$ where $z'_{\tilde\gamma} = {\rm Int}(g)^{-1} \, z_{\tilde\gamma} \, {^{{\tilde\gamma}}}{\rm Int}( g)$, we
may assume that $\psi_z(\tilde{\gamma}). (\phi_+,  x^\lambda_{-} )= 
 z_\gamma . (\phi_+, x^\lambda_{-} ) = (\phi_+, x^\lambda_{-} )$ for every
 $\tilde\gamma \in \tilde\Gamma$.
In particular, $$
z_{\tilde\gamma} \in \Stab_{\bAut(\bG)(\widehat L_+)}(\phi_+)
= \bAut(\bG)(\widehat A_+)
$$
hence
$$
z_{\tilde\gamma} \in \,  \bAut(\bG)(\tilde{R}) \cap \bAut(\bG)(\widehat A_+)\bAut(\bG)( \tilde{k}[\, \tilde {t} \,])
$$
for each $\tilde\gamma \in \tilde\Gamma$.
Let $g_\lambda:=\lambda(\, \tilde{t} \,) \in \bT(\tilde{R}) \subset  \bG(\tilde{R})$.
We have $g_\lambda. \phi_{-} =  x^\lambda_{-}$ and therefore
$$
z_{\tilde{\gamma}} \, \in \,  \Stab_{\bAut(\bG)(\widehat L_{-})}(x^\lambda_{-}) 
$$
$$
{\rm Int}(g_\lambda)\,  \Stab_{\bAut(\bG)(\widehat L_{-})}(\phi_{-}) \,{\rm Int}(g_\lambda)^{-1}
= {\rm Int}(g_\lambda) \bAut(\bG)(\widehat A_{-}) {\rm Int}(g_\lambda^{-1}) .
$$
It follows that for each $\tilde\gamma \in \tilde\Gamma$
 \begin{eqnarray} \nonumber
z_{\tilde\gamma} \in J_\lambda & := &\bAut(\bG)( \tilde{k}[\, \tilde{t} \, ]) \, \cap
\, {\rm Int}(g_\lambda) \, \bAut(\bG)(\widehat A_{-})
 \, {\rm Int}(g_\lambda^{-1}) \\ \nonumber
& = & \bAut(\bG)( k[\, \tilde{t} \,]) \, \cap \,
 {\rm Int}(g_\lambda) \, \bAut(\bG)(\tilde{k}[ \, {\tilde{t}}\,^{^{-1}}\, ])
\, {\rm Int}(g_\lambda^{-1}) .
\end{eqnarray}

Note that for $\tilde\gamma \in \tilde\Gamma$ we have
$$
{^ {\tilde\gamma}} g_{\lambda} = {^ {\tilde\gamma}}\lambda(t') = \lambda(t') \, \big(\lambda(t')^{-1} \, {^{\tilde \gamma}}\lambda(t')\big) \in \lambda(t')\bT(\tilde{k}) \subset \lambda(t')\bG(\widehat{A}_{-}) = g_{\lambda}\, \bG(\widehat{A}_{-})
$$
\noindent From this it follows not only that the  subgroup $J_{\lambda}$ of
 $\bAut(\bG)(\tilde{R})$ is  stable under the (standard) action of
$\tilde\Gamma,$ but also that
$$
[z] \in {\rm Im} \Bigl( H^1\bigl(\tilde\Gamma, J_\lambda \bigr)
 \to H^1\bigl(\tilde\Gamma, \bAut(\bG)(\tilde{R}) \bigr) \Bigr).
$$
It turns out that  the structure of the group $J_\lambda$
 is known by a computation of Ramanathan, as we shall see in Proposition \ref{global} below.
We have
$$
J_\lambda = \bU_\lambda(\tilde{k}) \rtimes
\bZ_{\bAut(\bG)} (\lambda) (\tilde{k}) \,  \subset \, \bAut(\bG) ( \tilde{k}[\, \tilde{t} \,]) ,
$$
where $\bU_\lambda$ is a  unipotent $k$--group.  Lemma \ref{sansuc}
shows that the map
$$
H^1\bigl(\tilde\Gamma,  \bZ_{\bAut(\bG)} (\lambda) (\, \tilde{k} \,) \bigr)
\to H^1\bigl(\tilde\Gamma, J_\lambda \bigr)
$$
is bijective.
Summarizing, we have the commutative diagram
$$
\begin{CD}
H^1\bigl(\tilde\Gamma, \bZ_{\bAut(\bG)} (\lambda) (\, \tilde{k} \,) \bigr) @>>>
H^1\bigl(\tilde\Gamma, \bAut(\bG) (\, \tilde{k} \,) \bigr) \\
@VV{\cong}V @VVV \\
H^1\bigl(\tilde\Gamma, J_\lambda \bigr) @>>> H^1\bigl(\tilde\Gamma, \bAut(\bG)(\tilde{R}) \bigr) \\
\end{CD}
$$
which shows that $$
[z] \in {\rm Im} \Bigl( H^1\bigl(\tilde\Gamma, \bAut(\bG) (\, \tilde{k} \,) \bigr)
 \to H^1\bigl(\tilde\Gamma, \bAut(\bG)(\tilde{R}) \bigr) \Bigr).
$$
This means that  $[z]$ is cohomologous to a loop cocycle, and we can now conclude by Proposition \ref{existenceoftori} \qed

\section{ Internal characterization of loop torsors and applications}

We continue to assume that  our base field $k$ of characteristic zero. Let $R_n = k[t_1^{\pm 1},...,t_n^{\pm 1}]$
and $\gX = \Spec(R_n).$ As explained in Example \ref{FGLaurent} we have
$\pi_1(R_n, a) \simeq \widehat {\Bbb Z}^n \rtimes \Gal(k),$ where the action of $\Gal(k)$ on $\widehat{\Z}^n$ is given by our fixed choice of compatible roots of unity in $\overline{k}$. For convenience in what follows we will denote $\pi_1(R_n, a) $ simply by $\pi_1(R_n).$

Throughout this section  $\bG$ denotes a linear algebraic $k$--group.

\subsection{Internal characterization of loop torsors}
We first observe that loop torsors make sense over 
$R_0=k$, namely $H^1_{loop}(R_0, \bG)$ is the usual Galois cohomology $H^1(k, \bG).$

Section \ref{basepoint} shows that $Z^1\big( \pi_1(R_n), \bG(\ol k)\big)$
is given by couples $(z, \eta^{geo})$ where
$z \in Z^1( \Gal(k), \bG(\ol k))$
and $\eta^{geo} \in \Hom_{\Gal(k)}( \pi_1(\ol X ,\ol a), {_z\bG}(\ol k)) = \Hom_{\Gal(k)}(\widehat {\Bbb Z}^n , {_z\bG}(\ol k)) \simeq \Hom_{\Gal(k)}({_\infty}\bmu , {_z\bG})$. We are now ready to state and establish the
 internal characterization of $k$-loop torsors as toral
torsors.

\begin{theorem}\label{jo} Assume that $\bG^0$ is reductive. Then
$H^1_{toral}(R_n, \bG) = H^1_{loop}(R_n, \bG)$.
\end{theorem}

\bigskip

First we establish an auxiliary useful result.

\begin{lemma}\label{tori1} 1) Let $\gH$ be an $R_n$ group of multiplicative type. Then for all $i \geq1$ the natural abstract group homomorphisms
$$H^i\big(\pi_1(R_n) , {\gH}(\ol R_{n,\infty})\bigr)
\to H^i(R_n, \, \gH) \to  H^i(F_n, \,  \gH).
$$
are all isomorphisms.

2) Let $\bT$ be a $k$--torus. Let $c \in Z^1\big(\pi_1(R_n), \bAut(\bT)(\ol k)\big) \subset 
\newline Z^1\big(\pi_1(R_n) ,  \bAut(\bT)(\ol R_{n,\infty})\big)$ be a cocycle, and consider the twisted
$R_n$--torus $_c\bT = {_c}(\bT \times_k R_n).$ Consider the  natural maps
$$
H^i\big(\pi_1(R_n) ,\,  _c(\bT(\ol k)\big) \to
H^i\big(\pi_1(R_n) , {_c\bT}(\ol R_{n,\infty})\bigr)
\to H^i(R_n, \, _c\!\bT) \to  H^i(F_n, \,  _c\!\bT)
$$
Then.

(i) If $i = 1$ then the first group homomorphism is surjective and the last one is an isomorphism.

(ii) If $i > 1$ then all the maps are group isomorphisms.

\end{lemma}

\begin{proof} 1) The second isomorphism is proposition 3.4.3 of \cite{GP2}. As for the first isomorphism we consider the Hochschild-Serre spectral sequence \newline
$H^p\big(\pi_1(R_n), H^q(\ol{R}_{n, \infty}, \gH)\big) \implies H^{p + q}(R_n, \gH).$ From the fact that the group $H^p(\ol{R}_{n, \infty}, \gH)$ is torsion for $p \geq 1$ it follows  that the map $\limind_m H^p(\ol{R}_{n, \infty}, {_m}\gH) \to H^p(\ol{R}_{n, \infty}, \gH)$, where  ${_m}\gH$ stands for the kernel of  the ``multiplication by $m$" map, is surjective. By {\it loc. cit.} cor. 3.3 $H^p(\ol{R}_{n, \infty}, {_m}\gH)$ vanishes for all $m \geq 1.$ Hence $H^p(\ol{R}_{n, \infty}, \gH) = 0.$ The spectral sequence degenerates and yields the isomorphisms $H^i\big(\pi_1(R_n) , {\gH}(\ol R_{n,\infty})\bigr)
\simeq H^i(R_n, \, \gH)$ for all $i \leq 1.$

2) We begin with an observation about the notation used in the statement of the Lemma. The subgroup $\bT(\overline{k})$ of $\bT(\ol R_{n,\infty})$  is stable under the (twisted) action of $\pi_1(R_n) $ on ${_c\bT}(\ol R_{n,\infty}).$  To view $\bT(\overline{k})$ as a $\pi_1(R_n)$--module
with this twisted action we write ${_c}(\bT(\overline{k})).$

The fact that the last two maps are  isomorphism for all $i \geq 1$  is
a special case (1).  For the first map, we first analyse the
 $\pi_1(R_n)$--module  $A= \bT( \ol R_{n,\infty})/
 \bT( \ol k)$.
We have
\begin{eqnarray} \nonumber
A & = & \limind\limits_m \enskip  \bT( \ol R_{n,\infty}) \, / \,  \bT( \ol k)
\\ \nonumber
 & = & \limind\limits_m  \enskip
(\widehat \bT)^0 \otimes_\Z   \ol R_{n,m}^\times \,
/   (\widehat \bT)^0 \otimes_\Z   {\ol k}^\times \\ \nonumber
 & = &  (\widehat \bT)^0 \otimes_\Z  \,
\limind\limits_m   \enskip  \ol R_{n,m}^\times
\, / \,  {\ol k}^\times \\ \nonumber
 & = &  (\widehat \bT)^0 \otimes_\Z \,
\limind\limits_m \enskip   (\Z^n)_m  \,\,\text{\rm (where} \,\, (\Z^n)_m =  \Z^n {\rm )}\\ \nonumber
 & = &  (\widehat \bT)^0 \otimes_\Z     \Q^n
\end{eqnarray}
\noindent given that the transition map  $(\Z^n)_m \to (\Z^n)_{md}$ is
multiplication by $d$.
It follows that $A,$ hence also ${_cA},$  is uniquely divisible.

We consider the sequence of continuous $\pi_1(R_n)$--modules

\begin{equation}\label{divisible}
1 \to  {_c}\big(\bT(\ol k)\big) \to
\,  {_c\bT}(\ol R_{n,\infty}) \to \, {_cA} \to 1.
\end{equation}
From the fact that ${_cA}$  is uniquely divisible it follows that the group homomorphisms
\begin{equation}\label{maps}
H^i\big(\pi_1(R_n) ,\,  {_c}(\bT(\ol k))\big) \to
H^i\big(\pi_1(R_n) , \,  {_c\bT}(\ol R_{n,\infty})\big)
\end{equation}
are surjective for all $i \geq 1$ and bijective if $i > 1.$

\end{proof}

We can proceed now with the proof of Theorem \ref{jo}.

\begin{proof}
Let us show first show that
$H^1_{loop}(R_n, \bG) \subset H^1_{toral}(R_n, \bG)$.

\smallskip

\noindent{\it Case 1: $\bG=\bAut(\bH_0)$ where $\bH_0$ is a semisimple Chevalley 
$k$--group :} Let $\phi: \pi_1(R_n) \to \bAut(\bH_0)(\overline{k})$ 
be a loop cocycle. Consider the twisted $R$--group $_\phi\bG.$\footnote{We recall, for the last time, that $_\phi\bG$ is short hand notation  for $_\phi(\bG_R).$} Proposition \ref{existenceoftori} shows that
 the connected component of the identity  $(_\phi\bG)^0=  {_\phi}(\bG^0)$ of $_\phi\bG$
admits a   maximal $R$-torus. Therefore  $_\phi\bG$ admits a maximal $R$-torus, hence
$\phi$ defines a toral $R$-torsor.

\smallskip

\noindent{\it Case 2: $\bG=\bAut(\bH)$ where  $\bH$ is a  semisimple 
$k$--group :} Denote by $\bH_0$ the Chevalley $k$--form of $\bH$.
There exists a cocycle $z: \Gal(k) \to \bG(\overline{k})$ such that
$\bH$ is isomorphic to the twisted $k$--group ${_z\bH_0}$.
We can assume then that  $\bH={_z\bH_0}$ and
 $\bG={_z\bAut(\bH_0)}$.
The torsion bijection $
\tau_z: H^1(R, \bG)=  H^1(R, {_z\bAut(\bH_0)})  \simlgr
H^1(R, \bAut(H_0) )
$
exchanges loop classes (resp. toral classes) according to Remark
 \ref{klooptorsion}. Case 1 then yields 
$H^1_{loop}(R_n, \bG) \subset H^1_{toral}(R_n, \bG)$.

\smallskip

\noindent{\it General case.} The $k$--group acts by conjugacy 
on $\bG^0$, its center $\bZ(\bG^0)$ and then on its  adjoint quotient
$\bG^0_{ad}$. Denote by 
$f: \bG \to \bAut(\bG^0_{ad})$ this action.
Let $\phi : \pi_1(R_n)  \to \bG(\overline{k})$ be a loop cocycle.
We have to show that the twisted $R$--group scheme $_\phi\bG$
admits a   maximal torus. Equivalently, we need to show that
 $(_\phi\bG)^0 =  _\phi(\bG^0)$ admits a maximal torus
which is in turn  equivalent to the fact that 
  $(_{f_*\phi}\bG)^0_{ad}={_{f_*\phi}(\bG^0_{ad})}$ 
admits a maximal torus \cite[XII.4.7]{SGA3}.
But $f_*\phi$ is a loop cocycle for $\bAut(\bG^0_{ad})$, so defines 
a toral $R$--torsor under $\bAut(\bG^0)$ according to Case 2.
Thus $_{f_*\phi}(\bG^0_{ad})$ admits a maximal torus as desired.

\medskip

To establish the reverse inclusion we consider the  quotient group $\bnu= \bG/\bG^0,$ which is a finite and \'etale $k$--group. In particular $\bnu \times_k \ol{k}$ is constant and finite, and one can
easily see as a consequence that the natural map $\bnu(\ol{k}) \to \bnu(\ol R_{n,\infty})$  is an isomorphism. We first establish the result for tori and then the general case.

\smallskip

\noindent{\it $\bG^0$ is a torus $\bT$:}
We again appeal to the isotriviality theorem of \cite{GP1} to see that
$H^1( \pi_1(R_n), \bG( \ol R_{n,\infty})) \simlgr  H^1(R_{n}, \bG)$.
We consider the following commutative diagram of continuous
$\pi_1(R_n)$--groups
$$
\begin{CD}
1 @>>> \bT(\ol k) @>>> \bG(\ol k) @>>> \bnu(\ol k) @>>>  1 \\
&&@VVV @VVV \mid \mid \\
1 @>>> \bT(\ol  R_{n,\infty}) @>>> \bG( \ol R_{n,\infty}) @>>> \bnu(\ol R_{n,\infty}) @>>>  1. \\
\end{CD}
$$
This gives rise to an exact sequence of pointed sets
$$
\begin{CD}
1 @>>> H^1\bigl( \pi_1(R_n),  \bT(\ol k) \bigr) @>>>
H^1\bigl( \pi_1(R_n), \bG(\ol k)\bigr) @>>>
H^1\bigl( \pi_1(R_n),\bnu(\ol k) \bigr)  \\
&&@VVV @VVV \mid \mid \\
1 @>>> H^1\bigl( \pi_1(R_n),\bT(\ol  R_{n,\infty})\bigr)
@>>>  H^1\bigl( \pi_1(R_n), \bG( \ol R_{n,\infty})\bigr) @>>>
 H^1\bigl( \pi_1(R_n), \bnu(\ol R_{n,\infty}) \bigr) \,. \\
\end{CD}
$$
We are given a cocycle
 $z \in  Z^1\bigl( \pi_1(R_n), \bG( \ol R_{n,\infty})\bigr) .$
Denote by $c$ the image of $z$ in  $Z^1\bigl( \pi_1(R_n),\bnu(\ol k) \bigr) = Z^1\bigl( \pi_1(R_n),\bnu(\ol R_{n,\infty})\bigr)$ under the bottom map.
Under the top map, the obstruction to lifting $[c]$ to $H^1\bigl( \pi_1(R_n), \bG(\ol k)\bigr)$
is given by a class $\Delta([c]) \in
H^2\bigl(\pi_1(R_n) ,\,  _c\!(\bT(\ol k)) \bigr)$ \cite[\S I.5.6]{Se}.
This class vanishes in
$H^2\bigl(\pi_1(R_n), \,  {_c}\bT(\ol R_{n,\infty})) \bigr)$, so
Lemma \ref{tori1} shows that
 $\Delta([c])=0$. Hence $c$ lifts to a loop
 cocycle $a \in  Z^1\bigl( \pi_1(R_n), \bG( \ol k)\bigr)$.
By twisting by $a$ we obtain the following commutative diagram of pointed sets
$$
\begin{CD}
1 @>>> H^1\Bigl( \pi_1(R_n),\,   {_a}\big(\bT(\ol k)\big) \,  \Bigr) @>>>
H^1\Bigl( \pi_1(R_n),\,  {_a}\big(\bG(\ol k)\big) \, \Bigr) @>>>
H^1\Bigl( \pi_1(R_n),\,  {_c}\big(\bnu(\ol k)\big) \, \Bigr)  \\
&&@V{}VV @VVV \mid \mid \\
1 @>>> H^1\Bigl( \pi_1(R_n),\, { _a}\big(\bT(\ol  R_{n,\infty})\big) \, \Bigr)
@>>>  H^1\Bigl( \pi_1(R_n), {_a}\big(\bG( \ol R_{n,\infty})\big) \Bigr) @>>>
 H^1\Bigl( \pi_1(R_n),\,  {_c}\big(\bnu(\ol R_{n,\infty})\big) \,  \Bigr) \, \\
&& @VVV \\
&& 1
\end{CD}
$$
where the surjectivity of the left map comes from Lemma \ref{tori1},  ${_a}(\bG(\overline{k}))$ denotes $\bG(\overline{k})$ as a $\pi_1(R_n)$--submodule of ${_a}\bG( \ol R_{n,\infty}),$ and similarly for ${_c}\big(\bnu)(\overline{k})\big) = {_c}\big(\bnu(\ol R_{n,\infty})\big).$
We consider the  torsion map
$$
\tau_a:  H^1\bigl( \pi_1(R_n), {_a\bG}( \ol R_{n,\infty}) \bigr) \simlgr
 H^1\bigl( \pi_1(R_n), {\bG}( \ol R_{n,\infty}) \bigr).
$$
Then
$$
\tau_a^{-1}( [z]) \in \ker\Bigl(  H^1\bigl( \pi_1(R_n),\,  {_a\bG}( \ol R_{n,\infty}) \,  \bigr) \to
 H^1\bigl( \pi_1(R_n), \,  {_c}\big(\bnu(\ol R_{n,\infty})\big)  \Bigr).
$$
The diagram above shows that
$\tau_a^{-1}( [z])$ comes from
$H^1\bigl( \pi_1(R_n),\,  _a\!(\bG(\ol k)) \, \bigr)$, hence
$[z]$ comes from $H^1\bigl( \pi_1(R_n), \bG(\ol k)\bigr)$ as desired.
We conclude that   $H^1(R_{n}, \bG)$ is covered by loop torsors.

\smallskip
\noindent{\it General case : } Let $\bT$ be a maximal torus of $\bG$. Consider the commutative diagram
$$
\begin{CD} 
H^1_{loop}\big(R_n,\bN_\bG(T)\big) @>>> H^1\big(R_n,\bN_\bG(\bT)\big) \\
@VVV   @VVV \\
H^1_{loop}(R_n,\bG) @>>> H^1_{toral}(R_n,\bG). \\
\end{CD}
$$
The right vertical map is surjective according to Lemma \ref{toral}.1. The
top horizontal map is surjective by the previous case.
We conclude that the bottom horizontal map is surjective as desired.
\end{proof}

\begin{corollary}\label{corjo}
Let $\gG$ be a reductive $R$--group. 
Then $\gG$ is loop reductive if and only if $\gG$ admits a maximal 
torus.
\end{corollary}

\begin{proof}    
Let $\bG$ be the  Chevalley $k$--form of $\gG$.
 To $\gG$ corresponds  a class $[\bE] \in H^1\big(R_n, \bAut(\bG)\big).$ 
When  $\bG$ is semisimple,  $\bAut(\bG)$ is an  affine algebraic $k$--group and
the Corollary follows 
from  $H^1_{toral}\big(R_n, \bAut(\bG)\big) = H^1_{loop}\big(R_n, \bAut(\bG)\big)$.

We deal now with the general case. We already know tby \ref{existenceoftori} that every loop reductive group  is toral.  Conversely assume that 
$\gG$ admits a maximal  torus.
Consider the exact sequence of smooth $k$--groups [XXIV.1.3.(iii)]
$$
1 \to \bG_{ad} \to \bAut(\bG) \buildrel p \over \lgr  \bOut(\bG) \to 1,
$$
where $\bG_{ad}$ is the adjoint group of $\bG$ and $\bOut(\bG)$ is a constant $k$--group. Since $R_n$ is a noetherian normal domain,
we know that $R_n$--torsors under $\bOut(\bG)$ are isotrivial  [X.6]. Furthermore by \cite[XI \S5]{SGA1}
$$
{\rm Hom}_{ct}\big( \pi_1(R_n), \bOut(\bG)(k)\big) / \sim \enskip \simlgr \enskip H^1\big(R_n, \bOut(\bG)\big).
$$
So $p_*[\bE]$ is given by a continous homomorphism $\pi_1(R_n) \to  \bOut(\bG)(k)$ whose image 
we denote by $\bOut(\bG)^\sharp$. This is a finite group so that
 $\bAut(\bG)^\sharp:= p^{-1}( \bOut(\bG)^\sharp)$ is an affine algebraic $k$--group. 
We consider the square of  pointed sets
$$
\begin{CD}
H^1\big(R_n, \bAut(\bG)^\sharp\big) @>{p^\sharp_*}>> H^1\big(R_n, \bOut(\bG)^\sharp\big)  \\
@VVV @VVV \\
H^1\big(R_n, \bAut(\bG)\big) @>{p_*}>> H^1\big(R_n, \bOut(\bG)\big) . \\
\end{CD}
$$
Since $\bAut(\bG)/ \bAut(\bG)^\sharp = \bOut(\bG)/ \bOut(\bG)^\sharp$, 
this square is cartesian as can be seen by using the criterion of reduction of a torsor to  a subgroup \cite[III.3.2.1]{Gi}.
By construction,  $[p_* \bE]$ comes from  $H^1\big(R_n, \bOut(\bG)^\sharp\big)$, hence
$[\bE]$ comes from a class $[\gF] \in H^1\big(R_n, \bAut(\bG)^\sharp\big)$.
Our assumption is that the $R_n$--group $\gG= {_\bE\bG}= {_\gF\bG}$ contains a maximal torus,
so $\gG_{ad}= {_\gF\bG_{ad}}$ contains a  maximal torus and
${_\gF}\big(\bAut(\bG)^\sharp\big)$ contains a  maximal torus.
In other words, $\gF$ is a toral $R_n$--torsor under  $\bAut(\bG)^\sharp$.
From the equality
$H^1_{toral}\big(R_n, \bAut(\bG)^\sharp\big) = H^1_{loop}\big(R_n, \bAut(\bG)^\sharp\big)$, it follows 
that that $\gF$ is a loop torsor under  $\bAut(\bG)^\sharp$. By applying the  change of groups
$\bAut(\bG)^\sharp \to \bAut(\bG)$, we conclude that $\bE$ is a loop torsor under $\bAut(\bG),$ hence 
that $\gG$ is  loop reductive.
\end{proof}

Lemma \ref{toral}.2 yields the following fact.

\begin{corollary}
 Let $1 \to \bS \to \bG' \buildrel p \over \to \bG \to 1$ be a central extension of
$\bG$ by a $k$--group $\bS$ of multiplicative type. Then the  diagram
$$
\begin{CD}
H^1_{loop}(R_n, \bG') & \enskip \subset \enskip  & H^1(R_n, \bG') \\
@V{p_*}VV @V{p_*}VV \\
H^1_{loop}(R_n, \bG) & \enskip  \subset \enskip  & H^1(R_n, \bG) \\
\end{CD}
$$
 is cartesian.

\end{corollary}

\begin{remark} For a $k$--group $\bG$ satisfying the condition
of Corollary \ref{corbogo}, one can prove in a simpler way
 that toral $\bG$--torsors over $R_n$
are loop torsors by reducing to the case of
a finite \'etale group.
\end{remark}

\begin{remark} Given an integer $d \geq 2$, the Margaux algebra 
(both the Azumaya and Lie versions)
\cite[3.22 and example 5.7 ]{GP2} provides an example of a $\bPGL_d$-torsor over $\C[t_1^{\pm 1}, t_2^{\pm 1}]$ which is not a loop torsor. The underlying $\PGL_d$-torsor is therefore not toral.
This  means that  the Margaux Azumaya algebra
 does not contain any (commutative) \'etale $\C[t_1^{\pm 1}, t_2^{\pm 1}]$-subalgebra of rank $d,$ and that the Margaux Lie algebra, viewed as a Lie algebra over $\C[t_1^{\pm 1}, t_2^{\pm 1}]$, does not contain any Cartan subalgebras (in the sense
of \cite{SGA3}).
 \end{remark}

\begin{remark}\label{laurentremark}
More generally, for each each   positive integer $d$, we
 have
$H^1_{toral}(R_n[x_1, ..., x_d], \bG) = H^1_{loop}(R_n[x_1,...,x_d], \bG)$.  Since $\pi_1( R_n[x_1,...,x_d]) \simeq\pi_1(R_n)$ and
 $\bN_\bG(\bT)(S)=  \bN_\bG(\bT)(S[x_1,...,x_d])$ for every finite \'etale covering $S$ of $R_n$,  the proof we have given works just the same in this case.
\end{remark}

\subsection{Applications to (algebraic) Laurent series.}

Let $F_n = k((t_1))((t_2))...((t_n)).$ In an analogous fashion  to what  we did in the case of $R_n$ we define $F_{n,m} = k((t_1^{\frac{1}{m}}))((t_2^{\frac{1}{m}}))...((t_n^{\frac{1}{m}}))$ and $F_{n, \infty} = \limind F_{n,m}.$

\begin{remark}\label{series} {\rm  (a) If $\tilde{k}$ is a field extension of $k$ the natural map $\tilde{k} \otimes_k F_{n,m}  \to  \tilde{k}((t_1^{\frac{1}{m}}))((t_2^{\frac{1}{m}}))...((t_n^{\frac{1}{m}}))$ is injective. If the
extension is {\it finite}, then this map is an isomorphism.
We will find it convenient (assuming that the field $\tilde k$ is fixed in our discussion)  to denote $\tilde{k}((t_1^{\frac{1}{m}}))((t_2^{\frac{1}{m}}))...((t_n^{\frac{1}{m}}))$ simply by $\tilde{F}_{n,m}.$
\smallskip

(b) The field $\limind \ol{k}((t_1^{\frac{1}{m}}))((t_2^{\frac{1}{m}}))...((t_n^{\frac{1}{m}}))$ is algebraically closed.  We will denote  the algebraic closure of
$F_n$ (resp. $F_{n,m}$, $F_{n,\infty}$) in this field by 
$\overline{F_{n}}$ (resp.  $\overline{F_{n,m}}$,
$\overline{F_{n, \infty}}$). As mentioned in (a) we
have a natural injective ring homomorphsm $\ol{k} \otimes_k F_{n, \infty} \to
\overline{F_{n}}.$

\smallskip

(c) There is a natural group morphism $\pi_1(R_n) \to \Gal(F_n)$ given by considering the Galois extensions $\tilde{R}_{n,m} = \tilde{k} \otimes_k R_{n,m}$ of $R_n$ and $\tilde{F}_{n,m}$ of $F_n$ respectively, where $\tilde{k}
\subset \ol{k}$ is a finite Galois extension of $k$ containing all $m$-roots of unity. These homomorphisms are in fact  isomorphisms.\footnote{ If  $k$ is algebraically closed  this was proved in \cite{GP3} cor 2.14.}  For by applying successively
the structure theorem for local fields \cite{GMS} \S 7.1 p. 17, we have
$\Gal(F_n)= {_{\infty} \bmu}^n(\ol k) \rtimes \Gal(k)$. This means that
$$
\Gal(F_n) = \limproj \Gal\bigl(   \tilde{k}((t_1^{\frac{1}{m}}))((t_2^{\frac{1}{m}}))...((t_n^{\frac{1}{m}})) / F_n\bigr)
$$
for $m$ running over all integers and $\tilde k$ running over all  finite Galois extensions of
$k$ inside $\ol k$  containing a primitive $m$--root of unity. Since at each step we have an isomorphism
$$
\Gal\bigl(   \tilde{k} \otimes R_{n,m} / R_n\bigr) \cong \Gal\bigl(   \tilde{k}((t_1^{\frac{1}{m}}))((t_2^{\frac{1}{m}}))...((t_n^{\frac{1}{m}})) / F_n\bigr) \cong \bmu_m^n(\tilde k) \rtimes
\Gal(\tilde k/k),
$$
we conclude that  $\pi_1(R_n) \cong \Gal(F_n)$.

\smallskip (d) It follows from (c) that the base change $R_n \to F_n$  induces an equivalence of categories between finite \'etale coverings of  $R_n$ and
finite \'etale coverings of  $F_n$. Furthermore, if $\bE/R_n$ is a finite \'etale covering
of $R_n$, we have $\bE(R_n)= \bE(F_n)$. Indeed, 
$\bE$ is split by some Galois covering $\tilde{R}_{n,m} = \tilde{k} \otimes_k R_{n,m}$ and
$\bE(R_n)=  \bE( \tilde{R}_{n,m})^{\Gal(\tilde{R}_{n,m}/R_n)} 
= \bE( \tilde{F}_{n,m})^{\Gal(\tilde{F}_{n,m}/F_n)} = \bE(F_n)$.
}
\end{remark}

\begin{proposition} \label{onto}
The canonical map  $$
H^1_{loop}(R_n, \bG) \to  H^1(F_n, \bG)
$$ is surjective.
\end{proposition}

\begin{proof} We henceforth identify  $\pi_1(R_n)$ with $\Gal(F_n)$ as described in Remark \ref{series}(c).
The proof is very similar to that of Theorem \ref{jo}, and we maintain the notation therein. Again we proceed
in two steps.

\smallskip

\noindent{\it First  case: $\bG^0$ is a torus $\bT$:}
We consider the following commutative diagram of continuous
$\pi_1(R_n)$--groups
$$
\begin{CD}
1 @>>> \bT(\ol k) @>>> \bG(\ol k) @>>> \bnu(\ol k) @>>>  1 \\
&&@VVV @VVV \mid \mid \\
1 @>>> \bT( \ol {F_n}) @>>> \bG(\ol {F_n}) @>>> \bnu(\ol{F_n}) @>>>  1 . \\
\end{CD}
$$
This gives rise to an exact sequence of pointed sets
$$
\begin{CD}
1 @>>> H^1\bigl( \pi_1(R_n),  \bT(\ol k) \bigr) @>>>
H^1\bigl( \pi_1(R_n), \bG(\ol k)\bigr) @>>>
H^1\bigl( \pi_1(R_n),\bnu(\ol k) \bigr)  \\
&&@VVV @VVV \mid \mid \\
1 @>>> H^1\bigl( F_n,\bT\bigr)
@>>>  H^1\bigl( F_n, \bG \bigr) @>>>
 H^1\bigl( F_n, \bnu \bigr) \,. \\
\end{CD}
$$
We are given a cocycle
 $z \in Z^1\bigl( \Gal(F_n), \bG( \ol {F_n} )\bigr) =  Z^1\bigl( \pi_1(R_n), \bG( \ol {F_n} )\bigr) $, and consider its image
$c \in  Z^1\bigl( \pi_1(R_n),\bnu(\ol {F_n}) \bigr)$. By reasoning as in Theorem \ref{jo}
we see that
$[z]$ comes from $H^1\bigl( \pi_1(R_n), \bG(\ol k)\bigr)$ as desired. We conclude that   $H^1(F_{n}, \bG)$ is covered by $k$-loop torsors.

\bigskip

\bigskip

\smallskip
\noindent{\it General case: } Let $\bT$ be a maximal torus of $\bG$.
$$
\begin{CD}
H^1_{loop}\big(R_n, \bN_\bG(T)\big) @>>> H^1\big(F_n,\bN_\bG(\bT)\big) \\
@VVV   @VVV \\
H^1_{loop}(R_n,\bG) @>>> H^1(F_n,\bG). \\
\end{CD}
$$
The reasoning is again identical to the one used in Theorem \ref{jo}.
\end{proof}

\section{Isotropy of loop torsors}

As before $\bG$ denotes  a linear algebraic group
over a field $k$ of characteristic zero. $R_n$ and  $\pi_1(R_n)$ are as in the previous section.

\subsection{Fixed point statements}

Let $\eta:  \pi_1(R_n) \to \bG(\ol{k})$ be a continuous cocycle.
Consider as before a Galois extension $\tilde{R}_{n,m} = \tilde{k} \otimes_k R_{n,m} $ of $R_n$ where $\tilde{k} \subset \ol{k}$ is a
finite Galois extension of $k$ containing all $m$--roots of unity
in $\ol k$, chosen so that our cocycle $\eta$ factors through the Galois group
\begin{equation}\label{Galoisnm}
\tilde{\Gamma}_{n,m} = \Gal(\tilde{R}_{n,m}/R_n) \cong \bmu_m^n(\tilde k) \rtimes
\Gal(\tilde k/k)
\end{equation}
We assume henceforth that $\bG$ acts on a $k$--scheme $\bY.$ The Galois group
$\tilde{\Gamma}_{n,m}$ acts naturally on
$\bY(\tilde{R}_{n,m}),$  and we
denote this action by $\gamma : y \mapsto {^{\gamma}y}.$ By means of
$\eta$ we get a twisted action of $\tilde{\Gamma}_{n,m}$  on
$\bY(\tilde{R}_{n,m})$  which we denote by $\gamma : y \mapsto
{^{\gamma '}y}$ where
\begin{equation}\label{Gammadescent}
{^{\gamma '}y} = \eta_\gamma \, . {^{\gamma}y}
\end{equation}

By Galois descent (\ref{Gammadescent}) leads to a twisted form of
the $R_n$--scheme $\bY_{R_n}.$ One knows that this twisted form is up
to isomorphism independent of the Galois extension
$\tilde{R}_{n,m}$ chosen through which  $\eta$ factors. We
will denote this twisted form  by  ${_\eta \bY_{R_n}},$ or simply by $_{\eta}\bY$ following the conventions that have been previously mentioned regarding this matter.

Let $(z, \eta^{geo})$ be the  couple associated to $\eta$ according
to Lemma \ref{Theta}. Thus $z \in Z^1\big(\Gal(k), \bG(\ol{k})\big)$ and $\eta^{geo} \in
\Hom_{k-gp}( {_\infty \bmu^n}, {_z\bG} )$ by taking into account
 Lemma \ref{salade2}.  By means of $z$ we
construct a twisted form ${_z \bY}$ of the $k$--scheme $\bY$
which comes equipped with an action of  ${_z \bG}$.
Via $\eta^{geo},$ this defines an  algebraic   action  of the affine
$k$-group ${_\infty \bmu}^n$ on ${_z \bY}.$
At the level of $\ol k$--points of ${_\infty \bmu}^n$, the action is given by

\begin{equation}\label{etageotwistedaction}
\widehat{n} . y = \eta^{geo}(\widehat{n}) . y
\end{equation}
for all $\widehat{n} \in {_\infty \bmu}^n(\ol k) = \limproj_m  \bmu_m^n({\ol k})$ and $y \in {_z\bY}(\ol k)$.
We denote by ${(_z\bY)}^{\eta^{geo}}$ the closed fixed point subscheme for the action of
${_\infty \bmu}^n$ (see \cite{DG} II \S 1 prop. 3.6.d). We have
$$
{(_z \bY)}^{\eta^{geo}}(\ol k)= \Big\{  y \in  {_z \bY}(\ol k) = \bY(\ol k) \, \mid  y= \eta^{geo}(\widehat{n}) . y  \, \, \enskip \forall \, 
\widehat{n} \in {_\infty \bmu}^n(\ol k) 
    \Big\}
$$
and in terms of rational points
\begin{equation}\label{fix}
{(_z \bY)}^{\eta^{geo}}(k)= {_z \bY}(k) \cap {({_z \bY})}^{\eta^{geo}}(\ol k)
= \Big\{  y \in  {_z \bY}(k) \, \mid  y= \eta^{geo}(\widehat{n}) . y \,  \enskip \, \forall \, 
\widehat{n} \in {_\infty \bmu}^n(\ol k)
    \Big\}.
\end{equation}
where we recall that
$$
{_z \bY}(k) 
= \Big\{  y \in  \bY(\ol k) \, \mid  y= z_{\gamma} . {^ \gamma}y \,  \enskip \, \forall \, 
\gamma \in \Gal(k)
    \Big\}.
$$

\begin{theorem} \label{godown}

\begin{enumerate}
\item Let $\bG$ act on  $\bY$ as above, and assume that $\bY$ is 
projective ( i.e. a 
closed subscheme in $\mathbb{P}^n_k$).
Let $\eta:  \pi_1(R_n) \to \bG(\ol{k})$ be  a (continuous)  cocycle,
 and ${_\eta \bY}$ be the corresponding twisted form of $\bY_{R_n}.$
 The following are equivalent:

\begin{enumerate}
\item $(_\eta\bY)(R_n) \not= \emptyset$,
\item $(_\eta \bY)(K_n) \not= \emptyset$,
\item $(_\eta \bY)(F_n) \not= \emptyset$,
\item ${(_z \bY)}^{\eta^{geo}}(k) \not= \emptyset$.
 \end{enumerate}

\item Let $\bS$ be a closed $k$--subgroup of $\bG$.
Let $\bY$ be a smooth $\bG$--equivariant compactification of
$\bG/ \bS$ (i.e., $\bY$ is projective $k$--variety that
 contains $\bG/ \bS$ as an open dense
$\bG$-subvariety). Then the following are equivalent:

\begin{enumerate}
\item $[\eta]_{ K_n} \in
{\rm Im}\bigl( H^1(K_n, \bS) \to   H^1(K_n, \bG)  \bigr) $,
\item $[\eta]_{F_n} \in
{\rm Im}\bigl( H^1(F_n, \bS) \to   H^1(F_n, \bG)  \bigr) $,
\item   $ {(_z \bY)}^{\eta^{geo}}(k)  \not= \emptyset$.
 \end{enumerate}

 \end{enumerate}

\end{theorem}

\begin{proof}
(1) Again we twist the action $\bG \times \bY \to  \bY$  by $z$ to obtain an action ${_z \bG} \times {_z\bY} \to {_z\bY}$.
Lemma \ref{looptwist}
enables us to
 assume without loss of generality that  $z$ is the trivial
 cocycle.
We are thus left to deal with a $k$--homomorphism
$\eta^{geo}: {_\infty \bmu}^n \to \bG$ which factors at the finite level through
$ \bmu_m^n \to \bG$ for $m$ large enough.
This allows us to reason  by means  of  a suitable  covering
$\tilde{R}_{n,m}$ as in (\ref{Galoisnm}).

\smallskip

\noindent  $(a) \implies (b) \implies (c)$ are obtained by applying
the base change $R_n \subset K_n \subset F_n.$

\smallskip

\noindent $(c) \Longrightarrow (d)$: Each $\gamma \in
\tilde{\Gamma}_{n,m}$ induces an automorphism  of
$\tilde{R}_{n,m} \otimes_{R_n} F_n \simeq F_{n,m}
\otimes_k \tilde{k} = \tilde{F}_{n,m}$ which we also denote by
$\gamma$ (even though the notation $\gamma \otimes 1$ would be more
accurate.) Since $\tilde{R}_{n,m}$ trivializes $_\eta \bY,$ the Galois extension $\tilde{F}_{n,m}$ of $F_n$ (whose Galois group we identify with $\tilde{\Gamma}_{n,m}$) splits ${_\eta}\bY_{F_n}.$ By Galois descent
$$
_\eta \bY(F_n) = \Bigl\{ y \in \bY(\tilde{F}_{n,m}) \, \mid \,   \eta_{\gamma}\, . {^\gamma}y = y
\enskip \forall \gamma \in \tilde{\Gamma}_{n,m} \Bigr\}.
$$
Since $z$ is trivial, 
this last equality reads
$$
_\eta \bY(F_n) = \Bigl\{ y \in \bY(\tilde{F}_{n,m}) \, \mid \,   \eta^{geo}({\ol \gamma})\, . {^\gamma}y = y
\enskip \forall \gamma \in \tilde{\Gamma}_{n,m} \Bigr\}.
$$
where $\ol \gamma$ is the image of $\gamma$ under the map $\tilde{\Gamma}_{n,m} \to\bmu_m^n(\tilde k)$ given by (\ref{Galoisnm}).
Hence we  have  ${_\eta \bY}(F_n) \subset  \bY(F_{n,m})$ and
$$
_\eta \bY(F_n) = \Bigl\{ y \in \bY(F_{n,m}) \, \mid \,   \eta^{geo}({
\gamma})\, . {^\gamma}y = y
\enskip \forall  \gamma \in  \bmu_m^n( \ol k)  \Bigr\}.
$$

\noindent Since $\bY$ is proper over $k$, we have
$$
{_\eta}\bY(F_n)  = \Bigl\{ y \in \bY({F}_{n-1,m}[[t_n^{\frac{1}{m}}]]) \, \mid \,   \eta^{geo}({ \gamma}). {^\gamma}y =  y \enskip \,   \gamma \in  \bmu_m^n( \ol k) \Bigr\}.
$$
Our hypothesis is that this last set is not empty. By specializing at
$t_n=0$, we get that

\begin{equation}\label{t=0}
\Bigl\{ y \in \bY({F}_{n-1,m}) \, \mid \,   \,    \eta^{geo}({\gamma}). {^\gamma}y = y \enskip \,\forall \gamma \in  \bmu_m^n( \ol k) \Bigr\} \not = \emptyset.
\end{equation} 
We write now $\bmu_m^n( \ol k)= \bmu_m^{n-1}(\ol k) \times \bmu_m(\ol k)$
which provides a decomposition of $\eta^{geo}$ into two $k$-homomorphisms ${\eta'}^{geo} : \bmu_m^{n-1} \to \bG$  and  ${\eta_n}^{geo} : \bmu_m \to \bG$.
We define $\eta'=(1, {\eta'}^{geo})$, 
 ${\eta_n}=(1 , {\eta_n}^{geo})$
and
  $$
\bY' := \bY^{\eta_n^{geo}}.
$$ 
By \cite{DG} II \S 1 prop. 3.6.(d) we know that $\bY' $ is a closed subscheme of $\bY$, hence 
a projective $k$--variety.
Observe that $\bmu_m^{n-1}$ acts on  $\bY'$.

\begin{claim} The set (\ref{t=0}) is included in 
${_{\eta'}}\bY' (F_{n-1})$.\footnote{This inclusion is in fact an equality, but this stronger statement is not needed.}
\end{claim}

To look at the invariants under the action of $\bmu_m^n( \ol k)$, we first look at
the invariants under the last factor $\bmu_m( \ol k),$ and then the first $(n-1)$-factor
$\bmu_m^{n-1}( \ol k)$
By restricting the  condition  to  elements of the form $(1, \gamma_n)$ for
$\gamma_n \in \bmu_m(\ol k)$, we see that our set is included in  
$$
\Bigl\{ y \in \bY({F}_{n-1,m}) \, \mid \,   \,    \eta_n^{geo}({\gamma_n}). y = y \enskip \,\forall \gamma_n \in  \bmu_m( \ol k) \Bigr\}
$$ 
because $\bmu_m(\ol k)$ acts trivially on $F_{n-1,m}$.
By identity (\ref{fix}) applied to the base change of
$\eta_n^{geo}$ to the field $F_{n-1}$, this is nothing but
$\bY^{\eta_n^{geo}}({F}_{n-1,m})$. Looking now at the invariant condition for 
the elements of the form  $(\gamma', 1)$ for 
$\gamma' \in \bmu^{n-1}_m(\ol k)$, it follows that 
$$ 
\Bigl\{ y \in \bY({F}_{n-1,m}) \, \mid \,   \,    \eta^{geo}({\gamma}). {^\gamma}y = y \enskip \,\forall \gamma \in  \bmu_m^n( \ol k) \Bigr\}
 \qquad \qquad \qquad \qquad\qquad \qquad
$$ 
$$
\qquad \qquad \qquad \qquad  \, \subset \, 
\Bigl\{ y \in \bY^{\eta_n^{geo}}({F}_{n-1,m}) \, \mid \,   \,    {\eta'}^{geo}({\gamma'}). {^{\gamma'}}y = y \enskip \,\forall \gamma' \in  \bmu_m^{n-1}( \ol k) \Bigr\}
= {_{\eta'}}\bY' (F_{n-1}).
$$ 
By induction on $n$, we get that inside $({_{\eta'}}\bY') (F_{n-1})$ we have
${\bY'}^{\eta'_{geo}}(k)  \not = \emptyset.$ 
 Thus $\bY(k)^{\eta^{geo}}  \not = \emptyset$
as desired.

\smallskip

\noindent $(d) \Longrightarrow (a)$:
Since
 $$
({_\eta} \bY)(R_n) = \Bigl\{ y \in \bY(\tilde{R}_{n,m}) \, \mid \,  
 \eta^{geo}({\ol \gamma})\, . {^\gamma}y = y
\enskip \forall \gamma \in \tilde{\Gamma}_{n,m} \Bigr\},
$$
 the inclusion $\bY(k) \subset \bY(\tilde{R}_{n,m})$ yields the  inclusion
$$
(\bY^{\eta^{geo}})(k) \subset ({_\eta}\bY)(R_n).
$$
Thus if $(\bY^{\eta^{geo}})(k) \not = \emptyset$, then
$({_\eta }\bY)(R_n) \not = \emptyset$.

\smallskip

\noindent (2) The quotient $\bG/\bS$ is representable by Chevalley's theorem \cite[\S III.3.5]{DG}.
The only non trivial implication is 
$(c) \Longrightarrow (a).$ Let $\bX= (\bG/\bS) \times_k R_n.$
By (1), we have
${_\eta}\bY(K_n) \not = \emptyset$.
In other words, the $K_n$-homogeneous space
${_\eta}\bX$ under ${_\eta}\bG$
has a $K_n$-rational point on the compactification ${_\eta}\bY$.
By Florence's theorem \cite{F}, ${_\eta\bX}(K_n) \not= \emptyset$, hence
$(a)$. \end{proof}

\subsection{Case of flag varieties}

The $k$--group $\bG^0/R_u(\bG)$ is  reductive.
Let $\bT$ be a maximal $k$--torus of $\bG^0/ R_u(\bG)$.
This data permits to choose   a basis $\Delta$ of the root system
$\Phi(\bG^0/R_u(\bG) \times_k {\ol k}, \bT \times_k {\ol k})$
or in other words to choose a Borel 
 subgroup $\bB$ of the $\ol{k}$--group $\bG^0/R_u(\bG) \times_k {\ol k}$.
It is well known that there is a one-to-one correspondence between
the  subsets of $\Delta$ and 
the parabolic subgroups of $\bG^0 \times_k \ol k$ containing
$\bB$, which is provided by the standard parabolic subgroups
 $(\bP_I)_{I \subset \Delta}$ of $\big(\bG^0/ R_u(\bG)\big) \times_k \ol k$ \cite[\S 21.11]{Bo}.
We have $\bP_\Delta= \big(\bG^0/ R_u(\bG)\big) \times_k \ol k$ and $\bP_\emptyset= \bB$.
Furthermore we
 know that each parabolic subgroup of $\big(\bG^0/ R_u(\bG)\big)\times_k \ol{k}$ is  $\big(\bG^0/ R_u(\bG)\big)(\ol k)$--conjugate to a unique standard parabolic subgroup. This allows us to define the {\it type}
of an arbitrary   parabolic
subgroup
of $\bG^0/ R_u(\bG)$. It can happen that 
two different standard parabolic 
subgroups
of the $\ol k$--group $\big(\bG^0/ R_u(\bG)\big) \times_k \ol k$ are conjugate under $\bG(\ol k)$:
There are in general fewer conjugacy classes of parabolic subgroups.
If $\bP$ is a parabolic subgroup of the $k$--group $\bG^0/R_u(\bG)$,
we denote by $\bN_\bG(\bP)$ its normalizer for the conjugacy  action  
of $\bG$ on $\bG^0/R_u(\bG)$.

\begin{lemma}\label{proj}
The quotient $\bG/\bN_\bG(\bP)$ is a projective $k$--variety.
\end{lemma}

\begin{proof} We can assume that $\bG^0$ is reductive.
Since $\bG^0/\bP$ is projective and is a connected component of  $\bG/\bP$,
$\bG/\bP$ is projective as well.
The point is that the morphism $\bG/\bP \to \bG/\bN_\bG(\bP)$
is a $\bN_\bG(\bP)/ \bP$--torsor. Since the affine $k$--group 
$\bN_\bG(\bP)/ \bP$ is finite, \'etale descent tells us that 
$\bG/\bN_\bG(\bP)$ is proper \cite[prop. 2.7.1]{EGAIV}.
But   $\bG/\bN_\bG(\bP)$ is quasiprojective, hence projective. 
\end{proof}

Given a loop cocycle  $\eta: \pi_1(R_n) \to \bG(\ol k)$ with coordinates
$(z, \eta^{geo})$ as before, we focus on the special case of 
  flag varieties  of parabolic
subgroups of $\bG^0/R_u(\bG)$.

\begin{corollary}\label{flag}
\begin{enumerate}
\item Let $I \subset \Delta$. The following are equivalent:

\begin{enumerate}
\item The $R_n$--group ${_\eta}\big(\bG^0/ R_u(\bG)\big)$ admits a parabolic 
subgroup scheme of type $I$;
\item The $R_n$--group ${_\eta}\big(\bG^0/ R_u(\bG)\big)_{R_n}$ admits a parabolic 
subgroup of type $I$;
\item The $F_n$--group ${_\eta}\big(\bG^0/ R_u(\bG)\big)_{F_n}$ admits a parabolic subgroup   of type $I$;

\item There exists a parabolic subgroup 
$\bP$ of the $k$--group$ _z(\bG^0/ R_u(\bG))$ which is of type $I$ and
which is normalized by $\eta^{geo}$, i.e.,
$\eta^{geo}$ factorizes through $\bN_{_z\bG}(\bP)$.
\end{enumerate}

\item  The following are equivalent:
\begin{enumerate}
\item ${_\eta\big(\bG^0/ R_u(\bG)\big)}_{R_n}$ is irreducible
 (i.e has no proper parabolic subgroups);
\item ${_\eta\big(\bG^0/ R_u(\bG)\big)}_{K_n}$ is irreducible;
\item ${_\eta\big(\bG^0/ R_u(\bG)of type \big)}_{F_n}$ is irreducible;
\item The $k$--group homomorphism $\eta^{geo}: {_\infty  \bmu}^n \to
{_z \bG} \to \bAut({_z \bG}^0)$ is irreducible (see \S \ref{sec-irr}).
\end{enumerate}

\item  The following are equivalent:
\begin{enumerate}
\item ${_\eta  \big(\bG^0/ R_u(\bG)\big)}_{R_n}$ is anisotropic;
\item ${_\eta  \big(\bG^0/ R_u(\bG)\big)}_{K_n}$ is anisotropic;
\item ${_\eta  \big(\bG^0/ R_u(\bG)\big)}_{F_n}$ is anisotropic;
\item The $k$--group homomorphism $\eta^{geo}: {_\infty \bmu}^n \to
{_z \bG} \to  \bAut({_z \bG}^0)$ is anisotropic (see \S \ref{sec-irr}).
\end{enumerate}

\end{enumerate}

\end{corollary}

\begin{proof} Without loss of generality, we can factor out by $R_u(\bG)$
and assume that $\bG^0$ is reductive.
As in the proof of Theorem \ref{godown}, we 
can assume by twisting that $z$ is trivial and
reason ``at the finite level":

\begin{claim}\label{killed} There exists a positive integer $m$ such that
$[\eta] \in H^1(R_n, \bG)$ is trivialized by the base change $R_{n,m}/R_n$. 
\end{claim}
 
Indeed by continuity $\eta^{geo}: {_\infty\bmu}^n \to \bG$ factors through a morphism $f:  {\bmu}^n_m \to \bG$ and
$[\eta]= f_*[ \bE_{n,m}]$ where $\bE_{n,m}= \Spec(R_{n,m})/ \Spec(R_n)$
stands for the standard    ${\bmu}^n_m$-torsor.
In particular, the class $[\eta] \in H^1(R_n, \bG)$ is trivialized by
the covering $R_{n,m}/R_n$  as above.

\smallskip

\noindent (1) $(a) \Longrightarrow (b) \Longrightarrow (c)$: obvious.

\smallskip

\noindent $(c) \Longrightarrow (d)$: We assume that  ${_\eta \bG^0}_{F_n}$ 
admits a $F_n$-parabolic subgroup $\bQ$ of type $I$.
Hence \break ${_\eta \bG^0}_{F_n} \times_{F_n} F_{n,m} =  \bG^0_{F_{n,m}}$
admits a $F_{n,m}$--parabolic subgroup of type $I$. Since $F_{n,m}$ is an
iterated Laurent serie field over $k$, 
it implies that  $\bG^0$ admits a parabolic subgroup $\bP$ of type $I$  
 (see the proof of \cite[lemma 5.24]{CGP}). We consider 
the  $R_n$--scheme
 $\bX:={ _\eta \big(\bG/ \bN_\bG(\bP)\big)} \times_k R_n$ 
which by  descent considerations \cite[2.7.1.vii]{EGAIV} is proper since $\bG/ \bN_\bG(\bP)$ is.

\begin{claim}
 $\bX(F_n) \not = \emptyset$. 
\end{claim}

The $F_n$--group ${_\eta \bG}^0/F_n$ admits a subgroup $\bQ$
such that  ${\bQ \times_k \ol{F}_n}$ is 
 $\bG^0(\ol F_n)$-conjugate to  ${\bP \times_k \ol{F}_n} \subset \bG^0 \times_k \ol{F}_n$.
Let  $g \in \bG({\ol F_n})$ be such that 
$\bQ \times_{F_n}{\ol F_n} =   g \,(\bP  \times_{k}{\ol F_n}) \, g^{-1}$.
As in  \cite[III.2, lemme 1]{Se}, we check that  the cocycle 
$\gamma \mapsto g^{-1} \, \eta_\gamma \, ^\gamma g$ is cohomologous to
$\eta$ and has value in  $\bN_\bG(\bP)({\ol F_n})$. 
In other words, 
the $F_n$--torsor corresponding to $\eta$ admits a 
reduction to the subgroup $\bN_\bG(\bP),$ i.e.
$$
[\eta] \in {\rm Im} \Bigl(   H^1\big(F_n, \bN_\bG(\bP)\big) \to   H^1(F_n, \bG)   \Bigr) .
$$
 This implies that  $\bX(F_n) \not = \emptyset$ ({\it ibid}, I.5, prop. 37)
and the Claim is proven.

 By Theorem \ref{godown}.1, we have
$(\bX^{\eta^{geo}})(k) \not = \emptyset$, so that there exists an element $x \in (\bX^{\eta^{geo}})(k)$. 
Since $ H^1\big(k, \bN_\bG(\bP)\big)$ injects in $H^1(k,\bG)$ (see \cite[cor. 2.7.2]{Gi4}), we have 
$\bX(k)= \bG(k)/ \bN_\bG(\bP)(k)$, i.e. $\bX(k)$ is homogeneous under 
$\bG(k)$. 

Hence there exists $g \in \bG(k)$ such that $x= g.x_0$ where $x_0$ stands for 
the image of $1$ in $X(k)$. 
We have
$$
\eta^{geo}(\bmu^n_m(\ol k)) \subset \Stab_{ \bG(\ol k)}(x)g \, \Stab_{ \bG(\ol k)}(x_0) g^{-1}
=g\bN_\bG(\bP)(\ol k) g^{-1}= \bN_\bG(g \bP g^{-1} )(\ol k) .
$$ 
Thus $\eta^{geo}$ normalizes a parabolic subgroup  of type $I$ of the $k$--group of $\bG^0.$

\smallskip

\noindent $(d) \Longrightarrow (a)$: We may assume that $\eta$ has values in
$\bN_\bG(\bP)(\ol k)$. In that case, the twisted $R_n$--group scheme ${_\eta\bG}^0$
admits the parabolic subgroup  ${_\eta\bP}/R_n$.

\smallskip

\noindent (2)  Follows of (1).

\smallskip

\noindent (3) Recall that a $k$--group $\bH$ with
reductive connected component of the identity  $\bH^0$ is anisotropic if and only if
it is irreducible  and its connected center $\bZ(\bH^0)^0$
is an anisotropic torus.  Statement (3) reduces then to the case
where $\bG^0$ is a $k$--torus $\bT$. We are then given a continuous action
of $\pi_1(R_n)$ on the cocharacter group $\widehat \bT^0(\ol k)$.
It is convenient to work with the opposite assertions  to $(a)$, $(b)$ $(c)$and $(d),$ which we denote by $(a')$,
 $(b')$ $(c')$ and $(c')$ respectively.

\smallskip

\noindent $(a') \Longrightarrow (b')$: If the $R_n$--torus ${_\eta}\bT:= {_\eta}\bT_{R_n}$ is isotropic, so
is the $K_n$--torus $_\eta\bT \times_{R_n} {K_n}$.
\smallskip

\noindent $(b') \Longrightarrow (c')$: If the $K_n$--torus $_\eta\bT \times_{R_n} K_n$ is isotropic, so
is the $F_n$--torus $_\eta\bT \times_{R_n} {F_n}$.

\smallskip

\noindent $(c') \Longrightarrow (d')$: By Lemma \ref{looptwist}
we have
$$
\Hom_{F_n-gr}( \bG_m, {_\eta\bT}_{F_n})  = \Hom_{F_n-gr}( \bG_m, {\bT}_{F_n})^{\eta_{geo}}.
$$
If $_\eta\bT$ is isotropic, then this group is not zero and
the $k$--group morphism $\eta^{geo} : {_\infty \bmu}^n \to {_z\bG}$ fixes a cocharacter of
${\bT} = ({\bG})^0$, hence $(c')$. 

\smallskip

\noindent $(d') \Longrightarrow (a')$:  We assume that
the morphism $\eta^{geo} :  {_\infty \bmu}^n  \to {\bG}$ fixes a cocharacter $\lambda : \bG_m \to 
{\bT}$. Since

$$
\Hom_{K_n-gr}( \bG_m, {_\eta\bT}_{K_n})  \simeq \Hom_{K_n-gr}( \bG_m, {\bT}_{K_n})^{\eta_{geo}}.
$$
it follows that $\lambda$ provides a non-zero cocharacter of
$_\eta\bG_{K_n}$, hence $(a')$.

\end{proof}

As in the case of loop torsors \cite[cor. 3.3]{GP2}, 
the Borel-Tits theorem has the following consequence.

\begin{corollary}\label{index}
The minimal elements (with respect to  inclusion) of the set of
parabolic subgroups of the $k$--group ${_z\bG}^0$ which are normalized by $\eta^{geo}$ are
all conjugate under $_z\bG^0(k)$. The type $I(\eta)$ of this
conjugacy class is the Witt-Tits index of the $F_n$--group
\break $_\eta \big( \bG^0/R_u(\bG)\big) \times_{R_n} F_n$.
\end{corollary}


\subsection{Anisotropic loop torsors}

For anisotropic loop classes, we have the following
beautiful picture.

\begin{theorem}\label{k-aniso} Assume that  $\bG^0$ is reductive.
Let $\eta, \eta' : \pi_1(R_n) \to
\bG(\ol k)$ be two loop cocycles.
Assume that $({_\eta}\bG)_{F_n}$ is anisotropic.
 Then the following are equivalent:

\begin{enumerate}

\item $[\eta] = [\eta'] \in H^1(R_n, \bG)$,

\item $[\eta]_{K_n} = [\eta']_{K_n}  \in H^1(K_n, \bG)$,

\item $[\eta]_{F_n} = [\eta']_{F_n}  \in H^1(F_n, \bG)$.
\end{enumerate}

\end{theorem}

We consider first the case of purely geometric
 loop cocycles. Note that this is the set of {\it all} loop cocycles
if $k$ is algebraically closed.

\begin{theorem}\label{k-aniso-bis} 
 Let $\eta, \eta' : \pi_1(R_n) \to
\bG(\ol k)$ be two loop cocycles of the form $\eta= (1, \eta^{geo})$ and
 $\eta'= (1, \eta^{'geo}).$  Assume that $\eta$ is anisotropic.
Then the following are equivalent:

\begin{enumerate}
\item $\eta^{geo}$ and ${\eta'}^{geo}$ are conjugate under $\bG(k)$,

\item $[\eta]= [\eta'] \in H^1(R_n, \bG)$,

\item $[\eta]_{K_n} = [\eta']_{K_n}  \in H^1(K_n, \bG)$,

\item $[\eta]_{F_n} = [\eta']_{F_n}  \in H^1(F_n, \bG)$.
\end{enumerate}

\end{theorem}


\begin{proof}  Recall that  $\eta^{geo}, {\eta'}^{geo} : {_\infty\bmu} \to \bG$ are affine $k$--group homomorphisms that factor through the algebraic group $\bmu_m^n$ for $m$ large enough. The meaning of (1) is that there exists $g \in \bG(k)$ such that ${\eta'}^{geo} = \Int(g) \circ {\eta}^{geo}.$

The implications $1) \Longrightarrow 2)
\Longrightarrow 3) \Longrightarrow 4)$ are obvious.
We shall prove the implication  $4) \Longrightarrow 1)$.
Assume, therefore,  that $[\eta]_{F_n}=  [\eta']_{F_n} \in H^1(F_n, \bG)$.

Let $\bT$ be a maximal torus of $\bG^0$ and let $\bN=\bN_\bG(\bT)$ and  $\bW= \bN/ \bT$.
Since the maximal tori of $\bG^0 \times_k \ol{k}$ are all conjugate under $\bG^0(\ol k)$, 
the map $\bN_\bG(\bT) \to \bG/\bG^0$ is surjective.
Let   $\tilde{k}$ be a finite Galois extension 
which contains   $\bmu_m(\ol k)$,    splits $\bT$ and such that the natural map
$\bN(\tilde k) \to ({\bG/\bG^0})(\ol k)$ is surjective. We furthermore assume without 
loss of generality that  our choice of $m$ and $\tilde k$ trivialize $\eta$ and $\eta'.$

Set
$\tilde \Gamma_{n, m}= \bmu_m^n(\ol k) \rtimes \Gal(\tilde{k}/k)$.
In terms of cocycles, our hypothesis  means that there exists
$h_n \in {\bG}(\tilde{F}_{n,m})$ such that
\begin{equation}\label{*}
h_n^{-1} \, \eta(\gamma) \,{^{\gamma }}h_n =  \eta'( \gamma)
\quad \forall \gamma \,  \in \tilde {\Gamma_n}. 
\end{equation}

\noindent Our goal is to show that we can  actually find such an
element inside ${\bG}(k)$.
 We reason by means of a building argument, and
 appeal to Remark \ref{series} to 
view $\tilde{F}_{n,m}$ as a complete local field
with residue field $\tilde{F}_{n-1,m}$.
Note that $F_n= (\tilde{F}_{n,m})^{\tilde \Gamma_{n,m}},$ and that $F_n$ can be viewed as
complete local field with residue field $F_{n-1}$.

Let $\bC= \bG^0/ \bD\bG^0$ be the coradical  of $\bG^0.$ This is a $k$--torus which is
 split by $\tilde k$. Recall that the (enlarged) Bruhat-Tits building  ${\cal B}_n$ 
of the $\tilde{F}_{n,m}$--group $\bG \times_k \tilde{F}_{n,m}$ \cite[\S 2.1]{T2} is defined by
$$
{\cal B}_n
= {\cal B} \times \enskip  V 
$$
where $V=  {\widehat \bC}^0 \otimes_\Z \R$, and ${\cal B}$ is the building of the semisimple $\tilde{F}_{n,m}$--group $\bD\bG^0 \times_k \tilde{F}_{n,m}.$ The building ${\cal B}_n$ is equipped with a natural action of
${\bG}(\tilde{F}_{n,m}) \rtimes \tilde{\Gamma}_{n,m}$.
By  \cite[9.1.19.c]{BT1} the group  $\bD\bG^0(\tilde F_{n-1,m}[[t_n^{1 \over m}]])\, $ fixes a 
unique (hyperspecial) point $\phi_d \in {\cal B}(\bD\bG^0 \times_k \tilde{F}_{n,m})$ and
  ${\rm Stab}_{\bD\bG^0(\tilde F_{n,m})} (\phi_d) = 
\bD\bG^0\bigl(\tilde{F}_{n-1,m}[[t_n^{1 \over m}]] \bigr)$.

Since the bounded group $\bG(\tilde F_{n-1,m}[[t_n^{1 \over m}]])\rtimes \tilde \Gamma_{m,n}$
fixes at least one point of the building ${\cal B}(\bD\bG^0 \times_k \tilde{F}_{n,m})$; such a point is necessarily $\phi_d$ which is then fixed under 
$\bG(\tilde F_{n-1,m}[[t_n^{1 \over m}]])\rtimes \tilde \Gamma_{m,n}$.

\begin{claim}\label{claim5}  There exists a point $\phi=(\phi_d,v) \in {\cal B}_n$ such that 

\begin{enumerate}

\item   $\tilde \Gamma_{m,n}$ fixes $\phi$; 

\item ${\rm Stab}_{\bG(\tilde F_{n,m})} (\phi) = 
\bG\bigl(\tilde{F}_{n-1,m}[[t_n^{1 \over m}]] \bigr)$.

\end{enumerate}
\end{claim}

We use the fact that 
$\bG^0(\tilde F_{n,m})$ acts on  $V$
by  translations under the map
$$
\bG^0(\tilde F_{n,m}) \buildrel  q \over \lgr \bC(\tilde{F}_{n,m}) = {\widehat \bC}^0 \otimes_\Z {\tilde F_{n,m}^\times} \buildrel - {\rm ord}_{t_n} \over \lgr 
 {\widehat \bC}^0.
$$
It follows that for each $v \in V$, we have
$$
{\rm Stab}_{\bG^0(\tilde F_{n,m})}(v) = {\rm Stab}_{\bG^0(\tilde F_{n,m})}(V)= 
q^{-1}\bigl(\bC( \tilde F_{n-1,m}[[t_n^{1 \over m}]]  )\bigr). \leqno{(*)}
$$
Since $q$ maps  $\bG^0( \tilde F_{n-1,m}[[t_n^{1 \over m}]])$ into
$\bC( \tilde F_{n-1,m}[[t_n^{1 \over m}]])$, it follows that 
$\bG^0( \tilde F_{n-1,m}[[t_n^{1 \over m}]])$ fixes pointwise $\phi_d \times V$.

Let us choose now the vector $v$ by considering
 the action of the group $\bN(\tilde k) \rtimes \tilde \Gamma_{m,n}$
on $V$. Since this action is trivial on $\bT(\tilde k)$, it provides an action of the finite group $\bW(\tilde k) \rtimes \tilde \Gamma_{m,n}$ on $V$. But  this action is affine, so 
there is at least one $v \in V$  which is fixed under  $\bN(\tilde k) \rtimes \tilde \Gamma_{m,n}$.
The point $\phi=(\phi_d,v)$ is $\tilde \Gamma_{m,n}$-invariant, hence $(1)$.
We now use  that $\bN(\tilde k)$ surjects onto 
$(\bG/\bG^0)(\tilde k)= (\bG/\bG^0)(\tilde F_{n-1,m}[[t_n^{1 \over m}]]) = (\bG/\bG^0)(\tilde F_{n,m})$, hence that
$$
\bG(\tilde F_{n-1,m}[[t_n^{1 \over m}]])= \bG^0(\tilde F_{n-1,m}[[t_n^{1 \over m}]]) \, . \, \bN\bigl(\tilde k\bigr), \enskip
\bG(\tilde F_{n,m})= \bG^0(\tilde F_{n,m} ) \, . \, \bN\bigl(\tilde k\bigr).
$$
Since $\bN\bigl(\tilde k\bigr)$ fixes $\phi$, we have
$$
{\rm Stab}_{\bG(\tilde F_{n,m})} (\phi) =  {\rm Stab}_{\bG^0(\tilde F_{n,m})} (\phi) \, . \, \bN\bigl(\tilde k\bigr), 
$$
and it remains to show that  $\bG^0(\tilde F_{n-1,m}[[t_n^{1 \over m}]]) = {\rm Stab}_{\bG^0(\tilde F_{n,m})} (\phi)$. 
Since $\bT \times_k \tilde k $ is split, the map $\bT\times_k \tilde k \to \bC\times_k \tilde k$ is split and we have the 
decompositions
$$
\bG^0(\tilde F_{n-1,m}[[t_n^{1 \over m}]])= \bD\bG^0(\tilde F_{n-1,m}[[t_n^{1 \over m}]]) \, . \,
\bT(\tilde F_{n-1,m}[[t_n^{1 \over m}]] ) $$
and
$$
\bG^0(\tilde F_{n,m})= \bD\bG^0(\tilde F_{n,m} ) \, . \, 
\bT( \tilde F_{n,m}) .
$$
The first equality shows that $\bG^0(\tilde F_{n-1,m}[[t_n^{1 \over m}]])$ fixes $\phi$
hence that $\bG^0(\tilde F_{n-1,m}[[t_n^{1 \over m}]]) \subset  {\rm Stab}_{\bG^0(\tilde F_{n,m})} (\phi)$. 

As for the reversed inclusion consider an element $g \in {\rm Stab}_{\bG(\tilde F_{n,m})} (\phi)$. Then $q(g) \in 
\bC( \tilde F_{n-1,m}[[t_n^{1 \over m}]])$. The map
$\bG(\tilde F_{n-1,m}[[t_n^{1 \over m}]]) 
\buildrel  q \over \lgr \bC(\tilde F_{n-1,m}[[t_n^{1 \over m}]])$ is surjective, hence we can assume that $g \in \bD\bG^0(\tilde F_{n,m})$. Since $g. \phi_d=\phi_d$,
$g$ belongs to  $\bD\bG^0\bigl(\tilde{F}_{n-1,m}[[t_n^{1 \over m}]] \bigr)$ as desired.
This finishes the proof of our claim.

\medskip

We consider the twisted action of $ \tilde{\Gamma}_{n,m}$
on  ${\cal B}_n$  defined by 
$$
{\gamma } \, * \,   x = \eta(\gamma) \, . \, {^\gamma} x.
$$

The extension of local fields (with respect to $t_n$)
$\tilde{F}_{n,m}/F_n$ is tamely ramified. The
Bruhat-Tits-Rousseau theorem
states that the Bruhat-Tits  building of
$({_\eta\bG^0})_{F_n}$
can be identified with  ${\cal B}_n^{\tilde \Gamma_{n,m}},$
i.e. the fixed points of the building ${\cal B}_n$ under the twisted action (\cite{Ro} and \cite{Pr}).
But by Corollary  \ref{flag}.3, the  $F_{n-1}((t_n))$--group   ${_\eta\bG}^0 \times_{R_n} F_{n-1}((t_n))$ is anisotropic, so its building consists of a single
point, which is in fact $\phi.$ Indeed since our loop cocycle has value in $G(\tilde k) \rtimes \Gamma_{n,m}$, 
$\phi$ is  fixed under the twisted action of 
$\tilde \Gamma_{n,m}$. This shows that
$$
{\cal B}_n^{ \Gamma_{n,m}} = \{ \phi \}.
$$

We claim that $h_n. \phi=\phi$. We have

\vskip-4mm

\begin{eqnarray} \nonumber
\gamma * (h_n \, . \, \phi) & = & \eta(\gamma) \, \, ^\gamma (h_n) \, . \, {^\gamma \phi} \\  \nonumber 
& = & \eta(\gamma) \, \, ^\gamma (h_n) \, . \,  \phi \qquad 
\hbox{ [$\phi$ is invariant under $\tilde \Gamma_{m,n}$]}  \\  \nonumber 
& = &   h_n \, . \, \eta'(\gamma)  \, \,  \phi \qquad 
\hbox{[relation \enskip \ref{*}]}  \\  \nonumber 
& = &   h_n \, . \,  \phi \qquad 
\hbox{[$\eta'(\gamma) \in \bG(\tilde k)]$ and claim \ref{claim5}]}  \\  \nonumber 
\end{eqnarray}

\noindent for every $\gamma \in  \tilde \Gamma_{n,m}$.   
Hence $h_n \, . \, \phi \in {\cal B}_n^{ \Gamma_{n,m}}$ and therefore
$h_n \, . \, \phi=\phi$ as desired.

It then follows that $h_n \in
 {\bG}(\tilde{F}_{n-1,m}[[t_n^{1 \over m}]])$.
By specializing (\ref{*}) at $t_n=0$, we obtain an element
 $h_{n-1} \in \bG( \tilde{F}_{n-1,m})$ such that
\begin{equation}\label{**}
h_{n-1}^{-1} \, \eta(\gamma  ) \, \ {^\gamma}h_{n-1} =  \eta^{'}(\gamma)
\quad \forall \gamma \in \tilde \Gamma_{n,m}.
\end{equation}
Since $\eta$ and $\eta'$ have  trivial arithmetic part, it follows
that $h_{n-1}$ is invariant under $\Gal(\tilde{k}/k)$.
We apply now the relation (\ref{**}) to the generator
$\tau_n$ of $\Gal\bigl(\tilde{F}_{n,m}/ \tilde{ F}_{n-1,m}((t_n)) \bigr).$
This yields
\begin{equation}\label{***}
h_{n-1}^{-1}\, \eta(\tau_n)   \, h_{n-1} =  \eta'(\tau_n),
\end{equation}
where $\eta(\tau_n),\,  \eta'(\tau_n) \in {\bG}(\tilde{k})$
and $h_{n-1} \in  \bG(F_{n-1,m})=\bG(\tilde{F}_{n-1,m})^{\Gal(\tilde k/k)}$. If we denote by
$\bmu_m^{(n)}$  the last factor
 of $\bmu_m^{n}$ then (\ref{***}) implies that ${\eta^{geo}}_{\mid \bmu_m^{(n)}} $ and ${ \eta^{'geo}}_{\mid \bmu_m^{(n)}}$
are conjugate under $\bG(F_{n-1,m})$.

\begin{claim} 
 ${\eta^{geo}}_{\mid \bmu_m^{(n)}} $ and 
$ {\eta^{'geo}}_{\mid \bmu_m^{(n)}}$ are conjugate  under  $\bG(k)$.
\end{claim}

The transporter $\bX:=\{h \in \bG  \, \mid \,  \Int(h) \circ {\eta^{geo}}_{\mid \bmu_n^{(n)}} = {\eta'^{geo}}_{\mid \bmu_n^{(n)}}
 \}$ is 
a non-empty $k$--variety which is a
 homogeneous space under the group
 $\bZ_{\bG}({\eta^{geo}}_{\mid \bmu_n^{(n)}})$. 
Since $\bX(F_{n-1,m}) \not = \emptyset$ and $F_{n-1,m}$ is an iterated
Laurent series field over $k$,  Florence's theorem  \cite[\S 1]{F}
shows that 
$\bX(k) \not= \emptyset$. The claim is thus proven.

\smallskip

Without loss of generality we may therefore assume that
${\eta^{geo}}_{\mid \bmu_m^{(n)}} =  {\eta^{'geo}}_{\mid \bmu_m^{(n)}}$.
The finite multiplicative $k$--group $ \bmu_m^{(n)}$ acts on $\bG$ via $\eta^{geo}$, and we let $\bG_{n-1}$ denote the $k$--group which is the centralizer of this action \cite[II 1.3.7]{DG}. The connected component of the identity of $\bG_{n-1}$ is   reductive (\cite{Ri}, proposition 10.1.5).
Since the action of $ \bmu_m^{(n-1)}$  on $\bG$ given by  $\eta^{geo}$ commutes with that of $\bmu_m^{(n)}, $ the $k$--group morphism $\eta^{geo}: \bmu_m^n \to \bG$ factors through  $\bG_{n-1}.$ Similarly for  $\eta^{'geo}$.
Denote by $\eta_{n-1}^{geo}$ (resp. ${\eta'}_{n-1}^{geo}$) the restriction of 
$\eta^{geo}$ (resp. ${\eta'}^{geo}$) to the $k$--subgroup 
$\bmu_m^{n-1}= \bmu_n^{(1)} \times \cdots  \times \bmu_n^{(n-1)}$ of $\bmu_m^{n}$.
Set $\tilde \Gamma_{n-1, m} := \bmu_m^{n-1}(\ol k) \rtimes \Gal(\tilde k/k)$ 
and consider the loop cocycle $\eta_{n-1}: \tilde \Gamma_{n-1,m} \to \bG_{n-1}(\tilde k)$
attached to $(1, \eta^{geo}_{n-1}),$ and similarly  for ${\eta'}_{n-1}^{geo}$.

The crucial point for the induction argument we want to use is the fact  that $\eta^{geo}_{n-1}: \bmu_m^{n-1} \to \bG_{n-1}$ is anisotropic. For otherwise the $k$--group
$\bG_{n-1}^0$ admits a non-trivial split subtorus 
$\bS$ which is normalized by  $\bmu_m^{n-1}$. But then $\bS$ is
a non-trivial split subtorus of $\bG^0$ which is normalized by
$\bmu_m^{n}$, and this contradicts the anisotropic assumption on   $\eta^{geo}$.

Inside  $\bG_{n-1}( \tilde F_{n-1,m} )$, relation (\ref{**}) yields that 
$$
h_{n-1}^{-1} \, \, \eta_{n-1}(\gamma) \,  {^\gamma}h_{n-1} = 
 \eta'(\gamma)
\quad \forall \gamma \in \tilde \Gamma_{n-1}.
$$
which is similar to (\ref{*}).
 By induction on $n$, we may assume that $\eta^{geo}_{n-1}$ is $\bG_{n-1}(k)$--conjugate to 
${\eta'}_{n-1}^{geo}$. Thus $\eta^{geo}$ is $\bG(k)$--conjugate to 
${\eta'}^{geo}$ as desired.
\end{proof}

Before establishing Theorem \ref{k-aniso}, we need the following preliminary step.

\begin{lemma}\label{bateau} Let $\bH$ be a linear algebraic $k$--group.
It two loop classes $[\eta],  \, [\eta'] $ of $H^1\big(\pi_1(R_n),
\bH(\ol k)\big)$ have same image in $H^1(F_n, \bH)$,
then  $[\eta^{ar}]=  [{\eta'}^{ar}] $ in $ H^1(k,\bH)$.
\end{lemma}

\begin{proof}
Up to twisting $\bH$ by $\eta^{ar}$, the standard torsion argument
allows us   to assume with no loss of generality that   $\eta^{ar}$ is trivial, i.e. that $\eta$ is purely geometrical.
We are thus left to deal with the case of  a $k$--group homomorphism $\eta^{geo} : {_\infty\bmu} \to \bH$ that factors through  some
$\bmu_m^n \to \bH$ for $m > 0$ large enough. Hence $[\eta]$ is trivialized  by the extension $\tilde{R}_{n,m}/R_n$ and
its image in $H^1(F_n, \bH)$ by the extension $\tilde{F}_{n,m}/F_n,$ where $\tilde{R}_{n,m}$ and $\tilde{F}_{n,m}$ are as above.

By further  increasing  $m$,  the same reasoning allows us to assume that the image of $\eta'$ in $H^1\big(R_{n,m}, \bH(\ol k)\big)$ is
purely arithmetic. More precisely, that the map
$$
Z^1\big( \pi_1(R_n), \bH(\ol k)\big) \to Z^1\big( \pi_1(R_{n,m}), \bH(\ol k)\big)
$$
maps $({\eta'}^{geo}, {\eta'}^{ar})$ to  $(1, {\eta'}^{ar})$ where  the coordinates are as in Section \ref{basepoint}.
But our hypothesis implies that the image of $[\eta']$ in $H^1(F_{n,m}, \bH)$ is trivial, hence
$$
[{\eta'}^{ar} ] \in \ker\bigl(    H^1(k, \bH) \to  H^1(F_n, \bH) \bigr) .
$$
Since $F_n$ is an iterated  Laurent series field over $k$, this kernel is trivial (see \cite[\S 5.4]{F}),
and we conclude that $[{\eta'}^{ar} ]=1 \in H^1(k, \bH)$.

\end{proof}

We are now ready to proceed with the proof of Theorem \ref{k-aniso}.

\begin{proof} The implications $1) \Longrightarrow 2)
\Longrightarrow 3) $ are obvious.
We shall prove the implication  $3) \Longrightarrow 1)$ by using the previous result.
By assumption $[\eta]_{F_n}=  [\eta']_{F_n} \in H^1(F_n, \bG)$.
It is convenient  to work at finite level as we have done previously, namely with cocycles
$$
\eta, \eta^{'}:  \tilde \Gamma_{n,m} \to {\bG}(\tilde k)
$$
with $\tilde \Gamma_{n,m} :=  \bmu_m^n( \tilde k) \rtimes \Gal(\tilde{k}/k)$ where
$m > 0 $ large enough and $\tilde{k}/k$ is a finite Galois extension extension containing
all $m$--roots of unity in $\ol{k}.$
We associate to $\eta$ its arithmetic-geometric coordinate pair $(z, \eta^{geo})$ where
$z \in Z^1(  \Gal(\tilde{k}/k), \bG(\tilde k))$
and $\eta^{geo} :  \bmu_m^n \to {_z\bG}$ is a $k$--group homomorphism. Similar considerations apply to
$\eta',$ and its corresponding pair $(z', {\eta'}^{geo}).$
By Lemma \ref{bateau}, we have
 $[z]=[z'] \in H^1(k, \bG)$.
Without lost of generality we may assume that $z=z'$.
Consider  the commutative diagram
$$
\begin{CD}
H^1\big(  \tilde \Gamma_{n,m} , {_z\bG}(\tilde k)\big)
 @>>> H^1\big(  \tilde \Gamma_{n,m} , {_z\bG}(\tilde F_{n,m})\big)  \\
 @V{\tau_z}V{\wr}V @V{\tau_z}V{\wr}V \\
H^1\big(  \tilde \Gamma_{n,m} , {\bG}(\tilde k)\big)  @>>> H^1\big(  \tilde \Gamma_{n,m} , {\bG}(\tilde F_{n,m})\big) . \\
\end{CD}
$$
where the vertical arrows are the torsion bijections. Thus  $\tau_z^{-1}[\eta] = \tau_z^{-1}[\eta'] \in H^1\big(  \tilde \Gamma_{n,m} , {_z\bG}(\tilde F_{n,m})\big)$.
 By Corollary \ref{flag}.3, $\eta^{geo} : \bmu_m^n \to {_z\bG}$ is an anisotropic $k$--group homomorphism.
We can thus apply  Theorem \ref{k-aniso-bis} to conclude that $\eta^{geo}$ and ${\eta'}^{geo}$ are conjugate
under ${_z\bG}(k)$ hence  $\tau_z^{-1}[\eta]  = \tau_z^{-1}[\eta'] $ in $H^1\big(  \tilde \Gamma_{n,m} , {_z\bG}(\tilde k)\big),$ and therefore  $[\eta] = [\eta']$ in $H^1\big(  \tilde \Gamma_{n,m} , {_z\bG}(\tilde k)\big)$ as desired.
\end{proof}

\begin{corollary} \label{inject} Let the assumptions be as in the  Theorem, and
let $H^1\big(\pi_1(R_n), \bG(\ol k)\big)_{an}$ denote
the preimage of  $H^1(F_n, \bG)_{an}$ under the composite map
 $H^1\big(\pi_1(R_n), \bG(\ol k)\big) \to  H^1(R_n, \bG) \to  H^1(F_n, \bG)$. Then
 $H^1\big(\pi_1(R_n), \bG(\ol k)\big)_{an}$ injects into $H^1(F_n, \bG)$. \qed
\end{corollary}

\medskip

\section{Acyclicity}

We have now arrived to one of the main results of our work

\begin{theorem}\label{acyclic}  Let  $\bG$ be  a linear algebraic group over a field $k$ of characteristic $0$.
Then
the natural restriction map $H^1(R_n,\bG) \to H^1(F_n,\bG)$ induces a bijection
$$
H^1_{loop}(R_n,\bG) \simlgr H^1(F_n,\bG).
$$
In particular, the inclusion map $H^1_{loop}(R_n,\bG) \to H^1(R_n,\bG)$ admits a canonical section.
\end{theorem}

\subsection{The proof}

For any $k$--scheme $\gX$, we denote by
  $H^1(\gX , \bG)_{irr} \subset
 H^1(\gX , \bG)$\footnote{We remind the reader that $H^1(\gX , \bG)$ stands for $H^1(\gX , \bG_{\gX}).$} the subset  consisting of  classes of
$\bG$--torsor $\bE$ over $\gX$ for which the twisted reductive $\gX$--group scheme  $_{\bE}\big(\bG^0/ R_u(\bG^0)\big)_{\gX}$ does not contain a proper parabolic subgroup which admits a Levi subgroup.\footnote{Recall that the assumption on the existence of the Levi subgroup is superfluous whenever $\gX$ is affine.}
Set  $H^1_{loop}(\gX , \bG)_{irr}
= H^1_{loop}(\gX , \bG) \, \cap  \, H^1(\gX , \bG)_{irr}$.
We begin with the following special case.

\begin{lemma}\label{irr-case}  $H^1_{loop}(R_n , \bG)_{irr}$ injects into  $H^1(F_n, \bG)$.
\end{lemma}

\begin{proof} 
By Lemma \ref{sansuc}, we can assume without loss of generality that
$\bG^0$ is reductive.
We have an exact sequence
$1 \to \bG^0 \buildrel i \over \lgr   \bG \buildrel p \over \lgr  \bnu \to 1$ where $\bnu$ is a finite \'etale $k$--group.
We are given two loop cocycles $\eta$, $\eta'$ in $Z^1(R_n, \bG)$ which have the same image
in $H^1(F_n, \bG)$, and for which  the twisted  $F_n$--groups ${_\eta \bG}^0$, ${_{\eta'} \bG}^0$
are irreducible. 
Since $H^1\big(\pi_1(R_n), \bnu(\ol k)\big) \simlgr H^1(F_n, \bnu)$, it follows that
$p_*[\eta]= p_*[\eta']$ in $H^1\big(\pi_1(R_n), \bnu(\ol k)\big)$.
We can thus assume without loss of generality that
$p_*\eta= p_*\eta' $ in  $Z^1\big(\pi_1(R_n), \bnu(\ol k)\big)$.
Furthermore, 
as in the proof of Theorem \ref{k-aniso} the standard twisting argument reduces the problem to the case of purely
geometric loop torsors. In particular, the group actions of $\eta^{geo}$ and 
${\eta'}^{geo}$ are irreducible according to Corollary \ref{flag}.3.

Let   $\bC$ be the connected center of $\bG^0$. Then $\bC$ is a $k$--torus
equipped with an action  of $\bnu$.
We consider its $k$--subtorus
$\bC^\sharp := ( \bC^{p\circ \eta^{geo}})^0$
and denote by $\bC_0$ its maximal $k$--split subtorus which is  defined
by $\widehat{\bC_0}^0 = \widehat{ \bC^\sharp}^0(k)$.
By construction, $\eta^{geo}: {_\infty\bmu}^n \to \bG$ centralizes 
$\bC^\sharp$ and $\bC_0.$ Similarly for  ${\eta'}^{geo}$.
The $k$--torus $\bC_0$ is  a split subtorus of $\bC$ centralized by $\eta^{geo}$
and maximal for this property.
We consider the exact sequence of
$k$--groups
$$
1 \to \bC_0 \to \bG \buildrel q \over \lgr  \bG/\bC_0 \to 1
$$
\begin{claim} The composite $q\circ \eta^{geo} :  {_\infty\bmu}^n \to \bG/\bC_0$
 is anisotropic.
\end{claim}

Let us establish the claim. We are given a  split subtorus $\bS$ of the $k$--group $\bG^0$
which is centralized by  $q\circ \eta^{geo}$. Since  $q\circ \eta^{geo}$  is irreducible,
$\bS$ is central in $\bG^0/ \bC_0$. We consider $\bM=q^{-1}(\bS)$, this is an extension
of $\bS$ by $\bC_0$, so it is a split $k$--torus. By the semisimplicity of the 
category of representations of ${_\infty\bmu}^n$, we see that ${_\infty\bmu}^n$ 
acts trivially on $\bM$.
Then $\bM=\bC_0$ and $\bS=1,$ and the claim thus holds.

Next we twist  the  sequence of $R_n$--groups
$$
1 \to \bC_0 \to \bG \buildrel q \over \lgr  \bG/\bC_0 \to 1
$$
 by $\eta$ to obtain
$$
1 \to \bC_0 \to {_\eta \bG} \to {_\eta(\bG /\bC_0)} \to 1.
$$
This leads to  the commutative exact diagram of pointed sets
$$
\begin{CD} 
0= H^1(R_n, \bC_0) @>>>  H^1(R_n, \bG) @>{q_*}>>  H^1(R_n, \bG / \bC_0)  \\
   && @A{\tau_{\eta}}A{\simeq}A @A{\tau_{\eta}}A{\simeq}A \\
 0= H^1(R_n, \bC_0) @>>> H^1(R_n, {_\eta\bG}) @>>>  H^1\big(R_n, {_\eta(\bG / \bC_0)}\big)  \\
 \end{CD}
$$
where the vertical maps are the  torsion bijections.
Note that  $H^1(R_n, \bC_0)$ vanishes since $\Pic(R_n)=0$.
By diagram chasing we have  $[\eta]=[\eta']$ in $H^1(R_n, \bG)$
if and only if  $q_*[\eta]=q_*[\eta']$ in $H^1(R_n, \bG/ \bC_0)$.
Since $q_*[\eta]_{F_n}=q_*[\eta']_{F_n}$ in $H^1(F_n, \bG/ \bC_0)$ it will suffice to establish the Lemma for $\bG/ \bC_0$. The claim states  that $q_*\eta^{geo}$ is anisotropic,
therefore $q_*[\eta]=q_*[\eta']$ in $H^1(R_n, \bG/ \bC_0)$ by Theorem \ref{k-aniso-bis}.
\end{proof}

We can now proceed to prove Theorem \ref{acyclic}.

\begin{proof} The surjectivity of the map
$H^1_{loop}(R_n,\bG) \to H^1(F_n,\bG)$ is a special case of Proposition
\ref{onto}. Let us establish injectivity.
We are given  two loop cocycles $\eta, \eta' \in
Z^1\big(\pi_1(R_n),\bG(\ol k)\big)$ having the same image in $H^1(F_n, \bG)$.
Lemma \ref{bateau} shows that  $[\eta^{ar}]= [{\eta'}^{ar}]$ in
$H^1(k,\bG)$. Up to twisting $\bG$ by $\eta^{ar}$, we may assume that
$\eta$ and $\eta'$ are purely geometrical loop torsors.
The proof now proceeds  by  reduction to the irreducible case, i.e. to the case when
${_\eta\bG}^0 \times_{R_n} F_n$ is irreducible.

Let $\bQ$ be a minimal $F_n$--parabolic subgroup
of ${_\eta\bG}^0 \times_{R_n} F_n$. Corollary \ref{flag}
shows that the $k$--group $\bG^0$ admits a parabolic subgroup $\bP$
of the same type as $\bQ$ which is normalized by $\eta.$
The same statement shows that $\eta'$ normalizes a parabolic
subgroup, say $\bP'$, of  the same type than $\bP.$  Since by Borel-Tits theory $\bP'$ is $\bG^0(k)$--conjugate
to $\bP$  we may assume that
$\eta'$ normalizes $\bP$ as well.
Furthermore, $\bP$ is minimal for $\eta$ (and $\eta'$)
with respect to this property. We can view then
$\eta$, $\eta'$ as elements of
$Z^1_{loop}\big(R_n, \bN_{\bG}(\bP)\big)_ {irr}$.
We look at the following commutative diagram
$$ 
\begin{CD}
H^1_{loop}\big(R_n, \bN_{\bG}(\bP)\big)_ {irr}
@>>> H^1(R_n, \bG) \\
@VVV@VVV \\
 H^1\big(F_n, \bN_{\bG}(\bP)\big)_ {irr}
@>{\sim}>> H^1(F_n, \bG) \\
\end{CD}
$$
Since the bottom map is injective (see \cite[th. 2.15]{Gi4}), it   will suffice  to show that
$H^1_{loop}\big(R_n, \bN_{\bG}(\bP)\big)_ {irr}$ injects in
$H^1\big(F_n, \bN_{\bG}(\bP_I)\big)$.
Since the unipotent radical $\bU$ of $\bP$ is a split unipotent group, 
we have $$
H^1\big(R_n, \bN_{\bG}(\bP)\big)
 \simeq H^1\big(R_n, \bN_{\bG}(\bP)/ \bU\big),
$$ and similarly for $F_n$ by 
 Lemma \ref{sansuc}. So we are reduced to showing that
$H^1_{loop}\big(R_n, \bN_{\bG}(\bP)/ \bU\big)_ {irr}$ injects in
$H^1\big(F_n, \bN_{\bG}(\bP)/ \bU\big)$,
which is covered by Lemma \ref{irr-case}. This completes the proof of
injectivity.
\end{proof}

\medskip

\subsection{Application:  Witt-Tits decomposition}

By using the Witt-Tits decomposition over $F_n$ \cite[th. 2.15]{Gi4}, we get the following.

\begin{corollary}\label{WT3} Assume that $\bG^0$ is a split reductive $k$--group.
Let $\bP_{I_1}$,  ...  , $\bP_{I_l}$ be representatives
of the $\bG(\ol k )$-conjugacy classes of parabolic subgroups of $\bG^0.$
Let $\bL_{I_j}$ be a Levi subgroup of $\bP_{I_j}$ for $j=1,...l$.
Then
$$
\bigsqcup\limits_{j=1,...,l} \, H^1_{loop}\big(R_n, \bN_{\bG}(\bP_{I_j}, \bL_{I_j})\big)_ {irr}
\enskip \simeq \enskip  H^1_{loop}(R_n, \bG) \simeq  H^1(F_n,\bG).
$$
\end{corollary}

\begin{remark}
 It follows from the Corollary that we have a ``Witt-Tits decomposition" for loop torsors.
Furthermore, if we are interested in purely geometrical irreducible  loop torsors,
then we have a nice description in terms of $k$--group homomorphisms ${_\infty \bmu} \to \bG$ as described in Theorem \ref{k-aniso-bis}. This corresponds
 to the embedding of $\Hom_{k-gp}({_\infty \bmu}_{irr}^n, \bG) / \bG(k)$ in $H^1(R_n,\bG)$.
\end{remark}

For future use we record the  connected case.

\begin{corollary}\label{WT4}  Assume that $\bG$ is a split reductive group.

\begin{enumerate} 
\item Let $\bP_{I_1}$,  ...  , $\bP_{I_l}$ be the standard 
$k$--parabolic subgroups containing a given Borel subgroup of $\bG/k$.
Then
$$
\bigsqcup\limits_{j=1,...,l} \, H^1_{loop}(R_n, \bP_{I_j})_ {irr}
\enskip \simeq \enskip  H^1_{loop}(R_n, \bG) \simeq  H^1(F_n,\bG).
$$
\item If $k$ is algebraically closed$$
\Hom_{k-gp, irr}({_\infty \bmu}^n, \bG) / \bG(k) \simeq
H^1_{loop}(R_n, \bG)_{irr} \simeq  H^1(F_n,\bG)_{irr}.
$$
\end{enumerate}
\end{corollary}

Using our  choice of roots of unity (\ref{laurent}), 
we have ${_\infty\bmu} \simeq \widehat \Z$. So the left  handside is
$\Hom_{gp}\big({\widehat \Z}^n, \bG(k)\big)_{irr} / \bG(k)$, namely
the  $\bG(k)$--conjugacy classes
of finite order irreducible
 pairwise commuting elements $(g_1,...,g_n)$ (irreducible in the sense that the elements do not belong to a proper parabolic subgroup).

\subsection{Classification of semisimple $k$--loop adjoint groups}

Next  we discuss in detail the important  case where our algebraic group is the group  $\bAut(\bG)$ of automorphisms of a  split semisimple group $\bG$ of adjoint type. This is the situation needed to classify forms of the $R_n$--group $\bG \times_k R_n$
 and of the corresponding $R_n$--Lie algebra $\mathfrak{g} \otimes_k R_n$ where $\mathfrak{g}$ is the Lie algebra of $\bG.$ Indeed it is this particular case, and its applications to infinite- dimensional Lie theory as described in \cite{P2} and  \cite{GP2} for example,  that have motivated our present work.

We fix a  ``Killing couple"  $\bT \subset \bB$ of  $\bG,$ as well as  a base  $\Delta$ of the  corresponding root system.
For each subset $I \subset \Delta$ we define as usual $$
\bT_I = \bigl( \bigcap\limits_{\alpha \in I} \ker(\alpha) \bigr)^0.
$$
Since $\bG$ is adjoint, we know that the roots define an isomorphism
$\bT \simeq (\GG_m)^{| \Delta |},$ hence
 $\bT_I \simeq (\GG_m)^{| \Delta \setminus I |}$.
The centralizer  $\bL_I:= \bC_{\bG}(\bT_I),$ is the standard Levi subgroup of the
parabolic subgroup $\bP_I = \bU_I \rtimes \bL_I$ attached to $I$.
Since $\bG$ is of adjoint type, we know that $\bL_I/ \bT_I$ is a semisimple  $k$--group of adjoint type.

We have a split exact sequence of $k$--groups
$$
1 \to \bG \to \bAut(\bG) \to \bOut(\bG) \to 1
$$
where $\bOut(\bG)$ is the finite constant $k$--group corresponding to the finite (abstract) group $\Out (\bG)$ of symmetries  of the Dynkin diagram of $\bG.$ [XXIV \S3].\footnote{In \cite{SGA3} the group $\bOut(\bG)$ is denoted by
$\bAut({\rm Dyn}).$}
For  $I \subset \Delta$, we need to describe the normalizer $\bN_{\Aut(\bG)}(\bL_I)$
of $\bL_I$.
Following \cite[16.3.9.(4)]{Sp}, we define the subgroup of
$I$-automorphisms of $\bG$ by
$$
\bAut_I(\bG) = \bAut(\bG,\bP_I, \bL_I) 
$$
where the latter group is the subgroup of $\bAut(\bG)$ that stabilizes both $\bP_I$ and $\bL_I.$ There is then an exact sequence
$$
1 \to \bL_I \to \bAut_I(\bG) \to \bOut_I(\bG) \to 1,
$$
where $\bOut_I(\bG)$ is the finite constant group corresponding to  the subgroup of $\Out(\bG)$ consisting of elements that stabilize $I \subset \Delta.$ Then the preceding Corollary reads
$$
\bigsqcup\limits_{ [I] \subset \Delta/ \Out(\bG)} \, H^1_{loop}\big(R_n, \bAut_I(\bG)\big)_ {irr}
\enskip \simeq \enskip  H^1_{loop}\big(R_n, \bAut(\bG)\big) \simeq  H^1\big(F_n,\bAut(\bG)\big).
$$
By \cite[cor. 3.5]{Gi4},
 $H^1_{loop}\big(R_n, \bAut_I(\bG)\big)_ {irr} \simeq H^1\big(F_n, \bAut_I(\bG)\big)_ {irr}$
can be seen as a subset of
$H^1_{loop}\big(R_n, \bAut_I(\bG)/\bT_I\big)_ {an} \simeq H^1\big(F_n, \bAut_I(\bG)/\bT_I\big)_ {an}$.
We come now to another of the central results of the paper.

\begin{theorem}\label{mainn} Assume that $k$ is algebraically closed and of characteristic $0.$ Let $\bG$ be a 
 simple $k$--group of adjoint type. Let $\bT \subset \bB$, $I$, and $\Delta$ be as above.
On the set $\Hom_{k-gp}\big( {_\infty\bmu}^n, \bAut_I(\bG)\big)$ define the equivalence relation
$\phi \sim_I \phi'$ if there exists $g \in \bAut_I(\bG)(k)$ such that
$\phi$ and  $g \phi' g^{-1}$ have same image in $\Hom_{k-gp}\big( {_\infty\bmu}^n, \bAut_I(\bG)/ \bT_I\big)$.
 Then we have a decomposition
 $$
  \bigsqcup\limits_{ [I] \subset \Delta/ \Out(\bG)} \, 
\Hom_{k-gp}( {_\infty\bmu}^n, \bAut_I(\bG))_{an} / \sim_I
\enskip \simlgr \enskip H^1_{loop}(R_n, \bG) \enskip \simlgr H^1(F_n, \bG).
$$
where $\Hom_{k-gp}\big(  {_\infty\bmu}^n, \bAut_I(\bG)/\bT_I\big)_{an}$ stands for the set of anisotropic
group homomorphisms ${_\infty \bmu}^n \to  \bAut_I(\bG)/\bT_I$.

\end{theorem}

\begin{remark}\label{Leftschetz} As an application of Margaux's rigidity theorem \cite{Mg2}, the right handside does not change by extension of algebraically closed fields. Hence
$H^1_{loop}(R_n, \bG)$ does not change by extension of algebraically closed fields.
This allows us in practice whenever useful to work over $\overline \Q$ or $\C$.
\end{remark}

\begin{proof}  The group $(\bAut_I(\bG)/\bT_I)(k)$  acts naturally on
 the set $\Hom_{k-gp}\big(  {_\infty\bmu}^n, \bAut_I(\bG)/\bT_I\big)_{an}$ by conjugation, and we denote the resulting quotient set by $\overline{\Hom}_{k-gp}\big(  {_\infty\bmu}^n,
\bAut_I(\bG)/\bT_I\big)_{an}.$ The commutative square
$$
\begin{CD}
\Hom_{k-gp}\big( {_\infty\bmu}^n, \bAut_I(\bG)\big)_{irr}  @>>> \overline{\Hom}_{k-gp}\big(  {_\infty\bmu}^n, \bAut_I(\bG)/\bT_I\big)_{an} \\
@VVV @VVV \\
 H^1_{loop}\big(R_n, \bAut_I(\bG)\big)_ {irr} @>>> H^1_{loop}\big(R_n, \bAut_I(\bG)/ \bT_I\big)_ {an} \\
\end{CD}
$$
is well defined as one can see by taking into account Corollary \ref{flag}.
Since $k$ is algebraically closed, loop torsors are purely geometric, hence
the two vertical maps are onto.
As we have seen, the bottom horizontal map is injective;
 this  defines an equivalence relation $\sim_I'$ on the set
$\Hom_{k-gp}\big(  {_\infty\bmu}^n, \bAut_I(\bG)\big)$ such that
$$ 
\Hom_{k-gp}\big(  {_\infty\bmu}^n, \bAut_I(\bG)\big)_{irr} / \sim_I'  \enskip
\simlgr
\enskip
H^1_{loop}(R_n, \bAut_I(\bG)).
$$
It remains to establish that  the equivalence relations $\sim_I'$ and $\sim_I$ coincide, and we do this by using
that the right vertical map in the above diagram is injective.
(Theorem \ref{k-aniso-bis}).
We are given
$\phi_1, \phi_2 \in \Hom_{k-gp}\big(  {_\infty\bmu}^n, \bAut_I(\bG)\big)_{irr}$.
Then $\phi_1 \sim_I' \phi_2$ if and only if the image of
$\phi_1$ and $\phi_2$ in $\Hom_{k-gp}\big(  {_\infty\bmu}^n, \bAut_I(\bG)/\bT_I\big)_{an}$ are conjugate by an element
of $\bigl(\bAut_I(\bG)/\bL_I\bigr)(k)$.
Since  the map $\bAut_I(\bG)(k) \to  \big(\bAut_I(\bG)/\bL_I\big)(k)$ is onto,
it follows that $\phi_1 \sim_I' \phi_2$ if and only if $\phi_1 \sim_I \phi_2$ as desired.
\end{proof}

\begin{corollary} Under the hypothesis of Theorem \ref{mainn},  the classification of
loop torsors on $R_n$ ``is  the same"  as the classification, for each subset $I \subset \Delta,$   
 of irreducible
commuting $n$--uples of elements of finite order  of $\bAut_I(\bG)(k)$ up to the equivalence relation $\sim_I$.
\end{corollary}

\subsection{Action of $\GL_n(\Z)$} 

The assumptions are as in the previous section. We fix a 
 pinning (\'epinglage)  $(\bG,\bB,\bT)$ [XXIV \S1]. This determines a section  $s : \bOut(\bG) \to \bAut(\bG).$

 The group $\GL_n(\Z)$ acts on the left as automorphisms of the $k$--algebra $R_n$ via
\begin{equation}\label{GLonRn}
 g =(a_{ij}) \in \GL_n(\Z) : t_i \mapsto t_1^{a_{1i}} t_2^{a_{2i}} \dots t_n^{a_{ni}}
 \end{equation}
We denote the resulting $k$--automorphism of $R_n$ corresponding to $g$ also by $g$ since no confusion will arise. By Yoneda considerations $g$ (anti)corresponds to an automorphism $g^*$ of the $k$--scheme $\Spec(R_n).$

Applying (\ref{GLonRn})  where we now replace $t_i$ by $t_1^{1/m}$ and $k$ by $\ol k$ gives a left action of $\GL_n(\Z)$ as automorphisms of $\overline{R}_{n,m} =  \overline{k}[t^{\pm{1/m}}, \dots, t_n^{\pm{1/m}}].$ If we denote by $g_m$ the automorphism corresponding to $g$ then the diagram

 $$
\begin{CD}\label{gaction}
 R_n@>{g}>>  R_n\\
@VVV @VVV\\
\overline{R}_{n,m}@>{g_m}>> \overline{R}_{n,m}\\
\end{CD}
$$ 
commutes. Passing to the direct limit on (\ref{gaction}) the element $g$ induces an automorphism $g_{\infty}$ of  $\overline{R}_{n,\infty} =  \limind\overline{k}[t^{\pm{1/m}}, \dots, t_n^{\pm{1/m}}].$ If no confusion is possible, we will denote $g_{\infty}$ and $g_m$ simply by $g.$

Recall that $\pi_1(R_n) = \widehat{\Z}(1)^n \rtimes \Gal(k).$ Our fixed choice of compatible roots of unity $(\xi_m)$ allows us to identify $\pi_1(R_n)$ with $\widehat{\Z}^n \rtimes \Gal(k)$ where the left action of $\Gal(k)$ on each component $\widehat{\Z} = \limproj \Z/m\Z$ is as follows:  If $a \in \Gal(k)$ and $m \geq 1$ there exists a unique $1 \leq a(m) \leq m-1$ such that $a (\xi_m) = \xi_m^{a(m)}.$ This defines an automorphism $a_m$ of the additive group $\Z/m\Z.$ Passing to the limit on each component yields the desired group automorphism $\widehat{a}$ of $\widehat{\Z}^n.$

View $\big(\Z/m\Z\big)^n$ as row vectors. Then $\GL_n(\Z)$ acts on the right on this group by right multiplication
$$
g : {\bf \rm e}_m \mapsto  {\bf \rm e}_m^g = {\bf \rm e}_mg
$$
where $\big(\Z/m\Z\big)^n$ is viewed as a $\Z$--module in the natural way. By passing to the inverse limit we get a right action of  $\GL_n(\Z)$ as automorphisms of $\widehat{\Z}^n$ that we denote by ${\bf \rm e} \mapsto {\bf \rm e}^g.$ We extend this to a right action on $\pi_1(R_n) = \widehat{\Z}^n \rtimes \Gal(k)$ by letting $\GL_n(\Z)$ act trivially on $\Gal(k).$ Thus if $\gamma = ({\bf \rm e}, a) \in \widehat{\Z}^n \rtimes \Gal(k)$ and $g \in \GL_n(\Z),$ then $\gamma^g = ({\bf \rm e}^g, a).$\footnote{After our identifications, this is nothing but the natural action of $g^*$ on $\pi_1\big(\Spec(R_n)\big).$}

By taking the foregoing discussion into consideration we can define the (right) semidirect product group $\GL_n(\Z) \ltimes \pi_1(R_n)$ with multiplication
\begin{equation}\label{semiproduct1}
\big(h,({\bf \rm e}, a)\big)\big(g , ({\bf \rm f}, b)\big) = \big( hg, ({\bf \rm e}^g, a)({\bf \rm f}, b)\big) = \big( hg, ({\bf \rm e}^g \widehat{a}({\bf \rm f}), ab)\big)
\end{equation}
for all $h,g \in \GL_n(\Z)$, $ {\bf \rm e}, {\bf \rm f} \in \widehat{\Z}^n$ and $a,b \in \pi_1(R_n).$ For future use we point out that under that under the natural identification of  $\GL_n(\Z)$ and $ \pi_1(R_n)$ with subgroups of $\GL_n(\Z) \ltimes \pi_1(R_n)$ we have
\begin{equation}\label{semiproduct}
\gamma g = g\gamma^g
\end{equation}
for all $g \in \GL_n(\Z)$ and $\gamma \in \pi_1(R_n).$

By definition $\pi_1(R_n)$ acts naturally on $\overline{R}_{n,\infty}.$ Under our identification  $\pi_1(R_n) = \widehat{\Z}^n \rtimes \Gal(k)$ the action is given by

\begin{equation}\label{semiproduct2}
({\bf \rm e}, a) : \lambda t_i^{1/m} \mapsto a(\lambda)\xi_m^{e_{m,i}}t_i^{1/m}
\end{equation}
where ${\bf \rm e} = ({ \bf \rm e}_1, \dots, {\bf \rm e}_n) \in \widehat{\Z}^n,$ ${\bf \rm e}_i = {({\rm e}_{m,i})}_{m \geq 1}$ with $0 \leq {\rm e}_{m,i}  < m,$ and $\lambda \in k.$
Using (\ref{semiproduct1}) and (\ref{semiproduct2}) it is tedious but straightforward to check that the group $\GL_n(\Z) \ltimes \pi_1(R_n)$ defined above acts on the left as automorphisms of the $k$--algebra $\overline{R}_{n,\infty}$ in a way which is compatible with the left actions of each of the groups, i.e.
\begin{equation}\label{compatible1}
(g\gamma). x = g . (\gamma . x)
\end{equation}  
for all $g \in  \GL_n(\Z),$  $\gamma \in \pi_1(R_n)$ and $x \in\overline{R}_{n,\infty}.$

In the reminder of this section we let $\bH$ denote  a linear algebraic group over $k$. Each element $g \in \GL_n(\Z)$  viewed as an automorphism $g^*$ of the $k$--scheme $\Spec(R_n)$ induces by functoriality  a bijection, also denoted by $g^*,$ of the pointed set $H^1(R_n, \bH)$ onto itself. This leads to a left action of $ \GL_n(\Z)$ on this pointed set which we called {\it base change.}  Our objective is to have a precise description of this action.\footnote{Our main interest is the case when $\bH = \bAut(\bG)$ with $\bG$ simple.  The reason behind 
the importance of this case lies in the applications to infinite-dimensional 
Lie theory.} 
The isotriviality theorem \cite[th. 2.9]{GP3} shows tha it will suffice to trace the base change action at the level of
$1$-cocycles in $Z^1\bigl(\pi_1(R_n), \bH({\ol R}_{n,\infty})\bigr).$
Following standard conventions for cocycles  we denote the action of an element $\gamma \in \pi_1(R_n)$ on an element $h \in \bH({\ol R}_{n,\infty})$ by ${^\gamma}h.$ Then (\ref{compatible1}) implies that
\begin{equation}\label{compatible2}
\gamma . h = {^\gamma}h.
\end{equation}

\begin{lemma}\label{cocycles-action} 
The base change action of $\GL_n(\Z)$ 
on $H^1(R_n, \bH)$ is induced by the  action $ \eta \mapsto {^g}\eta$ of  $\GL_n(\Z)$ on
 $Z^1\bigl(\pi_1(R_n), \bH({\ol R}_{n,\infty})\bigr)$  given by
 $$
({^g}\eta)(\gamma)= g. \eta(  \gamma ^g)$$
for all $\gamma \in \pi_1(R_n)$ and $  g \in \GL_n(\Z).$
\end{lemma}

\begin{proof} For all $\alpha, \beta \in \pi_1(R_n)$ we have
\begin{eqnarray} \nonumber
{^g} \eta(\alpha \beta) &=&   g.\eta\big((\alpha \beta)^g)\big)   \quad\hbox{[definition]} \\ \nonumber
&=&   g.\eta(\alpha ^g \beta^g) \\ \nonumber 
&=&   g.\big(\eta(\alpha ^g) \, {^{\alpha^g}} \eta(\beta ^g)\big)   \quad\hbox{[$\eta$ a cocycle]}  \\ \nonumber
&=&   g.\Big(\eta(\alpha ^g) \, \big({\alpha^g} .\eta(\beta ^g)\big)\Big)   \quad\hbox{[by (\ref{compatible2})]}  \\ \nonumber
&=&   \big(g.(\eta(\alpha ^g)\big)\big(g. {\alpha^g} .\eta(\beta ^g)\big)   \quad\hbox{[by action axiom]}  \\ \nonumber
&=&   \big(g.(\eta(\alpha ^g)\big)\big(\alpha . g .\eta(\beta ^g)\big)   \quad\hbox{[by action axiom and (\ref{semiproduct})]}  \\ \nonumber
&=&   {^g}\eta(\alpha) \big(\alpha . {^g}\eta(\beta)\big)   \quad\hbox{[by definition ]}  \\ \nonumber
&=&   {^g}\eta(\alpha) {^ \alpha} {\big({^g}\eta(\beta)\big)}   \quad\hbox{[by (\ref{compatible2}) ].}  \\ \nonumber
\end{eqnarray}

This shows that ${^g}\eta$ is a cocycle (which is clearly continuous since $\eta$ is). That this defines a left action of 
$\GL_n(\Z)$  on $Z^1\bigl(\pi_1(R_n), \bH({\ol R}_{n,\infty})\bigr)$ is easy to verify using the definitions.

Next we verify that the action factors through $H^1.$ Assume $\mu$ is a cocycle cohomologous to $\eta,$ and let $h \in \bH({\ol R}_{n,\infty})$ be such that $\mu (\gamma) = h^{-1} \eta{\gamma} {^ \gamma}h$ for all $\gamma \in \pi_1(R_n).$ Then 

\begin{eqnarray} \nonumber
{^g} \mu(\gamma) &=&   g.\mu({^g}\gamma) \quad\hbox{[definition]} \\ \nonumber
&=&   g.\big(h^{-1} \eta(\gamma^g) {^ \gamma}h\big)     \\ \nonumber
&=&   g.h^{-1} \, g.\eta(\gamma^g) \,  g.{^ \gamma}h    \quad\hbox{[action axiom]}  \\ \nonumber
&=&   (g.h)^{-1} \, {^g}\eta(\gamma) \,  {^\gamma}g^.{^ \gamma}h    \quad\hbox{[action axiom, definition, and $g = {^\gamma}g$]}  \\ \nonumber
&=&   (g.h)^{-1} \, {^g}\eta(\gamma) \,  {^\gamma}(g.h). \\ \nonumber 
 \end{eqnarray}
Thus ${^g} \mu$ and ${^g} \eta$ are cohomologous.

It remains to verify that the action we have defined on $H^1\bigl(\pi_1(R_n), \bH({\ol R}_{n,\infty})\bigr) = H^1\bigl(R_n, \bH\bigr)$ coincides with the base change action. To see this we consider 
a faithful representation  $\bH \to \bGL_d$  
and  the corresponding quotient variety  $\bY= \bGL_d/ \bH$.
Since $H^1(R_n, \bGL_n)=1$ by a variation of a theorem of  Quillen and Suslin (\cite{Lam} V.4), we have a short exact sequence of pointed sets 
$$
1 \to \bH(R_n) \to \bGL_d(R_n) \to \bY(R_n) \buildrel \varphi \over
\to H^1(R_n, \bH) \to 1.
$$
Therefore it is enough to verify our assertion for the image of the characteristic map $\varphi.$  Given  $y \in \bY(R_n)$,
by definition $\varphi(y)$ is the class  of the  cocycle 
$$\gamma \to 
\eta(\gamma)= Y^{-1} \, {^\gamma}Y = Y^{-1} \, \gamma.Y
$$
where $Y \in \GL_d({\ol R}_{n,\infty})$ is a lift of $y$ [the last equality holds by (\ref{compatible2})].
If $g \in \GL_d(\Z)$ we have
$$
g_*(\varphi(y)) = \varphi( g.y) 
$$  
by the equivariance of the characteristic map relative to $k$--schemes.
Since $g.Y$ is a lift of $g.y$, we conclude that $\varphi( g.y)$ is the class of the cocycle
$(g.Y)^{-1} {^\gamma}(g.Y).$ Using identities and compatibility of actions that have already been mentioned, we have 
$$
(g.Y)^{-1} {^\gamma}(g.Y) = (g.Y)^{-1} \gamma.(g.Y)= g. Y^{-1} \, g \gamma^g.Y =
$$
$$ 
= g. Y^{-1} \, g. \gamma^g.Y
=  g. \big(Y^{-1}  \gamma^g.Y\big) = g.\eta(\gamma^g) = {^g}\eta(\gamma)$$
as desired.
\end{proof}

\begin{remark}\label{GLpreserves} The action of $\GL_n(\Z)$ on $Z^1\big(\pi_1(R_n), \bH({\ol R}_{n,\infty})\big)$
stabilizes $Z^1\big(\pi_1(R_n), \bH(\ol k)\big)$. In particular, it preserves loop cocycles.
\end{remark}

We pause to observe that the  decomposition \ref{WT3} is
equivariant under the action of $\GL_n(\Z)$ on $R_n$. Thus

\begin{corollary}\label{WT3equi}  With the assumptions and notation as above

$$
 \bigsqcup\limits_{j=1,...,l} \, \GL_n(\Z) \backslash H^1_{loop}\big(R_n, \bAut(\bG,\bP_{I_j})\big)_ {irr} 
\enskip \simeq \enskip  \GL_n(\Z) \backslash H^1_{loop}\big(R_n, \bAut(\bG)\big) 
$$
In particular, if $\bE$ is a loop $R_n$--torsor under $\bAut(\bG)$,  the 
 Witt-Tits index of the loop group scheme $_\bE\bG/R_n$ 
depends only of the class of $\bE$ in  $\GL_n(\Z) \backslash H^1_{loop}\big(R_n, \bAut(\bG)\big)$.

\end{corollary}

\begin{remark}\label{GLacts} Assume $\eta$ is a loop cocycle. Since $\GL_n(\Z)$ acts trivially on $
\bH(k)$  we have 
$({^g}\eta)(\gamma)= \eta( \gamma ^g)$ for all $\gamma \in \pi_1(R_n)$ and $g \in \GL_n(\Z)$.
\end{remark}

\begin{lemma}\label{cocycles} Assume that  $\bH$ acts
on a quasi-projective $k$--variety $\bM$.
Let $\eta \in Z^1\big(\pi_1(R_n), \bH(\ol k)\big)$  be a loop
cocycle. Let $\Lambda_\eta \subset \bGL_n(\Z)$ be the stabilizer of
$\eta$ for the (left) action of $\bGL_n(\Z)$ on $Z^1\big(\pi_1(R_n), \bH(\ol k)\big)$. 

\medskip

(1) $\Lambda_\eta= \Bigl\{  \, g \in \bGL_n(\Z) \, \mid \,  \eta(\gamma ^g)= 
\eta(\gamma ) \enskip \forall \gamma \in \pi_1(R_n) \,  \Bigr\}. $

\smallskip

(2) The  map  $$
\big( \bGL_n(\Z) \ltimes \pi_1(R_n) \big) \times  \bM({\ol R}_{n, \infty}) 
\to \bM({\ol R}_{n, \infty}) , \qquad  \big((g,\gamma), x\big) \mapsto 
g.\eta(\gamma).\gamma.x
$$
defines an  action of $\Lambda_{\eta} \ltimes \pi_1(R_n)$ on 
$({_\eta\bX})({\ol R}_{n, \infty})$.

\smallskip

(3)  Assume that $\bM$ is a linear algebraic $k$--group on which 
$\bH$ acts  as group automorphisms.
Let $g\in \Lambda_{\eta}$ and $\zeta \in Z^1\big(\pi_1(R_n), 
{_\eta\bM}({\ol R}_{n, \infty} )\big)$, and set
$$
{^g}\zeta(\gamma)= g. \zeta(\gamma^g). 
$$
This defines a (left) action of $\Lambda_{\eta}$  on 
$Z^1\big(\pi_1(R_n),  {_\eta\bM}({\ol R}_{n, \infty} )\big)$ which induces
an action of $\Lambda_{\eta}$ on $H^1(R_n, {_\eta \bM})$. The action is 
functorial in $\bM.$
If $\bH=\bM$, the  diagram
$$
\begin{CD}
H^1(R_n, {_\eta\bH}) @>{\tau_\eta}>{^\sim}> H^1(R_n, \bH)  \\
@V{g_*}V{\wr}V @V{g_*}V{\wr}V \\
 H^1(R_n, {_\eta\bH}) @>{\tau_\eta}>{^\sim}> H^1(R_n, \bH)  \\
\end{CD}
$$
commutes for all $g \in \GL_n(\Z),$ where $\tau_{\eta}$ is the twisting bijection. 

\smallskip

(4) Assume that in (3) $\bH$ is finite and that $\bM$ is 
of multiplicative type.
For  $g \in \Lambda_\eta$ and 
 an inhomogeneous (continuous) cochain  $y \in {\mathcal C}^i\big(\pi_1(R_n),  {_\eta \bM}\big)$
of degree  $i \geq 0$, set
$$
(^gy)( \gamma_1, \dots  , \gamma_i)=  
{^g\bigl(} y ( \gamma_1^g, \dots , \gamma_i^g )  \bigr) .
$$ 
This defines a left action of $\Lambda_\eta$
on the chain complex ${\mathcal C}^*\big(\pi_1(R_n),  {_\eta\bM}(\ol R_{n,\infty})  \big)$
 of (contnuous)  inhomogeneous  cochains and on 
$H^*\big(\pi_1(R_n), {_\eta \bM}({\ol R}_{n,\infty})\big)= 
H^*(R_n, {_\eta \bM})$ which is functorial
with respect to short  exact  sequences of 
$\bH$--equivariant $k$--groups of multiplicative type.

\smallskip

(5) Assume that $\bH$ is finite and 
let $1 \to \bM_1 \to \bM_2 \to \bM_3 \to 1$ be a an exact sequence
of linear algebraic $k$--groups equipped with an equivariant action of $\bH$ as group automorphism.
The action of $\Lambda_\eta$ commutes with the characteristic map
${_\eta\bM}_3(R_n) \to H^1(R_n, {_\eta\bM}_1)$.
If  $\bM_1$ is central in $\bM_2$, then the 
action of $\Lambda_\eta$ commutes with the boundary map
$\Delta: H^1(R_n, {_\eta\bM}_3 ) \to H^2(R_n,
 {_\eta\bM_1})$.

\smallskip

(6) Assume $k$ is algebraically closed and  let $d$ be a postive integer. 
The base change action of $\GL_n(\Z)$ on  $H^2(R_2, \bmu)$ and on  $\Br(R_2)$  is given by $g. \alpha= \det(g). \alpha$.

\end{lemma}

\begin{proof} (1) is obvious by taking into account Remark \ref{GLacts}.

\smallskip

In what follows we take the ``Galois'' point of view and 
notation: ${_\eta\bM}({\ol R}_{n ,\infty})$ 
coincides with $\bM({\ol R}_{n ,\infty})$ as a set, but the action of $\pi_1(R_n)$ is the 
twisted action, which we denote by $\star$: 
$$\gamma \star x = \eta(\gamma).(\gamma.x)$$

\noindent (2) The groups $
 \bGL_n(\Z)$ and $\pi_1(R_n)$ act on $\bX({\ol R}_{n, \infty}) 
$ and $\bH({\ol R}_{n, \infty})$ via their natural action on ${\ol R}_{n, \infty}.$
We will denote 
these actions by $x \mapsto g.x$,  and $x \mapsto \gamma . x$  
for all $g \in 
 \bGL_n(\Z)$,  $\gamma \in  \pi_1(R_n)$ and $x \in  \bX({\ol R}_{n, \infty}).$ 
It follows from (\ref{compatible1}) and (\ref{compatible2}) that for all $\gamma \in  \pi_1(R_n)$ we have 
$$ \gamma.g.x = g. \gamma^g .x
$$
One also verifies using the axioms of action that 
$$ \gamma.(h.x) = (\gamma.h).(\gamma.x)$$
for all $h \in \bH({\ol R}_{n, \infty})$.
The content of (2) is that
\begin{equation}\label{actionofGamma}
 (g,\gamma) \star x = g. (\gamma \star x)
\end{equation}
defines an action of $\Lambda_{\eta} \ltimes \pi_1(R_n)$ on 
$({_\eta\bX})({\ol R}_{n, \infty}).$ Write for convenience $g. \gamma \star x$ instead of $g. (\gamma \star x)$ since no confusion is possible. Then

\begin{eqnarray} \nonumber
(f,\alpha)\star (g, \beta) \star x &=&  f.\eta(\alpha).\alpha.g.\eta(\beta).b.x
    \quad\hbox{[definition of the twisted action]} \\ \nonumber
&=&  f.\eta(\alpha).g.\alpha^g . \eta(\beta).\beta.x
    \quad\hbox{} \\ \nonumber 
&=&  f.g.\eta(\alpha).\alpha^g . \eta(\beta).\beta.x
     \quad\hbox{[$\eta $ is a loop cocycle]} \\ \nonumber
&=&  f.g.\eta(\alpha).(\alpha^g . \eta(\beta)).\alpha^g.(\beta.x)
     \quad\hbox{} \\ \nonumber
&=&  f.g.\eta(\alpha^g).(\alpha^g . \eta(\beta)).\alpha^g.(\beta.x)
     \quad\hbox{[$g \in \Lambda_{\eta}$]}\\ \nonumber
&=&  fg.\eta(\alpha^g \beta).\alpha^g . (\beta.x)
     \quad\hbox{[$\eta$ a cocycle]} \\ \nonumber
&=&  (fg,\alpha^g\beta) \star x
     \quad\hbox{} \\ \nonumber
&=&  \big((f,\alpha)(g,\beta)\big) \star x.
     \quad\hbox{}\\ \nonumber 
\end{eqnarray}

\smallskip

\noindent (3)  One checks that 
${^g}\eta$ is a cocycle and that two equivalent cocycles
remain equivalent under this action along the same lines as for the proof of Lemma \ref{cocycles-action}.

The commutativity of the diagram takes place already at the level of cocycles. Indeed.
Consider the square
$$
\begin{CD}
Z^1\big(\pi_1(R_n), {_\eta \bH}({\ol R}_{n,\infty})\big)  @>{\tau_\eta}>{^\sim}> 
Z^1\big(\pi_1(R_n), {\bH}({\ol R}_{n,\infty})\big)  \\
@V{g_*}V{\wr}V @V{g_*}V{\wr}V \\
Z^1\big(\pi_1(R_n), {_\eta \bH}({\ol R}_{n,\infty})\big) @>{\tau_\eta}>{^\sim}> Z^1\big(\pi_1(R_n), {\bH}({\ol R}_{n,\infty})\big)  \\
\end{CD}
$$
Given a cocycle $\phi \in Z^1\big(\pi_1(R_n), {_\eta \bH}({\ol R}_{n,\infty})\big)$ 
recall that $(\tau_\eta \phi)(\gamma)= \phi(\gamma) \,  \eta(\gamma)$, hence
$$
{^g}\big((\tau_\eta)(\phi)\big)(\gamma) = 
g.\big(\tau_{\eta}(\phi)(\gamma^g)\big)
  =  g.\big((\phi)(\gamma^g) \, \eta(\gamma^g)\big)  \quad, \gamma \in \Lambda_\eta.
$$ 
$$
g.(\phi)(\gamma^g)g \, . \, \eta(\gamma^g) = 
{^g}\phi(\gamma)g \, . \, \eta(\gamma^g) = 
{^g}\phi(\gamma) \, \eta(\gamma)
$$
since $g.\eta(\gamma^g) = \eta(\gamma^g)$ because $\eta$ is a loop cocycle, and 
$\eta(\gamma^g) = \eta(\gamma)$ because $g \in \Lambda_\eta.$  On the other hand by
 definition of the twisting map
$$
\tau_\eta ({^g}\phi)(\gamma)= {^g}\phi(\gamma)  \,  \eta( \gamma )
$$
so that the diagram above commutes.

\smallskip
   
\noindent (4) The continuous profinite  cohomology is the direct limit
of discrete group cohomology of finite quotients.
 Hence it is enough to establish the desired results  at the ``finite level", namely
for a group $\Gamma=\Gal( \tilde R_{n,m}/R_n)$ where 
$\tilde R_{n,m}= \tilde k \otimes_k R_{n,m}$ is a finite 
Galois covering of $R_n$ through wich $\eta$ factors, and such that $\bH(\tilde k)=\bH(\ol k)$.
Recall that the action of  $\bGL_n(\Z)$ 
 on ${\ol R_{n,\infty}}$  preserves  $\tilde k \otimes_k R_{n,m}$, so 
that $\bGL_n(\Z)$ acts on $\Gamma$.

We need to check that the given action of $\Lambda_{\eta}$
on  $C^*(\Gamma, {_\eta A})$ commutes with the differentials.
We are given  $g \in \Lambda_{\eta}$ and
 $y \in {\mathcal C}^{i}(\Gamma, {_\eta A})$.
Recall that the boundary map $\partial_i :  
{\mathcal C}^{i}(\Gamma, {_\eta A}) \to {\mathcal C}^{i+1}(\Gamma, {_\eta A})$ is given by

\begin{eqnarray} \nonumber
\bigl(\partial_i(y)\bigr) ( \gamma_1, \dots, \gamma_{i+1})
= \qquad \qquad \qquad \qquad \qquad \qquad \qquad \qquad 
\qquad \qquad \qquad \qquad \qquad \qquad  \\ \nonumber
  \gamma_1 \, . \, \eta(\gamma_1) .  y(\gamma_1, \dots, \gamma_{i+1})
+ \sum\limits_{j=1}^i \,  (-1)^j \,  y(\gamma_1, \dots, \gamma_{j-1},
\gamma_j \gamma_{j+1}, \gamma_{j+2} ,  \dots\gamma_{i+1}) + 
(-1)^{i+1} y(\gamma_1,\dots, \gamma_i) .
\end{eqnarray}
 
\noindent Thus

\begin{eqnarray*} \nonumber
\Bigl( ^g\bigl( \partial_i(y) \bigr)\Bigl)( \gamma_1, \dots, \gamma_{i+1})
 &=&g. \bigl( \partial_i(y)( \gamma_1^g, \dots, \gamma_{i+1}^g) \bigr) 
 \\ \nonumber
 &&  = g. \bigl(  \gamma_1^g \, \eta(\gamma_1^g) \, . \,  
y(\gamma_1^g, \dots, \gamma_{i+1}^g) \bigr) \
 \\ \nonumber
&& \qquad +  g. \bigl( \sum\limits_{j=1}^i \,  (-1)^j \, \,  y(\gamma_1^g, \dots, \gamma_{j-1}^g,
\gamma_j^g \gamma_{j+1}^g, \gamma_{j+2}^g ,  \dots\gamma_{i+1}^g)  \bigr)
 \\ \nonumber
&& \qquad + g. \bigl(  (-1)^{i+1}\,  y(\gamma_1^g,\dots, \gamma_i^g) \bigr) \\  
\nonumber
&& = \gamma_1 \, \eta(\gamma_1) \, . \,  {^gy}(\gamma_1, \dots, \gamma_{i+1})
\qquad [g \in \Lambda_\eta]
\\ \nonumber
&& \qquad + \sum\limits_{j=1}^i \,  (-1)^j \, \,  {^gy}(\gamma_1, \dots, \gamma_{j-1},
\gamma_j \gamma_{j+1}, \gamma_{j+2} ,  \dots\gamma_{i+1} )
\\ \nonumber
&& \qquad +  (-1)^{i+1}\,  {^gy}(\gamma_1,\dots, \gamma_i) \\ \nonumber
&&= \bigl(\partial_i(^gy)\bigr) ( \gamma_1, \dots, \gamma_{i+1}).
\end{eqnarray*}

This shows that the action of $\Lambda_\eta$ on 
${\mathcal C}^{i}(\Gamma, {_\eta A})$ commutes with the boundary maps as desired.

\smallskip
   
\noindent (5) We are given 
 an exact sequence
of linear algebraic groups  $1 \to \bM_1 \to \bM_2 \to \bM_3 \to 1$ equipped with an action of $\bH$.
We twist it by $\eta$ to obtain $1 \to {_\eta\bM_1} \to {_\eta\bM_2} \to {_\eta\bM_3}
 \to 1,$ and look at the characteristic map
$$
\psi: {_\eta\bM_3}(R_n) \to H^1(R_n, {_\eta\bM_1}).
$$
Let $x_3 \in  {_\eta\bM_3}(R_n) \subset   {\bM_3}(\ol R_{n, \infty})$.
Lift $x_3$ to an element   $x_2 \in  {\bM_2}(\ol R_{n, \infty}).$
Then $\psi(x_3)= [ z_\gamma]$ with 
$z_\gamma = x_2^{-1} \big( \eta(\gamma) \, {^\gamma x_2}\big)$.
Now if $g \in \Lambda_\eta$  the element $^gx_2$ lifts $^gx_3$, hence
$\psi(^gx_3)$ is represented by the $1$--cocycle
$$
 {(^gx_2)}^{-1} \, \big ( \eta(\gamma) \, {^\gamma({^gx_2})}\big)
=  {^gx_2}^{-1}\big( \eta(\gamma) \, {^\gamma {^gx_2}}\big)
=  g. \bigl( {x_2}^{-1} ( \eta(\gamma) \, {^{\gamma^g}{x_2}}) \bigr)
= (^gz)_\gamma
$$
by using again $\eta(\gamma^g)= \eta(\gamma)$ and the fact that $g$ acts trivially 
on $\bH(\ol k)$. This shows that $\psi( {^gx_2}) = {^g \psi}(x_2)$.

Assuming that $\bM_1$ is central and of multiplicative type,
we consider the boundary map $\Delta:  H^1(R_n, {_\eta\bM_3}) \to 
H^2(R_n, {_\eta\bM_1})$.
By isotriviality, the precise nature of this map can be computed at the ``finite level"  by means of Galois cocycles.
Let $(a_\gamma)$ be a cocycle with value   ${_\eta\bM_3}(\ol R_{n, \infty})
= {\bM_3}(\ol R_{n, \infty})$ and choose a  lifting 
$(b_\gamma) $  in ${\bM_2}(\ol R_{n, \infty})$ which is trivial on an open subgroup
of $\pi_1(R_n)$. Recall that
$\Delta( [a_\gamma]) \in H^2\big(\pi_1(R_n), {_\eta\bM_1}(\ol R_{n, \infty})\big) $ is the
 class of the $2$--cocycle \cite[I.5.6]{Se}
$$
c_{\gamma,\tau}=  b_\gamma  \, ( \eta(\gamma) . {^\gamma b}_\tau) \,  
b_{\gamma \tau}^{-1}.
$$ Similarly, 
the  $(^gb_{\gamma^g})$ lift the $({^ga_{\gamma^g}})$, so
$\Delta(g.[a_\gamma])$ is the class of the $2$--cocycle
$$
{^gb}_{\gamma^g}  \, \big( \eta(\gamma) . {^\gamma(^gb_{\tau^g}})\big) \,  
{^gb}_{\gamma^g \tau^g}^{-1}=
g. \Bigl( b_\gamma  \, ( \eta(\gamma) . {^{\gamma^g}b_\tau^g})
\,  {b}_{\gamma^g \tau^g}^{-1}\Bigr)=
g. \Delta([a_\gamma])
$$
as desired.

\smallskip

\noindent 6)  Since  $H^2(R_2,\bmu_d)$ injects in $\Br(R_2)= \Q/\Z$ \cite[2.1]{GP2}, it is
enough to check the formula on $\Br(R_2)$. 
Since $\GL_2(\Z)$ is generated by the matrices
$\left(\begin{matrix} 
 -1& 0  \\ 
 0& -1  
\end{matrix}\right)$, 
$\left(\begin{matrix} 
 1& 0  \\ 
 0& -1  
\end{matrix}\right)$
$\left(\begin{matrix} 
 1& 1  \\ 
 0& 1  
\end{matrix}\right)$ and 
$\left(\begin{matrix} 
 0& 1  \\ 
 1& 0  
\end{matrix}\right)$,
it is enough to show that the desired  compatibility holds when $g$ is one of these four elements.
Consider the cyclic Azumaya $R_2$-algebra  $A= A(1,d)$ 
with presentation $T_1^d= t_1, T_2^d= t_2, T_2T_1 = \zeta_d T_1 T_2$.
Then for $g$ in the above list we have $g. [A] =  [A]$ (resp. $-[A]$, $[A]$, $-[A]$) respectively, so that
$g. [A] =\det(g). [A].$ Since  the classes of these cyclic Azumaya algebras generate $\Br(R_2)$ the result follows. \end{proof}

\begin{remark} \label{rem-cocycles}
(a) In (4), we have $H^*\big(\pi_1(R_n), {_\eta \bM}({\ol R}_{n,\infty})\big)
\simlgr H^*(R_n, {_\eta \bM})$ \cite[prop 3.4]{GP3}, hence we have a natural action 
of $\Lambda_\eta$ on  $H^*(R_n, {_\eta \bM})$.
We have used an explicit description of this action in our proof, but the result can also be established in a more abstract setting.
For $g \in \Lambda_\eta$, we claim  that the map $g_* : A \to A, a \mapsto g.a$
applies $H^0(\Gamma, {_\eta A})$ into itself. Indeed
for  $ a\in  H^0(\Gamma, {_\eta A})$ and
$\gamma \in \Gamma$, we compute the twisted action just as we did  in (2) of the Lemma.

\begin{eqnarray} \nonumber
\gamma \star (g.a) &=& (\eta(\gamma) \, \gamma g) \, . a \\ \nonumber
&=& (\eta(\gamma) \, g \, \gamma^g ) \, . a \qquad \hbox{[definition of $\gamma^g$]}\\
\nonumber
&=& ( g \, \eta(\gamma) \,  \gamma^g ) \, . a \qquad 
\hbox{[$\bGL_n(\Z)$ commutes with  $\bH(\tilde k)]$} \\
\nonumber
&=& ( g \, \eta(\gamma^g)  \gamma^g ) \, . a \quad [g \in \Lambda_\eta]\\
\nonumber
&=& g.a \quad [a \in H^0(\Gamma, {_\eta A})].
\end{eqnarray}

\noindent We get then a  morphism of functors $g_*: F \to F$ which extends uniquely as a morphism
of $\delta$-functors \cite[\S 2.5]{W}. This then yields the desired  natural 
transformations $g_*: H^i(\Gamma, {_\eta A})
\to H^i(\Gamma, _\eta A)$ for each
 $\bGL_n(\Z) \ltimes \bigr(\bH(\tilde k) \rtimes \Gamma \bigl)$-module.

(b) There is an analogous statement to (5)  for
 homogeneous spaces.
 
\end{remark}

\bigskip

For each class $[\bE] \in  H^1\big(R_n,\bOut(\bG)\big)$, we denote by
$H^1\big(R_n,\bAut(\bG)\big)_{[\bE]}$ the fiber at $[\bE]$
of the map $H^1\big(R_n,\bAut(\bG)\big) \to H^1\big(R_n,\bOut(\bG)\big)$. We then have
the decomposition
 \begin{equation}\label{partition}
H^1\big(R_n,\bAut(\bG)\big) = \bigsqcup\limits_{[\bE] \in H^1(R_n,\bOut(\bG))} \,  H^1\big(R_n,\bAut(\bG)\big)_{[\bE]}
\end{equation}
The group $\bGL_n(\Z)$ acts on $H^1\big(R_n,\bOut(\bG)\big)$ and on  $H^1\big(R_n,\bAut(\bG)\big)$ by base change (see Lemma \ref{cocycles-action}).
It follows that $\bGL_n(\Z)$
permutes the subsets of the partition (\ref{partition}), and that 
for each class 
$[\bE] \in H^1\big(R_n,\bOut(\bG)\big)$, its stabilizer under the action of $\bGL_n(\Z)$
preserves $H^1\big(R_n,\bAut(\bG)\big)_{[\bE]}$.

Let $\Out(\bG) = \bOut(\bG)(k).$ The (abstract) group $\Out(\bG)$ acts naturally 
on the right 
on the set of (continuous) homomorphisms ${\Hom}\big(\pi_1(R_n), \bOut(\bG)\big).$ This action, 
which we denote  by ${\rm int},$ is given by 
${\rm int}(a)(\phi)(\gamma) = \phi^a(\gamma) = a^{-1} \, \phi(\gamma) \, a .$ 
We have $H^1\big(R_n,\bOut(\bG)\big) = \Hom \big(\pi_1(R_n), \Out(\bG) \big)/ {\rm int}\big(\Out(\bG)\big)$

We consider a system of representatives $([\phi_j])_{j \in J}$ of 
the set of double cosets $\bGL_n(\Z) \, \backslash \, \Hom\big(\pi_1(R_n), \Out(\bG)\big )/ {\rm int}\big(\Out(\bG)\big)$.
Consider a fixed element $j \in J$.
Denote by $\Lambda_j \subset \bGL_n(\Z)$ the stabilizer 
of $[\phi_j] \in  H^1\big(R_n,\bOut(\bG)\big)$ 
for the base change  action of $\bGL_n(\Z)$ on $\Spec(R_n)$.
An element $g\in \bGL_n(\Z)$ belongs to $\Lambda_j$ if and only  if there exists
$a_g \in \Out(\bG)$ such that the following diagram commutes
$$
\begin{CD}
 \phi_j : \pi_1(R_n) @>>> \bOut(\bG) \\
@A{g^*}A{\wr}A @A{{\rm int}(a_g)}A{\wr}A \\
 \phi_j : \pi_1(R_n) @>>> \bOut(\bG) \\
\end{CD}
$$
Note that $\Lambda_{\phi_j} \subset \Gamma_j$. We have

\begin{eqnarray}\label{action35}
\bGL_n(\Z) \backslash  H^1\big(R_n,\bAut(\bG)\big) = & \bigsqcup\limits_{j \in J} \,  \Lambda_j \backslash H^1\big(R_n,\bAut(\bG)\big)_{[\phi_j]}.
\end{eqnarray}

Recall that our section $s : \bOut(\bG) \to \bAut(\bG)$ is determined by our choice of pinning of $(\bG, \bB, \bT).$  This allows us to trace the action of $\Lambda_j$. Indeed 
 $[s_*(\phi_j)] \in H^1\big(R_n,\bAut(\bG)\big)_{[\phi_j]}$, so  that the classical twisting argument (see \cite[lemme 1.2]{Gi4})
shows that the map
$$
H^1(R_n, {_{s_*(\phi_j)}\bG}) \to H^1\big(R_n, {_{s_*(\phi_j)}\bAut(\bG)}\big) \buildrel
  \tau_{s_*(\phi_j)} \over \to H^1\big(R_n, \bAut(\bG)\big)
$$
 induces a bijection
\begin{eqnarray}\label{action4} 
  H^1(R_n, {_{s_*(\phi_j)}\bG}) / H^0\big(R_n, {_{\phi_j}\bOut(\bG)} \big ) &\simlgr & H^1\big(R_n, \bAut(\bG)\big)_{[\phi_j]}.
\end{eqnarray} 

Note that the action of an element $a \in H^0\big(R_n, {_{\phi_j}\bOut(\bG)} \big )$ on 
$ H^1\big(R_n, {_{s_*(\phi_j)}\bG}\big)$ is given
by $$
 H^1(R_n, {_{s_*(\phi)}\bG})  \buildrel  {(_{\phi_j}s_*)}(a)   \over \simlgr  H^1(R_n, {_{s_*(\phi_j)}\bG}).
$$ 
where ${(_{\phi_j}s_*)}$ is the twist of $s_*$ by the cocycle $\phi_j.$ Furthermore the map (\ref{action4})
 preserves toral or, what is equivalent,  loop classes. Feeding this information into the decomposition
(\ref{action35}), we get

\begin{eqnarray}\label{action3} 
\\ \nonumber
\bGL_n(\Z) \backslash  H^1\big(R_n,\bAut(\bG)\big)  \simlgr &
&\bigsqcup_{j \in J} \, 
\Lambda_j \backslash 
\Bigl( H^1(R_n, {_{s_*(\phi_j)}\bG}) / H^0\big(R_n, {_\phi{_j}}\bOut(\bG)  \big) \Bigr) .
\end{eqnarray}

At least in certain cases, the action of 
$\Lambda_j$ on $H^1(R_n, {_{s_*(\phi_j)}\bG}) / H^0\big(R_n, {_{\phi_j}\bAut(\bG)}\big)$
can be understood quite nicely (see Remark \ref{nicecase} below).

\begin{lemma}  

\smallskip

(1)  For each $g  \in \Lambda_{\phi_j}$,  the following diagrams 
$$ 
\begin{CD}
H^1(R_n, {_{s_*(\phi_j)}\bG}) @>>> H^1\big(R_n, {_{s_*(\phi_j)}\bAut(\bG)}\big) @>{\tau_{s_*(\phi_j)}}>> H^1\big(R_n, {\bAut(\bG)}\big)_{[\phi_j]} \\ 
 @V{g_*}VV @V{g_*}VV  @V{g_*}VV \\
H^1({R_n}, {{_{s_*(\phi_j)}\bG}}) @>>> H^1\big({R_n}, {{_{s_*(\phi_j)}\bAut(\bG)}}\big) @>{\tau_{{s}_*({\phi_j})}}>> H^1\big({R_n}, {{\bAut(\bG)}}\big)_{[\phi_j]} ,\\ 
\end{CD}
$$

$$
\begin{CD}
H^1(R_n, {_{s_*(\phi_j)}\bG})  
&\, \times \, & H^0\big(R_n, {_{\phi_j}\bOut(\bG)}\big) @>>>  H^1(R_n, {_{s_*(\phi_j)}\bG}) \\ 
@V{{g_*}}VV @V{id}VV  @V{{g_*}}VV \\
H^1(R_n, {_{{s_*}(\phi_j)}\bG})  &\, \times \, & H^0\big(R_n, {{_{\phi_j}\bOut(\bG)}}\big) @>>> 
H^1(R_n,  {{_{s_*(\phi_j)}\bG}}) \\ 
\end{CD}
$$
commute where the maps $g_*$ are the base change maps defined in Lemma \ref{cocycles}.

\smallskip

(2) The map  (\ref{action4})
$$
H^1(R_n, {_{s_*(\phi_j)}\bG})  \to  
H^1\big(R_n, \bAut(\bG)\big)_{[\phi_j]}
$$
is $\Lambda_{\phi_j} \times H^0(R_n, {_{\phi_j}\bOut(\bG))}^{\rm op}$--equivariant and 
$$
\Lambda_{\phi_j} \times H^0\big(R_n, {_{\phi_j}\bOut(\bG)\big)}^{\rm op} \, \,  \backslash \, \,   H^1\big(R_n, {_{s_*(\phi_j)}\bG}\big) \,  \simlgr  \, 
\Lambda_{\phi_j} \, \, \backslash \, \,  H^1\big(R_n, \bAut(\bG)\big)_{[\phi_j]}.
$$
\end{lemma}

\begin{remark} \label{nicecase} Of course (2) is useful provided that
$\Lambda_{\phi_j}=\Gamma_j$. This is the case for simple  groups which are not of type $D_4$
since $\bOut(\bG)=1$ or $\Z/2\Z$.
\end{remark}

\begin{proof} (1) We are given $g \in \Lambda_{\phi_j}$. 
The left square of the first diagram
commutes by the functoriality of the base change map $g_*$.
The commutativity of the right square  follows from Lemma \ref{cocycles}.(3) applied to the $k$--group $\bAut(\bG)$ and the cocycle $s_*(\phi_j)$.
The commutativity of the second diagram follows from the action on cocycles given in Lemma \ref{cocycles-action}.

\smallskip

\noindent (2) By (1), the left action of $\Lambda_{\phi_j}$ and the right action
of $H^0\big(R_n, {_{\phi_j}\bOut(\bG)}\big)$ on $H^1(R_n, {_{s_*(\phi_j)}\bG})$ commute.
Hence 
$$
\Lambda_{\phi_j} \backslash \Bigl( H^1(R_n, {_{s_*(\phi_j)}\bG}) / H^0\big(R_n, {_{s_*(\phi_j)}\bAut(\bG)} \big) \Bigr)  \simlgr 
$$

$$
\qquad \qquad 
 \Lambda_{\phi_j} \times H^0(R_n, {_{s_*(\phi_j)}\bOut(\bG)})^{\rm op} \, \,  \backslash \, \,   H^1(R_n, {_{s_*(\phi_j)}\bG}) 
$$
and this set maps bijectively onto 
$\Lambda_{\phi_j} \backslash H^1(R_n, \bAut(\bG))_{[\phi_j]}$.
\end{proof}

\section{Small dimensions}

\subsection{The one-dimensional case}

By combining Theorem \ref{acyclic}, Corollary \ref{corpunctural} and Lemma \ref{sansuc}
we get the following generalization (in characteristic $0$) of
theorem 2.4 of \cite{CGP}.

\begin{corollary}\label{dim-one}
Let $\bG$ be a linear algebraic $k$--group 
Then we have bijections
$$
H^1_{toral}(k[t^{\pm 1}], \bG) \simlgr H^1_{loop}(k[t^{\pm 1}], \bG) \simlgr  H^1(k[t^{\pm 1}], \bG) \simlgr H^1(k((t)), \bG).
$$
\end{corollary}

In the case when $k$ is algebraically closed, we also recover the original results of
\cite{P1} and {\cite{P2} that began the ``cohomological approach" to classification problems 
in infinite-dimensional Lie theory.

\subsection{The two-dimensional case }

Throughout this section we assume that $k$ is algebraically closed of characteristic $0$ and $\bG$  a semisimple Chevalley $k$--group of 
adjoint type. We let $\bG^{sc} \to \bG$ be its
simply connected covering and  denote by $\bmu$ its kernel.

\subsubsection{Classification of  semisimple loop $R_2$-groups.}

 Serre's conjecture II holds for the field $F_2$ by Bruhat--Tits theory \cite[cor. 3.15]{BT3}, i.e.
$H^1(F_2, \bH)=1$ for every semisimple simply connected group $\bH$ over $F_2$.
Furthermore,  we know explicitly how to compute the  Galois cohomology of an arbitrary semisimple $F_2$ group \cite[th. 2.1]{CTGP} and \cite[th. 2.5]{GP2}.  We thus have.

\begin{corollary} \label{class-loop}
We have a decomposition 
$$
 H^1_{loop}\big(R_2, \bAut(\bG)\big)   \enskip \simlgr \enskip   
\bigsqcup\limits_{ [\bE ] \in H^1(R_2, \bOut(\bG))}   {_\bE\bOut(\bG)}(R_2) \, \backslash H^2( R_2, {_\bE\bmu}) 
$$
and the inner $R_2$--forms of $\bG$ are classified by the coset ${\bOut(\bG)}(R_2)\, \backslash H^2( R_2, {\bmu})$.
\end{corollary}

Note that the case when $\bOut(\bG)$ is trivial recovers theorem 3.17 of \cite{GP2}. We can thus view the last Corollary as an extension of this theorem to the case when the automorphism group of $\bG$ is not connected.

\begin{proof} Our choice of splitting $s : \bOut(\bG) \to \bAut(\bG)$ of the exact sequence  
$$
1 \to \bG \to \bAut(\bG) \to \bOut(\bG) \to 1
$$ 
easily leads to the decomposition (see \cite[lemme 1.2]{Gi4}) 
$$
H^1\big(F_2, \bAut(\bG)\big)  \simlgr
\bigsqcup\limits_{ [\bE ] \in H^1(F_2, \bOut(\bG))}   H^1(F_2, {_\bE\bG}) /  {_\bE\bOut(\bG)}(F_2)
$$
with respect to the Dynkin-Tits invariant. On the other hand, the boundary map 
$H^1(F_2, {_\bE\bG}) \to  H^2(F_2, {_\bE\bmu})$ is bijective by \cite[th. 2.1]{CTGP} and \cite[th. 2.5]{GP2}. 
The right action of ${_\bE\bOut(\bG)}(F_2)$ can then be transferred to 
$H^2(F_2, _\bE\bmu),$ and is the opposite  of the natural left action
of ${_\bE\bOut(\bG)}(F_2)$ on $H^2(F_2, _\bE\bmu)$.
Hence  
$$
H^1(F_2, \bAut(\bG))  \simlgr
\bigsqcup\limits_{ [\bE ] \in H^1(F_2, \bOut(\bG))}     {_\bE\bOut(\bG)}(F_2) \, \backslash
 H^2(F_2, {_\bE\bmu}).
$$
But   ${_\bE\bOut(\bG)}$ is finite \'etale over $R_n$ hence
${_\bE\bOut(\bG)}(R_2) = {_\bE\bOut(\bG)}(F_2)$ by 
Remark \ref{series}.(d).
On the other hand, we have 
 $H^2(R_2, {_\bE\bmu}) \simlgr   H^2(F_2, {_\bE\bmu})$ since $ {_\bE\bmu}$ is 
an $R_2$--group of multiplicative type \cite[prop. 3.4]{GP3}. 
Taking into account the acyclicity theorem for $\bAut(\bG)$ and $\bOut(\bG)$, we get the square of bijections
$$
\begin{CD}
H^1(F_2, \bAut(\bG))  @>>{^\sim}> 
\bigsqcup\limits_{ [\bE ] \in H^1(F_2, \bOut(\bG))}     {_\bE\bOut(\bG)}(F_2) \, \backslash
 H^2(F_2, {_\bE\bmu}). \\
@AA{\wr}A @AA{\wr}A \\
H^1_{loop}(R_2, \bAut(\bG))  @>>{^\sim}> 
\bigsqcup\limits_{ [\bE ] \in H^1(R_2, \bOut(\bG))}     {_\bE\bOut(\bG)}(R_2) \, \backslash
 H^2(R_2, {_\bE\bmu}),  \\
\end{CD}
$$ 
and this establishes the Corollary.
\end{proof}

Next we give a complete list of the isomorphism classes of loop $R_2$--forms of $\bG$ in the case 
when $\bG$ is simple of adjoint type. We have $\bOut(\bG)=1$ in type $A_1$ $B$, $C$, $E_7$, $E_8$, $F_4$ and $G_2$,
$\bOut(\bG)=\Z/2 \Z$ in type $A_n$ ($n \geq 2$), $D_n$ ($n \geq 5$) and $E_6,$ and $\bOut(\bG)=S_3$ in type $D_4$. \footnote{Of course here $1, \Z/2\Z$ and $S_3$ are here viewed as constant $R_2$ groups or finite (abstract) groups as the situation requires.}

\medskip

In the case $\bOut(\bG)=1$, then by theorem 3.17 of \cite{GP2} we have
\newline
 $ H^1_{loop}(R_2, \bAut(\bG))   \enskip \simlgr \enskip   H^2( R_2, {\bmu})$. 
But ${\bmu}= {\bmu_n}$ for  $n =1$ or $2$.
We have
$H^2( R_2, {\bmu}_2) \cong \Z/2 \Z$ \cite[\S 2.1]{GP2}.Thus

\begin{corollary}\label{inner}

(1)  If $\bG$ has type $A_1$, then $\ H^1_{loop}\big(R_2, \bAut(\bG)\big) \simeq  \Z/2 \Z$.

\smallskip

 (2) If $\bG$ has type  $B$, $C$ or $E_7$, then $ H^1_{loop}\big(R_2, \bAut(\bG)\big) \simeq  \Z/2 \Z$.

\smallskip

(3) If $\bG$ has type $E_8$, $F_4$ or $G_2$, then $H^1_{loop}\big(R_2, \bAut(\bG)\big) =1$.  \qed

\smallskip

\end{corollary}

\begin{remark}\label{trivialauto} In Case (1) and Case (2) the non-trivial twisted groups are not quasisplit (because their ``Brauer invariant" in $H^2(R_2, \bmu)$ is not trivial.) In Case (1) the non-trivial twisted group is in fact anisotropic (see \cite{GP2} for details).
\end{remark}

In the case $\bOut(\bG)=\Z/2 \Z$, we have $H^1(R_2,\Z/2 \Z) \cong \Z/2 \Z \oplus \Z/2 \Z.$ The $\Z/2\Z$-Galois extensions 
of $R_2$ under consideration are $R_2 \times R_2$,  $R_2[\sqrt{t_1}],$  $R_2[\sqrt{t_2}]$  and  $R_2[\sqrt{t_1 t_2}]$ which 
correspond to the elements $(0,0), (1,0), (0,1)$ and $(1,1)$ respectively. These can also be thought 
as $\Z/2\Z$-torsors over $R_2$ that we will denote by $\bE_{0,0}, \bE_{1,0}, \bE_{0,1}$ and $\bE_{1,1}$ respectively. In the first 
case the generator of the Galois group acts by permuting the two factors, while in the other three is 
of the form $\sqrt{x} \mapsto -\sqrt{x}.$

 Since ${_{\bE}\bOut(\bG)} \cong {\bOut(\bG)}= \Z/2 \Z$, for any of our four torsors we have
 $$
 H^1_{loop}\big(R_2, \bAut(\bG)\big)   \enskip \simlgr \enskip  
 \Z/2 \Z \, \backslash H^2( R_2, {\bmu})  \enskip \bigsqcup_{\bE = \bE_{1,0},  \bE_{0,1},  \bE_{1,1} } \enskip  
\Z/2 \Z \, \backslash H^2( R_2, {_{\bE}{\bmu}})      .
$$ 

This leads to  a case by case discussion.

\begin{corollary}\label{quadratic}

(1) For $\bG$ of type $A_{2n}$ ($n \geq 1)$
$$
H^1_{loop}\big(R_2, \bAut(\bG)\big) \simeq  \{\pm 1 \}\,  \backslash   \Bigl( \Z  / (2n+1) \Z \Bigr)  
 \enskip \bigsqcup_{\bE = \bE_{1,0},  \bE_{0,1},  \bE_{1,1} } \enskip   \{ _{\bE}\bG  \}.
$$ 
There are $n+1$ inner  and three outer loop $R_2$--forms of $\bG.$ 
 All outer forms are quasisplit.
 
(2) For $\bG$ of type $A_{2n-1}$ ($n \geq 2$)
$$
H^1_{loop}\big(R_2, \bAut(\bG)\big) \simeq 
  \{\pm 1 \} \, \backslash  \Bigl( \Z/ 2n \Z \Bigr) 
\enskip \bigsqcup_{\bE = \bE_{1,0},  \bE_{0,1},  \bE_{1,1} } \enskip   \{ _{\bE}\bG^{\pm}  \}.
$$
 There are $n+1$ inner and six outer  loop $R_2$--forms of $\bG.$ 
 The outer forms come in three pairs. Each pair has one form which is quasisplit and one which is not.
 
(3) For $\bG$ of type $D_{2n-1}$ ($n \geq 3$)
$$
H^1_{loop}\big(R_2, \bAut(\bG)\big) \simeq \{\pm 1 \} \,   \backslash \Bigl( \Z/ 4 \Z  \Bigr) 
 \enskip \bigsqcup_{\bE = \bE_{1,0},  \bE_{0,1},  \bE_{1,1} } \enskip   \{ _{\bE}\bG^{\pm}  \}.
$$ 
There are three  inner and six outer  loop $R_2$--forms of $\bG.$ 
The outer forms come in three pairs. Each pair has one form which is quasisplit and one which is not.
 
(4) For $\bG$ of type $D_{2n}$ ($n \geq 3$) 
$$
H^1_{loop}\big(R_2, \bAut(\bG)\big) 
\simeq \hbox{\rm switch} \,  \backslash \Bigl( \Z/2 \Z \oplus  \Z/ 2\Z \Bigr)    
\enskip \bigsqcup_{\bE = \bE_{1,0},  \bE_{0,1},  \bE_{1,1} } \enskip   \{ _{\bE}\bG^{\pm}  \}. 
$$ 
There are three  inner and six outer  loop $R_2$--forms of $\bG.$ 
The outer forms come in three pairs. Each pair has one form which is quasisplit and one which is not.

(5) For $\bG$ of type $E_6$
$$
H^1_{loop}(R_2, \bAut(\bG)) \simeq    \{\pm 1 \} \,  \backslash \Bigl( \Z/ 3 \Z \Bigr) 
 \enskip \bigsqcup_{\bE = \bE_{1,0},  \bE_{0,1},  \bE_{1,1} } \enskip   \{ _{\bE}\bG  \}.
$$ 
There are two  inner and three outer  loop $R_2$--forms of $\bG.$ 
  All outer forms are quasisplit. \end{corollary}

\begin{proof} 
(1) We have $\bmu= \bmu_{2n+1}= \ker\bigl( \bmu_{2n+1}^2 \buildrel \prod \over \to \bmu_{2n+1} \bigr)$ and the
action of $\Z/2\Z$ switches the two factors.
We have $H^2(R_2, \bmu) \simeq \Z/(2n+1) \Z$ 
and  the outer action of $\Z/ 2 \Z$
is by signs.\footnote{Strictly speaking we are looking, here and in what follows, at the 
action of $\Out(\bG)$  on $R_2$--groups or cohomology of $R_2$--groups which are of multiplicative type. 
Since we have an equivalence of categories between $R_2$ and $F_2$ groups of multiplicative type 
\cite{GP3}. By Remark \ref{series}(d)) we can carry all relevant calculations at the level of fields, in which case the situation is well understood.
See for example the table in page 332 of \cite{PR}.}

Let $\bE = \bE_{(1,0)}.$ It follows that $_{\bE}\bmu=  \ker\bigl( \prod\limits_{R_2[\sqrt{t_1}]/R_2} \bmu_{2n+1} 
\buildrel norm \over \to \bmu_{2n+1} \bigr)$.
  Since $2n+1$ is odd, the norm is split and Shapiro lemma yields
$$
H^2(R_2, \, _{\bE}\bmu)= \ker\bigl( H^2(R_2[\sqrt{t_1}],  \bmu_{2n+1}) \buildrel {\rm Cores} \over
\to   H^2(R_2,  \bmu_{2n+1}) \bigr).
$$ This reads
$ \ker\bigl( \Z/(2n+1) \Z \buildrel id  \over \to \Z/(2n+1) \Z\big) =0$ 
by taking into account proposition 2.1 of \cite{GP3}. The same calculation holds for $E_{(0,1)}$ and $E_{(1,1)}$ and we obtain 
the desired decomposition. In particular there are $n + 1$ inner forms 
and three outer forms. The outer forms are all quasiplit.

\smallskip

\noindent (2) We have $\bmu= \bmu_{2n}= \ker\bigl( \bmu_{2n}^2 \buildrel \prod \over \to \bmu_{2n} \bigr)$
 and the action of $\Z/2\Z$ switches the two factors.
The coset $\Z/2\Z  \, \backslash H^2(R, \bmu_{2n})$ 
is as before   $\{ \pm 1\}  \, \backslash  ( \Z/ 2n \Z)$.
However, the  computation of $H^2(R_2, \, {_{\bE}\bmu})$ is different.
The exact sequence $$
1 \to {_{\bE}\bmu_{2n}} \to 
\prod\limits_{R_2[\sqrt{ t_1}]/R_2}\bmu_{2n}  \,  \buildrel norm  \over \to \bmu_{2n} \to 1  $$
gives rise to the long exact
sequence of \'etale cohomology 
$$
\dots \to H^1(R_2[\sqrt{t_1}], \bmu_{2n}) \buildrel norm \over \to
H^1(R_2, \bmu_{2n}) \buildrel \delta \over \to   H^2(R_2, {_{\bE}\bmu}) \to H^1(R_2[\sqrt{t_1}], \bmu_{2n}) \buildrel norm \over \to 
H^2(R_2, \bmu_{2n}). 
$$ 
The norm map appearing on the righthand side is the identity map $id : \Z/2n\Z \to \Z/2n\Z$, so $\delta$ is onto.
 By the choices
of coordinates $\sqrt{t_1}$ and $t_2$ on 
$R_2[\sqrt{t_1}]$ and $t_1,t_2$ on $R_2$, the beginning of the exact sequence decomposes as
$$
\Z/2n\Z \oplus \Z/2n\Z  \buildrel  (id, \times 2 ) \over \longrightarrow \Z/2n\Z \oplus \Z/2n\Z.   
$$ 
So $H^2(R_2, {_{\bE}\bmu}) \simeq \Z/2\Z$ and the action of 
$\Z/2\Z$ on $H^2(R_2, {_{\bE}\bmu})$ is therefore necessarily trivial. Thus $\bE$ leads to two distinct twisted 
forms ${_\bE}\bG^{\pm}.$ More precisely ${_\bE}\bG^{+} = {_\bE}\bG$ (which is quasiplit), while ${_\bE}\bG^{-}$ is not quasiplit (since its 
``Brauer invariant'' in $H^2(R_2, {_{\bE}\bmu})$ is not trivial). Similarly for 
$\bE_{(0,1)}$ and $\bE_{(1,1)}.$ This gives the desired decomposition. There are $n +1$ inner forms and six outer forms (three of which are quasisplit).  
  
\smallskip

\noindent (3) In this case $\bmu= \bmu_4$. The computation of the $H^2$ are exactly as
in case (2) for $n=2$. There are three inner forms and six outer forms (three of which are quasisplit).

\smallskip

\noindent (4) This case is rather different since 
$\bmu= \bmu_2 \times \bmu_2$ where $\Z/2\Z$ switches the two summands. 
We have $H^2(R_2, \bmu) \simeq \Z/2\Z \oplus \Z/2\Z$ where again $\Z/2\Z$  acts by switching the two summands
 
 Given that ${_{\bE}\bmu} =   
\prod\limits_{R_2[\sqrt{ t_1}]/R_2}\bmu_{2}$, we have 
$H^2(R_2, {_{\bE}\bmu}) \simlgr H^2(R_2[\sqrt{ t_1}], \bmu_2) = \Z/2\Z.$ Similarly  for $E_{(0,1)}$ and $E_{(1,1)},$ whence 
our decomposition. Again we have  three inner forms and six outer forms (three of which are quasisplit).

\smallskip

\noindent (5) This is exactly as in case (1) for $n=1$. There are two inner forms and three outer forms (all three of them quasisplit).
\end{proof}

It remains to look at the case when $\bG$ is of type  $D_4$.
The set $H^1(R_2, S_3)$ classifies all degree $3$ \'etale extensions $S$ of $R_2.$ Then $S$ is a 
direct product of connected extensions.
There are tree cases: $S=R_2 \times R_2 \times R_2$ (the split case), $S=S' \times R_2$ with
$S'/R_2$ of degree $2,$ and the connected case.

The case of $S' \times R_2$ is already understood: They correspond to a $1$-cocycle $\phi: \pi_1(R_n) \to
\Z/2\Z \subset S_3$, where
we view  $\Z/2\Z$ as a subgroup of  $S_3$ generated by a permutation. Note that up to conjugation by $S_3,$ there are exactly three such maps $\phi.$ These are three non-isomorphic quadratic extensions which were denoted by 
 $\bE_{(i,j)}$ above for $(i,j) \neq (0,0).$ We shall denote them by $\bE_2^{(i,j)}$ in the present 
situation to avoid confusion.

In the connected case there are four cubic extensions of $R_2.$ They correspond to adjoining to $R_2$ a cubic root  in $R_{2, \infty}$
of $t_1,$ $t_2,$ $t_1t_2$ and $t_1^2t_2$ respectively. We will denote the corresponding four $S_3$--torsors by 
$\bE_3^{(i,j)}$ with the obvious values for $(i,j).$ The cubic case, which a priori appears as the most complicated,  ends up being quite simple due to  cohomological vanishing reasons, as we shall momentarily see.

 According to Corollary \ref{class-loop} we have the decomposition
\begin{eqnarray}\label{looptriality}
H^1_{loop}(R_2, \bAut(\bG))& \simeq & \,\,\,\,
S_3 \backslash H^2(R_2, \bmu)  \\ \nonumber
&&\enskip \bigsqcup_{\bE_2^{(i,j)}}  \enskip (_{\bE_2^{(i,j)}}S_3)(R_2) \, \backslash H^2( R_2,
_{\bE_2^{(i,j)}}\bmu)  \\ \nonumber
&&  \enskip \bigsqcup_{\bE_3^{(i,j)}}  \enskip (_{\bE_3^{(i,j)}}S_3)(R_2) \, \backslash H^2(R_2,
_{\bE_3^{(i,j)}}\bmu).
\end{eqnarray}

The centre  is $\bmu= \bmu_2 \times \bmu_2 = \ker\bigl( \bmu^3_2 \buildrel
\prod \over \to \bmu_2 \bigr)$ and
$S_3$ acts by  permutation on $\bmu^3_2$.
Hence $H^2(R_2, \bmu)= \ker\bigl( H^2(R_2, \bmu_2)^3 \to H^2(R_2,
\bmu_2) \bigr)
\subset  H^2(R_2, \bmu_2)^3 \simeq (\Z/2\Z )^3$. There are two orbits
for the action of $ S_3$ on $H^2(R_2, \bmu)$, namely $(0,0,0)$ and $(1,1,0)$.

For simplicity we will denote $\bE_2^{(1,0)}$ by $\bE_2$ and $\bE_3^{(1,0)}$ by $\bE_3.$ 
Since the group $\bGL_2(\Z)$ acts transitively on the set of quadratic and cubic extensions 
of $R_2$ we may consider only the case 
 of $\bE_2:= \bE_2^{1,0}$ [resp.  $\bE_3:= \bE_3^{1,0}$] for the purpose of determining
the coset $(_{\bE_2^{(i,j)}}S_3)(R_2) \backslash H^2(R_2, {_{\bE_2^{(i,j)}}\bmu})$ [resp. $(_{\bE_3^{(i,j)}}S_3)(R_2) \backslash H^2(R_2, {_{\bE_3^{(i,j)}}\bmu})]$.   that all the twists of $\bmu$ and $S_3$ by quadratic or cubic torsors 
are of the form ${_{\bE_i}\bmu}$ and ${_{\bE_i}S_3}$  for $i = 2$ (resp. $i= 3$) in the quadratic (resp. cubic) case.

We have  ${_{\bE_2}\bmu} = \ker\bigl( \prod\limits_{R_2[\sqrt{
t_1}]/R_2}\bmu_2  \, \times \bmu_2 \buildrel norm \times id \over \to
\bmu_2 \bigr)$, hence
$$
H^2(R_2, {_{\bE_2}\bmu})= \ker\bigl( H^2(R_2[\sqrt{ t_1}],  \bmu_2) \oplus
H^2(R_2,  \bmu_2) \to H^2(R_2,  \bmu_2) \bigr) \cong H^2(R_2[\sqrt{
t_1}],  \bmu_2)
= \Z/2 \Z.  
$$
Since $\Z/2\Z$ has trivial automorphism group, we get three copies of $\Z/2\Z$ in the
second summand of the decomposition
(\ref{looptriality}).

In the cubic case  we have  ${_{\bE_3}\bmu }= \ker\bigl( \prod\limits_{R_2[\root
3 \of t_1]/R_2} \bmu_2 \buildrel norm \over \to \bmu_2 \bigr)$.
Since $2$ is prime  to $3$, the norm is split and
$$
H^2(R_2, {_{\bE_3}\bmu})= \ker\bigl( H^2(R_2[\root 3 \of t_1],  \bmu_2)
\buildrel {\rm Cores}
\over \longrightarrow H^2(R_2,  \bmu_2) \bigr)
= \ker\bigl( \Z/2 \Z \buildrel id  \over \to \Z/2 \Z) =0.  
$$
Finally we observe, with the aid of Remark \ref{series}(d), that $(_{\bE_2}S_3)(R_2) \cong \Z/2\Z$ and
$(_{\bE_3}S_3)(R_2) \cong \Z/3\Z$. 

Looking at (\ref{looptriality}) we obtain.

\begin{corollary}\label{cubic} 
For $\bG$ of type $D_4$ there are twelve  loop $R_2$--forms, two inner and ten outer. Six of the outer forms  are ``quadratic", and come  divided into three pairs, 
where each pair  contains exactly one quasiplit group. The remaining  
four outer forms are ``cubic"  and are all  quasiplit.  \qed
\end{corollary}

\subsubsection{Applications to the classification of EALAs in nullity 2.}

The Extended Affine Lie Algebras (EALAs), as their name suggests, are a class of Lie algebras 
which generalize the affine Kac-Moody Lie algebras. To an EALA $\cal{E}$ one can attached its so called centreless core,
which is usually denoted by $\cal{E}_{{\rm cc}}.$ This is a Lie algebra over $k$ (in general infinite-dimensional) which satisfies
the axioms of a Lie torus.\footnote{This terminology is due to Neher and Yoshii. It may seem strange to call a Lie algebra  a Lie torus (since tori have already a meaning in Lie theory). The terminology was 
motivated by the concept of Jordan tori, which are a class of Jordan algebras.}  Neher has shown that all Lie torus arise as centreless cores of EALAs, and 
conversely. He has also given an explicit procedure that constructs all EALAs having a given Lie torus $\cal{L}$ as their 
centreless cores. To some extent this reduces many central questions about EALAs (such as their classification) to that of Lie tori.

The centroid of a Lie tori $\cal{L}$ is always of the form $R_n.$ 
This gives a natural $R_n$--Lie algebra structure 
to $\cal{L}.$ If  $\cal{L}$ as an $R_n$--module is of finite type, 
then $\cal{L}$ is necessarily a multiloop algebra
 $L(\gg, {\bs })$ as explained  in the Introduction. 
Let $\bG$ be a Chevalley $k$--group of adjoint type with Lie 
algebra $\gg.$
Since $\bAut(\gg) \simeq \bAut(\bG)$ the $n$--loop algebras based on $\gg$ 
(as $R_n$--Lie algebras) are in bijective correspondence  with the loop $R_n$--forms 
of $\bG.$ Indeed,  they are precisely the Lie
 algebras of the loop $R_n$--groups. The subtlety comes from the fact 
that in infinite-dimensional Lie theory one is interested in these Lie algebras 
as Lie algebras over $k,$
 {\it and not} $R_n.$ In the present context the ``centroid trick'' (see \cite[\S 4.1]{GP2}) translates into the $\bGL_n(\Z)$ action  on $H^1\big(R_n, \bAut(\gg)\big)$ we have defined. This allows us to describe, in terms of orbits, all  the isomorphism classes of $R_n$--multiloop algebras $L(\gg, {\bs})$ that become isomorphic when viewed as Lie algebras over $k.$

 \medskip
 
In what follows ``loop algebras based on $\gg$" will be though as Lie algebras  {\it over} $k.$
\smallskip

In the case $\bOut(\bG)=1$ we have seen that
$H^1_{loop}\big(R_2, \bAut(\bG)\big) \simeq H^2( R_2, {\bmu}),$ and this latter $H^2$ is either trivial or $\Z/2 \Z .$  In both cases the 
action 
of $\bGL_2(\Z)$ on $H^1_{loop}\big(R_2, \bAut(\bG)\big)$ is necessarily trivial. In particular.

\medskip

\begin{corollary}\label{inner}

(1)  If $\gg$ has type $A_1$ $B$, $C$ or $E_7$, there exists two isomorphism 
classes of $2$--loop algebras based on $\gg$ denoted by $\gg_0$ (the split case) and $\gg_1.$\footnote{As pointed out 
in \cite{GP2}, the case of $E_7$ has an amusing story behind it. 
The existence of a $k$--Lie algebra 
$L(\gg, \sigma_1, \sigma_2) $ which is not isomorphic to $\gg\otimes_k R_2$
 was first established by van de Leur with the aid of 
a computer. In nullity $1$ inner automorphisms always lead to trivial 
loop algebras. van de Leur's example shows that this fails already 
in nullity two.}

\smallskip

(2) All 2--loop algebras based on $\gg$ of type $E_8$, $F_4$ or $G_2$ are trivial, i.e isomorphic 
as $k$--Lie algebras to $\gg_0 = \gg \otimes_k R_2.$  

\smallskip

\end{corollary}

In the case $\bOut(\bG)=\Z/2 \Z$, we have $H^1(R_2,\Z/2 \Z) \cong \Z/2 \Z \oplus \Z/2 \Z$ and
the action of $\bGL_2(\Z)$ on  $H^1(R_2,\Z/2 \Z)$
is given by the linear action  mod $2$. Since $\SL_2(\Z/2\Z)= \GL_2(\Z/2\Z)$ and 
$\SL_2(\Z/2\Z)$ is generated by elementary matrices, the reduction map 
$\bGL_2(\Z) \to \GL_2(\Z/2 \Z)$ is onto.
Hence there are two orbits for the action of $\bGL_2(\Z)$ on 
$H^1(R_2,\Z/2 \Z)$, namely the trivial one   and $H^1(R_2,\Z/2 \Z) \setminus \{0\}$. The last one
is represented by the quadratic Galois  extension $R_2[\sqrt{t_1}]/R_2,$ denoted by  
 $\bE_{(1,0)}$ in the previous section, which we will again denote simply by $\bE$ in what follows. 
The action of $\bGL_2(\Z)$ we have just 
described shows that in all cases the outer forms, which came in three families (each with one or two classes)
in the case of loop $R_2$--groups, collapse into 
a single family.  This single family consists of either a single class, namely 
the quasi-split algebra
${_{\bE}}\gg = {_{\bE}}\gg^+$, or two classes  ${_{\bE}}\gg^{+}$ and ${_{\bE}}\gg^{-}$. The algebra ${_{\bE}}\gg^{-}$ is not quasisplit. 

\begin{corollary}\label{quadratic} If  $\bOut(\bG)=\Z/2 \Z$ the classification of isomorphism classes of $2$--loop algebras based on $\gg$ is as follows:

(1) In  type $A_{2n}$ ($n \geq 1)$  
$$
\bGL_2(\Z) \, \backslash H^1_{loop}(R_2, \bAut(\gg)) \simeq  \{\pm 1 \}\,  \backslash   \Bigl( \Z  / (2n+1) \Z \Bigr)  
 \enskip \bigsqcup  \enskip   \{ _{\bE}\gg  \}.
$$ 
 There are $n + 1$ inner forms, denoted by $\gg_q$ with $0 \leq q \leq n$, and one outer form (which is quasisplit).
 \medskip
 
(2) In type  $A_{2n-1}$ ($n \geq 2$)  
$$
\bGL_2(\Z) \, \backslash H^1_{loop}(R_2, \bAut(\gg)) \simeq 
  \{\pm 1 \} \, \backslash  \Bigl( \Z/ 2n \Z \Bigr)
   \enskip  \bigsqcup  \enskip   \{ _{\bE}\gg^+  \}
\enskip \bigsqcup  \enskip   \{ _{\bE}\gg^-  \}.
$$
There are $n +1$ inner forms,  denoted by $\gg_q$ with $0 \leq q \leq n,$ and two outer forms  (one of them quasiplit, the other one not).
 \medskip
 
(3) In type $D_{2n-1}$ ($n \geq 3$)
$$
\bGL_2(\Z) \, \backslash H^1_{loop}(R_2, \bAut(\gg)) \simeq  \{\pm 1 \} \,   \backslash \Bigl( \Z/ 4 \Z  \Bigr) 
 \enskip  \bigsqcup  \enskip   \{ _{\bE}\gg^+  \}
\enskip \bigsqcup  \enskip   \{ _{\bE}\gg^-  \}.
$$ 
There are $3$ inner forms, denoted by $\gg_{0,1,2}$,  and two outer forms  (one of them quasiplit, the other one not).
 \medskip
 
(4) In type $D_{2n}$ ($n \geq 3)$, we have 
$$
\bGL_2(\Z) \, \backslash H^1_{loop}(R_2, \bAut(\gg)) 
\simeq \hbox{\rm switch} \,  \backslash \Bigl( \Z/2 \Z \oplus  \Z/ 2\Z \Bigr)   
 \enskip  \bigsqcup  \enskip   \{ _{\bE}\gg^+  \}
\enskip \bigsqcup  \enskip   \{ _{\bE}\gg^-  \}.
$$ 
There are $3$ inner forms, denoted by $\gg_{0,1,2}$, and two outer forms  (one of them quasiplit, the other one not).

\medskip

(5) In  type $E_6$ 
$$
\bGL_2(\Z) \,  \backslash H^1_{loop}(R_2, \bAut(\gg)) \simeq  \{\pm 1 \} \, 
 \backslash \Bigl( \Z/ 3 \Z \Bigr) 
 \enskip \bigsqcup  \enskip  \{ _{\bE}\gg  \}.
 $$
 There are $2$ inner forms, $\gg_0$ and $\gg_1$, and one outer form (which is quasisplit).
 
 \end{corollary}

\begin{proof}
The nature of the collapse of outer forms when passing from $R_2$ to $k$ was explained before the statement of the Corollary. It remains to understand the inner cases.
According to Corollary \ref{class-loop} and 
(\ref{action35}), we need to trace the action of $\bGL_2(\Z)$ on $\Z/2\Z \backslash H^2(R_2, \bmu) $
and use the fact that this action lifts to an action of  $\bGL_2(\Z)$ on $H^2(R_2, \bmu)$
which commutes with that of $\Z/2\Z$.

 \smallskip

\noindent (1) We have $\bmu= \bmu_{2n+1}= \ker\bigl( \bmu_{2n+1}^2 \buildrel \prod \over \to \bmu_{2n+1} \bigr)$ and the
action of $\Z/2\Z$ switches the two factors.
We have $H^2(R_2, \bmu) \cong \Z/(2n+1) \Z$ and
the action of $\bGL_2(\Z)$ on  $H^2(R_2, \bmu)$ is  given by the determinant (Lemma \ref{cocycles}.6),
that  of $\Z/ 2 \Z$ is given by signs. Thus  $\bGL_2(\Z)$ acts trivially on  $\Z/2\Z \backslash H^2(R_2, \bmu) $
and the result follows.

\smallskip

\noindent (2) We have
$\bmu= \bmu_{2n}= \ker\bigl( \bmu_{2n}^2 \buildrel \prod \over \to \bmu_{2n} \bigr)$
 and the action of $\Z/2\Z$ switches the two factors.
The action of $\bGL_2(\Z)$ on $H^2(R, \bmu)$ is given by the determinant,  
hence the set of cosets $(\bGL_2(\Z)\times \Z/2\Z))  \, \backslash H^2(R, \bmu_{2n})$ 
can still be identified with  $\{ \pm 1\}  \, \backslash  ( \Z/ 2n \Z)$.

\smallskip

\noindent (3) In this case $\bmu= \bmu_4$. The computation of  $H^2$ and reasoning are exactly as
in  (2)  above for $n=2$. 

\smallskip

\noindent (4) This case is rather different since 
$\bmu= \bmu_2 \times \bmu_2$ where $\Z/2\Z$ switches the two summands. 
We have $H^2(R_2, \bmu) \cong \Z/2\Z \oplus \Z/2\Z$ with respect to the  switch action.
 Again $\bGL_2(\Z)$ acts by  $g . \alpha= \det(g). \alpha$, hence trivially.

\smallskip

\noindent (5) This is exactly as in case (1) for $n=1$.
\end{proof}

It remains to look at the case when $\bG$ is of type  $D_4$.

\begin{lemma} \label{cubic2} There are three orbits  for the action of  $\bGL_2(\Z)$
on $H^1(R_2, S_3)$ :

- the trivial class;

-  $\Bigl\{ \bE_2^{0,1}, \bE_2^{1,0}, \bE_2^{1,1} \Bigr\}$;

- $\Bigl\{ \bE_3^{1,0}, \bE_3^{0,1}, \bE_3^{1,1}, \bE_3^{2,1} \Bigr\}$

\medskip

\noindent where the notations is as in Corollary \ref{cubic} supra.

\end{lemma}

\begin{proof}
  The three classes above correspond to case of the split \'etale cubic $R_2$-algebra, the 
case $S' \times R_2$ where $S'/R_2$ is quadratic and the cubic case.
Obviously each of the above sets is $\GL_2(\Z)$-stable, so we need to check 
that there is a single orbit. The quadratic case was dealt with in
Corollary \ref{quadratic}. In the cubic case, we have  $\bE_3^{1,0}= R_2[\root 3 \of t_1]$. 
By applying the base change corresponding to 
$\left(\begin{matrix} 
 1& 0  \\ 
 0& 1  
\end{matrix}\right)$,
$\left(\begin{matrix}
 1& 0\\ 
 1& 1  
\end{matrix}\right)$,  
 $\left(\begin{matrix} 
 2&  5\\ 
 1& 3 
 \end{matrix}\right)$   
 we obtain $\bE_3^{0,1}$, $\bE_3^{1,1}$ and $\bE_3^{2,1}$ respectively.
\end{proof}

\begin{corollary}\label{cubic3} 
Up to $k$-isomorphism there are five 2-loop algebras  based on $\gg$ of type $D_4$: two inner forms, denoted by $\gg_0$ and $\gg_1$;  two  ``quadratic"  algebras, $_{\bE_2^{1,0}}\gg^+$ (which is quasisplit)
and $_{\bE_2^{1,0}}\gg^-$ (which is not quasiplit);   and   one  ``cubic"   algebra $_{\bE_3^{1,0}}\gg$  ( which is quasisplit).  
\end{corollary}

\begin{proof} By Lemma \ref{cubic2}, the quadratic (resp. cubic) classes
of Corollary \ref{cubic} are in the $\GL_2( \Z)$--orbit of 
those having Dynkin-Tits invariant ${\bE_2^{1,0}}$ (resp. ${\bE_3^{1,0}}$).
So the cubic case is done. In the quadratic case, 
there are then one or two non-isomorphic ``quadratic'' 2-loop algebras. 
Since one of these $R_2$--algebras is quasisplit and the other one is not, Corollary \ref{WT3equi} shows that they remain non-isomorphic as $k$--algebras.
Finally, there are two orbits
for the action of $ S_3$ on $H^2(R_2, \bmu)$, namely $(0,0,0)$ and $(1,1,0)$, and these correspond to the two inner $R_2$-forms. The action of $\GL_2(\Z)$ is trivial in this set, so the algebras remain non-isomorphic over $k.$\end{proof}

\subsubsection{Rigidity in nullity 2 apart from type $A.$} 
The following theorem extends results of Steinmetz from  classical types \cite[th. 6.4]{SZ} (which involves certain small rank restrictions) to all types. This establishes Conjecture 6.4 of \cite{GP2}.\footnote{An even stronger version of this Conjecture will be established in the next section.}

\begin{theorem}\label{conjecture} Let $\gg$ be a finite dimensional simple Lie algebra  over $k$ which is not of type $A$.
Let $\mathcal L$ and $\mathcal L'$ be two $2$-loop algebras based on $\gg.$
The following are equivalent: 

\smallskip

(1) $\mathcal L$ and $\mathcal L'$ are isomorphic (as Lie algebras over $k$);

\smallskip
 
(2)  $\mathcal L$ and $\mathcal L'$  have the same Witt-Tits index.   

\end{theorem}

\begin{proof}
Of course, we use that  $\mathcal L$  (resp. $\mathcal L'$) arise
as the Lie algebra of $R_2$--loop adjoint groups $\bH$ (resp. $\bH'$)
which are forms of $\bG= \Aut(\gg)^0$.\footnote{Strictly speaking ``...the Lie algebra..." is an $R_2$-Lie algebra, but we view this in a natural away as a $k$-Lie algebra.}
Then condition (1) reads that $[\bH]= [\bH']$ in 
$\bGL_2(\Z) \,  \backslash H^1_{loop}\big(R_2, \bAut(\gg)\big)$
and condition (2) reads that $\bH \times_{R_2}K_2$ and
$\bH' \times_{R_2}K_2$ have the same Witt-Tits index.

\smallskip

\noindent $(1) \Longrightarrow (2):$  This is is the simple part of the
the equivalence (and it is not necessary to exclude type $A$).  Let $\bG$ be the corresponding adjoint group.
If  $[\bH]$ and $[\bH']$  are equal in $\bGL_2(\Z) \,  \backslash H^1_{loop}\big(R_2, \bAut(\bG)\big)$,
it is obvious that their Dynkin-Tits invariant coincide 
in  $\bGL_2(\Z) \backslash H^1_{loop}\big(R_2, \bOut(\bG)\big)$, 
and also  that their Tits index over $K_2$ coincide by Corollary \ref{WT3equi}.

\smallskip

\noindent $(2) \Longrightarrow (1):$  Without loss of generality we can assume that $\bH$
and $\bH'$ have same Dynkin-Tits invariant in $H^1\big(R_2, \bOut(\gg)\big)$.
 The proof is given by a case-by-case discussion. 
The cases of type $E_8$, $F_4$ and  $G_2$  follow directly from Corollary \ref{inner}.2.
Types $B$, $C$ and $E_7$ are also straightforward since (over $R_2$)
there is  only one class of non-split  $2$--loop algebras. For obvious reasons, this non-split Lie algebras necessarily remain non-isomorphic to the split Lie algebra $\gg \otimes_k R_2$ when viewed as Lie algebras over $k.$

\smallskip

\noindent{\it Type $D_{2n}$, $n\geq 3:$}
If the Dynkin-Tits invariant is non-trivial, 
then the summand of  $\bGL_2(\Z) \, \backslash H^1_{loop}(R_2, \bAut(\gg))$ corresponding to  
$[\bE_1]$ has  only one non quasi-split class, so $[\bH]$ and $[\bH']$ are equal
in   $\bGL_2(\Z) \, \backslash H^1_{loop}\big(R_2, \bAut(\gg)\big)$. 
Corollary \ref{quadratic} states that  
the  inner part of $\bGL_2(\Z) \, \backslash H^1_{loop}(R_2, \Aut(\gg))$
is $ \Z/2 \Z \oplus  \Z/ 2\Z$ modulo the switch action, so is represented
by $(0,0)$, $(1,0)$ and $(1,1)$.  
It has then three elements, the split one and two others.
It is then enough to explicitly describe these two other elements and distinguish them by their
Witt-Tits index by means of Tits tables \cite{T1}. The first one is the $R_2$-loop group
 $\PSO(q)$ with $q= \langle 1, t_1, t_2, t_1 t_2 \rangle \perp (2n-1) \langle 1,-1 \rangle$.
Its Witt-Tits $K_2$-index is 

\vskip-8mm

\begin{equation}
\begin{picture}(250,20)
\put(00,00){\line(1,0){80}}
\put(100,00){\dots}
\put(120,00){\line(1,0){40}}
\put(170,00){\dots}
\put(190,00){\line(1,0){20}}
\put(210,0){\line(2,-1){20}}
\put(210,0){\line(2,1){20}}
\put(0,0){\circle*{3}}
\put(20,0){\circle*{3}}
\put(40,0){\circle{10}}
\put(40,0){\circle*{3}}
\put(60,0){\circle{10}}
\put(60,0){\circle*{3}}
\put(160,0){\circle{10}}
\put(160,0){\circle*{3}}
\put(190,0){\circle{10}}
\put(190,0){\circle*{3}}
\put(210,0){\circle{10}}
\put(210,0){\circle*{3}}
\put(230,10){\circle{10}}
\put(230,10){\circle*{3}}
\put(230,-10){\circle{10}}
\put(230,-10){\circle*{3}}
\put(140,0){\circle{10}}
\put(140,0){\circle*{3}}
\put(-5,8){$\alpha_1$}
\put(15,8){$\alpha_2$}
\put(65,8){}
\put(130,8){}
\put(240,10){$\alpha_{2n-1}$}
\put(240,-13){$\alpha_{2n}$}
\end{picture}       
\end{equation}

\vskip2mm 

\noindent The other one is ${\rm PSU}(A, h)$ where $A=A(2,1)$ is the $R_2$-quaternion algebra $T_1^2=t_1$,
$T_2^2=t_2$, $T_1 T_2+ T_2 T_1 =0$ and $h$ is the hyperbolic  hermitian form over $A^{2n}$
with respect to the quaternionic involution $q \mapsto \overline q$. 
 Indeed ${\rm PSU}(A, h)$ is an adjoint  inner loop $R_2$-group of type $D_{2n}$
 and its Witt-Tits $K_2$-index is

\vskip-8mm

\begin{equation}
\begin{picture}(250,20)
\put(00,00){\line(1,0){20}}
\put(20,00){\line(1,0){20}}
\put(40,00){\line(1,0){20}}
\put(60,00){\line(1,0){20}}
\put(120,00){\line(1,0){40}}
\put(170,00){\dots}
\put(190,00){\line(1,0){20}}
\put(210,0){\line(2,-1){20}}
\put(210,0){\line(2,1){20}}
\put(0,0){\circle*{3}}
\put(20,0){\circle{10}}
\put(20,0){\circle*{3}}
\put(40,0){\circle*{3}}
\put(60,0){\circle*{3}}
\put(60,0){\circle{10}}
\put(80,0){\circle*{3}}
\put(190,0){\circle*{3}}
\put(210,0){\circle{10}}
\put(210,0){\circle*{3}}
\put(230,10){\circle*{3}}
\put(230,-10){\circle*{3}}
\put(140,0){\circle{10}}
\put(140,0){\circle*{3}}
\put(-5,8){$\alpha_1$}
\put(15,8){$\alpha_2$}
\put(65,8){}
\put(130,8){}
\put(240,10){$\alpha_{2n-1}$}
\put(240,-13){$\alpha_{2n}$}
\end{picture}       
\end{equation}

\vskip2mm

\noindent So $[\bH]$ and $[\bH']$ are equal in   
 $\bGL_2(\Z) \, \backslash H^1_{loop}(R_2, \Aut(\gg))$ to the split form
or one of these two forms.

\smallskip

\noindent{\it Type $D_{2n-1}$, $n\geq 3:$}
As in the preceding case, we need to discuss only the inner case and it is enough
to provide two non-split $R_2$-loop groups with distinct $K_2$-Witt-Tits index.
The first one is the $R_2$-loop group
 $\PSO(q)$ with $q= \langle 1, t_1, t_2, t_1 t_2 \rangle \perp (2n-2) \langle 1,-1 \rangle$.
Its Witt-Tits $K_2$-index is 

\vskip-8mm

\begin{equation}
\begin{picture}(250,20)
\put(00,00){\line(1,0){80}}
\put(100,00){\dots}
\put(120,00){\line(1,0){40}}
\put(170,00){\dots}
\put(190,00){\line(1,0){20}}
\put(210,0){\line(2,-1){20}}
\put(210,0){\line(2,1){20}}
\put(0,0){\circle*{3}}
\put(20,0){\circle*{3}}
\put(40,0){\circle{10}}
\put(40,0){\circle*{3}}
\put(60,0){\circle{10}}
\put(60,0){\circle*{3}}
\put(160,0){\circle{10}}
\put(160,0){\circle*{3}}
\put(190,0){\circle{10}}
\put(190,0){\circle*{3}}
\put(210,0){\circle{10}}
\put(210,0){\circle*{3}}
\put(230,10){\circle{10}}
\put(230,10){\circle*{3}}
\put(230,-10){\circle{10}}
\put(230,-10){\circle*{3}}
\put(140,0){\circle{10}}
\put(140,0){\circle*{3}}
\put(-5,8){$\alpha_1$}
\put(15,8){$\alpha_2$}
\put(65,8){}
\put(130,8){}
\put(240,10){$\alpha_{2n-2}$}
\put(240,-13){$\alpha_{2n-1}$}
\end{picture}       
\end{equation}

\vskip2mm
 
\noindent The other one is ${\rm PSU}(A, h)$ where $h$ is the hyperbolic  hermitian form over $A^{2n-1}$
which is the orthogonal sum of $\langle 1 \rangle$ and the hyperbolic form over $A^{2n-2}$.  
 Its Witt-Tits $K_2$-index is

\begin{equation}
\begin{picture}(250,20)
\put(00,00){\line(1,0){20}}
\put(20,00){\line(1,0){20}}
\put(40,00){\line(1,0){20}}
\put(60,00){\line(1,0){20}}
\put(120,00){\line(1,0){40}}
\put(170,00){\dots}
\put(190,00){\line(1,0){20}}
\put(210,0){\line(2,-1){20}}
\put(210,0){\line(2,1){20}}
\put(0,0){\circle*{3}}
\put(20,0){\circle{10}}
\put(20,0){\circle*{3}}
\put(40,0){\circle*{3}}
\put(60,0){\circle*{3}}
\put(60,0){\circle{10}}
\put(80,0){\circle*{3}}
\put(190,0){\circle*{3}}
\put(190,0){\circle{10}}
\put(210,0){\circle*{3}}
\put(230,10){\circle*{3}}
\put(230,-10){\circle*{3}}
\put(140,0){\circle{10}}
\put(140,0){\circle*{3}}
\put(-5,8){$\alpha_1$}
\put(15,8){$\alpha_2$}
\put(65,8){}
\put(130,8){}
\put(240,10){$\alpha_{2n-2}$}
\put(240,-13){$\alpha_{2n-1}$}
\end{picture}       
\end{equation}

\smallskip

\noindent{\it Type $D_4$:} Follows from Corollary 
\ref{cubic3}.

\smallskip

\noindent{\it Type $E_6$:}
This case is straightforward because 
there is only  one class of $2$--loop algebras which  is not quasi-split. 
\end{proof}

\begin{remark}\label{redundant} There is some redundancy in the statement of the Theorem. It is well known, by descent considerations, that if $\mathcal L$ and $\mathcal L'$ are isomorphic as Lie algebras over $k,$ then their absolute type coincide, i.e. they are both 2--loop algebras based on the same $\gg$ (see \cite{ABP2.5} for further details). It will thus suffice  to assume in the Theorem that neither  $\mathcal L$ nor $\mathcal L'$ are of absolute type $A.$
\end{remark}

\subsubsection{Tables}

The following table summarizes   the classification on  2-loop algebras. The table includes
the Cartan-Killing (absolute) type $\gg$, its name, the
Witt-Tits index  (with Tits' notations) of an $R_n$--representative of the $k$--Lie algebra in question, and
the type of the relative root system.
 For example, van de Leur's algebra has absolute type $E_7$, 
Tits index $E_{7,4}^9$ and relative type $F_4$. The way in which the Witt-Tits index are determined was illustrated in the previous section. The procedure of how to obtain the relative type from the index is described by Tits.

In all cases the ``trivial" loop algebra $\gg \otimes k[t_1^{\pm},t_2^{\pm}]$ is denoted by $\gg_0.$ When the relative type is $A_0,$ the loop algebra in question is anisotropic. For example, in absolute type $A_1$ the Lie algebra $\gg_0$ is  $\mathfrak{sl}_2([t_1^{\pm},t_2^{\pm}]).$ The Lie algebra $\gg_1$ is the derived algebra of the Lie algebra that corresponds to the quaternion algebra over $k[t_1^{\pm},t_2^{\pm}]$ with relations $T_1T_2 = -T_2T_1$ and  $T_i^2 = t_i.$ This rank 3 free Lie algebra over $k[t_1^{\pm},t_2^{\pm}]$ is anisotropic, and is a twisted form of $\mathfrak{sl}_2  \otimes k[t_1^{\pm},t_2^{\pm}]$ split by the quadratic extension $k[t_1^{\pm},t_2^{\pm}](T_1).$

\bigskip

\begin{tabular}{|c|c|c|c|} \hline
Cartan-Killing type $\gg$& Name & Tits index & Relative root system  \\ \hline
$A_1$   & $\gg_0$ & $^1A^{(1)}_{1,1}$  & $A_1$ \\ \hline
$A_1$   & $\gg_1$ & $^1A^{(2)}_{1,0}$  & $A_0$ \\ \hline
$A_{2n}$ ( $n \geq 1$)  & $\gg_q$ & $^1A^{(\frac{2n +1}{r})}_{2n,{r}-1}\, \,\,\, r = {\text{gcd}(q, 2n +1)}$  & $A_{r-1}$ \\ \hline
$A_{2n}$ ( $n \geq 1$)  & $_{\bE}\gg$ & $^2A^{(1)}_{2n,n}$  & $BC_n$ \\ \hline
$A_{2n -1}$ ( $n \geq 2$)  & $\gg_q$ & $^1A^{(\frac{2n }{r})}_{2n-1,{r}-1}\, \,\,\, r = {\text{gcd}(q, 2n)}$  & $A_{r-1}$ \\ \hline
$A_{2n-1}$ ( $n \geq 2$)  & $_{\bE}\gg^+$ & $^2A^{(1)}_{2n-1,n}$  & $C_n$ \\ \hline
$A_{2n-1}$ ( $n \geq 2$)  & $_{\bE}\gg^-$ & $^2A^{(1)}_{2n-1,n-1}$  & $BC_{n-1}$ \\ \hline

$B_{n}$ ($n\geq 2)$  & $\gg_0$ & $B_{n,n}$  & $B_n$ \\ \hline
$B_n$ ($n\geq 2)$  & $\gg_1$ & $B_{n,n-1}$  & $B_{n-1}$  \\ \hline
$C_n$ ($n\geq 3)$   & $\gg_0$ & $C^{(1)}_{n,n}$  & $C_n$  \\ \hline
$C_{2n+1}$ ($n\geq 1)$   & $\gg_1$ & $C^{(2)}_{2n+1,n}$  &  $BC_n$ \\ \hline
$C_{2n}$ ($n\geq 2)$   & $\gg_1$ & $C^{(2)}_{2n,n}$  & $ C_n$ \\ \hline
$D_4$  & $\gg_0$ & $^1D_{4,4}^{(1)}$  & $D_4$ \\ \hline
$D_{4}$   &  $\gg_1$ & $^1D_{4,2}^{(1)}$ & $B_2$  \\ \hline
$D_{4}$   &  $_{\bE_2}\gg^{+}$ & $^2D_{4, 3}^{(1)}$ & $B_3$ \\ \hline
$D_{4}$   &  $_{\bE_2}\gg^{-}$ & $^2D_{4, 1}^{(2)}$ & $BC_1$ \\ \hline
$D_{4}$   &  $_{\bE_3}\gg$ & $^3D_{4, 2}^{2}$ & $G_2$ \\ \hline
$D_{2n-1}$ ($n \geq 3)$  &  $\gg_0$ & $^1D_{2n-1, 2n-1}^{(1)}$ & $D_{2n-1}$ \\ \hline
$D_{2n-1}$ ($n \geq 3)$  &  $\gg_1$ & $^1D_{2n-1,2n-3}^{(1)}$ & $B_{2n-3}$ \\ \hline
$D_{2n-1}$ ($n \geq 3)$  &  $\gg_2$ & $^1D_{2n-1, n-2}^{(2)}$ & $BC_{n-2}$  \\ \hline
$D_{2n-1}$ ($n \geq 3)$  &  $_\bE\gg^{+}$ & $^2D_{2n-1, 2n-2}^{(1)}$ & $B_{2n-2}$ \\ \hline
$D_{2n-1}$ ($n \geq 3)$  &  $_\bE\gg^{-}$ & $^2D_{2n-1, n-2}^{(2)}$ & $BC_{n-2}$ \\ \hline
$D_{2n}$ ($n \geq 3)$ &  $\gg_0$ & $^1D_{2n, 2n}^{(1)}$ & $D_{2n}$ \\ \hline
$D_{2n}$ ($n \geq 3)$  &  $\gg_1$ & $^1D_{2n,2n-2}^{(1)}$ & $B_{2n-2}$ \\ \hline
$D_{2n}$ ($n \geq 3)$  &  $\gg_2$ & $^1D_{2n,n}^{(2)}$ &  $C_n$\\ \hline
$D_{2n}$ ($n \geq 3)$  &  $_\bE\gg^{+}$ & $^2D_{2n, 2n-1}^{(1)}$ &  $B_{2n-1}$ \\ \hline
$D_{2n}$ ($n \geq 3)$  &  $_\bE\gg^{-}$ & $^2D_{2n, n-1}^{(2)}$ & $BC_{n-1}$\\ \hline
$E_6$ & $\gg_0$ & $^1E_{6,6}^0$  &  $E_6$ \\ \hline
$E_6$ & $\gg_1$ & $^1E_{6,2}^{16}$  &  $G_2$ \\ \hline
$E_6$  & $_\bE\gg$ & $^2E_{6,4}^2$  & $F_4$ \\ \hline
$E_7$ & $\gg_0$ & $E_{7,7}^0$  &  $E_7$ \\ \hline
$E_7$  & $\gg_1$ & $E_{7,4}^9$  &  $F_4$ \\ \hline
$E_8$ & $\gg_0$ & $E_{8,8}^0$  & $E_8$  \\ \hline
$F_4$ & $\gg_0$ & $F_{4,4}^0$  & $F_4$  \\ \hline
$G_2$ & $\gg_0$ & $G_{2,2}^0$  & $G_2$  \\ \hline
\end{tabular} 
\bigskip

By taking Remark \ref{redundant} into consideration, an inspection of the Table shows that a stronger version of Theorem \ref{conjecture} holds.

\begin{theorem}\label{stongconjecture} Let $\mathcal L$ and $\mathcal L'$ be two $2$-loop algebras neither of which is of absolute type $A.$ The following are equivalent: 

\smallskip

(1) $\mathcal L$ and $\mathcal L'$ are isomorphic (as Lie algebras over $k$);

\smallskip
 
(2)  $\mathcal L$ and $\mathcal L'$  have the same absolute and relative type.   

\end{theorem}

\begin{remark}This result was established, also by inspection, in Cor.13.3.3 of \cite{ABP3}. In this paper the classification of nullity 2 multiloop algebras over $k$ is achieved by considering loop algebras of the affine algebras. More precisely, it is shown that every multiloop algebra of nulllity2 is isomorphic as a Lie algebra over $k$ to a Lie algebra of the form  $L(\gg \otimes k[t_1^{\pm1}], \pi)$ where $\pi$ is a diagram automorphism of the untwisted affine Lie algebra $\gg \otimes k[t_1^{\pm1}].$ For example, van de Leur's algebra appears by taking $\gg$ of type $E_7$ and considering the diagram automorphism of order two of the corresponding extended Coxeter-Dynkin.

Note that in the present work we have outlined a general procedure to classify loop adjoint groups and algebras {\it over $R_n$}, and that the classification of multiloop algebras  {\it over $k$} follows by $\GL_n$-considerations from that over $R_n.$ This is not the case in \cite{ABP3}. The nullity 2 classification relies on the structure of the affine algebras and only yields results over $k.$ 

\end{remark}

\section{The case of  orthogonal groups}

These groups are related to quadratic forms, which allows for a very precise understanding of their nature based on our  results.

We consider the example of  the split orthogonal group $\bO(d)$ for $d
\geq 1$. If $d=2m$ (resp. $d=2m+1$),  this is the orthogonal group corresponding to
the quadratic form $\sum\limits_{i=1}^m X_i X_{2m+1-i}$ (resp.  $\sum\limits_{i=1}^m X_i X_{2m+1-i} + X_{2m+1}^2$).
Since $R_n$-projective modules of finite type are free, 
we know that  $H^1\big(R_n, \bO(d)\big)$  classifies regular quadratic forms over
$R_n^{d}$ \cite[\S 4.6]{K2}. We have $H^1_{loop}\big(R_n, \bO(d)\big) \simlgr 
H^1\big(F_n, \bO(d)\big)$. By iterating Springer's theorem for quadratic forms over
$k((t)))$ \cite[\S 6.2]{Sc},  the classification of $F_n$-quadratic forms
reads as follows:
For each subset  $I \subset \{1, ...,n\}$, we put $t_I= \prod_{i\in I} t_i$
with the convention $1=t_\emptyset$;
we denote by $\HH$ the hyperbolic plane, that is the rank two split form.
The isometry classes of $d$-dimensional $R_n$--forms are then of the form
$$
\perp_{I \subset \{1, ...,n\}} \, t_I \, q_I \enskip \perp \enskip    \HH^v
$$
where the $q_I$'s are
anisotropic quadratic $k$-forms and $v$ a non negative integer such that
$\sum\limits_{I \subset \{1, ...,n\}} \dim_k(q_I) \, + \, 2 v = d$.

\begin{corollary} The set $H^1_{loop}\big(R_n, \bO(d)\big)$ is parametrized  by the quadratic forms
$$
\perp_{I \subset \{1, ...,n\}} \, t_I \, q_I \enskip \perp \enskip    \HH^v
$$
where the $q_I$ are
anisotropic quadratic $k$-forms and $v$ a non-negative integer such that
$\sum\limits_{I  \subset \{1, ...,n\}} \dim_k(q_I) \, + \, 2 v = 2d$. \qed
\end{corollary}

We denote by   ${\cal P}(n)$ the set of subsets of $\{1, ...,n\}$
and by  ${\cal P}^{even}_{\leq d}(n) \subset {\cal P}(n)$  the set
of subsets of $\{1, ...,n\}$ of even cardinal $\leq d.$ In a similar fashion we define ${\cal P}^{odd}_{\leq d}(n).$

\begin{corollary} Assume that  $k$ is quadratically  closed.

\smallskip

(1) If $d=2m$, then   the map
$$
\begin{CD}
{\cal P}^{even}_{\leq d}(n) @>>> H^1_{loop}(R_n, \bO(d)) \\ 
S & \mapsto &  \perp_{I \subset S} \, \langle t_I \rangle \enskip \perp
\enskip
   \HH^{m - \frac{\mid S \mid}{2}} \\
\end{CD}
$$
is a bijection.

\smallskip

(2) If $d=2m+1$, then   the map
$$
\begin{CD}
{\cal P}^{odd}_{\leq d}(n) @>>> H^1_{loop}(R_n, \bO(d)) \\ 
S & \mapsto &  \perp_{I \subset S} \, \langle t_I \rangle \enskip \perp
\enskip
   \HH^{m + \frac{1 - \, \mid S \mid}{2}} \\
\end{CD}
$$
is a bijection.
\end{corollary}

\begin{corollary} Assume that  $k$ is quadratically  closed.
Inside $\bO'_d = \bO( \langle 1, \dots , 1\rangle) \simeq \bO_d$. 
there is a single   $\bO'_d(k)$-conjugacy class  of maximal
anisotropic abelian constant subgroup of $\bO'(d)$, that  of the diagonal subgroup $\bmu_2^{d}$. 
In particular anisotropic abelian subgroups of $\bO'(d)$ are $2$-elementary.
\end{corollary}

\begin{proof} Let $\bA$ be a finite  abelian constant group of $\bO'_d$.
There exist an even integer $m \geq 1$ and a surjective homomorphism
$\phi : (\Z/m\Z)^n \to \bA(k).$  Then the corresponding loop torsor
$[\phi] \in H^1(R_n, \bO'_d)$ is anisotropic.
Indeed the map $H^1(R_n, \bmu_2^{d}) \to H^1(R_n, \bO'_d)$ 
is surjective. Hence there exists $\psi: (\Z/m\Z)^n \to  \bmu_2^{d}$ such that
$[\phi]= [\psi] \in H^1(R_n, \bO'_d)$. 
Theorem \ref{k-aniso-bis} shows that
$\phi$ and $\psi$   are $\bO'_d(k)$-conjugate.
By considering  their images, we conclude that $\bA(k)$ is  $\bO'_d(k)$-conjugate
to a subgroup of  $\bmu_2^{d}(k)$.
\end{proof}

\begin{remark} (1) All anisotropic abelian constant subgroups of $\bO'_d$ are related
to codes, and these are not explicitly enumerated (see \cite{Gs} for details).

\smallskip

(2) Under the hypothesis of the Corollary, let
$f: \Spin'_d \to \SO'_d$ be the universal covering of $\bO'_d$.  Since the
image of
a finite  abelian constant anisotropic subgroup of $\Spin'_d$ in $\bO'_d$ is still
anisotropic,
it follows that an anisotropic finite constant abelian subgroup of  $\Spin'_d$
is of rank $\leq d$ and has $4$-torsion.
\end{remark}

\section{Groups of type $G_2$}

We denote by $\bG_2$ the split Chevalley group of type $G_2$ over $k.$
If $F$ is a field of characteristic zero containing $k$,
we know that $H^1(F_n, \bG_2)$ classifies octonion $F$-algebras or
alternatively
$3$-Pfister forms \cite[\S 8.1]{Se2}. This follows from the fact that the Rost invariant
\cite{GMS}
$$
r_F: H^1(F, \bG_2) \to H^3(F, \Z/ 2 \Z)
$$
 is injective and sends the class of an octonion algebra
to the Arason invariant of its norm form.

Consider the standard non-toral constant abelian subgroup $f: (\Z/2\Z)^3 \subset \bG_2$.
Then the composite map
$$
(F^\times/ {F^\times}^3) \cong H^1(F, (\Z/2\Z)^3)
\buildrel f_* \over \longrightarrow
H^1(F, \bG_2) \buildrel r_F \over \longrightarrow \,  H^3(F, \Z/ 2 \Z).
$$
sends an element  $\bigl( (a) , (b) , (c) \bigr)$ to the cup product
$(a) . (b) . (c) \in  H^3(F, \Z/ 2 \Z)$ \cite[\S 6]{GiQ}.
For $n \geq 0$, we consider the mapping
$$
(R_n^\times/ {R_n^\times}^2)^3 \simeq H^1\big(R_n, (\Z/2\Z)^3\big)
\buildrel f_* \over \longrightarrow
H^1(R_n, \bG_2).
$$
For a class  $\bigl( (x), (y), (z) \bigr) \in \bigl(R_3^\times /
(R_3^\times)^2\bigr)^3$
we write only  $(x,y,z)$.

\begin{corollary}\label{octonion1}  The map above surjects onto
$H^1_{loop}(R_n, \bG_2)$.

\end{corollary}

\begin{proof}  By the Acyclicity Theorem, it suffices to observe that the analogous
statement holds for  $H^1_{loop}(F_n, \bG_2)$.
\end{proof}

By using the Rost invariant, we  get  a full classification 
of the multiloop algebras based on the split Lie algebra of type $G_2$.

\begin{corollary}\label{octonion2}  Assume that $k$ is quadratically
closed. Assume that $n \geq 3$.

\smallskip

1) $H^1_{loop}(R_n, \bG_2) \setminus \{1 \}$
consists in  the images by $f_*$ of the 
$\Bigl( t_{I_1},  t_{I_2}, t_{I_3}\Bigr)$
where   $I_1$, $I_2$, $I_3$ are non-empty subsets 
of $\{1,..,n\}$ such that $i_1 < i_2 < i_3$ for all $(i_1,i_2,i_3) \in I_1 \times I_2 \times I_3$.

\smallskip

2)  $\bGL_n(\Z) \backslash \bigl( H^1_{loop}(R_n, \bG_2) \setminus \{1 \}
\bigr)$ consists of the image  by $f_*$ of $( t_1, t_2, t_3 )$.

\end{corollary}

\begin{proof} (1) Again by aciclicity it suffices to establish the analogous result over $F_n$.
Since $k$ is quadratically closed, we have
$R_n^\times/ (R_n^\times)^{\times 2} \cong F_n^\times/ (F_n^\times)^{\times 2} \cong (\Z/2\Z)^n$.
Hence  $H^1(F_n, \bG_2)$ consists of the image of $f_*( t_{I_1}, t_{I_2}, t_{I_3} )$
for $I_1,I_2,I_3$ running over the subsets of  $\{1,..,n\}$. The Rost invariant of such a class is  $(t_{I_1}) .(t_{I_2}). (t_{I_3}) \in H^3(F_n, \Z/2\Z)$. 
Since $(t_i).(t_i)=0$ and $(t_i).(t_j)=(t_j)(t_i) \in H^3(F_n, \Z/2\Z)$, 
it follows that $H^1(F_n, \bG_2)$ consists of the trivial class
and  the images by $f_*$ of the 
$\Bigl( t_{I_1},  t_{I_2}, t_{I_3}\Bigr)$
where   $I_1$, $I_2$, $I_3$ are non-empty subsets 
of $\{1,..,n\}$ such that $i_1 < i_2 < i_3$ for each $(i_1,i_2,i_3) \in I_1 \times I_2 \times I_3$.
The last classes are non-trivial pairwise distinct elements since 
the $(t_{I_1}) .(t_{I_2}). (t_{I_3}) \in H^3(F_n, \Z/2\Z)$ are
distinct pairwise elements by residue considerations (see for example  prop. 3.1.1 of \cite{GP3}).

\smallskip

\noindent (2) Follows easily from (1).
\end{proof}

The following corollary refines Griess' classification in the
$G_2$-case \cite{Gs}.

\begin{corollary}\label{octonion3} Assume that $k$ is algebraically closed.
Let $\bA$ be an anisotropic constant abelian subgroup of $\bG_2$. Then $\bA$ is 
$\bG_2(k)$-conjugate to the standard non-toral subgroup $(\Z/2\Z)^3$. 
\end{corollary} 

\begin{proof} Let $\bA$ be a finite  abelian constant anisotropic  subgroup of $\bG_2$. We reason as before.
There exist an even integer $m \geq 1$ and a surjective homomorphism
$\phi : (\Z/m\Z)^n \to \bA(k)$ so that the corresponding loop torsor
$[\phi] \in H^1(R_n, \bG_2)$ is anisotropic.
By part (1) of Corollary \ref{octonion1} there exists $\psi: (\Z/m\Z)^n
\to  (\Z/2\Z)^3$
such that $[\phi]= [\psi] \in H^1(R_n, \bG_2)$.
Theorem \ref{k-aniso-bis} shows that
$\phi$ and $\psi$   are $\bG_2(k)$-conjugate.
By taking  the images, we conclude that $\bA(k)$ is  $\bG_2(k)$--conjugate
to  the standard $(\Z/2\Z)^3$.
\end{proof}

\section{Case of  groups of type $F_4,$ $E_8$ and simply connected $E_7$ in  nullity $3$}

In this section, we assume that $k$ is algebraically closed.
We denote by $\bF_4$, and $\bbE_8$ the split algebraic $k$--group
of type  $F_4$ and $E_8$ respectively, and by $\bbE_7$ the split simply connected $k$--group of type $E_7.$
For either of these three groups we know that
$\bG= \bAut(\bG)$ and that  $H^1(R_2, \bG)= 1$  \cite[th. 2.7]{GP2}. The
goal is then to compute
$H^1_{loop}(R_3, \bG),$ or at least the anisotropic classes.

Since we  want to use  Borel-Friedman-Morgan's classification of
rank zero (i.e.  with finite  centralizer)
abelian subgroups and  triples  of
the corresponding compact Lie  group  \cite[\S 5.2]{BFM}, 
we will assume that $k= \C$. Note that there is no loss of generality in doing this as explained in Remark
\ref{Leftschetz}.\footnote{All the results that we need about rank zero abelian groups and triples can also be found in \cite{KS}.}

Denote by $\bG_0$ the anisotropic real form of $\bG$ (viewed as algebraic group over $\R$) and
let $K= \bG_0(\R).$ This is a compact  Lie group.

In the $F_4$  and  $E_7$ case $K$  has a single conjugacy class of
rank zero abelian subgroup  of rank $3$.
In the $E_8$ case, $K$ has two conjugacy classes of  rank zero
abelian subgroup  of rank $3$,  $(\Z/5\Z)^3$ and $(\Z/6\Z)^3$.
To translate this to the complex case we establish the following
 fact.

\begin{lemma}\label{ernest} Let $\bH$ be a complex affine algebraic group
whose connected component of the identity is reductive. Denote by 
$\bH_0$ its anisotropic real form, viewed as algebraic group over $\R$ (see \cite[\S 5.2, th. 12]{OV}).
 Set  $K_H= \bH_0(\R)$.

\smallskip

(1) Let  $A$ is a finite abelian subgroup of $K_H$ and denote by $\bA $ the  underlying
constant  subgroup of the algebraic $\R$-group  $\bH_0$.
Then $A$ is a rank zero subgroup of $K_H$ if and only if
   $\bA \times_\R \C$ is an anisotropic subgroup of $\bH$.

\smallskip

(2) Let $\bA$ be  an anisotropic  abelian constant  subgroup of $\bH$ and put $A=\bA(\C)$. 
Then there exists $h \in \bH(\C)$ such that
${^hA} \subset K_H$ and $N_{K_H}({^hA})= \bN_\bH({^h\bA})(\C)$, both groups being finite.
Furthermore $Z_{K_H}({^hA})= \bZ_\bH(^h\bA)(\C)$.

\end{lemma}

Recall that $K$ is a maximal subgroup of $\bH(\C)$ and that
maximal compact subgroups are conjugate under $\bH^0(\C)$.

\begin{proof} (1) Let  $\bC$ denote the connected component of the identity of the centralizer $\bZ_{\bH_0}(\bA)$. 
It is a  real reductive group \cite[10.1.5]{BMR}.   If  $\bA$ is an anisotropic
 subgroup of $\bH$, then the maximal tori of $\bZ_\bH(A)$ are trivial
and $\bC=1$. Hence $\bC(\C)$ is finite and $Z_{K_H}(A)$ is finite, i.e. $A$ is a rank zero subgroup of $K_H$.
Conversely, if $A$ is a rank zero subgroup of $K_H$ then
$\bC(\R)$ is finite. Since $\bC(\R)$ is Zariski dense in the connected group $\bC$, 
we see that $\bC=1$, and $\bA \times_\R \C$ is an anisotropic constant abelian subgroup of $\bH$.

\smallskip

\noindent (2) We are given a finite anisotropic constant subgroup $\bA$ of $\bH$.
  Since $1=\bZ_\bH(\bA)^0= \bN_\bH(\bA)^0$,  $\bN_\bH(\bA)$ is a finite algebraic group
 and $\bN_\bH(A)(\C)$ is finite. 
Since $\bN_\bH(\bA)(\C)$ is included in a maximal compact group of $\bH(\C)$,
we know that there exists $h \in \bH(\C)$ such that 
$A \subset \bN_\bH(\bA)(\C) \subset {^{h^{-1}}K_H}$.
We have then ${^hA}\subset \bN_\bH({^h\bA})(\C) \subset {K_H}$, hence
$N_{K_H}({^hA})= \bN_\bH({^h\bA})(\C)$.
It follows that $Z_{K_H}({^hA})= \bZ_\bH({^h\bA})(\C)$.

\end{proof}

\begin{lemma} \label{botanical} (1) The group $\bF_4$  has a single
conjugacy class of anisotropic finite abelian (constant) subgroups  of rank $3,$ 
 denoted by $f_3: (\Z/3\Z)^3 \subset \bF_4$.
Furthermore 
$$ \bN_{\bF_4}\big((\Z/3\Z)^3\big) /  \bZ_{\bF_4}\big((\Z/3\Z)^3\big) \simeq  \SL_3(\Z/3\Z). $$

\smallskip

(2) The group $\bbE_7$ has a single
conjugacy class of anisotropic finite abelian (constant) subgroups  of rank $3,$ 
 denoted by $f_4: (\Z/4\Z)^3 \subset \bbE_7.$ The finite group $f_4$
is a subgroup
the maximal subgroup $\SL_8 / \bmu_2$.
Furthermore $\bN_{\bbE_7}\big((\Z/4\Z)^3\big)/  \bZ_{\bbE_7}\big((\Z/4\Z)^3 \big) \simeq \SL_3(\Z/4\Z)$.

\smallskip

(3) The  group $\bbE_8$ has two conjugacy classes  of
anisotropic finite abelian (constant) subgroups  of rank $3,$  denoted by $f_5$ and $f_6.$ We have:

\smallskip

(a) $f_5: (\Z/5\Z)^3 \subset \bbE_8 $ and  $\bN_{\bbE_8}\big((\Z/5\Z)^3\big) /   \bZ_{\bbE_8}\big((\Z/5\Z)^3 \big) \simeq\SL_3(\Z/5\Z)$.

\smallskip

(b) $f_6: (\Z/6\Z)^3 \subset \bbE_8$ is a subgroup of
the subgroup  \break $(\SL_2 \times \SL_3 \times \SL_6) / \bmu_6$.
Furthermore $\bN_{\bbE_8}\big(\Z/6\Z)^3\big)/ \bZ_{\bbE_8}\big((\Z/6\Z)^3\big) \simeq \SL_3(\Z/6\Z)$.

\end{lemma}

\begin{remark} The finite subgroups (1) and 3 (a) are described precisely in  \cite[\S 6]{GiQ}. That the third one, namely that $f_6: (\Z/6\Z)^3 \subset \bbE_8$
sits inside the  subgroup   $(\SL_2 \times \SL_3 \times \SL_6) / \bmu_6,$
follows from its very construction (see the proof of lemma 5.1.1   \cite{BFM} for details).
\end{remark}

\begin{proof} As explained above we may assume that $k = \C.$ The previous Lemma \ref{ernest} shows that any rank $0$ finite abelian constant
 subgroup $\bA$ of $\bG$ arises from rank $0$ abelian subgroup $A$
of $K$, so  the list of Borel-Friedman-Morgan  \cite[\S 5.2]{BFM}
provides all relevant conjugacy classes, and this yields the inclusions
$f_3$, $f_4$, $f_5$ and $f_6$ described above.
Given two rank $0$ finite abelian constant  subgroups $\bA$ and $\bA'$ of $\bG$ arising respectively  from rank $0$ abelian subgroups $A, A'$
of $K$, it remains to check that $\bA(\C)$ and $\bA'(\C)$ are $\bG(\C)$--conjugate if and only if
 $A$ and $A'$ are $K$-conjugate.
But this is obvious since  the subgroups from the list are distinct as groups.
We investigate now the normalizers and centralizers.

\begin{claim} Let $A \subset K$ be a rank zero subgroup. Then $N_{K}({A})= \bN_\bG({\bA})(\C), \enskip Z_{K}({A})= \bZ_\bG(\bA)(\C).$ 

\end{claim}

Indeed   Lemma \ref{ernest}.(2) shows the existence of an element $g \in \bG(\C)$
 such that  ${^gA} \subset K$ and 
$$
N_{K}({^gA})= \bN_\bG({^g\bA})(\C), \enskip Z_{K}({A})= \bZ_\bG({^g\bA})(\C).
$$
But $A$ and $^gA$ are $K$--conjugate by Borel-Friedman-Morgan's theorem, 
so the same fact holds for $g=1$.

It is then enough to know the quotient ``normalizer/centralizer" 
in the compact group case.  
For each relevant  $d$, 
we have an exact sequence of groups 
$$
1 \to  Z_K\big((\Z/d\Z)^3\big) \to  N_K\big(\Z/d\Z)^3\big) \buildrel \theta \over \to \GL_3(\Z/d\Z)
$$
and we want to determine the image of $\theta$.
Denote by  ${\cal S}_d$ the set of $K$-conjugacy classes of rank zero  triples of $K$ of order $d$.
Since such a triple generates a rank zero abelian subgroup of order $d^3$ of $K$, 
the set  ${\cal S}_d$ is covered by rank zero triples inside $(\Z/d\Z)^3$, namely
$\GL_3(\Z/d\Z)$-conjugates of the standard triple $(1,1,1)$.
 So   we have $\GL_3(\Z/d\Z) / {\rm Im}(\theta) \cong {\cal S}_d$.
 Proposition 5.1.5 of \cite{BFM} states that the
$K$-conjugacy classes of rank zero  triples of $K$ of order $d$
consists of   the classes 
$f_d(1,1,i)$ for $i=1,..,d-1$ with $i$ prime to $d$.
Hence the image of $\theta$ in $\GL_3(\Z/d\Z)$ is exactly $\SL_3(\Z/d\Z)$
as desired.
\end{proof}

Given  $f_d: (\Z/d\Z)^3 \to \bG$ as above consider the map
$$
f_{d,*} : \bigl(R_3^\times / (R_3^\times)^d\bigr)^3 \simeq
  H^1\big( R_3, (\Z/d\Z)^3\big) \to H^1( R_3, \bG).
$$
A class  $\bigl( (x), (y), (z) \bigr) \in \bigl(R_3^\times /
(R_3^\times)^d\bigr)^3$
will for convenience simply be written as  $(x,y,z)$.

\begin{corollary}

(1) The set $H^1_{loop}(R_3,\bF_4)_{an}$ consists of
the classes of  $f_{3,*}(t_1,t_2,t_3)$ and $f_{3,*}(t_1,t_2,t_3^2)$.

\smallskip

(2) The set $H^1_{loop}(R_3,\bbE_7)_{an}$ consists of
the classes of   $f_{4,*}(t_1,t_2,t_3)$, $f_{4,*}(t_1,t_2,t_3^3)$.

\smallskip

(3) The set $H^1_{loop}(R_3,\bbE_8)_{an}$ consists in
the classes of   $f_{5,*}(t_1,t_2,t_3^i)$ for $i=1,2,3,4$, 
$f_{6,*}(t_1,t_2,t_3)$ and $f_{6,*}(t_1,t_2,t_3^5)$.

\end{corollary}

\begin{proof} We do in detail the case of $F_4,$ the other cases being similar.
The set $H^1_{loop}(R_3,\bF_4)_{an}$ is  covered
by the image of the anisotropic loop cocycles $\phi:\pi_1(R_3) \to \bF_4(\C)$.
The image  of such a $\phi$ is an anisotropic finite abelian subgroup
of $\bF_4$, so Lemma \ref{botanical}.1 allows us to assume that 
 its image is the subgroup
$(\Z/3\Z)^3$. Furthermore, we know that two such homomorphisms 
$\phi$ and $\phi'$ have the same image in $H^1_{loop}(R_3,\bF_4)_{an}$ 
if and only if there exists $g \in \bF_4(\C)$ such that $g \phi g^{-1}= \phi'$,
or equivalently if there exists $g \in \bN_{\bF_4}\big( (\Z/3\Z)^3 \big)(\C)$ such that $g \phi g^{-1}= \phi'$.
Note the importance of the  isomorphism
 $\bN_{\bF_4}( (\Z/3\Z)^3 ) / \bZ_{\bF_4}( (\Z/3\Z)^3 ) \simeq \SL_3(\Z/3\Z)$.

Rephrasing what has been said in terms of the mapping $f_{3,*}$, we see that 
 $H^1_{loop}(R_3,\bF_4)_{an}$ is the image under $f_{3,*}$
of the classes $(x,y,z)$ where $x,y,z \in R_3^\times$ 
are such that $(x,y,z)$ generates $R_3^\times/(R_3^\times)^3$;
furthermore,  two such classes 
$(x,y,z)$ and $(x',y',z')$ have the same image in $H^1_{loop}(R_3,\bF_4)_{an}$ 
if and only if there exists $\tau \in \SL_3(\Z/3\Z)$ such that
 $(x',y',z')=   \tau_*\big((x,y,z)\big)$.
We conclude  that $H^1_{loop}(R_3,\bF_4)_{an}$ consists of
the classes of  $f_{3,*}(t_1,t_2,t_3)$ and $f_{3,*}(t_1,t_2,t_3^2)$.
\end{proof}

\begin{corollary}

(1) The set $\GL_3(\Z) \backslash H^1_{loop}(R_3,\bF_4)_{an}$ consists of
the class of   $f_{3,*}(t_1,t_2,t_3)$.

\smallskip

(2) The set $\GL_3(\Z) \backslash H^1_{loop}(R_3,\bbE_7)_{an}$ consists of
the class $f_{4,*}(t_1,t_2,t_3)$.

\smallskip

(3) The set $\GL_3(\Z) \backslash H^1_{loop}(R_3,\bbE_8)_{an}$ consists of
the classes of $f_{5,*}(t_1,t_2,t_3)$ and 
$f_{6,*}(t_1,t_2,t_3)$.

\end{corollary}

\begin{remark} The above Corollary gives the full classification of nullity $3$
anisotropic multiloop algebras of absolute type $F_4$  or $E_8$. 
\end{remark}

\section{The case of $\PGL_d$}

\subsection{Loop Azumaya algebras}

For any base scheme $\gX$, 
the set  $H^1(\gX,\PGL_d)$ classifies the isomorphism classes of Azumaya ${\mathcal O}_{\gX}$-algebras $A$
of degree $d$, i.e. ${\mathcal O}_{\gX}$-algebras which are  locally  isomorphic for the \'etale topology
to the matrix algebra $M_d({\mathcal O}_{\gX})$ \cite{Gr2} and \cite[\S III]{K2}.

The exact sequence $1 \to {\bf G}_m \to \GL_d \buildrel p \over \lgr
\PGL_d \to 1$ induces the sequence of pointed sets
$$
\Pic(\gX) \to H^1(\gX, \GL_d) \to H^1(\gX, \PGL_d) \buildrel
\delta \over  \lgr H^2(\gX, {\bf G}_m)=\Br(\gX).
$$
We denote again by $[A] \in \Br(\gX)$ the class of $\delta([A])$ in
the cohomological Brauer group. 

By \cite[3.1]{GP3}, we have an isomorphism $\Br(R_n) \cong \Br(F_n)$.
We look now at the diagram
$$
\begin{CD}
 H^1_{loop}(R_n, \PGL_d) @>{\delta}>>  \Br(R_n)  \\
@V{\cong}VV @V{\cong}VV \\
 H^1(F_n, \PGL_d)@>{\delta}>>  \Br(F_n)
\end{CD}
$$
where the bottom map is injective \cite[\S 4.4]{GS} and the left map is bijective because of 
Theorem \ref{acyclic}. We thus have

\begin{corollary}\label{brauer} The boundary map  $H^1_{loop}(R_n, \PGL_d)\to  \Br(R_n)$ 
is injective. \qed
\end{corollary}

Azumaya $R_n$--algebras whose classes are in $H^1_{loop}(R_n, \PGL_d)$ are
called {\it loop Azumaya algebras.} They are isomorphic to twisted form of
$M_d$ by a loop cocycle. One can rephrase the last Corollary by saying that 
loop Azumaya algebras of degree $d$ are classified by their ``Brauer invariant".

Similarly, Wedderburn's theorem \cite[2.1]{GS} for $F_n$--central simple algebras has
its counterpart.

\begin{corollary}\label{wedderburn} Let $A$ be a loop Azumaya $R_n$--algebra of degree $d$.
Then there exists a unique positive integer $r$ dividing $d$
and a loop Azumaya $R_n$--algebra $B$ (unique up  to  $R_n$--algebras isomorphism) of degree $d/r$ such that 
$A \simeq M_r(B)$ and $B \otimes_{R_n} F_n$ is a division algebra. \qed
\end{corollary}

This reduces the classification  of loop Azumaya $R_n$-algebras to the
``anisotropic''  case, namely to the case of  loop Azumaya $R_n$--algebras $A$ 
such  that $A \otimes_{R_n} F_n$ is a division algebra.

In the same spirit, the Brauer decomposition \cite[4.5.16]{GS} for central $F_n$--division algebras
yields the following.

\begin{corollary}\label{brauer2} 
Write $d= p_1^{m_1} \cdots p_l^{m_l}$.
Let $A$ be an anisotropic loop Azumaya $R_n$-algebra of degree $d$.
Then there exists a unique decomposition 
$$ 
A \simeq A_1 \otimes_{R_n}   \cdots  \otimes_{R_n} A_l
$$
where $A_i$ is an anisotropic $k$--loop Azumaya $R_n$-algebra of 
degree $p_i^{m_i}$ for $i=1,..,l$. \qed
\end{corollary}

The two previous Corollaries show that the classification of  loop Azumaya $R_n$--algebra reduces the classification of anisotropic loop Azumaya $R_n$-algebras of degree $p^m.$ Though the  Brauer group of $R_n$ and $F_n$ are well understood,
the understanding of  $H^1_{loop}(R_n, \PGL_d)_{an}$ is much more delicate.

\medskip

We are given a loop cocycle $\phi=(\phi^{geo}, z)$ with values
in $\PGL_d(\ol k)$. Set $A= {_z(M_d)}.$  This is  a central  simple $k$--algebra
such that ${_z\PGL}_d = \PGL_1(A)$. Recall that  $\phi^{geo}$ is given by a 
$k$--group homomorphism $\phi^{geo}: \bmu_m^n \to \PGL_1(A)$.
To say that $\phi$ is anisotropic is 
to say that  $\phi^{geo}: \bmu_m^n \to \PGL_1(A)$ is anisotropic.

\medskip

We discuss in detail the following two special cases : the  one-dimensional case, and 
the geometric case  (i.e. $k$ is algebraically closed).

\subsection{The one-dimensional case}

If  $k$ is algebraically closed $H^1(R_1, \PGL_d)$  is trivial.
The interesting new case is when $k$ is not algebraically closed, e.g. the case of real numbers. Since the map $H^1(F_1, \PGL_d) \to  \Br(F_1)$ is injective,
as a consequence of Corollary  \ref{dim-one}, we have 
$H^1(R_1, \PGL_d) \simeq H^1(F_1, \PGL_d)$ 
and
the map 
\begin{equation}\label{PGLmap}
H^1(R_1, \PGL_d) \to  \Br(R_1)= \Br(k) \oplus H^1(k, \Q/\Z)
\end{equation}
is injective.

\begin{theorem}\label{local} The image of the map \ref{PGLmap} consists  of  all pairs $[A_0] \oplus \chi$ where
$A_0$ is a central simple algebra of degree $d$ and 
$\chi: \Gal(k_s/k) \to \Q/\Z$ a character for which  that there exists an \'etale 
algebra   $K/k$  of degree $d$ inside $A_0$ such that $\chi_K=0$.
\end{theorem}

\begin{remark}\label{JP}
The indices of such algebras over $F_1$ are known (\cite[prop. 2.4]{Ti} in the 
prime exponent case, and \cite[Lemma 4.6]{FSS} in the general case).
The index of a $F_1$--algebra  of invariant $[A_0] \oplus \chi$
is $\deg(\chi) \times {\rm ind}_{k_\chi}(A \otimes_k k_\chi)$
where  $k_\chi/k$ stands for the cyclic extension associated to $\chi$.
\end{remark}

The proof needs  some  preparatory material from homological algebra based
on Cartier duality for groups of multiplicative type.
More precisely, the dual of an extension of $k$--groups of 
multiplicative type
$$
1 \to {\bf G}_m \to \bbE \to \bmu_m \to 1 
$$ 
is the exact sequence
$$
0 \to \Z/m\Z \to \widehat E \to \Z \to 1.
$$
 We have then an isomorphism
$$
 Ext^1_{k-gr}(\bmu_m, {\bf G}_m ) \simeq
 Ext^1_{\Gal(k)}(\Z,\Z/m\Z)=H^1(k,\Z/m\Z)
$$
which permits to attach to the first extension a character.
Up to isomorphism, there exists a unique extension 
$\bbE_\chi$ of $\bmu_m$ by ${\bf G}_m$ of class $[\chi]$.

\begin{lemma}\label{tate} Let 
$\chi: \Gal(k) \to \Z/m\Z$ be a character for some $m \geq 1$.

\begin{enumerate}
 \item The boundary map 
$$
 k^\times/ (k^\times)^m \simlgr  H^1(k,\bmu_m) \to H^2(k, {\bf G}_m)=\Br(k)
$$
is given by $(x) \mapsto  \chi \cup (x)$.

 \item  Let $K/k$ be an \'etale algebra. 
The following are equivalent:
\begin{enumerate}
\item There exists a   morphism of extensions  $\bbE_\chi \to 
 R_{K/k}( {\bf G}_m)$ rendering the diagram
$$
\begin{CD}
1 @>>> {\bf G}_m @>>> \bbE_\chi @>>> \bmu_m @>>> 1 \\ 
&& @VV{\cong}V @VVV @VVV \\
1 @>>> {\bf G}_m @>>> R_{K/k}( {\bf G}_m) @>>> R_{K/k}( {\bf G}_m)/ {\bf G}_m  @>>> 1 ;\\ 
\end{CD}
$$
commutative.

\item $\chi_K=0$.
\end{enumerate}

\end{enumerate}

\end{lemma}

\begin{proof}
(1)  The cocharacter group $\widehat E_\chi$ is 
$\Z/m\Z \oplus \Z$ together with the Galois action
$\gamma( \alpha, \beta)= (  \alpha + \chi(\gamma) , \beta )$.  
The Galois action on  $\bbE_\chi(\ol k) \simeq {\ol k}^\times \times \bmu_m(\ol k)$
is then  given by $$
\gamma( y, \zeta)= \Bigl( \gamma(y) \,  \zeta^{\chi(\gamma)}, 
\gamma(\zeta) \Bigr)
$$  
for every $\gamma \in \Gal(k)$. 
The class $(x) \in H^1(k,\bmu_m)$ is represented by the cocycle
$c_\gamma= \gamma( \sqrt[m]{x})/ \sqrt[m]{x}$.
The element  $b_\gamma= (1, c_\gamma) \in \bbE_\chi(\ol k)$ lifts
$c_\gamma$. The boundary $\partial\bigl( (x) \bigr) \in  H^2(k, {\ol k}^\times)$
is then represented by the $2$--cocycle 
$$
a_{\gamma, \tau}= b_\gamma \times  \gamma(b_\tau)  \, b_{\gamma \tau}^{-1} 
=c_\tau^{\chi(\gamma)} \chi(\gamma) \, . \,  c_\tau \, \, .
$$

\smallskip

\noindent (2) 
We decompose $K=k_1 \times \cdots \times k_l$ as a product of field extensions 
and denote by $M_j$ the cocharacter module  of $R_{k_j/k}( {\bf G}_m)$.
Then the character module of $R_{K/k}({\bf G}_m)$
is ${M= \oplus M_j}$. By dualizing we are interested in  morphism of extensions  
$$
\begin{CD}
 0 @>>> I @>>> M @>>> \Z  @>>> 0 \\
&&@VVV @VVV @VV{\simeq}V \\
 0 @>>> \Z/m\Z @>>> \hat E_\chi @>>> \Z @>>> 0. \\
\end{CD}
$$
 By Shapiro's lemma $Ext^1( M_j, \Z/m\Z)= H^1(k_j, \Z/m\Z)$ 
and the map $Ext^1( \Z,  \Z/m\Z) \to Ext^1( M_j, \Z/m\Z)$ yields
 the restriction map $H^1(k, \Z/m\Z) \to  H^1(k_j, \Z/m\Z)$.
It follows that the bottom extension above is killed by the pull-back $M_j \to \Z$, and therefore that
$\chi_{k_j}=0$ for $j=1,..,l$. This shows that $(a) \Longrightarrow (b).$

\noindent  $(b) \Longrightarrow (a)$: We assume that $\chi_K=0$, namely $\chi_{k_j}=0$ for $j=1,..,l$.
Hence $\widehat E_\chi$ belongs  to the kernel of
  $Ext^1( \Z,  \Z/m\Z) \to Ext^1( M_j, \Z/m\Z)$  for for $j=1,..,l$
so $\widehat E_\chi$  belongs to the kernel of  $Ext^1( \Z,  \Z/m\Z) \to Ext^1( M, \Z/m\Z)$.
This means that the map $M \to \Z$ of Galois modules lifts  to   $\widehat E_\chi \to \Z$ as desired.
\end{proof}

We can proceed now with the proof of Theorem \ref{local}.

\begin{proof} 
We show first that the image of $\partial$ consists of
pairs with the desired properties.
Again by Corollary \ref{dim-one},
we have $H^1_{loop}(R_1, \PGL_d) = H^1(R_1, \PGL_d)$ and 
 we are reduced to twisted algebras given by
loop cocycles $\phi=(\phi^{geo}, z)$ with value in $\PGL_d(\ol k)$. 
Recall that $A_0= {_z(M_d)}$ and that we have then a  $k$--group homomorphism 
$\phi^{geo}: \bmu_m \to \PGL_1(A_0)$.  We may assume that
$\phi^{geo}$ is injective.
We pull back the central extension  
${1 \to {\bf G}_m \to \GL_1(A_0) \buildrel p \over \lgr
\PGL(A_0) \to 1}$ by $\phi^{geo}$ and get a central extension 
of algebraic $k$--groups
$$
1 \to {\bf G}_m \to \bbE \buildrel p' \over \lgr \bmu_m \to 1 
$$
such that $\bbE$ is a $k$--subgroup of $\GL_1(A_0)$.
By extending the scalars to $\ol k$, we see that $\bbE$ is commutative, 
hence is a $k$--group of multiplicative type.
Then $\bbE$  is contained in a maximal torus of the $k$--group 
$\GL_1(A_0)$   and is of the form $R_{K/k}( {\bf G}_m)$ where
$K \subset A_0$ is an \'etale algebra of degree $d$.
We have then the commutative diagram 
$$
\begin{CD}
 1 @>>> {\bf G}_m @>>> \bbE  @>>>  \bmu_m @>>> 1 \\
&& @VV{\simeq}V @VVV @VVV \\
1 @>>> {\bf G}_m @>>> R_{K/k}({\bf G}_m) @>>> R_{K/k}({\bf G}_m)/ {\bf G}_m  @>>> 1.
\end{CD}
$$
Lemma \ref{tate}.2 tells us that  $\chi_K=0$.
We compute the Brauer class of this loop algebra
by  taking into account the commutative diagram
$$
\begin{CD}
H^1(F_1, \PGL_d) @>>> \Br(F_1) \\
@A{\tau_z}A{\wr}A @A{\wr}A{? + [A_0]}A \\
H^1\big(F_1, \PGL(A_0)\big) @>{\partial}>> \Br(F_1) \\
@AAA @AAA \\
H^1(F_1, \bmu_m) @>{\partial}>> \Br(F_1) .\\
\end{CD} 
$$
The commutativity of the upper square is that of  the torsion map \cite[\S I.5.7]{Se}, while that of the bottom square is trivial.
The image of $(t_1) \in F_1^\times/ (F_1^\times)^\times$ is 
$\chi \cup (t_1)$ by Lemma \ref{tate}.1. The diagram yields the formula
$\partial( [\phi])= [A_0] \oplus \chi$ which has the required  properties.

Conversely, let $K/k$ be an  \'etale 
algebra    of degree $d$ inside $A_0$ and let $\chi$ be a character
 such that $\chi_K=0$. Let $m$ be the order of $\chi$; by restriction-corestriction considerations
$m$ divides $d$. Lemma \ref{tate}.2 shows that there exists a morphism
of  extensions $E_\chi \to R_{K/k}( {\bf G}_m).$ This yields a morphism
$\psi^{geo}: \bmu_m \to  R_{K/k}({ \bf G}_m)/ {\bf G}_m \to \PGL_1(A_0)$.
The previous computation shows that the loop torsor $(\psi^{geo}, z )$
has Brauer invariant $[A_0] \oplus \chi$.

\end{proof}

As an example, we consider the real case.

\begin{corollary} Assume that   $k= \R$. Then the image of the injective map
$$
H^1(R_1, \PGL_d) \to  \Br(R_1)= \Br(\R) \oplus H^1(\R, \Q/\Z)
$$
is as follows:

\begin{enumerate}
 \item $0 \oplus 0$    if  $d$ is odd;

\item $0 \oplus 0,$ $0 \oplus \chi_{\C/\R}$, $[(-1,-1)] \oplus 0$ and
$[(-1,-1)] \oplus \chi_{\C/\R}$
if $d$ is even. \qed

\end{enumerate}
 
\end{corollary}

\begin{remark} In the case $d=2$, the four classes under consideration corresponds to the quaternion algebras $(1,1)$,
$(1,t)$,  $(-1,-1)$, $(-1,t)$.
\end{remark}

\subsection{The geometric case}
We assume that $k$ is algebraically closed. According to Corollary \ref{WT4}.2,
our goal is to extract information from the bijections
$$
\Hom_{gp}\big({\widehat \Z}^n, \PGL_d(k)\big)_{irr} / \PGL_d(k) \simlgr
H^1_{loop}(R_n, \PGL_d)_{irr} \simlgr  H^1(F_n, \PGL_d)_{irr}. 
$$
The right hand set is known from the work of Amitsur \cite{Am}, 
Tignol-Wadsworth \cite{TiW}  and Brussel \cite{Br},\footnote{See also \cite{L}, \cite{Ne}, \cite[\S 8]{RY1} and  \cite{RY2}. These last two references relate to 
finite abelian constant subgroups of $\PGL_d$ which  have been investigated
by Reichstein-Youssin in their construction of 
division algebras with large essential dimension. \cite{Ne} is more interested in the ``quantum tori" point of view and its relation to EALAs of absolute type $A.$}
the left hand-side is known by a classification due to Mumford \cite[Prop. 3]{Mu}

As a byproduct  of our main result, we can provide then a proof of Mumford's classification 
of irreducible finite abelian constant subgroups of $\PGL_d$ from the knowledge of the Brauer group of the field $F_n$. 
 Let us state first Mumford's classification.
If $d= s_1....s_l \hat s$  with  $s_1 \mid s_2 \cdots  \mid s_l$ and $s_1 \geq 2$,
  we consider the embedding
$$
\PGL_{s_1} \times \cdots \times \PGL_{s_l}  \to \PGL_d
$$
and  define the subgroup $\bH(s_1,...,s_l)$ to be the image 
of the product of the standard anisotropic  
subgroups $\bH(s_j)=(\Z/s_j\Z)^2$ of $\PGL_{s_j}$ for $j=1,..,l$ defined by the generators

\begin{eqnarray} 
a_j=\left(\begin{matrix} 0 & 0 & \cdots&0&1 \\ 
1 &0 & 0 & \cdots & 0  \\
0 &1& \cdots  &&0  \\
& \cdots& &1&0 \end{matrix}\right) ,\quad &
b_j=\left(\begin{matrix} 1 & 0 & \cdots \\
0 &\zeta_{s_j} & 0 & \cdots & 0  \\
0 & \cdots & &&0  \\
0& \cdots& &0&\zeta_{s_j}^{{s_j-1}} \end{matrix}\right).
\end{eqnarray}

\begin{remark} The way of expressing the group  $\bH(s_1,s_2)$ in the form $\bH(s_1) \times \bH(s_2)$ is not unique when 
 $s_1$ and $s_2$ are coprime. There is then a unique way to  arrange 
such a group $\bH$ as $\bH(s'_1,...,s'_{l'})$ with 
$s'_1 \mid s'_2 \cdots  \mid s'_l$ and $s'_1 \geq 2$. Note that  $\rank\big(\bH(s'_1,...,s'_l)(k)\big)= 2l'$.
\end{remark}

We can now state and establish the  classification of 
irreducible finite abelian groups of the projective linear
group.

\begin{theorem}  \label{plinear} {\rm \cite[Prop. 3]{Mu} 
(see also\  \cite[\S 6]{BL}, \cite[Th. 8.28]{GM}).}
\begin{enumerate}
\item $d=s_1 \times ...\times s_l$ if and only  $\bH(s_1,...,s_l)$  is irreducible in $\PGL_d$.

\item If $\bH$ is an  irreducible finite abelian constant subgroup of $\PGL_d$, then 
$\bH$ is $\PGL_d(k)$--conjugate  to a unique  
 $\bH(s_1,...,s_l)$ 
with $d=s_1...s_l$,   $s_1 \mid s_2 \cdots  \mid s_l$ and $s_1 \geq 2$.
\end{enumerate}
\end{theorem}

As mentioned above, our  proof makes use of
 Galois cohomology results over $R_n$
for $n \geq 1$ (or equivalently $F_n$) 
collected from our previous  paper \cite{GP3}.

Our  convention on the cyclic algebra $(t_i,t_j)^{q}_p$
 is that of Tate\footnote{This is the opposite convention than that of \cite{Br} and
\cite{GP3}, but consistent with that of \cite{GS} which we use in the proof.}
 for  the Azumaya $R_n$--algebra with presentation
$$
X^q=t_i,\,  Y^q=t_j^p, \, YX =  \zeta_q \,  XY.
$$
Let us consider now Brussel $R_n$--algebras. 
Given  sequences of length $l \leq [\frac{n}{2}]$
$$
2 \leq s_1 \cdots \leq s_l , \qquad r_1, \cdots ,r_l
$$
 we define
$$
A(r_1,s_1,...,r_l,s_l):(t_1,t_2)_{s_1}^{r_1} \otimes_{R_n} \cdots  
\otimes_{R_n} (t_{2l-1},t_{2l})_{s_l}^{r_l} .
$$

\begin{lemma}\label{jojo} With the notation as above, 
set  $d=s_1...s_l$ and define 
$$
\phi: {\widehat \Z}^n \to \bH(s_1)(k) \times \cdots    \times \bH(s_l)(k)
= \bH(s_1,...,s_l)(k)  \, \, \subset  \, \, \PGL_d(k),
$$
$$
 (e_1,e_2,...,e_{2l-1}, e_{2l}) \mapsto 
(- a_1, -r_1 b_1,..., - a_{l}, -r_l \ol  b_{l}) 
$$

\begin{enumerate}
 \item Then $_\phi(M_d) \simlgr A(r_1,s_1,...,r_l,s_l)$ as $R_n$--algebras.
\item  The following are equivalent:
\begin{enumerate}

\item $A(r_1,s_1,...,r_l,s_l) \otimes_{R_n} F_n$ is division $F_n$--algebra;

\item $\phi$ is irreducible;

\item $\bH(s_1,...,s_l)$ is irreducible in $\PGL_d$
and $(r_j,s_j)=1$ for $j=1,...,l$.
\end{enumerate}
\end{enumerate}
\end{lemma}

\begin{proof} 
(1) This is done for $R_2$ and each $\bH(s_i)$ in \cite[proof of  Th. 3.17]{GP2}. This  ``extends" in an additive way to yield the general case.

\smallskip

\noindent(2) The equivalence $(a) \Longleftrightarrow (b)$ is
a special case of \cite[3.1]{GP3}. 

\smallskip

\noindent $(b) \Longrightarrow (c)$: 
Since $\phi$ is irreducible, its image  ${\rm Im}(\phi)$
is an irreducible subgroup of $\PGL_d$.
 This image is a product of the $\Im(\phi_i)$
which are then irreducible in $\PGL_{s_i}$.
According to  \cite[3.13]{GP3}, we have then 
$(r_j,s_j)=1$ for $j=1,...,l$. Hence 
${\rm Im}(\phi)= \bH(s_1,...,s_l)$ is irreducible in $\PGL_d$.

\smallskip

\noindent $(c) \Longrightarrow (a)$: Since $\bH(s_1,...,s_l)$ is irreducible in $\PGL_d$, we 
have $d=s_1....s_l$.  The condition $(r_j,s_j)=1$ for $j=1,...,l$ implies that
 the algebra  $A(r_1,s_1,...,r_l,s_l) \otimes_{R_n} F_n$ is division \cite[th. 3]{Am}.
\end{proof}


We can now proceed with the proof of  Theorem \ref{plinear}.

\begin{proof} 

(1) If $d= s_1...s_l$ then $A(1,s_1,...,1,s_l) \otimes_{R_n} F_n$ is a division $F_n$--algebra \cite[th. 3]{Am}, 
so Lemma \ref{jojo} shows that $\bH(s_1,...,s_n)$ is irreducible in $\PGL_d$.
If $d \not = s_1...s_l$, then this algebra is not division and  
$\bH(s_1,...,s_n)$ is reducible.

\noindent (2) 
If $\bH(s_1,...,s_l)$ is $\PGL_d(k)$--conjugate
to  some $\bH(s'_1,...,s'_{l'})$, then $\bH(s_1,...,s_l)$
is isomorphic to  $\bH(s'_1,...,s'_{l'})$ as finite abelian group.
So the divisibility conditions yield $l=l'$ and $s_j=s'_j$ for $j=1,...,l$.

\medskip

The delicate points are  existence and conjugacy. Let $\bH$ be  
a finite abelian constant  irreducible  subgroup of $\PGL_d$.
Denote by  $n$ the rank of  $\bH(k)$ and by $m$ its exponent.

Let us prove first that $\bH$ is $\PGL_d(k)$--conjugate
to some $\bH(s_1,...,s_l)$.
We view $\bH(k)$  as the image of an irreducible group homomorphism
 $\psi: (\Z/m\Z)^n \to  \PGL_d(k)$. Since $k$ is algebraically closed $\psi$ is a cocycle.
The loop construction then defines an Azumaya $R_n$--algebra of degree $d$
such that $A\otimes_{R_n} F_n$ is division (i.e. the group 
$\PGL_1(A)_{F_n}$ is anisotropic).
Up to base change by a suitable element of 
$\bGL_n(\Z)$,  Theorem 4.7 of \cite{GP3} provides an element
$g \in \bGL_n(\Z)$ such that $$
g^*(A) \cong A(r_1,s_1,...,r_l,s_l)
$$
with  $(r_j,s_j)=1$ for $j=1,..,l$.

 By Lemma \ref{jo}.1,
$A(r_1,s_1,...,r_l,s_l)$ is the loop Azumaya algebra defined
by the morphism
$$
\phi: {\hat \Z}^n \to \bH(s_1)(k) \times \cdots    \times \bH(s_l)(k)
= \bH(s_1,...,s_l)(k)  \subset \PGL_d(k),$$
$$
 (e_1,e_2,...,e_{2l-1}, e_{2l}) \mapsto 
(- a_1, -b_1,..., - a_{l}, -r_l b_{l}) .
$$
which is then irreducible by the second statement of the same lemma.
Theorem \ref{k-aniso-bis} tells us that $\phi$ and 
$\psi$ are $\PGL_d(k)$--conjugate, hence  
 $\bH(k)$ is $\PGL_d(k)$--conjugate to $\bH(s_1,...,s_l)(k)={\rm Im}(\psi)$.
Since $n=\rank( \bH(s_1,...,s_l)(k))$, we have $s_1 \mid s_2... \mid s_l$.
\end{proof}

We can now go back to Azumaya algebras.

\begin{theorem}  \label{brauer3}
 Let $A$ be an anisotropic  loop Azumaya $R_n$--algebra of degree $d.$
\begin{enumerate}
 \item There exists a sequence
$s_1,...,s_l$ and an integer $r_1$ prime to $s_1$
satisfying
$$
s_1 \mid \cdots \mid s_l, \quad 2 \geq s_1, \quad
d= s_1 \cdots s_l , \quad (r_1,s_1)=1
$$
and an element $g \in \bGL_n(\Z)$ such that
$$
g^*(A) \cong A(  r_1,s_1,1, s_2, 1, s_3, \cdots , 1, s_l)
\cong A(-r_1,s_1,1, s_2, 1, s_3, \cdots , 1, s_l).
$$
Such a sequence $s_1,...,s_l$ is unique.

\item If $n=2l$,  $\pm r_1$ is unique modulo $s_1$.

\item If $n > 2l$, $g^*(A) \cong A(  1,s_1,1, s_2, 1, s_3, \cdots , 1, s_l)$.

\end{enumerate}
 
\end{theorem}

\begin{proof} (1)
 By definition, $A$ is the twist of
$M_d(k)$ by a morphism $\phi: (\Z / m \Z)^n \to \PGL_d(k)$.
Since $A \otimes_{R_n} F_n$ is division, $\phi$ is irreducible \cite[th. 3.1]{GP2}. 
Theorem \ref{plinear} shows that  there exists a unique  sequence
$s_1,...,s_l$ such that
$s_1 \mid \cdots \mid s_l, \quad 2 \geq s_1$ and ${\rm Im}(\phi)$
is $\PGL_d(k)$--conjugate to $H(s_1,...,s_l) := \bH(s_1,...,s_l)(k)$.
Without lost of generality, we can assume that
${\rm Im}(\phi) = H(s_1,...,s_l)$.

Recall that  $a_1,b_1,...,a_l,b_l$  stand for the standard generators
of $\bH(s_1,...,s_l)$.
We shall use that  $\Lambda^l\big( H(s_1,...,s_l)\big) \simeq \Z / s_1 \Z$ 
generated by $a_1 \wedge b_1 \cdots  a_l \wedge b_l$ \cite[Lemma 2.1]{RY2}, as well as  the following invariant of $\phi$ ({\it ibid}, 2.5)
$$
\delta(\phi)= \phi(e_1) \wedge \phi(e_2) \wedge \dots \wedge   \phi(e_n) 
\, \in \,  \Lambda^n\big( \bH(s_1,...,s_l)(k)\big)
$$
This invariant has the remarkable property that a homomorphism
$\phi': (\Z/m\Z)^n \to \bH(s_1,...,s_l)(k)$ is  $\GL_n(\Z)$--conjugate 
to $\phi$ if and only if $\delta(\phi)= \pm \delta(\phi')$. 

We shall prove (1) together with (2) [resp. (3)]  in the case
$n=2l$ (resp. $n>2l$).

\smallskip

\noindent{\it First case. $n=2l$:} 
The family $(\phi(e_1),..., \phi(e_n))$ generates $H(s_1,...,s_l)$,  and we consider the class
$$
[r^\sharp_1] := \phi(e_1) \wedge \phi(e_2) \wedge \cdots  \wedge \phi(e_n) \in (\Z/s_1\Z)^\times.
$$
Let $r_1$ be an inverse of $r^\sharp_1$  modulo  $s_1$.
We have $$
\phi(r_1 e_1) \wedge \phi(e_2) \cdots  \wedge \phi(e_n)= 
a_1 \wedge b_1 \cdots  a_l \wedge b_l
$$
so
 there exists $g \in \bGL_n(\Z)$ such that ({\it ibid},  2.5)
$$
\psi( r_1 e_1)= a_1, \psi(e_2)=b_2,  \cdots , \, \phi(e_{n-1})=  a_l, \,  \phi(e_n)=  b_l
$$
where $\psi=\phi \circ g$.
In terms of algebras, this means that
$$
g^*(A) \simeq A(r_1,s_1,1, s_2, 1, s_3, \cdots , 1, s_l).
$$
Let us first prove the uniqueness assertions 
The uniqueness of $(s_1,..,s_l)$ follows from Theorem \ref{plinear}, 
hence (1) is proven. For (2),   
we are given then $r_1' \in \Z$ coprime to $s_1,$ and an element
$h  \in \bGL_n(\Z)$ such that
$$
h^*\big(A(r_1,s_1,1, s_2, 1, s_3, \cdots , 1, s_l)\big)
 \simeq A(r'_1,s_1,1, s_2, 1, s_3, \cdots , 1, s_l).
$$
Denote by $\psi': (\Z/m\Z)^n \to H(s_1,..,s_n)$ the group homomorphism
defined by $\psi'(e_1)=r'_1 a_1$, $\psi'(e_2)= b_1, \cdots,
\psi(e_n)=b_{2l}$. Since the ($F_n$--anisotropic) loop algebras  attached to
$h^*\psi$ and $\psi'$ are isomorphic,  
 Theorem \ref{k-aniso-bis} provides an element $u \in \PGL_d(k)$ such that 
$$
\psi \circ h = u \circ \psi'. 
$$
Since $H(s_1,...,s_l) = {\rm Im}(\psi)=  {\rm Im}(\psi')$, it follows that
$u \in \bN_{\PGL_d(k)}\big(H(s_1,...s_l)\big)$. 
But the map $u: \bH(s_1,...,s_l)(k) \to \bH(s_1,...s_l)(k)$ 
preserves  the symplectic pairing $\bH(s_1,...,s_l)(k) \times \bH(s_1,...,s_l)(k) 
\to k^\times$ arising by taking the commutator  of lifts in $\GL_d(k)$. 
It follows that $\Lambda^n(u)=id$ ({\it ibid}, 2.3.a) hence
$$
\delta( \psi)= \pm \delta( \psi \circ h) = \pm \delta( u \circ \psi') 
=\pm \delta( u \circ \psi').
$$Thus $r_1= \pm r'_1 \in \Z/s_1\Z$ as prescribed in (2) .

\smallskip

\noindent{\it Second case. $n > 2l$:} 
For $i=2l+1,...,n$ we set $c_i= 0 \in H(s_1,..,s_l).$ 
Both families $\big( \phi(e_1),  \dots \phi(e_n)\big)$
and $(r_1 \, a_1 , b_1 \cdots  , a_l,  b_l ,  c_{2l+1},  \cdots , c_{n})$
generate  $H(s_1,...,s_l)$ and satisfy
$$
\phi(e_1) \wedge \phi(e_2) \cdots  \wedge \phi(e_n) 
= (r_1 \, a_1) \wedge b_1 \cdots  a_l \wedge b_l \wedge c_{2l+1} \wedge  \cdots \wedge c_{n} \in  \Lambda^{n}\big(H(s_1,...,s_l)\big)=0.
$$
The same fact \cite[Lemma 2.5]{RY2} shows that
 there exists $g \in \bGL_n(\Z)$ such that 
$(g^*\phi)(e_1)= r_1 a_1$, $(g^*\phi)(e_2)= b_1$, 
$(g^* \phi)(e_{2i-1})= a_i$ 
and $(g^* \phi)(e_{2i})= b_i$  for
$i = 2,..,l$ and $(g^* \phi)(e_i)= c_{i}$ for $i=2l+1,...,n$. Therefore  the preceding case with  $2l$ variables
yields  the existence and the uniqueness of the $s_i$'s. 
It remains to prove (3), namely that we can assume that $r_1=1$.
But this follows along the same lines of the argument given above since 
$(r_1 \, a_1) \wedge b_1 \cdots  a_l \wedge b_l \wedge c_{2l+1}  \cdots \wedge c_{n}= ( a_1) \wedge b_1 \cdots  a_l \wedge b_l \wedge c_{2l+1}  \cdots \wedge c_{n}\in  \Lambda^{n}\big(H(s_1,...,s_l)\big)$.
\end{proof}

\subsection{Loop algebras of inner type $A$}

To the Azumaya $R_n$--algebra $A( r_1, s_1,...,  r_l,  s_{l})$ we can attach (using the commutator $[x,y] = xy -yx$) a Lie algebra over $R_n.$ We denote by  
$L( r_1, s_1,...,  r_l,  s_{l})$ the derived algebra
of this Lie algebra. It is a twisted form of $\sl_d(R_n)$ where $d = s_1\dots s_l.$

\begin{corollary} Let $d$ be a positive integer.
Let $L$ be a nullity $n$ loop algebra of inner (absolute) type $A_{d-1}$. 

\begin{enumerate}
 \item If $L$ is not split, it  is $k$--isomorphic to $L( r_1,s_1,1,s_2, \cdots , 1, s_l)$ 
where
$$
s_1 \mid \cdots \mid s_l, \quad 2 \geq s_1, \quad
d= s_1 \cdots s_l , \quad (r_1,s_1)=1 \hbox{\quad and  \quad}
l \leq \Bigl[\frac{n}{2}\Bigr]
$$
and such a sequence $s_1,...,s_l$ is unique.

\item  If $n=2l$,  $r_1$ is unique modulo $s_1$
and up to the sign.
\item If $n>2l$, $L$ is $k$--isomorphic to $L( 1,s_1,1,s_2, \cdots , 1, s_l)$ 
\end{enumerate}
 
\end{corollary}

\begin{proof} The classification in question is given by considering the image of the natural map
$$
H^1_{loop}(R_n, \PGL_d) \to \bGL_n(\Z) \backslash H^1(R_n, \bAut(\PGL_d))  
$$
The image can be identified with $(\Z/2\Z \times \bGL_n(\Z))\backslash H^1_{loop}(R_n, \PGL_d)$
where $\Z/2\Z$ acts by the opposite Azumaya algebra construction.
Corollary \ref{WT4} reduces the problem to  the ``anisotropic case''.
Theorem \ref{brauer3} determines the set 
$\bGL_n(\Z) \backslash H^1_{loop}(R_n, \PGL_d),$ and as it turns out the action of
$\Z/2\Z$ is trivial. Therefore the desired classification is
also provided by $\bGL_n(\Z) \backslash H^1_{loop}(R_n, \PGL_d) $
and we can now appeal to Theorem \ref{brauer3} to obtain the Corollary.
\end{proof}
\section{Invariants attached to EALAs and multiloop algebras}

Both the finite dimensional simple Lie algebras over $\ol{k}$  (nullity $0$) and their affine counterparts (nullity $1$) have Coxeter-Dynkin diagrams attached to them that contain a considerable amount of information about the algebras themselves. It  has been a long dream to find a meaningful way of attaching some kind of diagram to multiloop, or at least EALAs of arbitrary nullity (perhaps with as many nodes as the nullity). Our work shows that this can indeed be done and in a very natural way.

Let us recall (see the Introduction for more details) the multiloop algebras  based on a finite dimensional simple Lie algebra $\gg$ over an algebraically closed field $k$ of characteristic $0.$   Consider an
$n$--tuple $\bs = (\sigma  _1,\dots,\sigma  _n)$ of commuting
elements of ${\rm Aut}(\gg)$ satisfying $\sigma  ^{m}_i = 1.$  For
each $n$--tuple $(i_1,\dots ,i_n)\in \Z^n$ we consider the
simultaneous eigenspace
$
\gg_{i_1 \dots i_n} =\{x\in \gg:\sigma  _j(x) = \xi  ^{i_j}_{m} x \,
\, \text{\rm for all} \,\, 1\le j\le n\}.
$
The  multiloop algebra $L(\gg,\bs)$ corresponding to $\bs$ is defined by
$$
L(\gg,\bs) = \us{(i_1,\dots ,i_n)\in \Z^n}\oplus\,
\gg_{i_1\dots i_n} \otimes t^{\frac{i_1}{m}} \dots
t^{\frac{i_n}{m}}_n \subset \gg\otimes _k R_{n,m}\subset \gg \otimes_k R_{\infty}
$$
Recall that $L(A,\bs),$ which does not depend on the choice of common period $m,$ is not only a $k$--Lie algebra (in general infinite-dimensional), but also naturally an $R$--algebra. It is when $L(\gg,\bs)$ is viewed as an $R$--algebra that Galois cohomology and the theory of torsors enter into the picture. Indeed a rather simple calculation shows that 
$$ L(\gg,\bs) \otimes _{R_n} R_{n.m} \simeq \gg \otimes _k R_{n,m} \simeq (\gg \otimes _k
R_n)\otimes _{R_n} R_{n,m}.$$
 Thus $L(\gg,\bs)$ corresponds to a torsor $\bE_{\bs}$ over $\Spec(R)$ under $\bAut(\gg).$   It is, however, the $k$--Lie algebra structure that is of interest in infinite-dimensional Lie theory and Physics.
 
 Let $\bG$ be the $k$--Chevalley group of adjoint type corresponding to $\gg.$ Since $\bAut(\bG)$ and $\bAut({\gg})$ coincide we can also consider the twisted $R_n$--group ${_\bE}\bG_{R_n}.$ By functoriality and the definition of Lie algebra of a group functor in terms of dual numbers we see that  
 ${\mathfrak Lie}({_\bE}\bG_{R_n}) = L(\gg,\bs).$ By the aciclicity Theorem to ${_\bE}\bG_{R_n}$ we can attach a Witt-Tits index, and this is the ``diagram" that we attach to $L(\gg,\bs)$ as {\it as Lie algebra over} $k.$ Note that by Corollary \ref{WT3equi} this is well defined. The diagram carries the information about the absolute and relative type of $L(\gg,\bs).$\footnote{The relative type as an invariant of $L(\gg,\bs)$ is defined in in \S3 of \cite{ABP3} by means of the central closure.  If $C$ is the centroid of $L(\gg,\bs)$ and $\widetilde C$ denotes its field of quotients, then $L(\gg,\bs) \otimes_C \widetilde{C}$ is a finite dimensional central simple algebra over $\widetilde{C}$. As such it has an absolute and relative type. This construction applies to an arbitrary prime perfect Lie algebra which is finitely generated over its centroid.}  
 
  To reassure ourselves that this is the correct point of view we can look at the nullity one case. EALAs of nullity one are the same than the affine Kac-Moody Lie algebras. If one uses Tits methods to compute simple adjoint algebraic groups over the field $k((t))$ one obtains precisely the diagrams of the affine algebras.

\section{Appendix 1: Pseudo-parabolic subgroup schemes}

We extend the definition of pseudo-parabolic subgroups\footnote{In \cite{CoGP}
the groups that we are about to define are called {\it limit subgroups}. We have decided, since we are only dealing with analogues of pseudo-parabolic subgroups over fields, to abide by this terminology. This material can in part be recovered from their work, but we have decided to include it in the form that we needed for the sake of completeness. } 
of affine algebraic groups  (Borel-Tits
\cite{BoT2}, see also \cite[\S 13.4]{Sp}) to the case of
 a   group  scheme $\gG$ which is of finite type and affine over a fixed base scheme $\gX$.
 We begin by establishing some notation.
 
 \medskip

 We will denote by $\GG_{m, \gX}$ and $\GG_{a, \gX}$ the multiplicative and additive $\gX$-groups. The underlying 
schemes of these groups  will be denoted by $\Bbb{A}^\times_{\gX}$ and $\Bbb{A}_{\gX}$ respectively.
 After applying a base change $\gX \to \gX'$ we obtain corresponding $\gX'$-groups and schemes that we denote by
 $\GG_{m, \gX'}$, $\GG_{a, \gX'}$, $\Bbb{A}^\times_{\gX}$ and $\Bbb{A}_{\gX'}.$ 
 
 The structure morphism of the $\gX'$--scheme $\Bbb{A}^\times_{\gX'}$ gives by functoriality a group homomorphism
 \begin{equation}\label{nat}
\eta_{\gX'} : \gG(\gX') \to \gG(\Bbb{A}^\times_{\gX'})
  \end{equation}

Let  $\lambda: \GG_{m,\gX} \to \gG$ be a cocharacter. By applying $\lambda_{\gX'}$ to 
the identity map $\id_{\Bbb{A}^\times_{\gX'}} \in \GG_{m, \gX'}(\Bbb{A}^\times_{\gX'})$ we obtain 
an element of $\lambda_{\gX'}(\id_{\Bbb{A}^\times_{\gX'}}) \in \gG(\Bbb{A}^\times_{\gX'}).$ 

We have a natural group homomorphism $\gG(\Bbb{A}_{\gX'}) \to \gG(\Bbb{A}^\times_{\gX'}).$  
Given an element $x' \in \gG(\Bbb{A}^\times_{\gX'})$ we will write $x' \in \gG(\Bbb{A}_{\gX'})$ if $x'$ is in the image of this map.
\medskip

After these preliminary definitions we are ready to define the three group functors that are relevant 
to the definition of pseudo-parabolic subgroups.

Let $\bZ_\gG(\lambda)$ denote the centralizer of $\lambda$. Recall that this is the $\gX$-group functor that to
a scheme $\gX'$ over $\gX$ attaches the group
\begin{equation}\label{centralizer}
\bZ_\gG(\lambda)(\gX') = \{x' \in \gG(\gX') : x''  \, \text{\rm commutes with} \; \lambda\big(\GG_{m, \gX}(\gX'')\big) \subset \gG(\gX'')\}
  \end{equation}  
where $\gX''$ is a scheme over $\gX'$ and $x''$ denotes the image of $x'$ under the natural group homomorphism 
$\gG(\gX') \to \gG(\gX'').$

We consider the two following $\gX$--functors
$$
\bP(\lambda)(\gX')= \Bigl\{   g \in \gG(\gX') \, \mid \,
  \lambda_{\gX'}(\id_{\Bbb{A}^\times_{\gX'}})  \, \eta_{\gX'}(g) \,\big( {\lambda_{\gX'}(id_{\Bbb{A}^\times_{\gX'}}) }\big)^{-1}
\in \gG(\Bbb{A}_{\gX'})  \Bigr\}
$$
and
$$
\bU(\lambda)(\gX')=  \Bigl\{   g \in \gG(\gX')
\, \mid \,  \lambda_{\gX'}(\id_{\Bbb{A}^\times_{\gX'}})  \, \eta_{\gX'}(g) \,\big( {\lambda_{\gX'}(id_{\Bbb{A}^\times_{\gX'}}) }\big)^{-1}
\in \ker\bigl(  \gG(\Bbb {A}_{\gX'}) \to \gG(\gX') \bigr) \,  \Bigr\}
$$
for every $\gX$-scheme $\gX'.$  The centralizer
$\bZ_\gG(\lambda)$ is an $\gX$--subgroup functor of
$\bP(\lambda)$ which normalizes $\bU(\lambda)$.

\medskip

We look at the previous definitions in the case when $\gX = \Spec(R)$ and $\gX' = \Spec(R')$ are both affine.\footnote{As customary we write $ \GG_{m, R}$ instead of $ \GG_{m, \gX}$...}  We have $\Bbb{A}_{R'} = \Spec(R'[x])$ and $\Bbb{A}^\times_{R'} = \Spec(R'[x^{\pm 1}]).$ Then $x \in R'[x^{\pm 1}]^\times = \GG_{m, R'} (R'[x^{\pm 1}]) = \GG_{m,R'} (\Bbb{A}^\times_{R'}),$ and by applying our cocharacter we obtain an element $\lambda_{R'}(x) \in \gG(R').$ Under Yoneda's correspondence $\GG_{m, R'} (R'[x^{\pm 1}]) \simeq \Hom_{R'}(R'[x^{\pm 1}], R'[x^{\pm 1}])$ our element $x$ corresponds to the identity map, namely to the element $\id_{\Bbb{A}^\times_{R'}} \in \GG_{m, R'}(\Bbb{A}^\times_{R'})$ if we rewrite our ring theoretical objects in terms of schemes.  We thus have
$$
\bP(\lambda)(R')= \Bigl\{   g \in \gG(R') \, \mid \,
  \lambda_{R'}(x)  \, \eta_{R'}(g) \,\big( \lambda_{R'}(x)\big)^{-1}
\in \gG(R'[x])  \Bigr\}
$$ 
and
$$
\bU(\lambda)(R')=  \Bigl\{   g \in \gG(R')
\, \mid \,  \lambda_{R'}(x)  \, \eta_{R'}(g) \,\big( \lambda_{R'}(x)\big)^{-1}
\in \ker\bigl(  \gG(R'[x]) \to \gG(R') \bigr) \,  \Bigr\}
$$
where $\eta_{R'}(g)$ is the natural image of $g \in \gG(R')$ in $\gG(R'[x^{\pm 1}]),$ and the group homomorphism $\gG(R'[x]) \to \gG(R')$ comes from the ring homomorphism $R'[x] \to R'$ that maps $x$ to  $0.$

\subsection{The case of $\bGL_{n,\Z}.$}\label{linear}

Assume $S=\;\Spec (\Z)$ and let $\bG$ denote the general linear
group $\bGL_{n,\Z}$ over $\Z.$  We let $\bT$ denote the standard
maximal torus of $\bG.$  Let $\lambda  :\bG_{m,\Z} \to \bT
\hookrightarrow \bG$ be a cocharacter of $\bG$ that factors through $\bT.$
We review the structure of the groups $\bZ_{\gG}(\lambda  ),\,
\bP(\lambda  )$ and $\bU(\lambda  ).$

After replacing $\lambda  $ by $\text{int} (\theta  )\circ \lambda  $ for
some suitable $\theta  \in \bG(\Z)$ we may assume that there exists
(unique) integers $1\le \ell_1<\ell_2<\dots <\ell_j\le n$ and
$e_1,\dots,e_n$ such that
$$
\begin{aligned}
&e_i = e_j \q \text{\rm if} \q \ell_k\le i,\; j<\ell_{k+1}
\q\text{\rm for some}\q k\\
&e_{\ell_k + 1} > e_{\ell_k} \q \text{\rm for all} \q 1\le k\le j
\end{aligned}
$$
so that the functor of
points of our map $\lambda  : \bG_{m,\Z} \to \bT$ is  given by
$$
\lambda  _R: \bG_{m,\Z} (R) \lra \bT(R)
$$
$$
r\mapsto \begin{pmatrix}
r^{e_1} &&&0\\
&&\ddots\\
&0\\
&&&r^{e_n}\end{pmatrix} \leqno{(*)}
$$
for any (commutative) ring $R$ and for all $r\in \bG_{m,\Z}(R) = R^\times.$

At the level of coordinate rings if $\bG_{m,\Z} = \,\Spec\big(\Z[x^{\pm
1}]\big)$ and $\bT=\,\Spec\big(\Z[t^{\pm 1}_1,\dots,t^{\pm 1}_n]\big),$ then
$\lambda  $ corresponds (under Yoneda) to the ring homomorphism
$$
\lambda  ^* : \Z [t^{\pm 1}_1,\dots,t^{\pm 1}_n] \lra \Z
[x^{\pm1}]
$$
given by
$$
\lambda  ^* :t_i\mapsto x^{e_i}.
$$

From this it follows that $\bZ_{\bG}(\lambda  )(R)$ consists of block
diagonal matrices inside $\bGL_n(R)$ of size
$\ell_1,\dots,\ell_j.$  Note that one ``cannot see'' this by
looking at the centralizer of $\lambda  \big(\bG_{m,\Z}(R)\big)$
inside $\bG(R).$  This is clear, for example, if $n=2, \, R=\Z,$ $j = 1,$
$\ell_1 = 1$ and $1= e_1<e_2=3.$  The easiest way to eliminate
``naive'' contralizers in $\bZ_{\bG}(\lambda  )(R)$ is to look at
their image in $\bG\big(R[x^{\pm 1}]\big).$  In fact

\begin{lemma}\label{centralizer2} With the above notation we have 
$$\bZ_{\bG}(\lambda  )(R) =\bigl\{ \, A\in \bG(R)\subset \bG(R[x^{\pm 1}]): A\;
\text{\it commutes with} \; \lambda_R  \big(\bG_{m, \Z}(R[x^{\pm 1}])\big) \, \bigr\}.$$
\end{lemma}

\begin{proof} The inclusion $\subset$ follows from the definition of
$\bZ_{\bG}(\lambda  ).$  Conversely suppose that $A\in \bG(R)$ is not
an element of $\bZ_{\bG}(\lambda  )(R).$  Then there exists a ring
homomorphism $R\to S$ and an element $s\in S^\times$ such that the
image of $A$ in $\bG(S)$ does not commute with the diagonal matrix
$$
\lambda  _S(s) = \begin{pmatrix} s^{e_1}\\ &\ddots\\
&&s^{e_n}\end{pmatrix}\,.
$$
But then $A,$ viewed now as an element of $\gG\big(R[x^{\pm 1}]\big)$ cannot
commute with
$$
\lambda  _{R[x^{\pm 1}]}(x) = \begin{pmatrix} x^{e_1}\\ &\ddots\\
&&x^{e_n}\end{pmatrix}\,.
$$
For if it did, we would reach a contradiction by functoriality
considerations applied to the (natural) ring homomorphism $R[x^{\pm
1}] \to S$ that maps $x$ to $s.$
\end{proof}

Returning to our example we see that there are two extreme cases for
$\bZ_{\bG}(\lambda  ).$ If $j=1$ and $\ell_1=n$ then $\bZ_{\bG}(\lambda  )=  \bG.$  At the other extreme if $j=n$ then $\lambda  $ is regular and
$\bZ_{\bG}(\lambda  ) = \bT.$
In all cases we see from the diagonal block description that
$\bZ_{\bG}(\lambda  )$ is a closed subgroup of $\bG$ (in particular
affine).

We now turn our attention to $\bP(\lambda  )$ and $\bU(\lambda  ).$  By
using $(*)$ one immediately sees that
$$
\bP(\lambda  )(R) =\{A= (a_{ij})\in \gG(R): a_{ji} = 0\;\text{if} \;
e_i>e_j\}.
$$
Thus the $\bP(\lambda  )$ are the standard parabolic subgroups of
$\bG.$

\begin{example}\label{GL5} We illustrate with the case $n=5$ with $j=2$ and
$\ell_1=2,\; \ell_2 = 5, \; e_1=1,\; e_2=3.$  Then $A= (a_{ij})
\in \GL_5(R)$
 is of the form
$$
 A= \begin{pmatrix}
\begin{tabular}{c|c}
$\times$ &$+$\\ \hline
$-$ &$\times$\end{tabular}
\end{pmatrix}\,.
$$
We have two blocks, the top left of size $2$ and the bottom right of size $3.$
Given $A= (a_{ij}) \in \GL_5(R) \subset \GL_5(R[x^{\pm 1}])$ define $P$ by
$$
\begin{pmatrix}
x^1 &&&0\\
&x\\
&&x^3\\
&0 &&x^3\\
&&&&x^3\end{pmatrix} \q A\q\begin{pmatrix} x^{-1} &&&0\\
&x^{-1}\\
&&x^{-3}\\
&0&&x^{-3}\\
&&&&x^{-3}\end{pmatrix} = P
$$
that is
$$
\lambda  (x) A\lambda  (x)^{-1}=P
$$
where $P = (p_{ij})$ and $p_{ij} = \sum p_{ijk}x^k \in R[x^{\pm 1}].$
 To belong to $\bP(\lambda  )$ the element $A$ must be such that $p
_{ijk}=0$ for $k<0.$  This forces all entries in the $3\times 2$
block marked with a $-$ to vanish.  For the elements in
$\bZ_{\bG}(\lambda  )(R)$ both blocks $-$ and $+$ must vanish.

It is easy to determine that if $A\in \bP(\lambda  )$ the matrix $P$
is such that the $p_{ij} = a_{ij}\in R$ whenever $1\le i,\; j\le 2$
or $2\le i,\; j\le 5.$  If, on the other hand, $i\le 2< j$ then
$p_{ij} = a_{ij} x^2.$
\end{example}

This makes the meaning of $\bU(\lambda  )$ quite clear in general.  If
$\lambda  (x)A\lambda  (x)^{-1}\in \bG(R[x])$ is mapped to the
identity element of $\bG(R)$ under the map $R[x] \to R$ which sends
$x\mapsto 0$ then $a_{ii} =1$ and $a_{ij}=0$ if $e_i>e_j.$  That is
$$
\bU(\lambda  )(R) =\{A=(a_{ij})\in \bP(\lambda  ): a_{ii} = 1\;\text{and}\; a_{ij} =0\;\text{if} \; e_i>e_j\}.
$$
In particular $\bU(\lambda  )$ is an unipotent subgroup  of $\bP(\lambda).$

\subsection{The general case}

\begin{lemma}\label{pseudo} Assume that there exists  locally for the
$fpqc$-topology
 a closed embedding of $\gG$ in a linear  group scheme.\footnote{This condition is satisfied  if $\gX$ is locally noetherian of dimension $\leq 1$
\cite[\S 1.4]{BT2}, and also for reductive $\gX$--group schemes.} Then

\begin{enumerate}
\item the  $\gX$--functor $\bU(\lambda)$ (resp. $\bP(\lambda)$, resp.
 $\bZ_\gG(\lambda)$)
 is representable by a closed  subgroup scheme of  $\gG$ which is affine over $\gX.$
\item The geometric fibers  of $\bU(\lambda)$ are unipotent.
\item $\bP(\lambda)= \bU(\lambda) \rtimes \bZ_\gG(\lambda)$.
\item  $\bZ_\gG(\lambda)= \bP(\lambda) \times_{\gG} \bP(-\lambda)$.

\item  $\bP(\lambda) = \bN_\gG\bigl( \bP(\lambda) \bigr)$.
\end{enumerate}
\end{lemma}

\begin{proof}
{\it The case of $\bGL_{n,S}$:} The question is local with respect to
 the {\it fpqc} topology, so we can assume then that $\gX$ is the spectrum of a local ring $R.$  Since all
maximal split\footnote{Trivial, in the terminology of \cite{SGA3}.}  tori of the $R$--group $\bGL_{n,R}$ are conjugate under
$\bGL_n(R)$ \cite[XXVI.6.16]{SGA3}, we can assume
 that $\lambda :  \GG_{m,R} \to \bT_R  <  \bGL_{n,R} $ where $\bT$ is
the standard maximal torus of $\bGL_{n, \Z}$.
Since $\Hom_\Z(\GG_m, \bT ) \simeq \Hom_R(\GG_{m,R} \bT_R )$,
we can reduce our problem to the case when $R = \Z,$  which has been already done in Example \ref{linear}.

\smallskip

\noindent {\it General case:}

By $fpqc$-descent, we can assume that  $\gX$ is the spectrum of a ring $R,$ and that we are given a
 $R$--group scheme homomorphism
$\rho: \gG \to \gG'=\GL_{n,R}$ which is a closed immersion.

\smallskip

\noindent
(1) Denote by  $\bP'(\lambda)$ and
  $\bU'(\lambda)$ the $R$--subfunctors
of $\gG'$ attached to the cocharacter $\rho \circ \lambda$.
The identities $\bP(\lambda) = \bP'(\lambda) \times_{\gG'} \gG$ and
 $\bU(\lambda) = \bU'(\lambda) \times_{\gG'} \gG$ can be established by reducing  to the case
of $\gG'= \bGL_{n,R}.$ This reduces the representability questions to the case when  $\gG= \bGL_{n,R}$ considered above.

\smallskip

\noindent (2) This follows as well from the $\bGL_{n,R}$ case.

\smallskip

\noindent (3) We know that the result  holds for $\gG'$. Let  $R'$ be a ring extension of $R$
and let $g \in \gG(R')$. Then $g = u z$ with
$u \in \bU'(\lambda)(R')$ and $z \in \bP'(\lambda)(R')$.
We have
$$
  \lambda(x) \, g \,\lambda(x)^{-1}   \lambda(x) \, u  \,\lambda(x)^{-1} \, z \in \gG(\Bbb {A}_{R'}).
$$
By specializing at $0$, we get that $z \in \gG(R')$.
Thus $g \in \bZ_\gG(\lambda)(R')$ and $u \in   \bU(\lambda)(R')$.
We conclude that
 $\bP(\lambda)= \bU(\lambda) \rtimes \bZ_\gG(\lambda)$.

\smallskip

\noindent (4) and (5)  follows from the $\bGL_{n,R}$ case.
\end{proof}

\begin{definition} An $\gX$--subgroup of $\gG$ is {\it pseudo-parabolic} if it is of the form
$\bP(\lambda)$ for some  $\gX$--group homomorphism  $\GG_{m,X} \to \gG$.
\end{definition}

\begin{proposition}\label{pseudopseudo} Let $\gG$ be a reductive group scheme over $\gX$.

\begin{enumerate}

\item  Let $\lambda: \GG_{m, \gX} \to \gG$ be a cocharacter.
Then $\bP(\lambda)$ is a parabolic subgroup scheme of $\gG$ and
 $\bZ_\gG(\lambda)$ is a Levi subgroup of the $\gX$--group scheme  $\bP(\lambda)$.

\item Assume that $\gX$ is semi-local,  connected and non-empty. Then the
pseudo parabolic subgroup schemes of $\gG$  coincide with the
parabolic subgroup schemes of $\gG$.

\end{enumerate}
\end{proposition}

We shall use that this fact is known for reductive groups
over fields \cite[\S 15.1]{Sp}.

\begin{proof} We can assume that $\gX = \Spec(R)$ is affine.

\smallskip

\noindent (1) The geometric fibers of $\bP(\lambda)$ are parabolic subgroups.
By definition \cite[\S XXVI.1]{SGA3},
it remains to show that $\bP(\lambda)$ is smooth.
The question is then local with respect to
 the {\it fpqc} topology, so that we can assume that  $R$ is local and
that $\gG$ is split.
 By Demazure's  theorem \cite[XXIII.4]{SGA3}, we can assume that
 $\gG$ arises by base change from a (unique) split Chevalley group $\gG_0$ over $\Z$.
 
 We now reason along similar lines than the ones used in studying the $\bGL_{n,\Z}$ case above.
Let $\gT  \subset \gG_0$ be a maximal split torus.
Since all maximal split tori of $\gG$ are conjugate under
$\gG(R)$, we can assume that our cocharacter $\lambda$ factors  through $ \gT_R$.
Since $\Hom_\Z(\GG_{m,\Z}, \gT ) \cong \Hom_R(\GG_{m, R}, \gT_R )$,
the problem again reduces to the case  when $R = \Z$ and of $\gG = \gG_0,$ and 
$\lambda : \gG_{m,\Z} \to \gT$.  By the field case,
 the morphism  $\bP(\lambda) \to \Spec( \Z)$ is  equidimensional.
Since $\Z$ is a normal ring and the geometric fibers are smooth,
 we can conclude by   \cite[prop. II.2.3]{SGA1}
that $\bP(\lambda)$ is smooth and is a parabolic
subgroup scheme of the $\Z$--group $\gG_0$.

The geometric fibers of
 $\bP(\lambda) \times_S \bP(-\lambda)$ are Levi subgroups.
By applying \cite[th. XXVI.4.3.2]{SGA3}, we get that
  $\bZ_\gG(\lambda)= {\bP(\lambda) \times_\gG \bP(-\lambda)}$
is a Levi $S$-subgroup scheme of $\bP(\lambda)$.

\smallskip

\noindent (2) Using the theory of relative root systems \cite[\S XXVI.7]{SGA3}, 
the proof is the same as in the field case.
\end{proof}

\section{Appendix 2: Global automorphisms of $\bG$--torsors  over
 the projective line}

In this appendix there is no assumption on the characteristic of the  base field $k$.
Let $\bG$ be a linear algebraic $k$--group such that $\bG^0$ is reductive.
One  way to state Grothendieck-Harder's theorem is
to say that the natural map 
$$
\Hom_{gp} (\GG_m, \bG) / \bG(k) \to H^1_{Zar}( {\IP}^1, \bG)
$$
which maps a cocharacter $\lambda : \GG_m \to \bG$ to the $\bG$-torsor
$\bE_\lambda:=(-\lambda)_*\big( {\cal O}(-1)\big)$ over $\IP^1_{k}$
where  ${\cal O}(-1)$ stands for the Hopf
bundle ${\bf A}^2_{k} \setminus \{ 0 \} \to {\IP}^1_{k},$ is bijective.\footnote{This is not the usual  way to state the theorem
(see \cite[II.2.2.1]{Gi0}),
but it is easy to derive the formulation that we are using.}

We fix now a cocharacter $\lambda : \GG_m \to \bG$.
We are interested in the twisted ${\IP}^1_k$--group scheme $\bE_\lambda(\bG) = {\underline {\rm Isom}}_{\bG}(\bE_\lambda, \bE_\lambda),$ as well as the abstract
 group $\bE_\lambda(\bG)( {\IP}^1_k)$.
This group is the group of global automorphisms of the $\bG$--torsor $\bE_\lambda$ over $ {\IP}^1_k$.
It has a concrete description.

\begin{lemma}\label{hopf}
$\bE_\lambda(\bG)( {\IP}^1_k) =  \bG(k[t]) \cap  \lambda(t)  \bG(k[t^{-1}]) \lambda(t^{-1})$.
\end{lemma}

\begin{proof} We recover ${\IP}^1_k$ by two
affine lines $\bU_0= \Spec(k[t])$
and $\bU_1= \Spec(k[t^{-1}])$.  The Hopf bundle is isomorphic
to the twist of $\GG_m$ by the cocycle $z \in  Z^1( \bU_0 \sqcup \bU_1 / {\IP}^1_k, \GG_m)$
where $z_{0,0}=1$, $z_{0,1}=t^{-1}$, $z_{1,0}=t$,  $z_{1,1}=1$.
Then $\lambda(z) \in  Z^1( \bU_0 \sqcup \bU_1 / {\IP}^1_k, \GG_m)$ is the cocycle of $\bE_\lambda$.
Hence

\begin{eqnarray} \nonumber
\bE_\lambda(\bG)( {\IP}^1_k) &=& \Bigl\{ (g_0, g_1) \in \bG(\bU_0) \times \bG(\bU_1)
\, \mid \, \lambda^{-1}( z_{0,1}) . g_1 = g_0   \, \Bigr\} \\ \nonumber
&=&      \bG(k[t]) \cap  \lambda(t)  \bG(k[t^{-1}]) \lambda(t^{-1}) .
\end{eqnarray}

\end{proof}

In the split connected  case, this group has been computed by
Ramanathan  \cite[prop. 5.2]{Ra} and by the first author in the split case
(see proposition II.2.2.2 of \cite{Gi0}). We provide here the general
 case by computing the Weil restriction
$$
\bH_\lambda = \prod\limits_{ {\IP}^1_k/k}  \bE_\lambda(\bG),
$$
which is known to be a representable by an algebraic  affine $k$--group.
Let $\bP(\lambda)=\bU(\lambda) \rtimes \bZ_\bG(\lambda) \subset \bG$ be
 the parabolic subgroup attached to $\lambda$ (lemma \ref{pseudo}).

Denote by $\bZ(\lambda)$ the center of $\bZ_{\bG}(\lambda)$. Then
$\lambda$ factors through  $\bZ(\lambda)$ and this allows us to define the  $\bZ(\lambda)$--torsor
$\bS_\lambda:=(-\lambda)_*\big( {\cal O}(-1)\big)$ over ${\P}^1_{k}$.
We can twist the morphism $ \bZ_{\bG}(\lambda) \to \bG$ by $\bS_\lambda$,
 so we get  a morphism $\bZ_{\bG}(\lambda) \times_k \, {\IP}^1_k \to \bE_\lambda(\bG)$ and then a morphism
$\bZ_{\bG}(\lambda) \to   \bH_\lambda$.

\begin{proposition} \label{global}
 The homomorphisms of $k$--groups
$$
 \Bigl( \prod\limits_{ {\IP}^1_k/k} \,  \bS_\lambda(\bU(\lambda)) \Bigr)
\rtimes \bZ_{\bG}(\lambda) \enskip  \to \enskip
 \prod\limits_{ {\IP}^1_k/k} \,  \bS_\lambda \bigl(\bP(\lambda) \bigr)
\enskip \to \enskip  \bH_\lambda.
$$
are isomorphisms. Furthermore,
 $\prod\limits_{ {\IP}^1_k/k} \,  \bS_\lambda(\bU(\lambda))$
is a unipotent $k$--group.
\end{proposition}

\begin{proof}  Write $\bP$, $\bU$ for $\bP(\lambda)$,  $\bU(\lambda)$.
$$
 \prod\limits_{ {\IP}^1_k/k} \,  \bS_\lambda(\bP)
=  \prod\limits_{ {\IP}^1_k/k} \,  \bS_\lambda( \bU)   \, \rtimes
  \prod\limits_{ {\IP}^1_k/k} \,  \bZ_{\bG}(\lambda)
=   \prod\limits_{ {\IP}^1_k/k} \, \bS_\lambda (\bU) \,  \,  \rtimes \, \bZ_{\bG}(\lambda)
$$
so it remains to show that  $  \prod\limits_{ {\IP}^1_k/k} \,   \bS_\lambda(\bP)   \simlgr   \bH_\lambda$.
Consider a faithful representation $\rho:  \bG \to \bG'=\bGL_n$.
Denote by $\bP'$ the parabolic subgroup of $\bG'$  attached to $\lambda$.
We have $\bP =  \bG \times_{\bG'} \bP'$,
hence  $\bS_\lambda(\bP) =  \bE_\lambda(\bG) \times_{ \bE_\lambda(\bG')} \bS_\lambda(\bP')$.
It follows that
$$
\prod\limits_{ {\IP}^1_k/k}\bS_\lambda(\bP) \,
= \prod\limits_{ {\IP}^1_k/k} \bE_\lambda(\bG) \,
 \times_{\prod\limits_{ {\IP}^1_k/k} \bE_\lambda(\bG')  }
\prod\limits_{ {\IP}^1_k/k}\bS_\lambda(\bP')
$$
as can be seen by reducing to the case of $\bGL_n$ already done in Example \ref{linear}.
This case also shows that
 $\prod\limits_{ {\IP}^1_k/k} \,  \bS_\lambda(\bU(\lambda))$
is a unipotent $k$--group.
\end{proof}


\begin{thebibliography}{99}


\bibitem[A]{A} P. Abramenko, {\it Group actions on twin buildings},
  Bull. Belg. Math. Soc. Simon Stevin  {\bf 3}  (1996),   391--406.

\bibitem[Am]{Am} S.A. Amitsur, {\it On central division algebras},
  Israel J. Math.  {\bf 12}  (1972), 408--420.

\bibitem[AABFP]{AABGP} S. Azem, B. Allison, S. Berman, Y. Gao and A.
Pianzola, {\it Extended affine Lie algebras and their root systems},
Mem. Amer. Math. Soc. {\bf 603} (1997)  128 pp.

\bibitem[ABP2]{ABP2} B.N. Allison, S. Berman and A. Pianzola, {\it
Covering algebras II: Isomorphism of loop algebras}, J. Reine Angew.
Math. {\bf 571} (2004) 39-71.

\bibitem[ABP2.5]{ABP2.5} B.N. Allison, S. Berman and A. Pianzola, {\it
Iterated loop algebras}, Pacific Jour. Math.  {\bf  227}  (2006), 1--41.

\bibitem[ABP3]{ABP3} B.N. Allison, S. Berman and A. Pianzola, {\it
Covering algebras III: Multiloop algebras, Iterated loop algebras and Extended Affine ie Algebras of nullity 2. }, Preprint. arXiv:1002.2674

\bibitem[ABFP]{ABFP} B. Allison, S. Berman, J. Faulkner and A.
Pianzola, {\it Realization of graded-simple algebras as loop
algebras}, Forum Mathematicum  {\bf 20}  (2008),   395--432.


\bibitem[BMR]{BMR} M. Bate, B. Martin and G. R\"ohrle, {\it A geometric
approach to complete reducibility}, Invent. math {\bf 161}, 177-218 (2005).


\bibitem[BL]{BL} C. Birkenhake and H.  Lange, {\it Complex abelian varieties},
Gr\"undlehren der mathematischen Wissenschaften {\bf   302} (2004), Springer.

\bibitem[BFM]{BFM} A.  Borel, R. Friedman and J.W.  Morgan, {\it
Almost commuting elements in compact Lie groups} Mem. Amer. Math.
Soc. {\bf 157}, (2002), no. 747.


\bibitem[Bo]{Bo} A. Borel, {\it Linear Algebraic Groups} (Second enlarged edition),
Graduate text in Mathematics {\bf 126} (1991), Springer.

\bibitem[BM]{BM} A.  Borel and G.D Mostow, {\it
On semi-simple automorphisms of Lie algebras}, Ann. Math. {\bf 61}
(1955), 389-405.

\bibitem[BoT1]{BoT}  A.  Borel and J.  Tits,
{\it Groupes r\'eductifs}, Pub. Math. IHES {\bf 27}, (1965), 
55--152.


\bibitem[BoT2]{BoT2}  A.  Borel and J.  Tits,
{\it Th\'eor\`emes de structure et de conjugaison
 pour les groupes alg\'ebriques lin\'eaires},
C. R. Acad. Sci. Paris S\'er. A-B {\bf 287} (1978),55-57.

\bibitem[Bbk]{Bbk} N. Bourbaki, {\it Groupes et alg\`ebres de Lie}, Ch. 7 et 8, Masson.




\bibitem[BT1]{BT1} F. Bruhat and  J. Tits, {\it Groupes alg\'ebriques sur un
corps local}, Pub. IHES {\bf 41} (1972), 5--251.



\bibitem[BT2]{BT2} F. Bruhat and J. Tits, {\it Groupes alg\'ebriques sur un
corps local \deux. Sch\'emas en groupes. Existence d'une donn\'ee radicielle valu\'ee},
Pub. IHES {\bf 60}  (1984), 197--376.


\bibitem[BT3]{BT3} F. Bruhat and J. Tits, {\it Groupes alg\'ebriques sur un
corps local \trois. Compl\'ements et application \`a la cohomologie
galoisienne}, J. Fac. Sci. Univ. Tokyo  {\bf 34}  (1987),   671--698.



\bibitem[Br]{Br}  E. S. Brussel, {\it The division algebras and Brauer group of a strictly Henselian field},  J. Algebra {\bf 239}  (2001)  391--411. 

\bibitem[CGP]{CGP} V. Chernousov, P. Gille and A. Pianzola, {\it
Torsors  on the punctured line},  to appear in American Journal of Math. (2011).


\bibitem[CGR1]{CGR1} V. Chernousov, P. Gille and Z. Reichstein, {\it
Resolution of torsors by abelian  extensions}, Journal of Algebra {\bf 296} (2006), 561--581.


\bibitem[CGR2]{CGR2} V. Chernousov, P. Gille and Z. Reichstein, {\it
Resolution of torsors by abelian  extensions II}, Manuscripta Mathematica
{\bf  126} (2008), 465-480.








 \bibitem[CTGP]{CTGP} J.--L. Colliot--Th\'el\`ene, P. Gille and R.  Parimala, {\it  
Arithmetic of linear algebraic groups over two-dimensional geometric fields},  Duke Math. J.  {\bf 121} (2004), 285-321.

\bibitem[CTS]{CTS1} J.--L. Colliot--Th\'el\`ene and J.--J. Sansuc, {\it Fibr\'es
quadratiques et composantes connexes r\'eelles}, Math. Annalen {\bf
244} (1979), 105--134.

\bibitem[CoGP]{CoGP} B. Conrad, O. Gabber and G. Prasad, {\it Pseudo-reductive groups}, Cambridge University Press, 2010.

\bibitem[D]{D} M. Demazure,  {\it Sch\'emas en groupes r\'eductifs},    Bull. Soc. Math. France {\bf  93} (1965), 369--413.

\bibitem[DG]{DG} M. Demazure and  P. Gabriel,
{\it Groupes alg\'ebriques}, Masson (1970).

\bibitem[EGA IV]{EGAIV} A. Grothendieck (avec la collaboration de J.
Dieudonn\'e), {\it El\'ements de G\'eom\'etrie Alg\'ebrique IV},
Publications math\'ematiques de l'I.H.\'E.S. no 20, 24, 28 and 32
(1964 - 1967).




\bibitem[F]{F} M. Florence, {\it Points rationnels sur les
espaces homog\`enes et leurs compactifications}, Transformation
Groups {\bf 11} (2006), 161-176.

\bibitem[FSS]{FSS} B. Fein, D. Saltman  and M. Schacher,
{\it Crossed products over rational function fields}, J. Algebra {\bf 156} (1993), 454-493.

\bibitem[GMS]{GMS} R.S. Garibaldi, A.A. Merkurjev
and J.-P. Serre, {\it Cohomological invariants in Galois
   cohomology},  University Lecture Series, 28 (2003).
American Mathematical Society, Providence.


\bibitem[Gi0]{Gi0} P. Gille, {\it Torseurs
sur la droite affine et $R$-\'equivalence},
Th\`ese de doctorat (1994), Universit\'e Paris-Sud, on author's URL.


\bibitem[Gi1]{Gi1} P. Gille, {\it Torseurs sur la droite affine}, Transform. Groups {\bf 7} (2002), 231-245, errata dans Transform. Groups {\bf 10} (2005), 267--269.

\bibitem[Gi2]{Gi2} P. Gille, {\it Unipotent subgroups of reductive groups of
characteristic $p>0$}, Duke Math. J. {\bf 114} (2002),  307--328.



\bibitem[Gi3]{Gi3} P. Gille, {\it Serre's conjecture II: a survey}, 
Proceedings of the Hyderabad conference 
 ``Quadratic forms, linear algebraic groups, and cohomology", Springer, 41--56.

\bibitem[Gi4]{Gi4} P. Gille, {\it Sur la classification des sch\'emas en groupes semi-simples},  see author's URL.

\bibitem[GM]{GM} G. van der Geer and B. Moonen, {\it Abelian varieties}, book in preparation,
second author's URL. 

\bibitem[GMB]{GMB} P. Gille, L. Moret-Bailly, Actions alg\'ebriques de groupes
arithm\'etiques, to  appear in the procedings of the conference ``Torsors,
theory and applications", Edinburgh (2011), Proceedings of the London
Mathematical Society, edited by  V. Batyrev and A. Skorobogatov.



\bibitem[GiQ]{GiQ} P. Gille and A. Qu\'eguiner-Mathieu, {\it 
Formules pour l'invariant de Rost},  to appear in Algebra and Number Theory.



\bibitem[GP1]{GP1} P. Gille and A. Pianzola, {\it Isotriviality of
torsors over Laurent polynomials rings},  C. R. Acad. Sci. Paris,
Ser. I {\bf 340} (2005) 725--729.

\bibitem[GP2]{GP2} P. Gille and A. Pianzola, {\it Galois cohomology and forms of algebras over Laurent polynomial rings},  Mathematische Annalen  {\bf 338} (2007), 497-543.

\bibitem[GP3]{GP3} P. Gille and A. Pianzola, {\it Isotriviality and
\'etale cohomology of  Laurent polynomial rings},   J. Pure Appl. Algebra  {\bf 212}  (2008), 780--800.

\bibitem[GR]{GR} P. Gille and Z. Reichstein, {\it A lower bound on the
essential dimension of a connected linear group},  Commentarii Mathematici Helvetici {\bf 84} (2009), 189-212.

\bibitem[GS]{GS} P. Gille and T. Szamuely, {\it Lectures on the
Merkurjev-Suslin's theorem}, Cambridge Studies in Advanced
Mathematics {\bf 101} (2006), Cambridge University Press.




\bibitem[Gi]{Gi} J. Giraud, {\it Cohomologie non ab\'elienne},  Die Grundlehren der
mathematischen Wissenschaften {\bf 179} (1971), Springer-Verlag.



\bibitem[Gs]{Gs} R. Griess,  {\it Elementary abelian p-subgroups of algebraic groups}, Geom. Ded. {\bf 39} (1991), 253--305. 

\bibitem[Gr1]{Gr1} A.  Grothendieck, {\it Technique de descente et th\'eor\`emes 
d'existence en g\'eom\'etrie alg\'ebrique. I. G\'en\'eralit\'es. Descente par morphismes fid\`element plats}, 
 S\'eminaire Bourbaki (1958-1960), Expos\'e No. 190, 29 p. 



\bibitem[Gr2]{Gr2} A.  Grothendieck, {\it Le groupe de Brauer. II},
 Th\'eorie cohomologique (1968),  Dix Expos\'es
 sur la Cohomologie des Sch\'emas  pp. 67--87, North-Holland.


\bibitem[HM]{HM}  G. Hochschild and D.  Mostow, {\it  Automorphisms of affine algebraic groups},

\bibitem [KS]{KS} V.G. Kac and A.V  Smilga, {\it Vacuum structure in supersymmetric Yang-Mills theories with any
gauge group.} The many Faces of the Superworld, pp. 185Ð234. World Sci. Publishing, River
Edge (2000)
  J. Algebra  {\bf 13}  (1969),  535--543.



\bibitem[K]{K2} M. A. Knus, {\it Quadratic and hermitian forms over
 rings}, Grundlehren der mat. Wissenschaften {\bf 294} (1991), Springer.


\bibitem[L]{L} E. Landvogt, {\it  Some functorial properties of the Bruhat-Tits building},   J. reine angew. math. {\bf 518}  (2000), 213--241.

\bibitem[Lam]{Lam}T.Y. Lam, Serre's problem on projective modules,
Springer Monographs in Mathematics, Springer, Berlin, Heidelberg, New
York, 2006.

\bibitem[Mg]{Mg} B. Margaux, {\it Passage to the limit in non-abelian \v Cech cohomology},
Journal of Lie Theory  {\bf 17}  (2007),   591--596.



\bibitem[Mg2]{Mg2} B. Margaux, {\it Vanishing of Hochschild Cohomology for Affine Group Schemes and Rigidity of Homomorphisms between Algebraic Groups},
Doc. Math. {\bf  14} (2009), 653-672.



\bibitem[Mt]{Mt} G. D. Mostow, {\it
Fully reducible subgroups of algebraic groups}, Amer. J. Math. {\bf
78} (1956), 200--221.

\bibitem[M]{M} J.S. Milne, { \it Etale cohomology}, Princeton University Press (1980). 

\bibitem[Mu]{Mu} D. Mumford, { \it  On the equations defining 
abelian varieties. I.},  Invent. Math. {\bf 1} (1966), 287--354.



\bibitem[NSW]{NSW} J. Neukirch, A. Schmidt and  K. Wingberg, {\it Cohomology of number fields, 
second edition}, Grundlehren des math. Wiss. 323 (2008), Springer.

\bibitem[N]{N} Y. A. Nisnevich, {\it Espaces homog\`enes principaux rationnellement triviaux et arithm\'etique
 des sch\'emas en groupes r\'eductifs sur les anneaux de Dedekind}, C. R. Acad. Sci. Paris S\'er. I Math. {\bf  299}
  (1984),   5--8.

\bibitem[Ne]{Ne} K.H. Neeb, {\it On the classification of rational quantum tori and the structure of their automorphism groups},  Canad. Math. Bull. {\bf  51}  (2008),   261--282. 



\bibitem[OV]{OV} A.L. Onishik and B.E. Vinberg, {\it Lie groups and algebraic groups},
Springer (1990).


\bibitem[PR]{PR} V. Platonov and A. Rapinchuk, {\it Algebraic Groups and Number Theory}, Academic Press (1993). 



\bibitem[P]{P} P. Pansu,
{\it  Superrigidit\'e g\'eom\'etrique et applications harmoniques},
 S\'eminaires et congr\`es {\bf  18} (2008),  Soc. Math. France, 375-422.

\bibitem[P1]{P1} A. Pianzola, {\it Affine Kac-Moody Lie algebras as
torsors over the punctured line}, Indagationes Mathematicae N.S.
{\bf 13}(2) (2002) 249-257.

\bibitem[P2]{P2} A. Pianzola, {\it Vanishing of $H^1$ for Dedekind rings and applications to loop
algebras}, C. R. Acad. Sci. Paris, Ser. I {\bf 340} (2005), 633-638.


\bibitem[P3]{P3} A. Pianzola, {\it On automorphisms of semisimple Lie algebras} Algebras, Groups, and Geometries {\bf 2} (1985) 95--116.

\bibitem[Pr]{Pr} G. Prasad, {\it  Galois-fixed points in the Bruhat-Tits building of a reductive group},   Bull. Soc. Math. France  {\bf 129}, (2001),   169--174.


\bibitem[Ra]{Ra}  A. Ramanathan, {\it Deformations of principal
    bundles on the projective line},  Invent. Math. {\bf 71}  (1983),
  165--191.

\bibitem[RZ]{RZ}
L. Ribes and  P. Zaleski, {\it  Profinite Groups},
 Ergebnisse der Mathematik und ihrer Grenzgebiete vol. 40, Springer (2000).

\bibitem[Ri]{Ri} R. W. Richardson, {\it  On orbits of algebraic groups and Lie
 groups},  Bull. Austral. Math. Soc. {\bf 25} (1982), 1--28.

\bibitem[Ro]{Ro} G. Rousseau, {\it Immeubles des groupes r\'eductifs sur les corps locaux},   Th\`ese, Universit\'e de Paris-Sud (1977).


\bibitem[RY1]{RY1} Z.  Reichstein and  B. Youssin {\it Essential Dimensions of
Algebraic Groups and a Resolution Theorem for G-varieties , with an
appendix by J. Koll\'ar and E. Szab\'o},   Canada Journal of Math.
{\bf 52} (2000), 265-304.



\bibitem[RY2]{RY2} Z.  Reichstein and B. Youssin {\it
A birational invariant for algebraic group actions},   Pacific J. Math.  {\bf 204} 
(2002),  223--246. 





\bibitem[Sa]{Sa} J.-J. Sansuc, {\it  Groupe de Brauer et arithm\'etique
 des groupes alg\'ebriques lin\'eaires sur un corps de nombres}, J.
reine angew. Math. (Crelle) {\bf 327} (1981), 12--80.



\bibitem[Sc]{Sc} W. Scharlau, {\it  Quadratic and hermitian forms}, 
Grundlehren der math. Wiss. {\bf 270} (1985), Springer.


\bibitem[Se1]{Se} J.-P. Serre, {\it  Cohomologie Galoisienne},
cinqui\`eme \'edition r\'evis\'ee et compl\'et\'ee, Lecture Notes in
Math.  {\bf 5}, Springer-Verlag.

\bibitem[Se2]{Se2} J.-P. Serre, {\it 
Cohomologie galoisienne: progr\`es et probl\`emes},
S\'eminaire Bourbaki Exp. No. 783 (1994),
Ast\'erisque {\bf  227} (1995), 229–257.


\bibitem[Sp]{Sp} T.-A. Springer, {\it Linear algebraic groups},
 second edition (1998),  Birkha\"user.


\bibitem[SZ]{SZ} A Steinmetz-Zikesch, {\it 
Alg\`ebres de Lie de dimension infinie et th\'eorie de la descente}, 
to appear in M\'emoire de la Soci\'et\'e Math\'ematique de France.



\bibitem[SGA1]{SGA1} {\it S\'eminaire de G\'eom\'etrie alg\'ebrique de
l'I.H.E.S.,  Rev\^etements \'etales et groupe fondamental, dirig\'e
par  A. Grothendieck},  Lecture Notes in Math. 224. Springer (1971).

\bibitem[SGA3]{SGA3} {\it S\'eminaire de G\'eom\'etrie alg\'ebrique de
l'I.H.E.S., 1963-1964, sch\'emas en groupes, dirig\'e par M.
Demazure et A. Grothendieck},  Lecture Notes in Math. 151-153.
Springer (1970).

\bibitem[SGA4]{SGA4} {\it
Th\'eorie des topos et cohomologie \'etale des sch\'emas. Tome 2},
S\'eminaire de G\'eom\'etrie Alg\'ebrique du Bois-Marie 1963--1964,
dirig\'e par M. Artin, A. Grothendieck et J. L. Verdier, Lecture
Notes in Mathematics {\bf 270} (1972), Springer-Verlag.




\bibitem[Ti]{Ti} J.-P. Tignol, {\it  Alg\`ebres ind\'ecomposables d'exposant premier},  Adv. in Math. {\bf  65} (1987),  205 228. 

\bibitem[TiW]{TiW} J.-P. Tignol and A. R. Wadsworth  {\it 
  Totally ramified valuations on finite-dimensional division algebras},   Trans. Amer. Math. Soc.  {\bf 302}  (1987),  223--250. 


\bibitem[T1]{T1} J. Tits, {\it Classification of algebraic semisimple groups}, Algebraic Groups
and Discontinuous Subgroups (Proc. Sympos. Pure Math., Boulder,
Colo., 1965), 33--62 Amer. Math. Soc. (1966).


\bibitem[T2]{T2} J. Tits, {\it Reductive groups over local fields},
 Proceedings of the Corvallis conference on
$L$- functions etc., Proc. Symp. Pure Math. {\bf 33} (1979), part 1, 29-69.


\bibitem[T3]{T3} J. Tits,
 {\it Twin buildings and groups of Kac-Moody type},
  Groups, combinatorics and  geometry (Durham, 1990),  249--286,
London Math. Soc. Lecture Note Ser. {\bf 165} (1992), Cambridge
Univ. Press.

\bibitem[V]{V} A. Vistoli, {\it Grothendieck topologies,
 fibered categories and descent theory},
  Fundamental algebraic geometry,  1--104, Math. Surveys Monogr.
{\bf  123} (2005), Amer. Math. Soc.


\bibitem[W]{W} C. Weibel, {\it An introduction to homological
algebra}, Cambridge studies in a advanced mathematics {\bf 38} (1994),
Cambridge University Press.



\end{thebibliography}
\end{document}